\documentclass[11pt]{article}
\usepackage[left=2.54cm,top=2.54cm,right=2.54cm,bottom=2.54cm]{geometry}

\input{macros/preamble.tex}

\usepackage{rotating}
\usepackage{xcolor}
\usepackage{mathrsfs}
\usepackage{tikz}
\usetikzlibrary{arrows.meta,positioning}

\hypersetup{colorlinks=true,linkcolor=blue,citecolor=blue,urlcolor=blue}
\hypersetup{pdftitle={Optimal Linear Functional Estimation for Proportionally High-Dimensional Linear Models},pdfauthor={Junwei Lu}}

\newif\ifshowrevision
\showrevisiontrue

\newcommand{\E}{\mathbb E}
\newcommand{\R}{\mathbb R}\newcommand{\C}{\mathbb C}
\newcommand{\ii}{\mathrm i}\renewcommand{\dd}{\,\mathrm d}
\renewcommand{\tr}{\operatorname{tr}}\newcommand{\spec}{\operatorname{spec}}
\newcommand{\dist}{\operatorname{dist}}\newcommand{\rank}{\operatorname{rank}}
\renewcommand{\Var}{\operatorname{Var}}\renewcommand{\Cov}{\operatorname{Cov}}
\renewcommand{\diag}{\operatorname{diag}}\newcommand{\supp}{\operatorname{supp}}\newcommand{\Ddisk}{\mathbb D}

\newcommand{\op}{\mathrm{op}}\newcommand{\TV}{\mathrm{TV}}
\renewcommand{\norm}[1]{\lVert#1\rVert}
\renewcommand{\ind}{\mathbf1}\newcommand{\Fell}{\mathbf F}\newcommand{\Kell}{\mathbf K}

\newcommand{\proofstep}[1]{\par\smallskip\noindent\textit{#1}\enspace}

\usepackage{xcolor}

\begin{document}
\title{\huge Optimal Linear Functional Estimation for Proportionally  High-Dimensional Linear Models}
\author{Junwei Lu}
\date{}
\maketitle

\begin{abstract}
            We study estimation of a linear functional in random-design
            linear models when $\gamma=p/n \in (1, \infty)$ under the dense regime where the design covariance is unknown, and neither the regression
            vector nor the precision matrix is assumed sparse. Suppose the spectrum of the design covariance lies in $[\ell,L]$, and let $\kappa=L/\ell$. We establish that
            the minimax mean squared error for the linear functional is of order $n^{-\alpha(\kappa,\gamma)}$, where the exponent
            $\alpha(\kappa,\gamma)$ has an explicit elliptic-integral
            representation and $\alpha(\kappa,\gamma) < 1$ when both $\kappa > 1$ and $\gamma > 1$. We propose the Bocca estimator, which combines
            empirically calibrated resolvents through a complex contour integral, to attain the optimal rate. We also establish asymptotic normality of the Bocca estimator and develop a variance estimator for valid inference.
\end{abstract}


\section{Introduction}\label{sec:introduction}

Many regression problems concern a single coefficient or a prespecified
contrast rather than the entire regression vector. Examples include an
exposure coefficient after adjustment for many covariates and a contrast
among effects in a genomic study. Such targets can remain scientifically
meaningful when there is no convincing reason to assume that only a few
covariates affect the response. This distinction motivates inference with
dense regression coefficients
\citep{ZhuBradic2018a,BradicFanZhu2022,SankaranarayananEtAl2025}.
We study the linear model $Y=X\beta+\varepsilon$, with unknown design covariance $\Sigma = \Cov(X_i)$ and unknown noise variance $\sigma^2 = \Var(\varepsilon_i)$. The parameter of interest is the scalar functional
$\beta_v=\langle v,\beta\rangle$, where $v$ is a known deterministic unit
vector. There are two important types of $v$: (1)  $v = e_j$ for some $j \in \{1, \ldots, p\}$ to estimate a single coefficient $\beta_j$; (2)  $v = x_0$ for some deterministic unit test point $x_0$ to estimate the prediction $x_0^\top\beta$. Our focus is the dense proportional regime in which $\gamma=p/n \in (1, \infty)$ and both the coefficient vector $\beta$ and the precision matrix $\Sigma^{-1}$ may be dense.

In this regime, recovering a low-dimensional functional and recovering
the full vector are fundamentally different problems. Even under
isotropic Gaussian design, the minimax $\ell_2$-norm rate over a
fixed-radius ball need not vanish under proportional asymptotics
\citep{Dicker2016}. Nevertheless, a fixed coordinate can be estimated
consistently. The main difficulty in this regime is that  the sample covariance is singular when $p>n$, and estimating the
inverse-covariance direction associated with the target introduces an
additional high-dimensional nuisance parameter. Therefore, the main question this paper addresses is what the minimax rate of estimation of a fixed linear functional of $\beta$ is under this dense proportional regime.

When the design covariance is known, \citet{BradicFanZhu2022} estimate the coefficients of the linear model consistently via the method of moments; when the design covariance is unknown, they assume a sparse precision matrix. Debiasing methods are another important class of approaches.
\citet{BellecZhang2022} propose the degrees-of-freedom adjustment method which starts from a regularized regression estimator
and corrects its bias using a precise analysis of proportional asymptotics.
\citet{BellecZhang2023} extend debiasing to general convex penalties and
establish asymptotic normality for most normalized directions.
\citet{CelentanoWainwright2023} further develop debiasing for a population
mean under missingness and confidence intervals for regression
coefficients when both outcome and propensity estimators can be
inconsistent. From a Bayesian perspective, \citet{SaenzSur2025}
characterize fixed-dimensional posterior marginals under isotropic
Gaussian designs and strongly log-concave product priors. 

When the design covariance is unknown, 
the debiasing methods exploiting the sparsity of both the coefficient vector and the precision matrix are widely studied. The debiased Lasso methods 
\citep{ZhangZhang2014,vandeGeer2014} remove regularization bias through an
approximately orthogonal score. \citet{JavanmardMontanari2014} construct
an approximate inverse through constrained optimization. 
\citet{Buhlmann2013} considers correcting the projection bias of ridge regression.
These approaches impose different restrictions on regression sparsity,
precision-direction sparsity, and the design. In particular, the method
of \citet{JavanmardMontanari2014} and \citet{JavanmardMontanari2014Testing} does not require a sparse precision
matrix, although it requires a sufficiently sparse regression vector. The analysis is later sharpened
 for Gaussian designs  \citep{JavanmardMontanari2018}.
Similar approaches include post-double-selection inference
\citep{BelloniChernozhukovHansen2014}, decorrelated score tests
\citep{NingLiu2017}, and double machine learning with orthogonal
scores and cross-fitting \citep{ChernozhukovEtAl2018}.
Their $\sqrt{n}$-consistency results usually require suitable control of the nuisance
estimation errors, often through a product remainder. Orthogonality
reduces sensitivity to nuisance errors, but does not by itself make this
remainder negligible when both nuisance estimators have nonvanishing
errors under proportional asymptotics. \citet{ZhuBradic2018a,ZhuBradic2018b,BradicFanZhu2022} relax the sparsity assumption and develop tests that remain valid when $\beta$ is dense, if the precision direction associated with the loading vector $v$ is sparse.
\citet{CaiGuo2017} establish corresponding minimax confidence-interval
lengths and adaptation limits over sparse parameter spaces.

In the dense regime when both $\beta$ and $\Sigma^{-1}$ are dense, the source of covariance
information and the inferential target become decisive.
\citet{CelentanoMontanari2024} propose correlation adjusted debiasing,
which corrects bias caused by correlated errors in the outcome and
covariate regressions. Their analysis permits inaccurate nuisance
estimation, including overlapping samples. Other nonsparse linear-functional
methods exploit additional spectral structure or informative
initial estimators \citep{ZhaoZhouLiu2024}. \citet{LiSur2026} develop spectrum-aware debiasing for
right-rotationally invariant designs and establish a spectral universality
extension.  
The work of \citet{ChenLiuMukherjee2024} is the closest to the setting of our paper. They use higher-order moments to approximate the inverse covariance and estimate the resulting moments with higher-order $U$-statistics. Their construction
gives coordinate consistency when $p\ge n$ without requiring a sparse
precision matrix, but does not determine the optimal statistical rate. Related moment-based approaches address linear regression with affine constraints \citep{SankaranarayananEtAl2025} and generalized linear models \citep{SawayaUematsuImaizumi2026}.

When $\gamma$ converges to a constant in $(0,1)$,
\citet{DonohoMontanari2016} use approximate message passing (AMP) to
characterize the asymptotic variance of
$M$-estimators under Gaussian designs, implying a $\sqrt{n}$-estimation rate for $\beta_v$  when $\gamma < 1$.
For $p>n$ and assuming $\Sigma = I_p$, \citet{DonohoMalekiMontanari2011} analyze
the mean-squared error (MSE) of the Lasso estimator over sparse signals and \citet{BayatiMontanari2012} derive the MSE without sparsity.
Related AMP literature includes algorithms for compressed sensing, state
evolution, universality, and finite-sample distributional theory
\citep{DonohoMalekiMontanari2009,BayatiMontanari2011,BayatiLelargeMontanari2015,LiWei2025}.
See \citet{FengEtAl2022} for a tutorial and further references. However, these results do not establish the estimation rate
for $\beta_v$ under unknown general covariance in the dense regime.

There is also a wide range of work studying the estimation of scalar functionals of linear models other than linear functionals.
 \citet{Dicker2014} develops moment estimators of variance
components, and \citet{VerzelenGassiat2018} study adaptive estimation of
signal-to-noise ratios with known or unknown covariance.
\citet{KongValiant2018} show how polynomial approximation and moment
estimation can recover prediction-related quantities using fewer
observations than features. The dependence on covariance information in
these problems motivates estimating a functional directly rather than
first reconstructing the regression vector. Their targets and lower-bound
regimes, however, differ from the fixed-coordinate proportional problem.
For another functional, \citet{McGrathEtAl2025} obtain $\sqrt{n}$-consistency
of an expected conditional covariance under isotropic independent-entry
designs, despite inconsistent ridge nuisance estimators. 

Despite the rich literature discussed above, the optimal statistical rate of estimating a fixed linear functional in proportionally high-dimensional linear models remains open when both the regression vector and the precision matrix are dense \citep{MalekiEtAl2026}. 
This paper determines the minimax rate for this problem and characterizes the optimal estimator via the proportion $\gamma$ and the upper bound $\kappa$ on the design covariance matrix's condition number.
Specifically, let $\ell$ and $L$ be fixed positive lower and upper
 bounds of the covariance spectrum, and write $\kappa=L/\ell\ge 1$.
We assume that the design has the representation $X=Z\Sigma^{1/2}$, where the entries of
$Z$ are independent, identically distributed  with standardized variance and bounded moments. The coefficient vector $\beta$ has a bounded $\ell_2$-norm, and the independent noise has bounded moments and an unknown variance.
Let $\mathcal P(n,p,\gamma,\kappa)$ denote this model family, specified in
Section~\ref{sec:setup}. We prove
\begin{equation}\label{eq:intro-rate}
 \inf_{\widetilde\beta_v}
 \sup_{P\in\mathcal P(n,p,\gamma,\kappa)}
 \E_P(\widetilde\beta_v-\beta_v)^2
 \asymp n^{-\alpha(\kappa,\gamma)}.
\end{equation}
The exponent $\alpha(\kappa,\gamma)\in(0,1]$ has a closed-form expression \eqref{eq:alpha-exponent} in terms of
elliptic integrals. Specifically, we have $\alpha(\kappa,\gamma) \in (0,1)$ when both $\kappa > 1$ and $\gamma > 1$. Besides, $\alpha(\kappa,\gamma)$ converges to one as $\kappa$ or $\gamma$ converges to one. This aligns with the $\sqrt{n}$-rate when the design covariance is diagonal \citep{BradicFanZhu2022}. However, the $\sqrt{n}$-rate is not achievable under the general dense proportionally high-dimensional linear models, especially when $\kappa > 1$.
When both $\kappa$ and $\gamma$ diverge, we show that
$
 \alpha(\kappa,\gamma)
 \sim {4}/(\sqrt\kappa\,\log(16\gamma))
$; see Remark~\ref{rem:exponent-asymptotics} for details. 
This asymptotic formula further shows that the dimension-sample ratio $\gamma$ only affects the rate logarithmically, while the upper bound $\kappa$ on the covariance condition number affects the rate exponent on the order of $1/\sqrt{\kappa}$. Figure~\ref{fig:alpha} displays $\alpha(\kappa,\gamma)$ and its asymptotics.

\begin{figure}[t]
 \centering
 \setlength{\tabcolsep}{2pt}
 \begin{tabular}{ccc}
   \!\!$\alpha(\kappa,\gamma)$ &
   $1/\alpha(\kappa,\gamma)$ vs. $\gamma$ &
   $1/\alpha(\kappa,\gamma)$ vs. $\kappa$ \\
  \includegraphics[height=0.27\textwidth]{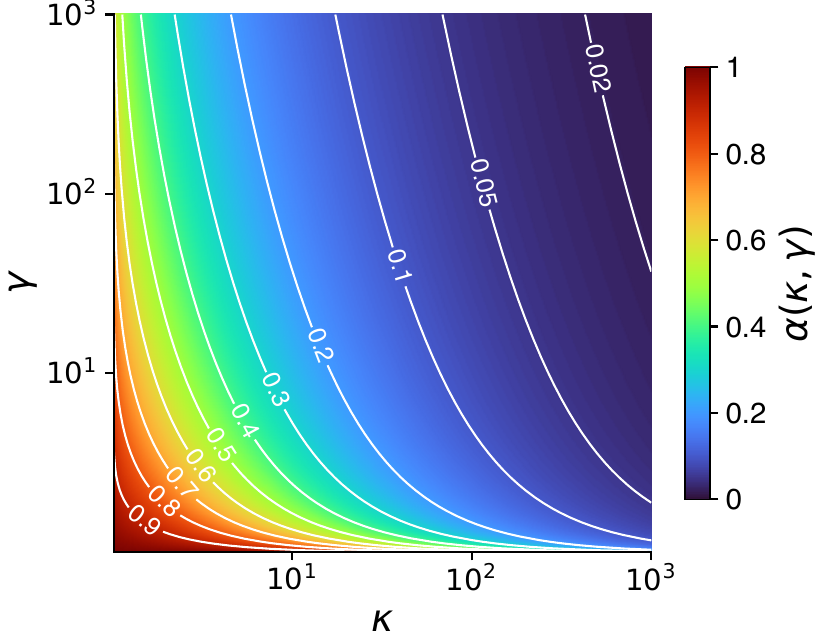} &
  \includegraphics[height=0.27\textwidth]{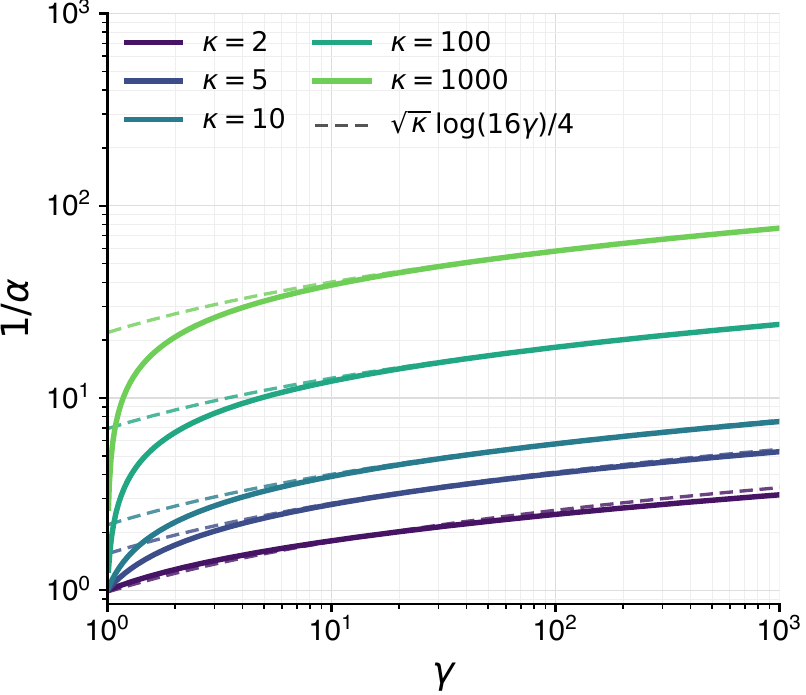} &
  \includegraphics[height=0.27\textwidth]{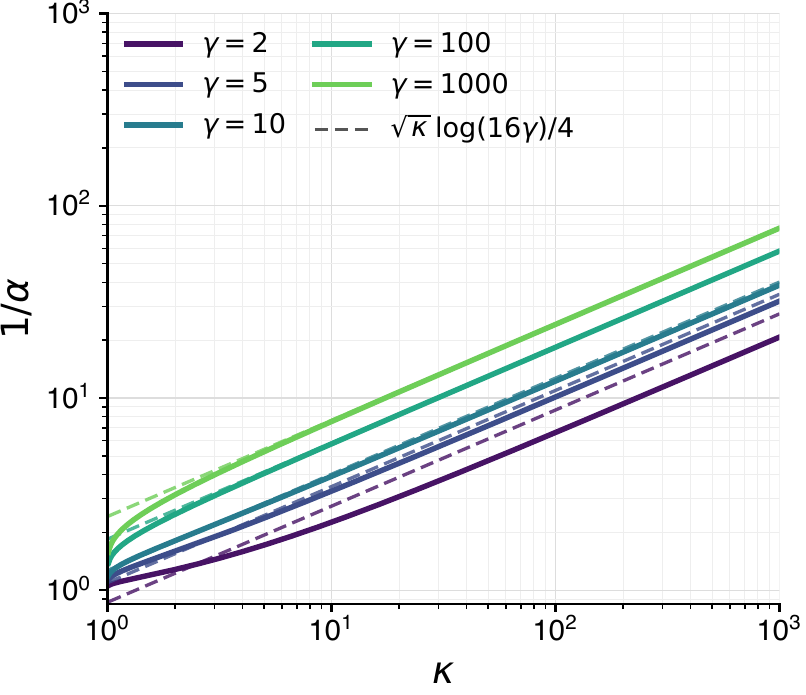}
 \end{tabular}
 \vspace{-15pt}
 \caption{The minimax exponent $\alpha(\kappa,\gamma)$ of \eqref{eq:intro-rate} with the closed-form expression \eqref{eq:alpha-exponent} and its asymptotics.}
 \label{fig:alpha}
 \vspace{-10pt}
\end{figure}

We propose the Bilinear Operator Calibration via Contour Aggregation (\emph{Bocca}) estimator, which combines empirically calibrated resolvents through a complex contour integral.
The
calibration is motivated by the companion Stieltjes equation \citep{SilversteinChoi1995} and
the anisotropic local law \citep{KnowlesYin2017} for
sample covariance matrices.
A contour integral then combines the estimated resolvent field with an
analytic filter. We show that the rate of the Bocca estimator is characterized by the bias terms determined by the approximation error of the filter on
the covariance spectrum, while its growth on the observation contour
controls the variance term. We identify a hyperbolic cosine of a scaled, normalized elliptic integral as the filter which balances the bias-variance trade-off and  attains the optimal rate.
Moreover, we establish asymptotic normality of the Bocca estimator and develop a variance estimator for valid inference.

\noindent {\bf Paper organization.} Section~\ref{sec:setup} specifies the statistical model and problem setup. Section~\ref{sec:estimator} presents the Bocca estimator and its rate. Section~\ref{sec:bocca-normality} establishes asymptotic normality of the Bocca estimator. Section~\ref{sec:lower} establishes the minimax lower bound. Section~\ref{sec:graph-debias} discusses the application of the Bocca estimator to debiasing methods. Section~\ref{sec:simulation} provides numerical simulations to validate the theoretical results.

\noindent {\bf Notation.}
For a positive integer $d$, write $[d]=\{1,\ldots,d\}$; $I_d$ and $e_j$ denote the $d$-dimensional identity matrix and the $j$th coordinate vector, respectively. We write $\lfloor t\rfloor=\max\{k\in\mathbb Z:k\le t\}$.
We use $\|x\|_2$ for the Euclidean norm and $x_{-j}$ for the vector obtained by deleting coordinate $j$.
For a set $\mathcal I$, $|\mathcal I|$ is its cardinality.
For a real symmetric matrix $A$, $\lambda_j(A)$ denotes its $j$th smallest eigenvalue, and $\lambda_{\min}(A)$ and $\lambda_{\max}(A)$ its minimum and maximum eigenvalues; $A\preceq B$ means that $B-A$ is positive semidefinite; $\tr A$ is the trace of $A$, and $\diag(a_1,\ldots,a_d)$ is the diagonal matrix with the indicated entries.
We write $\ii=\sqrt{-1}$, and $\operatorname{Re}z$, $\operatorname{Im}z$, $\bar z$, and $\arg z$ denote the real part, imaginary part, complex conjugate, and argument of $z$, respectively. 
We use $\delta_t$ for the unit point mass at $t$ and $\operatorname{TV}(P,Q)=\sup_A|P(A)-Q(A)|$ for total variation distance, with the supremum over measurable sets.
For nonnegative quantities, $a\lesssim b$ (equivalently, $b\gtrsim a$) means $a\le Cb$, with $C$ depending only on the fixed model constants, and $a\asymp b$ means both $a\lesssim b$ and $b\lesssim a$.
For positive sequences, $a_n\sim b_n$ means $a_n/b_n\to1$; $O(b_n)$ and $o(b_n)$ denote a bounded ratio to $b_n$ and a ratio tending to zero, respectively.
The notation $W_n=O_P(b_n)$ means that $W_n/b_n$ is bounded in probability, while $W_n=o_P(b_n)$ means that this ratio converges to zero in probability; $\overset{p}{\longrightarrow}$ and $\overset{d}{\longrightarrow}$ denote convergence in probability and in distribution, respectively.

\section{Problem Setup}\label{sec:setup}
Fix $0 < \ell < L$, $B_0,\bar\sigma>0$, and set $\kappa=L/\ell$. For integers $p>n\ge1$ and the compatible ratio $\gamma=p/n$, let $\mathcal P(n,p,\gamma,\kappa)$ be the family of joint laws of $(X,Y)\in\R^{n\times p}\times\R^n$ generated by
\begin{equation}\label{eq:model-family}
 X=Z\Sigma^{1/2}, \quad Y=X\beta+\varepsilon, \quad
 \ell I_p\preceq\Sigma\preceq LI_p,
 \quad \norm\beta_2\le B_0,
 \text{ and } 0 < \sigma\le\bar\sigma.
\end{equation}
Here $\Sigma$ and $\beta$ are deterministic parameters; the entries of $Z$ are independent with an identical distribution satisfying $\E Z_{11}=0$, $\E Z_{11}^2=1$, and $\E|Z_{11}|^r\le M_r$ for some $M_r < \infty$ for every integer $r\ge3$.
Similarly, the noise vector $\varepsilon$ is independent of $Z$ and has iid real entries satisfying
$
 \E\varepsilon_i=0,\E\varepsilon_i^2=\sigma^2,
 \E|\varepsilon_i|^r\le M_r$, for every integer $r\ge3$.
The identical distribution assumption on the entries of $Z$ is just for simplicity. Only the independence and moment conditions are essentially needed. The constants $L,B_0,\bar\sigma, M_r$ are held fixed and are suppressed from the family notation. The covariance matrix $\Sigma$, regression coefficients $\beta$, and noise variance $\sigma^2$ are all unknown.
Define $\mathcal G(n,p,\gamma,\kappa)$ as the subclass in which $Z_{11}\sim N(0,1)$ and $\varepsilon_i \sim N(0,\sigma^2)$. We take $M_r$ sufficiently large for every integer $r\ge3$ so that $\mathcal G(n,p,\gamma,\kappa)\subseteq\mathcal P(n,p,\gamma,\kappa)$. We will show the minimax lower bound is attained in the Gaussian design family $\mathcal G(n,p,\gamma,\kappa)$.
All rate statements in this paper concern sequences of integers $p=p(n)$ for which
\begin{equation}\label{eq:aspect-regime}
 n\rightarrow\infty \text{ and }
 1<\liminf_{n\to\infty}\frac{p(n)}n
 \le\limsup_{n\to\infty}\frac{p(n)}n<\infty.
\end{equation}
The notation $\gamma$ always denotes the current ratio $p/n$ and no convergence of $p/n$ to a particular number is required. This makes it more flexible to study the rate under non-asymptotic settings.
The number $\kappa=L/\ell$ specifies the eigenvalue bounds of the family; it need not be the condition number of each covariance matrix. As reviewed in Section~\ref{sec:introduction}, the isotropic covariance matrix case has been well studied. In the rest of the paper, we focus on the general dense proportionally high-dimensional linear models when $\kappa > 1$.

Our target is $\beta_v=\langle v,\beta\rangle$ for a known deterministic vector $v\in\R^p$ with $\norm v_2=1$; a nonunit loading is handled by rescaling. Write $\widehat\Sigma=X^\top X/n$ and $\widehat S=XX^\top/n$.

\subsection{Elliptic-Integral Notation}\label{sec:elliptic-notation}
The incomplete elliptic integral of the first kind is
\begin{equation}\label{eq:Fell}
 \Fell(\phi\mid m)=\int_0^\phi(1-m\sin^2 t)^{-1/2}\dd t,
 \text{ for some }0<m<1.
\end{equation}
Its complete value is $\Kell(m)=\Fell(\pi/2|m)$.  Also $\cosh z=(e^z+e^{-z})/2$ and $\sinh z=(e^z-e^{-z})/2$.
For the exponent, define
\begin{equation}\label{eq:alpha-exponent}
 \alpha(\kappa,\gamma)
 =1-\frac{\Fell(\arcsin(\xi/\rho)\mid\rho^2)}{\Kell(\rho^2)}, \text{ where $\xi=\frac{\sqrt\gamma-1}{\sqrt\gamma+1}$ and $\rho=\frac{\sqrt\gamma-1+\kappa^{-1}}{\sqrt\gamma+1-\kappa^{-1}}$}.
\end{equation}
Here $0<\xi<\rho<1$, so $0<\alpha(\kappa,\gamma)<1$. 

\section{The Bocca Estimator}\label{sec:estimator}

For the eigendecomposition $\Sigma=\sum_{j=1}^p\lambda_jo_jo_j^\top$, define
$
 \nu=\sum_{j=1}^p(v^\top o_j)(o_j^\top\beta)\delta_{\lambda_j}.
$
Then the parameter of interest $\beta_v=\int\nu(\dd\lambda)$ and 
\begin{equation}\label{eq:population-field}
      F(\mu)=v^\top\Sigma(\Sigma-\mu I_p)^{-1}\beta
            =\int_{[\ell,L]}\frac{\lambda}{\lambda-\mu}\nu(\dd\lambda).
     \end{equation}
     Given some holomorphic function $g$ with $g(0)=1$ and some contour $\Gamma$ enclosing $[\ell,L]\cup \{0\}$, we consider the following filter 
     \begin{align}\label{eq:bocca-bias}
           -\frac1{2\pi\ii}\oint_{\Gamma}\frac{1-g(z)}zF(z)\dd z 
           &=\int\left[-\frac1{2\pi\ii}\oint_{\Gamma}
            \frac{1-g(z)}z\frac{\lambda}{\lambda-z}\dd z\right]\nu(\dd\lambda)=\int_{[\ell,L]}(1-g(\lambda))\nu(\dd\lambda),
     \end{align}
    where we apply the Cauchy's integral formula \citep[Chapter~4, Eq.~(22)]{Ahlfors1979} in the second equality.
     As our parameter of interest is $\beta_v = \int1\,\nu(\dd\lambda)$, to approximate it using the right hand side of the above equation, we need to find a $g$ that is small enough on the slit $[\ell,L]$, while $g(0)=1$ removes the extraneous pole at zero. This is the key insight of our Bocca estimator to be proposed formally later.

We first construct an estimator for $F(\mu)$.
Let $m: \CC^+ \to \CC^+$, where $\CC^+=\{z\in\CC:\operatorname{Im}z>0\}$ is the upper half-plane, be the unique analytic solution of
\begin{equation}\label{eq:MP}
 z=-\frac1{m(z)}+\frac1n\tr\bigl(\Sigma(I_p+m(z)\Sigma)^{-1}\bigr) = -\frac1{m(z)}+\gamma \cdot \frac1p \sum_{j=1}^p \frac{\lambda_j}{1+\lambda_j m(z)}.
\end{equation}
This is essentially the equation in \citet[Eqs.~(1.3)--(1.4)]{SilversteinChoi1995} for the
companion Stieltjes transform of the generalized
Mar\v{c}enko--Pastur law, where they also established the existence and uniqueness of the solution in $\C^+$; see also \citet[Theorem~4.3]{BaiSilverstein2010} and \citet[Lemma~2.2]{KnowlesYin2017}.
In particular, if we additionally assume that $p/n\to\gamma>1$ and
the empirical spectral distribution of $\Sigma$
converges weakly to a probability measure $\mathsf H$,
then, by \citet[Theorem~1.1]{Silverstein1995},
the empirical spectral distribution of
$\widehat S$ converges almost surely
to a deterministic probability measure $\varrho_{\mathsf H}$,
whose Stieltjes transform satisfies
\[
 m_{\mathsf H}(z)
 =\int_{\R}\frac1{x-z}\,\varrho_{\mathsf H}(\dd x)
 \text{ and }
 z=-\frac1{m_{\mathsf H}(z)}
   +\gamma\int_{\R}
     \frac{\lambda}{1+\lambda m_{\mathsf H}(z)}
     \,\mathsf H(\dd\lambda).
\]
Thus $m$ in \eqref{eq:MP} is a non-asymptotic
deterministic approximation to $\widehat m(z)=n^{-1}\tr(\widehat S-zI_n)^{-1}$, the Stieltjes transform of the empirical eigenvalue distribution of $\hat S$, and converges
to $m_{\mathsf H}$ under the above additional asymptotic assumptions. However, as we want to keep the non-asymptotic nature without imposing asymptotic assumptions on $p/n$ and the spectrum of $\Sigma$, we use the non-asymptotic version $m$ in \eqref{eq:MP} instead of $m_{\mathsf H}$.

Equation \eqref{eq:MP} also gives the inverse parametrization
\begin{equation}\label{eq:zeta}
 \zeta(\mu)=\mu-\frac\mu n\tr\bigl(\Sigma(\Sigma-\mu I_p)^{-1}\bigr).
\end{equation}
The following proposition, whose proof is given in Appendix~\ref{app:pf:prop:calibration}, shows when we can invert $m$ using $\zeta(\mu)$ on a certain region of the complex plane.

\begin{proposition}\label{prop:calibration}
For every $\mu\in\C$ with $|\mu-L|>\sqrt\gamma L$, the function $m$ is analytic at $\zeta(\mu)$ and satisfies $m(\zeta(\mu))=-1/\mu$.
\end{proposition}

For a deterministic margin $\delta>0$, set
\begin{equation*}
      R_\delta=\sqrt\gamma(1+\delta)L, \,
       \Gamma_\delta=\{z\in\C:|z-L|=R_\delta\}, \text{ and } \Ddisk_\delta=\{z\in\C:|z-L|<R_\delta\}.
     \end{equation*}
The circle encloses $[\ell,L]$ and zero, and Proposition~\ref{prop:calibration} ensures $m(\zeta(\mu))=-1/\mu$ for all $\mu$ on the contour $\Gamma_{\delta}$, where we can define an estimator for $F(\mu)$. We denote the closure of $\Ddisk_\delta$ by $\overline{\Ddisk_\delta}$. See Figure~\ref{fig:potential}(a) for the illustration of the slit $[\ell, L]$,  $R_\delta$, $\Gamma_\delta$, and $\mathbb{D}_\delta$.

We construct an estimator for $F(\mu)$ by using the anisotropic local law for
sample covariance matrices 
\citep[Eq.~(1.5)]{KnowlesYin2017}.  They showed that
\[
 \frac1nX^\top(\widehat S-zI_n)^{-1}X
 \approx m(z)\Sigma(I_p+m(z)\Sigma)^{-1},
\]
where $\approx$ means local closeness for each matrix entry or matrix bilinear form.
Therefore, as $m(\zeta(\mu))=-1/\mu$, we have $F(\mu) \approx n^{-1} (Xv)^\top(\widehat S-\zeta(\mu)I_n)^{-1}Y$. Next is to estimate $\zeta(\mu)$.
Recall $\widehat m(z)=n^{-1}\tr(\widehat S-zI_n)^{-1}$. We next construct an estimator for $\zeta(\mu)$ by inverting $\widehat m(z)$. Whenever it exists, define $\widehat\zeta:\Gamma_\delta\to\C$ as the
continuous function satisfying 
\begin{equation}\label{eq:empirical-calibration}
      \widehat{m}(\widehat{\zeta}(\mu))
 =-1/\mu, \widehat\zeta(\mu)\notin [\lambda_{\min}(\widehat S), \lambda_{\max}(\widehat S)], \widehat\zeta(L-R_\delta)<\lambda_{\min}(\widehat S)
 \text{ and } 
 \widehat m'(\widehat\zeta(\mu))\ne0, 
\end{equation}
for all $\mu\in\Gamma_\delta$.
The derivative condition means that each selected root is simple. We set $\widehat{\zeta}(\mu)=0$ if no such function exists.
The following proposition shows that such a function exists with high probability. Its proof is deferred to Appendix~\ref{app:pf:prop:empirical-inverse}.

\begin{proposition}\label{prop:empirical-inverse}
Let $1/\log n\le\delta\le2/\log n$. For every fixed $D>0$ and all sufficiently large $n$, with probability at least $1-Cn^{-D}$ there exists a unique continuous function $\widehat\zeta:\Gamma_\delta\to\C$ satisfying \eqref{eq:empirical-calibration} for all $\mu\in\Gamma_\delta$. 
\end{proposition}

We discuss the concrete method to compute $\widehat\zeta(\mu)$ in Section~\ref{sec:computation}.
As the above proposition shows that $\widehat\zeta(\mu)$ is well-defined with high probability, we can define the estimator for $F(\mu)$ as
\begin{equation}\label{eq:empirical-field}
      \widehat F(\mu)=\frac1n v^\top X^\top(\widehat S-\widehat \zeta(\mu)I_n)^{-1}Y.
\end{equation}
If no such function exists, set $\widehat F$ identically equal to zero.

For every deterministic function $g$, we define the Bilinear Operator Calibration via Contour Aggregation (\emph{Bocca}) estimator as
\begin{equation}\label{eq:estimator}
 \widehat\beta_v(g)= \operatorname{Re}
 \left(-\frac1{2\pi\ii}\oint_{\Gamma_\delta}
 \frac{1-g(z)}z\widehat F(z)\,\dd z\right).
\end{equation}
See Figure~\ref{fig:potential}(a) for the illustration of the contour integration. 

\begin{remark}
 Comparing \eqref{eq:empirical-field} with the ridge regression estimator 
 \[
 \hat\beta_{\mathrm{ridge}} = \argmin_{\beta} \left\{ \frac1n \norm{Y-X\beta}_2^2 + \lambda \norm{\beta}_2^2 \right\} 
=\frac1nX^\top
\left(\widehat S+\lambda I_n\right)^{-1}Y,
\]
we can take $\widehat F(\mu)$ as a ridge estimator with a complex tuning parameter $\lambda = - \widehat \zeta(\mu)$. And the Bocca estimator \eqref{eq:estimator} is a weighted aggregation of ridge estimators with different tuning parameters.
\end{remark}

\subsection{Bias--Variance Trade-Off}\label{sec:bias-var}

Now we discuss how the function $g$ determines the bias-variance trade-off of the Bocca estimator and how to construct an ideal $g$ to achieve the optimal rate. To construct such a $g$, 
      we start with the harmonic potential $u_\delta$ satisfying
\begin{equation}\label{eq:potential}
 \Delta u_\delta=0\quad\text{in }\Ddisk_\delta\setminus[\ell,L],
 \qquad u_\delta=0\quad\text{on }[\ell,L],
 \qquad u_\delta=1\quad\text{on }\Gamma_\delta.
\end{equation}
Define $\alpha_\delta=u_\delta(0)$.  We can prove that (see Lemma~\ref{lem:geometry} in the Appendix)
\begin{equation}\label{eq:alpha-delta}
 \alpha_\delta=u_\delta(0)
 =\alpha\bigl(\kappa,\gamma(1+\delta)^2\bigr)
 =\alpha(\kappa,\gamma)+O(\delta).
\end{equation}
Therefore, the key insight is that the optimal exponent $\alpha(\kappa,\gamma)$ for the estimation rate of $\beta_v$ is determined by $u_0(0)$, the harmonic potential at zero. Figure~\ref{fig:potential}(b) illustrates how the potential changes with $\kappa$ and $\gamma$.

\begin{figure}[ht]
      \centering
      \begin{tabular}{cc}
       {(a)} Contour geometry & {(b)} Harmonic potential $u_0$ \\
       \includegraphics[height=0.36\textwidth]{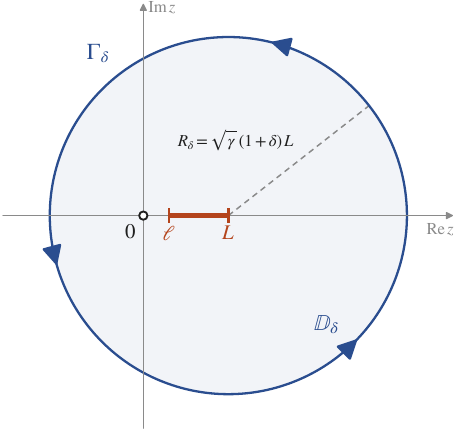} &
       \includegraphics[height=0.36\textwidth]{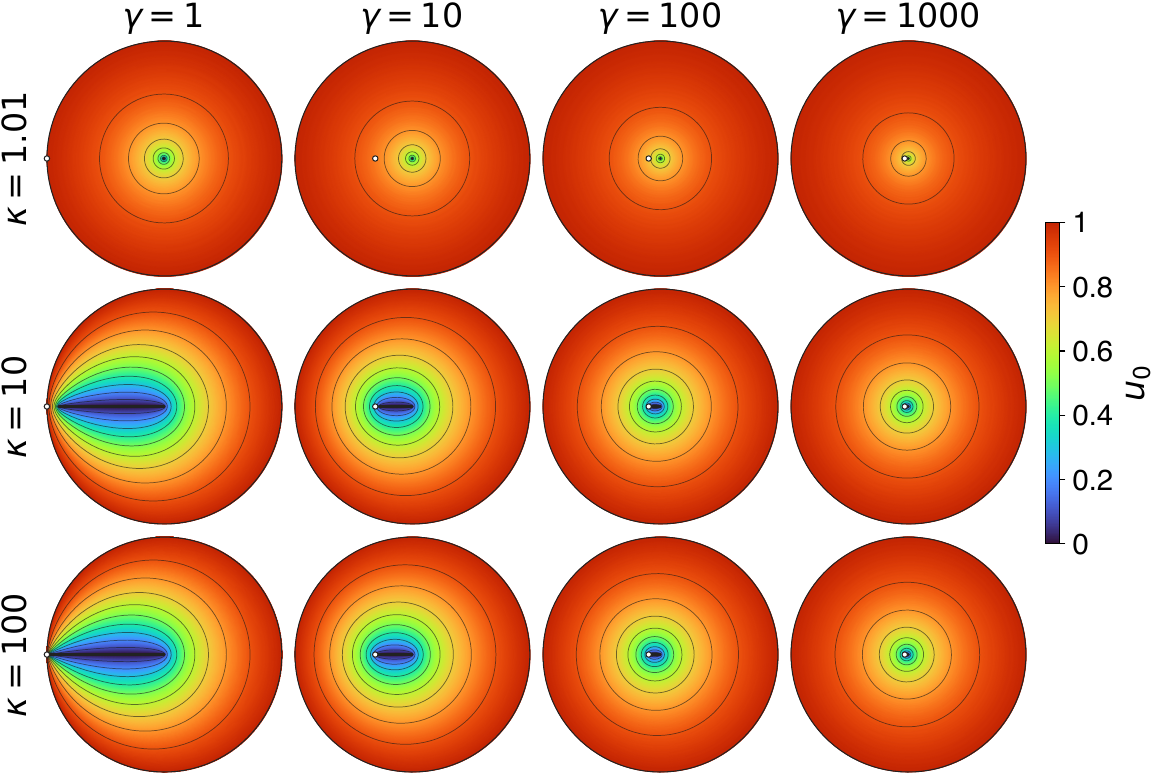}
      \end{tabular}
      \vspace{-10pt}
      \caption{(a) The slit disk $\Ddisk_\delta\setminus[\ell,L]$ and the contour $\Gamma_\delta$ used in the Bocca estimator \eqref{eq:estimator}. (b) The harmonic potential $u_0$ of \eqref{eq:potential} with $\delta=0$ and $L=1$. The white dot marks the origin $z=0$.}
      \label{fig:potential}
      \vspace{-10pt}
     \end{figure}

We next derive a closed-form expression for $u_{\delta}$.
Denote the slit disk by $\mathcal A_\delta=\Ddisk_\delta\setminus[\ell,L]$ and set
\begin{equation}\label{eq:k-T-h}
      h_\delta=\frac{\pi\Kell(k^2)}{\Kell(1-k^2)}, \quad k=\frac{R_\delta-L+\ell}{R_\delta+L-\ell},  \quad \text{ and }
      T_\delta(z)=\frac{R_\delta-L+z}{R_\delta+L-z}.  
\end{equation}
For some period $P$, let $\C / (P \ZZ)$ denote the quotient of the complex plane by the equivalence relation $x \sim y$ if $x-y \in P \ZZ$.
We define the following map:
\begin{equation}\label{eq:normalized-U}
 U_\delta(z)
 =1-{\Fell\Big(\!\arcsin\frac{T_\delta(z)}{k} \Bigm| k^2\Big)}\big/{\Kell(k^2)} \text{ such that }
 U_\delta:\mathcal A_\delta\rightarrow
 \C/\big((2\pi\ii/h_\delta)\mathbb Z\big).
\end{equation}
The map $U_\delta$ is well-defined by continuing from the real value on $(L-R_\delta,\ell)$ to the slit disk $\mathcal A_\delta$.
Essentially, $U_\delta$ is based on a composition of two conformal maps: 
$T_\delta^2$ maps the upper half-disk onto the upper
half-plane $\CC^+$, and the Schwarz--Christoffel mapping $\Fell\big(\!\arcsin({\sqrt{z}}/{k})
|k^2\big)$ maps the upper half-plane $\CC^+$ onto a rectangle. Reflection continues the map to the lower half-disk and linear normalization
ensures $U_\delta$ maps $\Gamma_\delta$ and $[\ell, L]$, in the sense of boundary limits, to the two sides of the rectangle 
$\operatorname{Re}U_\delta=1$ and $\operatorname{Re}U_\delta=0$
respectively. See Figure~\ref{fig:conformal} for the illustration of those conformal maps. Consequently, the two boundary conditions in \eqref{eq:potential} are satisfied. As the real part of a holomorphic function is harmonic, $u_\delta$ has the explicit form
\begin{equation}\label{eq:explicit-potential}
 u_\delta(z)=\operatorname{Re}U_\delta(z) \text{ and }
 u_\delta(0)=1-
 {\Fell\Big(\!
      \arcsin\frac{R_\delta-L}{k(R_\delta+L)}
      \Bigm| k^2\Big)}\big/{\Kell(k^2)} = \alpha(\kappa,\gamma(1+\delta)^2).
\end{equation}
This explains the equality in \eqref{eq:alpha-delta}.
We refer to Appendix~\ref{app:elliptic-preliminaries} for the detailed discussion on the M\"obius transformation and the Schwarz--Christoffel mapping mentioned above; see also \citet{driscoll2002schwarz} for a comprehensive introduction to the Schwarz--Christoffel mapping. A formal proof for the explicit form of $u_\delta$ in \eqref{eq:explicit-potential} is given in Lemma~\ref{lem:geometry} of the Appendix.

The following theorem, with its proof deferred to Appendix~\ref{app:pf:thm:generic}, shows the statistical rate of the Bocca estimator for a general filter.

\begin{theorem}\label{thm:generic}
Let $1/\log n\le\delta\le2/\log n$. Suppose $g$ is deterministic, holomorphic on an open neighborhood of $\overline{\Ddisk_\delta}$, and $g(0)=1$. Define
\begin{equation}\label{eq:B-V}
 B(g)=\sup_{x\in[\ell,L]}|g(x)|^2,
 \qquad V(g)=\sup_{z\in\Gamma_\delta}|g(z)|^2.
\end{equation}
Then, for all sufficiently large $n$, we have
\begin{equation}\label{eq:generic-rate}
 \sup_{P\in\mathcal P(n,p,\gamma,\kappa)}\E_P(\widehat\beta_v(g)-\beta_v)^2
 \lesssim B(g)+\frac{V(g)}n.
\end{equation}
\end{theorem}

The right-hand side of \eqref{eq:generic-rate} shows how the bias--variance trade-off depends on $g$: the variance term is controlled by the supremum of $|g|^2$ on the contour $\Gamma_\delta$, and the bias term is controlled by its supremum on the covariance spectral interval $[\ell,L]$.
The next theorem, whose proof is deferred to Appendix~\ref{app:pf:thm:two}, gives a constraint on these two values and gives the closed form of the ideal $g$ to achieve the optimal rate.

\begin{theorem}\label{thm:two}
For sufficiently small $\delta>0$, every function $g$ holomorphic on an
open neighborhood of $\overline{\Ddisk_\delta}$ and satisfying $g(0)=1$
obeys
\begin{equation}\label{eq:two-constants}
 B(g)^{1-\alpha_\delta}V(g)^{\alpha_\delta}\ge1.
\end{equation}
Recall $h_\delta$ in \eqref{eq:k-T-h} and $U_\delta$ in \eqref{eq:normalized-U}. For every integer $K\ge1$, define
\begin{equation}\label{eq:gK}
 g_K(z) =\frac{\cosh (Kh_\delta U_\delta(z))}
                    {\cosh(Kh_\delta U_\delta(0))}.
\end{equation}
This function extends holomorphically to a neighborhood of
$\overline{\Ddisk_\delta}$, satisfies $g_K(0)=1$, and obeys
\begin{equation}\label{eq:gK-bound}
 B(g_K)\le4\exp\big(-2Kh_\delta\alpha_\delta\big) \text{ and }
 V(g_K)\le4\exp\big(2Kh_\delta(1-\alpha_\delta)\big).
\end{equation}
Consequently,
$1\le B(g_K)^{1-\alpha_\delta}V(g_K)^{\alpha_\delta}\le4$.
\end{theorem}

\begin{remark}\label{rem:two-constants}
      The above trade-off \eqref{eq:two-constants} is essentially the two-constants theorem (see \citet[Theorem 4.3.7]{ransford1995potential} or Theorem~\ref{thm:two-constants} in the Appendix).  The reason $u_0(0)$ determines the minimax exponent $\alpha(\kappa,\gamma)$ is that the
      equality in the two-constants theorem is achieved by the harmonic measure $u_\delta(z)$.
      For \eqref{eq:two-constants}, the equality is achieved by the function $\exp(Kh_{\delta}(U_\delta(z)-U_\delta(0)))$ for an integer $K\ge1$. However, it does not extend holomorphically across $[\ell,L]$ as required in Theorem~\ref{thm:generic}. The hyperbolic cosine function in \eqref{eq:gK} symmetrizes the optimal exponential function, and makes it holomorphic on the whole disc. As is shown in Theorem~\ref{thm:two}, although $g_K$ does not attain the equality in \eqref{eq:two-constants}, it still achieves the constant upper bound of the left-hand side of  \eqref{eq:two-constants}, which is enough to achieve the optimal rate. The holomorphic extension of $g_K$ to a neighborhood of
      $\overline{\Ddisk_\delta}$ is shown in Lemma~\ref{lem:geometry}. Figure~\ref{fig:gK} illustrates $g_K$ for several $\kappa$, $\gamma$ and $K$. 
      When computing the Bocca estimator $\widehat\beta_v(g_K)$ in \eqref{eq:estimator}, we only need to evaluate $g_K$ on the contour $\Gamma_\delta$. Therefore, we do not need to compute $U_\delta(z)$ over the whole $\mathcal A_\delta$. See Section~\ref{sec:computation} and \eqref{eq:gK-contour} for the practical method to compute $\widehat\beta_v(g_K)$.
      \end{remark}

      \begin{figure}[ht]
            \vspace{-10pt}
            \centering
            \includegraphics[width=\textwidth]{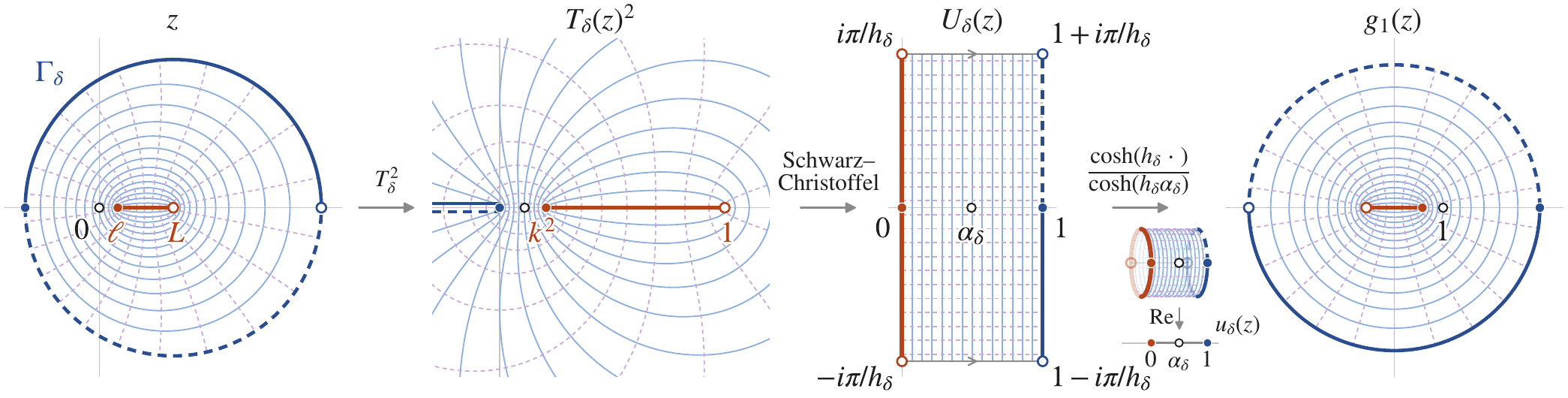}
            \vspace{-20pt}
            \caption{The conformal maps: $T^2_\delta$ in \eqref{eq:k-T-h}, $U_\delta$ in \eqref{eq:normalized-U}, and $g_1$ in \eqref{eq:gK}; and $u_\delta$ in \eqref{eq:explicit-potential}.}
            \label{fig:conformal}
            \vspace{-5pt}
           \end{figure}

Combining the generic rate \eqref{eq:generic-rate} and the bias-variance trade-off \eqref{eq:two-constants} already suggests the rate $n^{-\alpha_{\delta}}$ which converges to the minimax rate $n^{-\alpha(\kappa,\gamma)}$ as the margin $\delta$ shrinks. In particular, we can choose $K$ in $g_K$ to balance the two exponentials in \eqref{eq:gK-bound}. Following this idea, the theorem below gives the estimator to achieve the optimal rate.

\begin{theorem}\label{thm:upper}
Choose $\delta\in[1/\log n,2/\log n]$ and
            \begin{equation}\label{eq:scalar-degree}
             K= \left\lfloor \frac{\log n}{2h_\delta} \right\rfloor = \left\lfloor
             \frac{\Kell(1-k^2)\log n}
             {2\pi\Kell(k^2)}
             \right\rfloor.
            \end{equation}           
            Then the Bocca estimator with $g_K$ in \eqref{eq:gK} satisfies
            \begin{equation}\label{eq:upper}
             \sup_{P\in\mathcal P(n,p,\gamma,\kappa)}\E_P(\widehat\beta_v(g_K)-\beta_v)^2
             \lesssim n^{-\alpha(\kappa,\gamma)}.
            \end{equation}
            \end{theorem}
           The proof of the theorem is in Appendix~\ref{app:pf:thm:upper}. We will show in Theorem~\ref{thm:lower} that the above rate is minimax optimal.

           \begin{figure}[t]
            \centering
            \setlength{\tabcolsep}{2pt}\small
            \begin{tabular}{ccccccc}
             & $\gamma=1$, $K=2$ & $\gamma=1$, $K=10$ & $\gamma=10$, $K=2$ & $\gamma=10$, $K=10$ & $\gamma=100$, $K=2$ & $\gamma=100$, $K=10$ \\
             \rotatebox{90}{\makebox[0.14\textwidth][c]{$\kappa=1.01$}} & \includegraphics[width=0.15\textwidth]{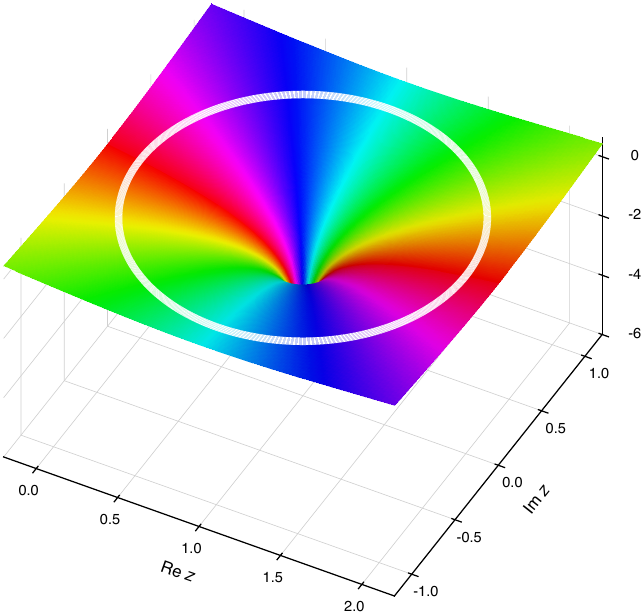} & \includegraphics[width=0.15\textwidth]{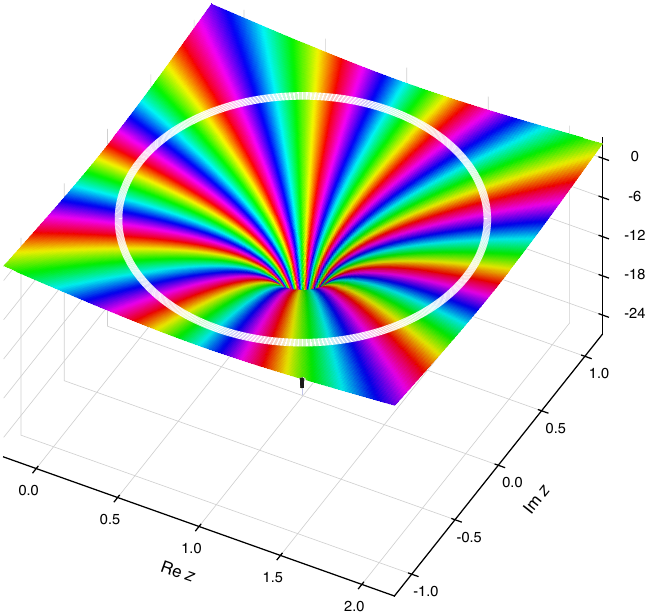} & \includegraphics[width=0.15\textwidth]{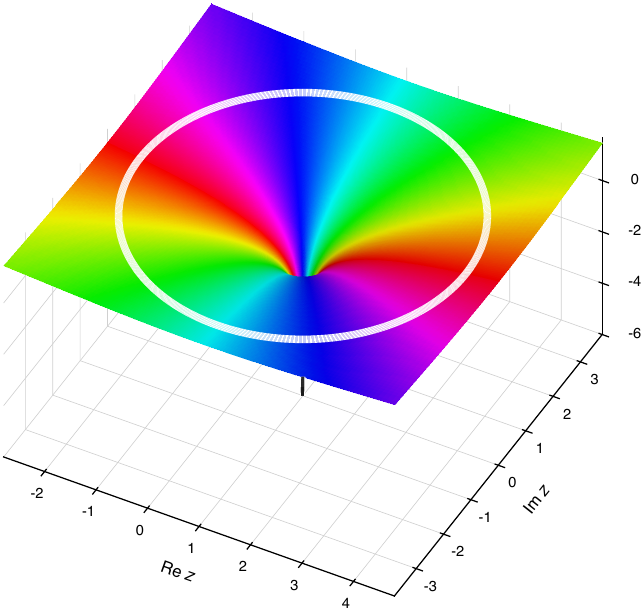} & \includegraphics[width=0.15\textwidth]{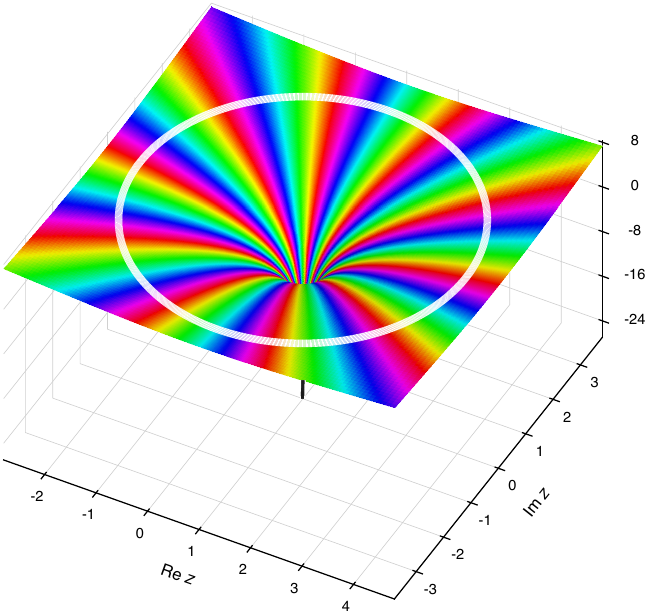} & \includegraphics[width=0.15\textwidth]{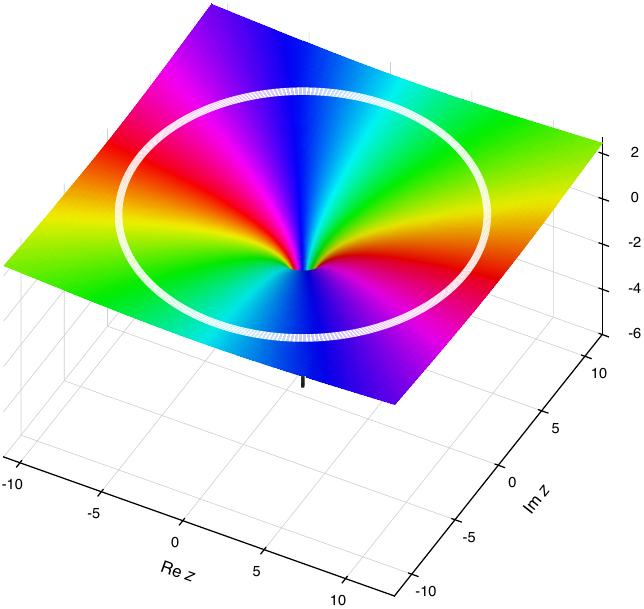} & \includegraphics[width=0.15\textwidth]{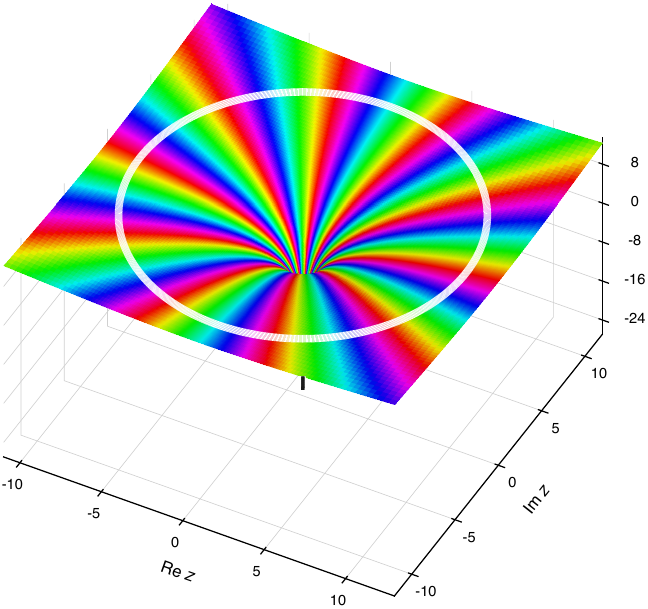} \\
             \rotatebox{90}{\makebox[0.14\textwidth][c]{$\kappa=10$}} & \includegraphics[width=0.15\textwidth]{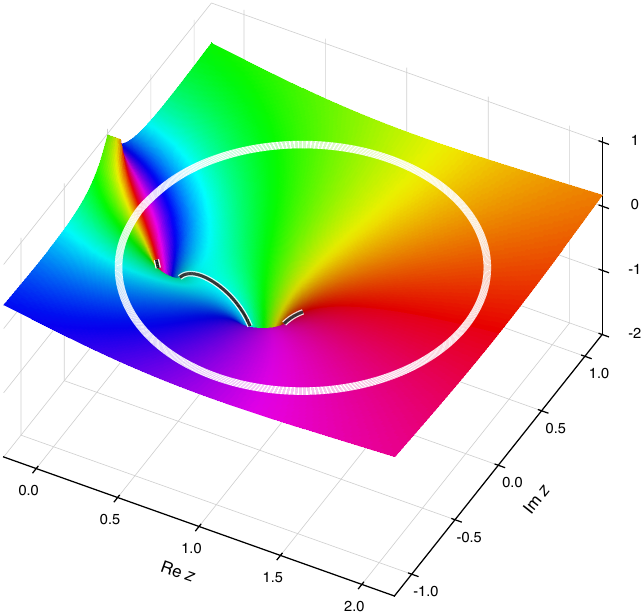} & \includegraphics[width=0.15\textwidth]{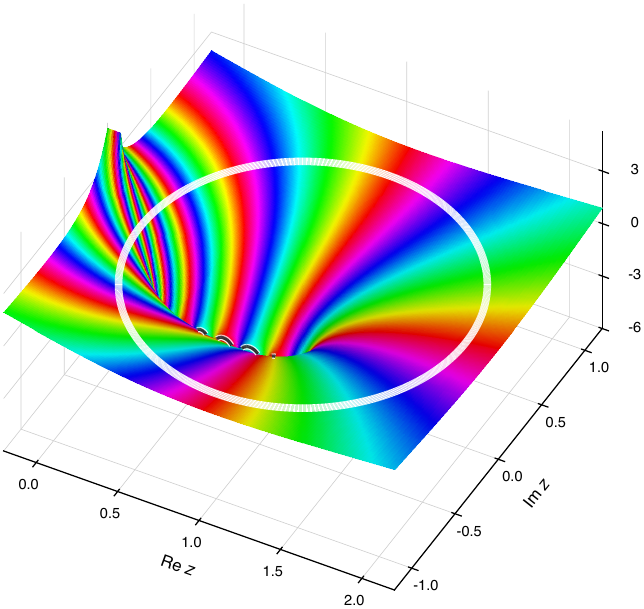} & \includegraphics[width=0.15\textwidth]{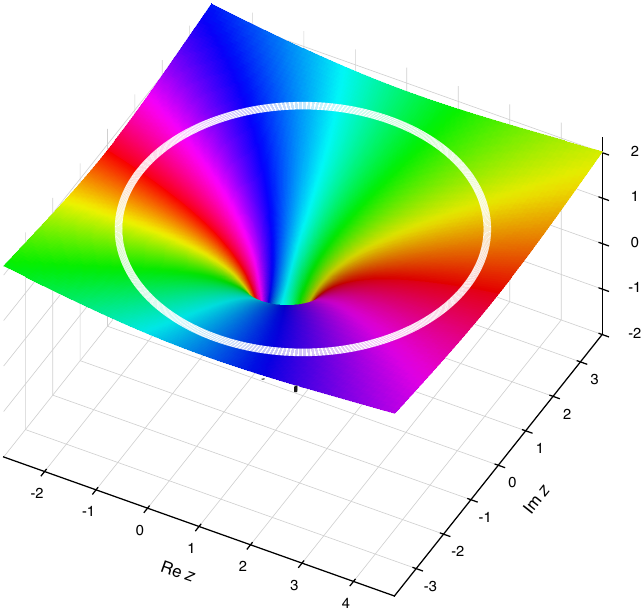} & \includegraphics[width=0.15\textwidth]{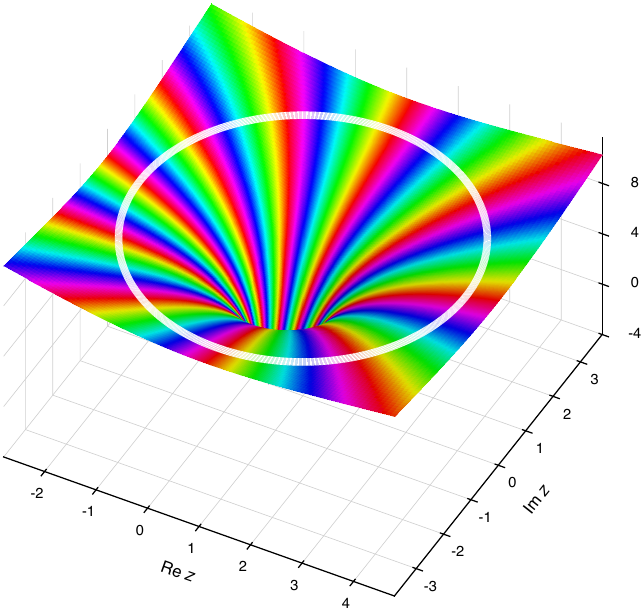} & \includegraphics[width=0.15\textwidth]{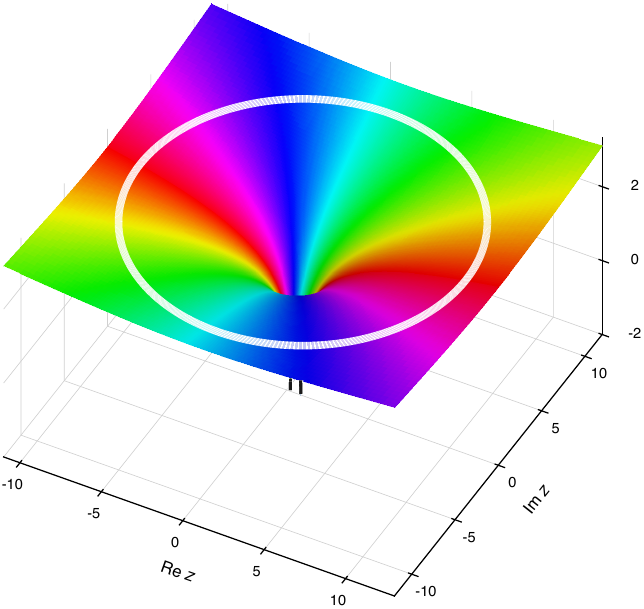} & \includegraphics[width=0.15\textwidth]{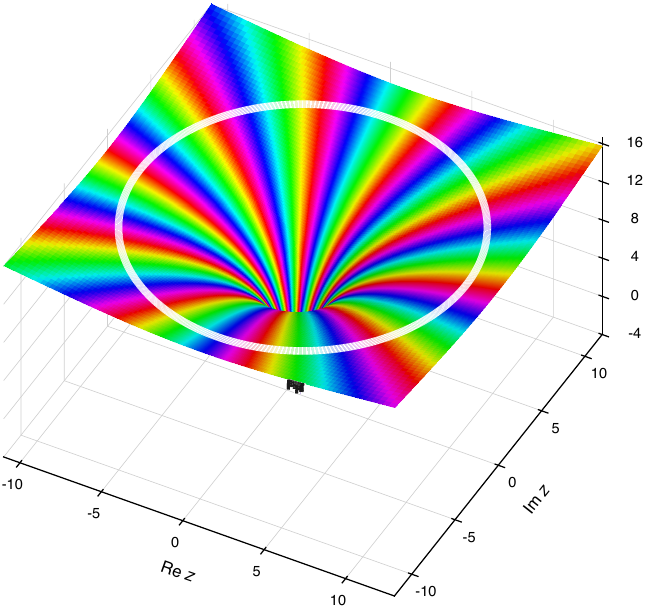} \\
             \rotatebox{90}{\makebox[0.14\textwidth][c]{$\kappa=100$}} & \includegraphics[width=0.15\textwidth]{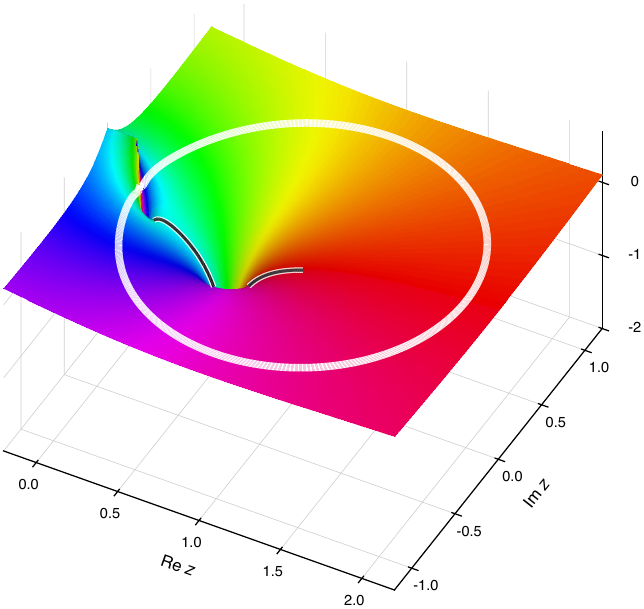} & \includegraphics[width=0.15\textwidth]{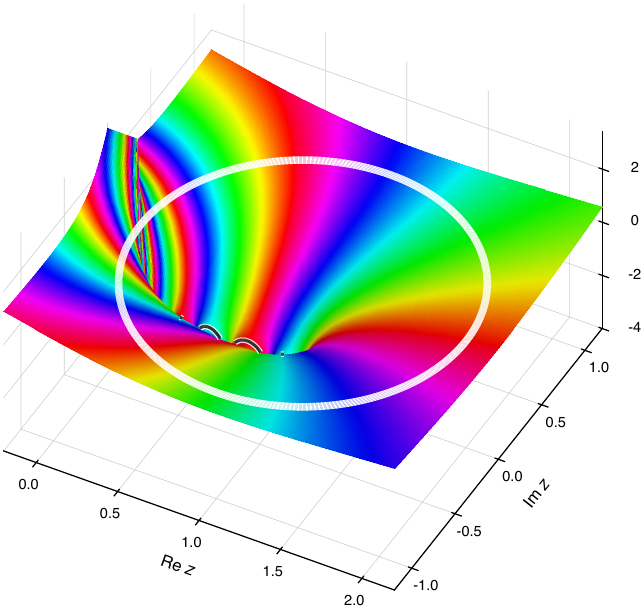} & \includegraphics[width=0.15\textwidth]{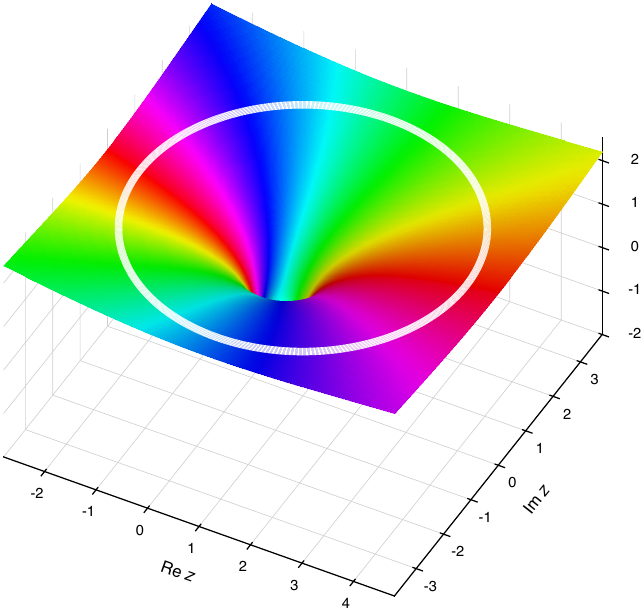} & \includegraphics[width=0.15\textwidth]{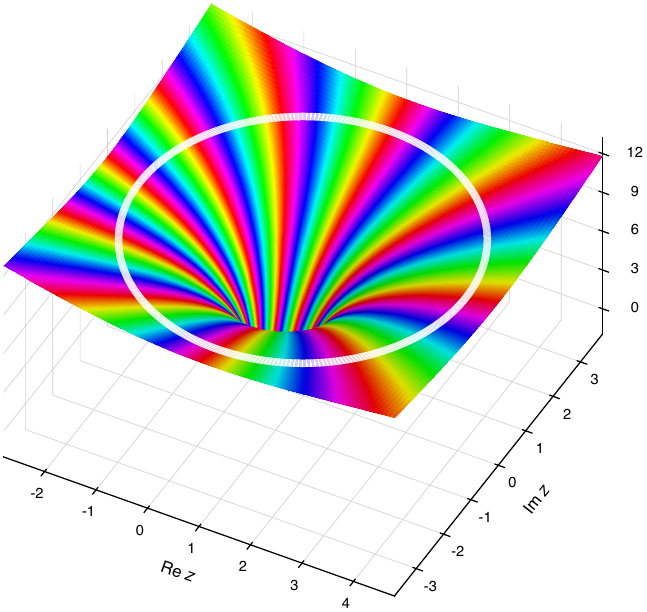} & \includegraphics[width=0.15\textwidth]{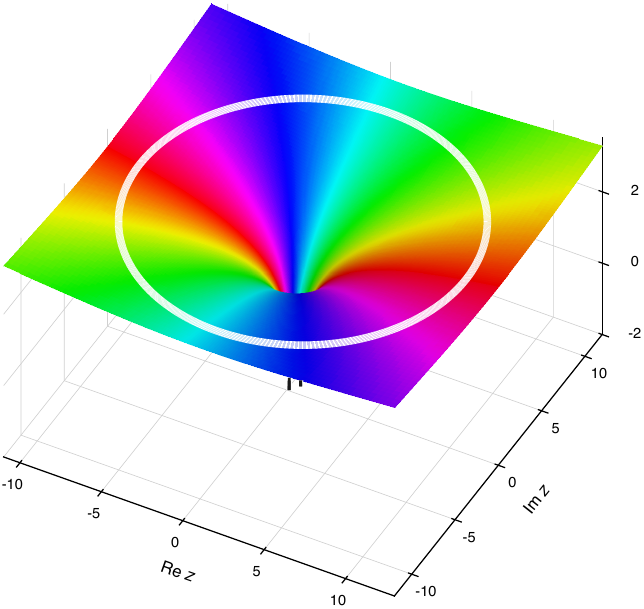} & \includegraphics[width=0.15\textwidth]{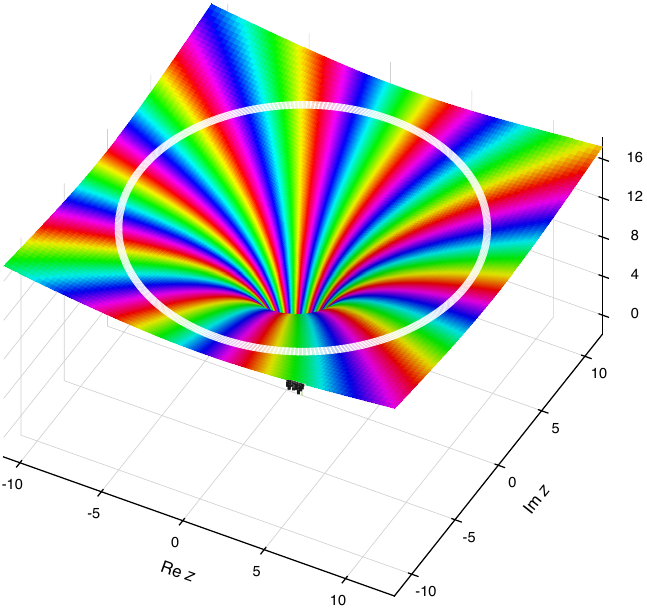} \\
             & \multicolumn{6}{c}{\includegraphics[width=0.36\textwidth]{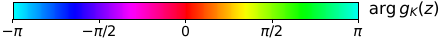}} \\
            \end{tabular}
            \vspace{-10pt}
            \caption{The filter $g_K$ of \eqref{eq:gK} with $\delta=0$ and $L=1$. The height is $\log_{10}|g_K(z)|$ and the hue is $\arg g_K(z)$; the white circle is the contour $\Gamma_0$.}
            \label{fig:gK}
            \vspace{-10pt}
           \end{figure}

           \begin{remark}
            \label{rem:exponent-asymptotics}
            The exponent $\alpha(\kappa,\gamma)$ is strictly decreasing in both $\kappa > 1$ and $\gamma > 1$. Therefore, increasing either the
            upper bound on the covariance condition number or the dimension ratio slows the minimax rate.
            When both $\kappa$ and $\gamma$ diverge, we can show that
            \begin{equation}\label{eq:alpha-joint}
             \alpha(\kappa,\gamma)
             \sim \frac{4}{\sqrt\kappa\,\log(16\gamma)}.
            \end{equation}
            See Proposition~\ref{prop:exponent-asymptotics} in the appendix for details.
            This shows that the optimal rate is less sensitive to the growth of the proportion $\gamma$ than the upper bound $\kappa$ on the covariance condition number. Figure~\ref{fig:alpha} validates the asymptotic approximation numerically. Although this paper focuses on the finite $\gamma$ case, it is not difficult to extend our results to the case when $\gamma \to \infty$. Suppose both our upper and lower bound results could be extended, then \eqref{eq:alpha-joint} suggests that we can consistently estimate the linear functional $\beta_v$ beyond the proportional regime to the case when $\log \gamma = o(\log n)$, i.e., $p = n^{1+o(1)}$. On the other hand, the lower bound in Theorem~\ref{thm:lower} suggests that there is no consistent estimator when $p \gtrsim n^{1+\epsilon}$ for any fixed constant $\epsilon>0$.
            \end{remark}

\begin{remark}\label{rem:prediction-double-descent} We compare the estimation rate in \eqref{eq:upper} in the context of prediction with the literature on double descent and benign overfitting \citep{BelkinEtAl2019,BelkinHsuXu2020,BartlettEtAl2020,HastieEtAl2022}.
For a deterministic $x_0\in\R^p$, as the Bocca estimator $\hat{\beta}_v$ is linear in $v$, we can estimate the prediction $x_0^\top\beta$ by $\widehat\beta_{x_0}(g_K)$, with the rate $\E\bigl(\widehat\beta_{x_0}(g_K)-x_0^\top\beta\bigr)^2
\lesssim \|x_0\|_2^2n^{-\alpha(\kappa,\gamma)}$. 
On the other hand, \citet{HastieEtAl2022} consider the out-of-sample prediction risk given a new independent random test point $x_{\mathrm{new}}$ with mean zero and covariance $\Sigma$,
\begin{equation*}
      R_X(\widehat\beta;\beta)
      =\E
        \bigl[(x_{\mathrm{new}}^\top(\widehat\beta-\beta))^2
               \mid X\bigr]
               =\E\bigl[
        (\widehat\beta-\beta)^\top\Sigma
        (\widehat\beta-\beta)\mid X\bigr].
     \end{equation*}
For the min-norm least squares estimator $\widehat\beta_{\mathrm{ols}}= \argmin_{\beta}\|\beta\|_2$ s.t. $X\beta=Y$, when 
$\Sigma=I_p$ and $p/n\to\gamma_0\in(1,\infty)$, Theorem~1 of \citet{HastieEtAl2022} shows almost surely that
\begin{equation}\label{eq:ridgeless-risk-limit}
 R_X(\widehat\beta_{\mathrm{ols}};\beta)
 \longrightarrow \|\beta\|_2^2\Bigl(1-\frac{1}{\gamma_0}\Bigr)
                  +\sigma^2\frac{1}{\gamma_0-1}.
\end{equation}
Theorem~2 of \citet{HastieEtAl2022} generalizes the above result under the nonasymptotic setting for general covariance.
Comparing our result with those in the double descent literature, we consider different metrics of prediction loss:
\eqref{eq:ridgeless-risk-limit} averages over random test covariates with
$\E\|x_{\mathrm{new}}\|_2^2=\tr\Sigma\asymp p$ and it converges to a positive constant even if it may decrease with $\gamma_0$, whereas our rate \eqref{eq:upper} considers the prediction risk of a fixed bounded covariate $\|x_0\|_2\le C$, and it could converge to zero but the rate always increases with $\gamma$.
\end{remark}

\subsection{Computation}\label{sec:computation}
We now discuss the computation of the Bocca estimator. One of the nice properties of $\widehat{\beta}_v(g)$ in \eqref{eq:estimator} is that it is linear in $v$. Therefore, we can compute the estimator before specifying $v$.

\noindent{\bf Contour discretization.} We compute the contour integral in \eqref{eq:estimator} by discretizing the contour circle $\Gamma_\delta$ into $M$ equally spaced points.
 Define $t_j=\pi+\frac{2\pi j}{M}$, $\mu_j=L+R_\delta e^{\ii t_j}$, for $j=0,\ldots,M$.
The nodes traverse $\Gamma_\delta$ counterclockwise, starting at
$\mu_0=L-R_\delta<0$, and $\mu_M=\mu_0$.
Compute the eigendecomposition of $\widehat S=XX^\top/n  
=U\diag(s_1,\ldots,s_n)U^\top
$, with $s_1\le\cdots\le s_n$.
 Then
$\widehat m(z)=n^{-1}\sum_{i=1}^n(s_i-z)^{-1}$.

\noindent{\bf Computing $\widehat\zeta$ by Newton's method.} As $\widehat m(\widehat\zeta(\mu))=-1/\mu$, for each node $\mu_j$, we can solve for $z$ using the equation
\begin{equation}\label{eq:zhat-scalar}
 \frac1n\sum_{i=1}^n\frac1{s_i-z}+\frac1{\mu_j}=0.
\end{equation}
We start with the real node $\mu_0=L-R_\delta<0$. Since $s_1+\mu_0
\le \widehat\zeta(\mu_0)
\le s_1+\frac{\mu_0}{n}$, we can solve $z_0 = \widehat\zeta(\mu_0)$ by bisection on the interval $[s_1+\mu_0, s_1+\frac{\mu_0}{n}]$.
We then continue to compute the next node by iteratively applying Newton's method to solve \eqref{eq:zhat-scalar}. Once $z_j = \widehat\zeta(\mu_j)$ has been computed, we describe how to compute $z_{j+1} = \widehat\zeta(\mu_{j+1})$.
Differentiate
$\widehat m(\widehat\zeta(\mu))=-1/\mu$ to obtain
\begin{equation}\label{eq:zhat-derivative}
 \widehat\zeta'(\mu)
 =
 \left[
 \mu^2\frac1n\sum_{i=1}^n
 (s_i-\widehat\zeta(\mu))^{-2}
 \right]^{-1}.
\end{equation}
We then apply Newton's method 
\[
z_{j+1}^{(r+1)}
      =
      z_{j+1}^{(r)}
      -
      \frac{
      n^{-1}\sum_{i=1}^n(s_i-z_{j+1}^{(r)})^{-1}+\mu_{j+1}^{-1}
      }{
      n^{-1}\sum_{i=1}^n(s_i-z_{j+1}^{(r)})^{-2}
      }, \text{ where }
      z_{j+1}^{(0)}
 =
 z_j+
 \frac{\mu_{j+1}-\mu_j}
 {\mu_j^2 n^{-1}\sum_{i=1}^n(s_i-z_j)^{-2}}.
\]
We will stop the iteration when the desired tolerances are achieved and denote the resulting value as $z_{j+1}$.
If the iteration does not converge, we terminate the computation and return the estimator as zero.

\noindent{\bf Deterministic kernel values.}
For the specific choice of $g = g_K$ in \eqref{eq:gK}, we compute its value on the contour 
$\mu(t)=L+R_\delta e^{\ii t}$ without computing $U_\delta(z)$ in its definition. For $t\notin2\pi\mathbb Z$, set
\[
 y(t)=
 \Fell\left(
 \arctan\frac{\cot(t/2)}k
 \,\middle|\,1-k^2
 \right).
\]
The filter $g_K$ can then be evaluated directly as
\begin{equation}\label{eq:gK-contour}
 g_K(\mu(t))
 =
 {
 \cosh\left(
 \dfrac{\pi K}{\Kell(1-k^2)}
 \bigl(\Kell(k^2)-\ii y(t)\bigr)
 \right)}\Big/
 {
 \cosh\left(
 \dfrac{\pi K\Kell(k^2)}{\Kell(1-k^2)}u_\delta(0)
 \right)
 },
\end{equation}
where $u_\delta(0)$ is given in \eqref{eq:explicit-potential}.
At $t\in2\pi\mathbb Z$, use either one-sided limit. 
Thus all evaluations require only real elliptic integrals
and a complex hyperbolic cosine.
Those formulas are verified in Lemma~\ref{lem:geometry} in the Appendix.

\noindent{\bf Bocca estimator computation for all loadings.}
Compute the real weights
\begin{equation}\label{eq:bocca-spectral-weights}
 w_i
 =
 -\frac1M\operatorname{Re}\sum_{j=0}^{M-1}
 \frac{(\mu_j-L)\bigl(1-g(\mu_j)\bigr)}
 {\mu_j(s_i-z_j)},
 \text{ for } i=1,\ldots,n,
\end{equation}
and the whole coefficient vector estimator
$
 \widetilde\beta_M(g)
 =
 \frac1nX^\top
 \left[
 U\left(
 \diag(w_1,\ldots,w_n)(U^\top Y)
 \right)
 \right].
$
For a given $v$, we compute the Bocca estimator by
$
 \widehat\beta_{v,M}(g)
 =
 v^\top\widetilde\beta_M(g).
$

For $g_K$ in \eqref{eq:gK}, which has conjugation symmetry, we have
\[
 \mu_{M-j}=\overline{\mu_j},
 \widehat\zeta(\mu_{M-j})
 =\overline{\widehat\zeta(\mu_j)},
 g_K(\mu_{M-j})=\overline{g_K(\mu_j)},
 \text{ for } 1\le j<M.
\]
Thus one may compute grids on a single
semicircle of $\Gamma_\delta$ and obtain the remaining values by conjugation.


\section{Asymptotic Normality}\label{sec:bocca-normality}

We study the asymptotic normality of the Bocca estimator \eqref{eq:estimator} in this section. Throughout this section, the parameters $\Sigma,\beta,\sigma,v, \gamma$ still may
vary with $n$, and we fix the choice 
$\delta=1/\log n$. To get the asymptotic normality, we need to assume the covariance matrix spectrum distribution satisfies the following condition.

\begin{assumption}[Spectral tube condition]\label{ass:bocca-normality}
For some fixed $\underline\sigma$, $\eta_0>0$, $\sigma\ge \underline\sigma >0$ and
\begin{equation}\label{eq:normality-tube-condition}
 \liminf_{n\to\infty}\inf_{x\in[\ell,L]}\gamma \cdot
 \frac1p\sum_{j=1}^p\frac{\lambda_j^2}{(\lambda_j-x)^2+\eta_0^2}>1.
\end{equation}
\end{assumption}

\begin{remark}\label{rem:normality-regularity}
The tube condition \eqref{eq:normality-tube-condition} permits atomic limiting spectral distributions,
including point masses. On the other hand, if the limiting spectral distribution has a density bounded below by a
positive constant on $[\ell,L]$, then \eqref{eq:normality-tube-condition} is also satisfied. 
\end{remark}

Recall $\Sigma=\sum_{j=1}^p\lambda_jo_jo_j^\top$ and define $g(\Sigma) = \sum_{j=1}^p g(\lambda_j) o_j o_j^\top$. 
The Bocca estimator $\widehat\beta_v(g)$ in \eqref{eq:estimator} has the variance-bias decomposition
\begin{equation} \label{eq:normality-variance-bias-decomposition}
     \begin{aligned}
          &\widehat\beta_v(g) - \beta_v = \operatorname{Re}
              \big(-\frac1{2\pi\ii}\oint_{\Gamma_\delta}
              \frac{1-g(z)}z(\widehat F(z)- F(z))\dd z\big) - v^\top g(\Sigma)\beta, \text{ where }\\
              &\quad \widehat F(\mu)-F(\mu) \approx \frac1n v^\top X^\top(\widehat S-  \zeta(\mu)I_n)^{-1}Y-v^\top\Sigma(\Sigma-\mu I_p)^{-1}\beta,
         \end{aligned} 
\end{equation}
and the bias term $v^\top g(\Sigma)\beta$ is from \eqref{eq:bocca-bias}. The variance term is driven by $\hat{F}-F$, with $F$ and $\widehat F$ defined in \eqref{eq:population-field} and \eqref{eq:empirical-field}, respectively,
and the $\approx$ replaces $\widehat\zeta$ by $\zeta$ in $\widehat F$.
We next introduce a covariance kernel related to the leading term of $\EE[(\widehat F(\mu)-F(\mu))(\widehat F(\nu)-F(\nu))]$ using the approximation in \eqref{eq:normality-variance-bias-decomposition}. For $\mu,\nu \in \Gamma_\delta$ and $j = 1, \ldots, p$, set 
\begin{equation}\label{eq:normality-covariance-kernel-auxiliary}
     R_\mu=(\Sigma-\mu I_p)^{-1},\,
A_{\mu\nu}=\Sigma R_\mu R_\nu,\,
\mathcal Q_{\mu\nu}=\frac1n\tr(\Sigma A_{\mu\nu}),\,
b_j(\mu)=(\Sigma^{1/2}R_\mu v)_j
                (\Sigma^{1/2}R_\mu\beta)_j.
\end{equation}
Write $\chi_4=\E Z_{11}^4-3$ and define the covariance kernel by
\begin{equation}\label{eq:normality-kernel}
 \begin{aligned}
 \mathcal K_n(\mu,\nu)
 ={}&\frac{\mu\nu}{1-\mathcal Q_{\mu\nu}}
 \Big((v^\top A_{\mu\nu}v)(\beta^\top A_{\mu\nu}\beta)+(v^\top A_{\mu\nu}\beta)^2\Big)\\
 &\quad +\frac{\sigma^2}{1-\mathcal Q_{\mu\nu}}v^\top A_{\mu\nu}v
   +\chi_4\mu\nu\sum_{j=1}^p b_j(\mu)b_j(\nu).
 \end{aligned}
\end{equation}
The second
term in \eqref{eq:normality-kernel} comes from the randomness of the noise $\varepsilon$. The first and third terms are from the design $X$, where the first term depends on the second moments and the third term is the
fourth-cumulant correction. We refer to Appendix~\ref{app:bocca-normality} for the details.

For a holomorphic filter with $g(0)=1$ and
$g(\bar z)=\overline{g(z)}$, as $\widehat\beta_v(g)$ in \eqref{eq:estimator} is a weighted integral of $\hat{F}$, its asymptotic variance thus is given by
\begin{equation}\label{eq:normality-variance}
 \sigma^2(g)=-\frac1{4\pi^2n}
 \oint_{\Gamma_\delta}\!\oint_{\Gamma_\delta}
 \frac{(1-g(\mu))(1-g(\nu))}{\mu\nu}
 \mathcal K_n(\mu,\nu)\dd\mu\dd\nu.
\end{equation}
Conjugation symmetry of $g$ makes $\sigma^2(g)$ real.

Next we propose an estimate of $\sigma^2(g)$ when we use the hyperbolic cosine filter $g = g_K$ in \eqref{eq:gK} with $K\asymp\log n$.
Put
$\widehat Q_\mu=(\widehat S-\widehat\zeta(\mu)I_n)^{-1}$ and define
\begin{equation}\label{eq:normality-estimated-kernel}
 \widehat{\mathcal K}_n(\mu,\nu)
 =\frac{
 \bigl((Xv)^\top\widehat Q_\mu\widehat Q_\nu Xv\bigr)
 \bigl(Y^\top\widehat Q_\mu\widehat Q_\nu Y\bigr)
 +\bigl((Xv)^\top\widehat Q_\mu\widehat Q_\nu Y\bigr)^2}
 {n\,\tr(\widehat Q_\mu\widehat Q_\nu)}.
\end{equation}
Comparing to \eqref{eq:normality-kernel}, $n^{-1}\tr(\widehat Q_\mu\widehat Q_\nu)$ corresponds to $(\mu\nu(1-\mathcal Q_{\mu\nu}))^{-1}$, and  $n^{-1}(Xu)^\top \widehat Q_\mu\widehat Q_\nu Xw$ to $u^\top A_{\mu\nu} w /(1-\mathcal Q_{\mu\nu})$ for $u, w$ equal to $v$ or $\beta$ as $Y = X\beta + \varepsilon$. The fourth cumulant correction in the last term of \eqref{eq:normality-kernel} is negligible when using $g_K$ as filter. This is typically not true for general filters. Plugging \eqref{eq:normality-estimated-kernel} into \eqref{eq:normality-variance}, we get the variance estimator
\begin{equation}\label{eq:normality-estimated-variance}
 \begin{aligned}
 \widehat\sigma^2(g_K)
 &=\max\biggl\{
 -\frac1{4\pi^2n}
 \oint_{\Gamma_\delta}\!\oint_{\Gamma_\delta}
 \frac{(1-g_K(\mu))(1-g_K(\nu))}{\mu\nu}
 \widehat{\mathcal K}_n(\mu,\nu)\,\dd\mu\,\dd\nu,
 \frac{\|Y\|_2^2}{n^2}\biggr\}.
 \end{aligned}
\end{equation}
If no continuous branch $\widehat\zeta$ satisfying
\eqref{eq:empirical-calibration} exists, we set
$\widehat\sigma^2(g_K)=\|Y\|_2^2/n^2$.
The double integral is real again by conjugation symmetry of $g_K$ and
$\widehat{\mathcal K}_n$.
The $n^{-2}{\|Y\|_2^2}$ term is a safeguard to avoid negative variance estimates, which will not happen with high probability. We refer to the proof of Lemma~\ref{lem:normality-variance-estimation} in the Appendix for the details.

We now discuss the selection of $K$ in $g_K$ for inference.
The choice in \eqref{eq:scalar-degree} is for estimation, which requires balancing bias and variance, and is not applicable to inference.
For inference, we need to choose $K$ so that the squared bias is negligible
relative to the variance. Thus, we need a lower bound estimate of variance instead of the upper bound in \eqref{eq:gK-bound}. Under Assumption~\ref{ass:bocca-normality}, recalling $\alpha_\delta$ in \eqref{eq:alpha-delta} and $h_\delta$ in \eqref{eq:k-T-h}, we have the following new variance-bias estimates:
\begin{equation}
     \label{eq:normality-variance-bias-estimates}
     \begin{aligned}      
              &\sigma^2(g_K)\gtrsim
         \frac{1}{nK}
         e^{2K(M_n-h_\delta\alpha_\delta)} \text{ and } |v^\top g_K(\Sigma)\beta|^2\lesssim e^{-2Kh_\delta\alpha_\delta}, \text{ where }\\
         &  M_n=h_\delta\max_{\mu\in\mathcal C_n}u_\delta(\mu) \text{ and }
         \mathcal C_n=\Big\{\mu\in\C\setminus[\ell,L]:
         \frac1n\sum_{j=1}^p\frac{\lambda_j^2}{|\lambda_j-\mu|^2}=1\Big\}. 
         \end{aligned}   
\end{equation}
We prove \eqref{eq:normality-variance-bias-estimates} in Lemma~\ref{lem:normality-variance-bounds} in the Appendix.
In \eqref{eq:gK-bound}, we give the upper bound $V(g_K)\lesssim \exp(2Kh_\delta(1-\alpha_\delta))$, where $1$ comes from $u_\delta=1$ on $\Gamma_\delta$ by \eqref{eq:potential}.
In the lower bound estimate \eqref{eq:normality-variance-bias-estimates}, the contour $\mathcal C_n$ depending on the covariance spectrum replaces $\Gamma_\delta$ as the variance scale could be different for different covariance spectra. In fact, \eqref{eq:normality-variance-bias-estimates} implies that $\mathcal C_n$ is the canonical contour for the Bocca estimator in \eqref{eq:estimator} though it depends on the population covariance spectrum. That is why we still choose $\Gamma_\delta$ as the contour for Bocca estimator in both estimation and inference.

Choosing $K=\lceil(\log n+a\log\log n)/(2M_n)\rceil$, with $a>1$, \eqref{eq:normality-variance-bias-estimates} gives the negligible ratio between squared bias and variance $|v^\top g_K(\Sigma)\beta|^2/\sigma^2(g_K) \lesssim nK e^{-2KM_n} = O((\log n)^{1-a})=o(1)$. As $M_n$ depends on the covariance spectrum, we propose its estimate next. By conjugation symmetry, the maximum in \eqref{eq:normality-variance-bias-estimates} over $\mathcal C_n$ can be computed on the upper half plane $\mathcal C_n \cap \{\operatorname{Im} \mu \ge 0\}$.  The key idea, informally speaking, is that $\mathcal C_n \cap \{\operatorname{Im} \mu \ge 0\} \approx \zeta^{-1}([s_1,s_n]) = -1/m([s_1,s_n])$, where $s_1\le\cdots\le s_n$ are the eigenvalues of $\widehat S$. We refer to \citet[Theorems~1.1 and~4.2]{SilversteinChoi1995} and Appendix~\ref{app:pf:normality-estimated-M} for the details. Discretize the interval $[s_1,s_n]$ into $N_n =\lceil\sqrt n\rceil$ subintervals and set $E_j=s_1+\frac{j}{N_n}(s_n-s_1)$, for $0\le j\le N_n$. We estimate $M_n$ in \eqref{eq:normality-variance-bias-estimates} by
\begin{align}\label{eq:normality-estimated-M}
 &\widehat M_n=\max\Big\{
 h_\delta\alpha_\delta,\,
 h_\delta\max_{0\le j\le N_n}u_\delta(\widehat\mu_j)\Big\}, \text{ where } \widehat\mu_j=-1/\widehat m(E_j+\ii n^{-1/4}),
\end{align}
and we evaluate $\hat m$ at $E_j + \ii n^{-1/4}$ to avoid singularities on the real axis. The $h_\delta\alpha_\delta$ in \eqref{eq:normality-estimated-M} is the safeguard to avoid $\widehat M_n$ being too small.
If any $\widehat\mu_j$ lies outside
$\Ddisk_\delta\setminus[\ell,L]$, set
$\widehat M_n=h_\delta\alpha_\delta$.
We then choose $K$ as
\begin{equation}\label{eq:normality-degree}
 \widetilde K
 =\left\lceil
 \frac{\log n+a\log\log n}{2\widehat M_n}+U_n
 \right\rceil.
\end{equation}
Here $a>1$ is fixed; we recommend $a=3$. The variable $U_n\sim\operatorname{Unif}(0,1)$ is
independent of the data. We add $U_n$ to randomize the rounding
threshold in \eqref{eq:normality-degree}, so the estimation error of $\widehat M_n$ changes the rounded $\tilde K$
with vanishing probability.

The following theorem gives the asymptotic normality of the Bocca estimator with its proof deferred to Appendix~\ref{app:bocca-normality}.

\begin{theorem}[Asymptotic normality of Bocca]\label{thm:bocca-normality}
Suppose Assumption~\ref{ass:bocca-normality} holds and use
$\delta=1/\log n$.
For every deterministic integer $K\asymp\log n$, we have
$\widehat\sigma^2(g_K)/\sigma^2(g_K)\xrightarrow{p}1$, and
\begin{equation}\label{eq:normality-biased}
     \sigma(g_K)^{-1}\big(\widehat\beta_v(g_K)-\beta_v+v^\top g_K(\Sigma)\beta\big)
 \overset{d}{\longrightarrow} N(0,1).
\end{equation}
For $\widetilde K$ in \eqref{eq:normality-degree}, we have $\widehat\sigma^2(g_{\widetilde K})/\sigma^2(g_{\widetilde K})\xrightarrow{p}1$ and 
\begin{equation}\label{eq:normality-unbiased}
 \widehat\sigma(g_{\widetilde K})^{-1}\big(\widehat\beta_v(g_{\widetilde K})-\beta_v\big)
 \overset{d}{\longrightarrow} N(0,1).
\end{equation}
\end{theorem}

\section{Information-Theoretic Lower Bound}\label{sec:lower}

In this section, we establish the lower bound which matches the rate of the Bocca estimator in Theorem~\ref{thm:upper}, with its proof in Appendix~\ref{app:mainproofs-lower-main}.

\begin{theorem}\label{thm:lower}
Given any fixed $v$ with unit length $\|v\|_2=1$, let $\widetilde\beta_v$ be any estimator for $\beta_v$. Then we have
\begin{equation}\label{eq:lower}
 \inf_{\widetilde\beta_v}
 \sup_{P\in\mathcal G(n,p,\gamma,\kappa)}
 \E_P(\widetilde\beta_v-\beta_v)^2
 \gtrsim n^{-\alpha(\kappa,\gamma)}.
\end{equation}
Consequently, the same lower bound holds with $\mathcal P(n,p,\gamma,\kappa)$.
\end{theorem}

\begin{remark}\label{rem:lower-overview-main}
We give a high-level overview of the proof of the lower bound, and refer to Appendix~\ref{app:mainproofs-lower} for the details. We consider the special case of $L=1$ and $v=e_1$. We implement Le Cam's two-point method \citep[Lemma~1]{Yu1997} to construct two hypothesis priors $P_0$ and $P_1$ with different values of $\beta_1$ separated by the desired rate but nearly indistinguishable data laws in total variation distance. We begin by constructing two measures $\nu_0$ and $\nu_1$ on the spectrum interval $[\kappa^{-1}+n^{-1/8},1]$ where we add the $n^{-1/8}$-margin to avoid the boundary singularity. Use $\tilde{g}_K(z) =\cosh (Kh_\delta U_\delta(z))$ as the unnormalized version of
$g_K$ in \eqref{eq:gK} with parameters $\delta=0$, $\ell=\kappa^{-1}+n^{-1/8}$, $L=1$,
and $\gamma=(p-1)/n$. Let $z_1<\cdots<z_K$ be its simple zeros on the spectrum interval. Then $\nu_0$ and $\nu_1$ are measures supported on these zeros with weights related to the residues $1/\tilde{g}_K'(z_k)$, where $\nu_0$ is supported on zeros with negative residues and $\nu_1$ on zeros with positive residues.

We next construct the regression models based on $\nu_0$ and $\nu_1$. Let $V\in\R^{(p-1)\times K}$ be a mixing matrix whose density is specified later.  Generate a centered Gaussian $X_{-1}$ with spiked precision matrix $I_{p-1}+V\bigl(\diag(z_1^{-1}, \ldots, z_K^{-1}) - I_K\bigr)
V^\top$. For each hypothesis $j\in\{0,1\}$, we then generate $X_1$ and $Y$ from linear models on $X_{-1}$ using the same coefficient vector $V\vartheta_j$ where $\vartheta_{jk}\propto\sqrt{\nu_j(\{z_k\})}$:
\[
      X_1=X_{-1}V\vartheta_j+\varepsilon^\circ+\xi,\;
      Y=X_{-1}V\vartheta_j+\varepsilon^\circ-\xi, \text{ then } \E[Y|X,V]=b_jX_1+(1-b_j)X_{-1}V\vartheta_j,
\]
where $\varepsilon^\circ,\xi$ are independent centered Gaussian noise vectors, and $b_j$ is the derived coefficient of $X_1$ in the linear model related to the variances of $\varepsilon^\circ$ and $\xi$. We then mix the two regression models over a prior on $V$ to obtain the two hypotheses $P_0$ and $P_1$.
We refer to Appendix~\ref{app:lower-construction} for the detailed construction. In Lemma~\ref{lem:lower-priors} of the Appendix, we show that $\operatorname{TV}(P_0,P_1)\le 1/3$ and $|b_1-b_0|\gtrsim n^{-\alpha(\kappa,\gamma)/2}$, which implies the lower bound in \eqref{eq:lower} by Le Cam's two-point method.

Finally, we provide an explanation of why we choose $\tilde{g}_K$ in the construction of the two spectral measures $\nu_0$ and $\nu_1$. There are three aspects to consider: (1) $\nu_0$ and $\nu_1$ need to have bounded total mass, which requires controlling the residues of $1/\tilde{g}_K$ at the zeros of $\tilde{g}_K$; (2) the condition that two hypotheses are indistinguishable, i.e., $\operatorname{TV}(P_0,P_1)\le 1/3$, requires $1/\tilde{g}_K$ to be small enough on the contour $\Gamma_0$; (3) the separation $|b_1-b_0|$ corresponds to the value of $1/\tilde{g}_K(0)$ and we want it to be as large as possible. Similar to the discussion on the upper bound in Remark~\ref{rem:two-constants}, 
by the two-constants theorem (Theorem~\ref{thm:two-constants} in the Appendix), the filter $\tilde{g}_K$ balances the trade-off of the three aspects and attains the optimal rate. We refer to Remark~\ref{rem:lower-measure-trade-off} in the Appendix for more detailed discussion.
\end{remark}


\section{Application to Debiasing Methods}\label{sec:graph-debias}

In this section, we discuss how to combine the Bocca estimator with the debiasing methods.
We work under the same linear model set up in \eqref{eq:model-family} and \eqref{eq:aspect-regime}.
Without loss of generality, the parameter of interest we consider in this section is $\beta_1$. The debiasing methods which correct the bias of an initial estimator typically through an
approximately orthogonal estimating equation have been widely studied in the literature
\citep{ZhangZhang2014,vandeGeer2014,JavanmardMontanari2014,ChernozhukovEtAl2018,MaLuLiu2021,
BellecZhang2023}. We consider using Bocca as the initialization and combine it with the debiasing methods to obtain a better estimation rate. By the lower bound in Theorem~\ref{thm:lower}, this is generally impossible. However, we can consider a smaller model family with a stronger regularity assumption on the design covariance matrix to achieve this.

We first review the idea of debiasing methods.
Split the $n$ observations into a training sample $\mathcal T$ and an evaluation sample $\mathcal E$ such that their sizes are proportional to $n$. Let
$\hat\beta$ be a generic initial estimator using the training sample $\mathcal T$.
We compute the debiasing step on the evaluation sample $\mathcal E$. Write
$\widehat\Sigma_{\mathcal E}
=|\mathcal E|^{-1} \sum_{i \in \mathcal E} X_i X_i^{\top}$ and
$\widehat\sigma_{xy}
=|\mathcal E|^{-1} \sum_{i \in \mathcal E} X_i Y_i$.
Let $\widehat H$ be a generic symmetric covariance matrix estimator and  $\widehat q$ be a generic correction
direction, both constructed using only the design matrix in $\mathcal E$. The generic debiasing estimator is given by
\begin{align}\label{eq:bocca-generic-correction}
 &\hat\beta_1^d = \hat\beta_1
 +\widehat q^\top
   (\widehat\sigma_{xy}-\widehat H\hat\beta) = \widehat q^\top
   \widehat\sigma_{xy} + (e_1 - \widehat H\widehat q)^\top \hat\beta
 \text{ such that } \\
 &\hat\beta_1^d - \beta_1 = \frac{1}{|\mathcal E|} \sum_{i \in \mathcal E} \hat{q}^\top X_i \varepsilon_i + \widehat q^\top(\widehat\Sigma_{\mathcal E}-\widehat H)\beta + (e_1-\widehat H\widehat q)^\top (\hat\beta-\beta). \nonumber
\end{align}
Notice the second term in \eqref{eq:bocca-generic-correction} is a linear functional of the initial estimator to which the Bocca estimator can be applied.
The first term above usually corresponds to the $\sqrt{n}$-rate, the second term is the bias to the sample covariance matrix, and the last remainder term is a linear functional estimation error. Suppose for any unit length loading $v$, the linear functional of the initial estimator has the rate $v^\top (\hat\beta-\beta) = O_P(n^{-\alpha_\tau/2})$ with some $\alpha_\tau > 0$. Then under the following conditions
\begin{equation}\label{eq:debiasing-conditions}
 \|\widehat q\|_2=O_P(1),\quad  \widehat q^\top(\widehat\Sigma_{\mathcal E}-\widehat H)\beta
 =O_P(n^{-1/2}),\text{ and } \|e_1-\widehat H\widehat q\|_2
 =O_P\big(n^{-(1-\alpha_\tau)/2}\big)
,
\end{equation}
we have $\hat\beta_1^d - \beta_1 = O_P(n^{-1/2})$. 
We typically prefer to further require the latter two terms above to be of $o_P$ order to achieve the asymptotic normality for $\hat\beta_1^d$. 

If the design covariance matrix only satisfies the condition in \eqref{eq:model-family}, then {Theorem~\ref{thm:lower} rules out a uniform $O_P(n^{-1/2})$ rate. Hence the initial-estimator rate and the conditions in \eqref{eq:debiasing-conditions} cannot all hold uniformly over the full model family.} Therefore, we consider the following stronger regularity assumption on the design covariance matrix. In fact, as we target to estimate $\beta_1$, we only need to impose the regularity assumption on the cross-covariance vector with the first covariate: 
$\Sigma_{1,-1}=(\Sigma_{12},\ldots,\Sigma_{1p})^\top$.

\begin{assumption}\label{ass:graph-regularity} Let $G=(V,E)$ be a known undirected and connected graph with $p-1$ vertices, $A_G$ be its adjacency matrix, and $L_G=\diag(A_G\mathbf 1)-A_G$ be its graph Laplacian, where $\mathbf 1$ is the vector of ones. There exists $q\in[1,\infty)$ such that $\lambda_{j+1}(L_G) \asymp(j/p)^{2/q}$, for all $1\le j\le p-2.$
      For fixed $R>0$ and
      $r\in(0,1/2)$, we impose
      \begin{equation}\label{eq:graph-covariance}
       \|\Sigma_{1,-1}\|_2^2
       +p^{2/q}\Sigma_{1,-1}^{\top}L_G\Sigma_{1,-1}
       \le R^2p^{-2r}.
      \end{equation}
\end{assumption}

The squared graph Sobolev norm on the left-hand side of \eqref{eq:graph-covariance} has been considered by
\cite{KirichenkoVanZanten2017,KirichenkoVanZanten2018}. They also introduced the same assumption as above on the graph Laplacian eigenvalue growth. \citet{KirichenkoVanZanten2017} illustrated that the $q$-dimensional grid graph has such eigenvalue growth. For the general graph, the order $q$ implies that the graph looks like a $q$-dimensional grid graph. The neighbor-difference regularity in \eqref{eq:graph-covariance} is also used in spatial lattice models \citep{Besag1974}, Gaussian Markov
random fields \citep{RueHeld2005}, signal processing \citep{ShumanEtAl2013}, and
image processing \citep{CheungEtAl2018}. The scale $p^{-2r}$ on the right-hand side of \eqref{eq:graph-covariance} characterizes the weakness of the aggregate correlation between the target covariate and the confounding covariates. Lastly, we emphasize that Assumption~\ref{ass:graph-regularity} is not intended to be a canonical regularity condition of covariance, but rather an example to illustrate how conditions weaker than sparsity can be exploited, in conjunction with Bocca, to construct a debiased estimator achieving the $\sqrt{n}$-rate.

\subsection{Bocca Debiasing Estimator}

Now we consider the concrete case of using the Bocca estimator as the initial estimator.
For any $w\neq0$, define
$\mathcal B_{\mathcal T}(w)
=\|w\|_2\widehat\beta_{w/\|w\|_2}(g_K)$, where the estimator is the same as the one in \eqref{eq:upper} only using the training sample $\mathcal T$.
Set $\mathcal B_{\mathcal T}(0)=0$. Denote $\tau = |\mathcal T|/n$ and then $\alpha_\tau=\alpha(\kappa,\gamma/\tau)$.
Theorem~\ref{thm:upper} implies that $\EE(\mathcal B_{\mathcal T}(w)-w^\top\beta)^2
\lesssim n^{-\alpha_\tau}\|w\|_2^2$.

We then implement the debiasing step in \eqref{eq:bocca-generic-correction} using the validation sample $\mathcal E$. Let $P_m$ be the orthogonal projection onto the first
$m$ eigenvectors of $L_G$, with ties resolved independently of the
data. We intend to use $P_m$ to regularize the sample covariance estimator. Specifically, define
\begin{equation}\label{eq:reg_cov_estimate}
      \widehat\sigma_1=\frac1{|\mathcal E|}\sum_{i\in\mathcal E}X_{i1}^2,
      \qquad
      \widehat\sigma_{1,y}=\frac1{|\mathcal E|}\sum_{i\in\mathcal E}X_{i1}Y_i,
      \qquad
      \widehat\sigma_{1,-1}
      = \frac1{|\mathcal E|}\sum_{i\in\mathcal E} P_m X_{i1}X_{i,-1}.           
\end{equation}
Choose $\widehat H$ with first column
$(\widehat\sigma_1,\widehat\sigma_{1,-1}^{\top})^\top$ and 
$\widehat q=e_1/\widehat\sigma_1$. The remaining columns of
$\widehat H$ do not enter the debiasing estimator. By Assumption~\ref{ass:graph-regularity}, we can show that
$\|(I-P_m)\Sigma_{1,-1}\|_2^2
\lesssim p^{-2r}m^{-2/q}$ for $m<p-1$. This approximation ensures the second condition in \eqref{eq:debiasing-conditions} {provided that $p^{-r}m^{-1/q}=O(n^{-1/2})$}. And the regularized estimator $\widehat\sigma_{1,-1}$ could reduce the dimension from $p-1$ to $m$ and guarantee the third condition in \eqref{eq:debiasing-conditions} {provided that $p^{-r}+\sqrt{m/|\mathcal E|}=O(n^{-(1-\alpha_\tau)/2})$}. See Appendix~\ref{app:pf:thm:graph-debias} for the detailed proof. We will specify the choice of $m$ in Theorem~\ref{thm:graph-debias}.

Replacing the second term in the second equality of
\eqref{eq:bocca-generic-correction} with the Bocca estimator, we have the Bocca debiasing estimator
\begin{equation}\label{eq:graph-two-sample}
 \widehat\beta_{1}^d = \widehat q^\top
 \widehat\sigma_{xy} + \mathcal B_{\mathcal T}(e_1 - \widehat H\widehat q)
 =
 \frac{1}{\widehat\sigma_1}\Big(\widehat\sigma_{1,y}
 -\mathcal B_{\mathcal T}((0,\widehat\sigma_{1,-1}^{\top})^\top)\Big)
 .
\end{equation}
We set $\widehat\beta_{1}^d$ in \eqref{eq:graph-two-sample} to zero when
$\widehat\sigma_1<\ell/2$, recalling $\ell$ is the covariance
eigenvalue lower bound in \eqref{eq:model-family}. This is to avoid the potential numerical singularity from dividing by small $\widehat\sigma_1$ in \eqref{eq:graph-two-sample} and ensure the first condition in \eqref{eq:debiasing-conditions}. As
$\E\widehat\sigma_1=\Sigma_{11}\ge\ell$, the safeguard is
asymptotically inactive.

We next generalize \eqref{eq:graph-two-sample} to the cross-fitting setting. Let $J=1/(1-\tau)$ and assume, for simplicity,
that $J$ and $n\tau$ are integers. Partition the data with $n$ observations into $J$ even splits:
$\mathcal I_1,\ldots,\mathcal I_J$, each of size $n/J$. For each $k \in [J]$, we fit the Bocca estimator in \eqref{eq:graph-two-sample} using $\mathcal T = \mathcal I_k^c$ with $n\tau$ observations and denote it as $\mathcal B_{-k}$. On $\mathcal I_k$, we form
$\widehat\sigma_{1,-1}^{(k)}
=(J/n)\sum_{i\in\mathcal I_k} P_m X_{i1}X_{i,-1}$.
Pool the scalar moments over all observations:
$\widehat\sigma_1=n^{-1}\sum_{i =1}^n X_{i1}^2$ and
$\widehat\sigma_{1,y}=n^{-1}\sum_{i =1}^n X_{i1}Y_i$.
The cross-fitted Bocca debiasing estimator is
\begin{equation}\label{eq:graph-debiased-estimator}
 \widehat\beta_1^d
 =
 \frac{1}{\widehat\sigma_1}\Big(\widehat\sigma_{1,y}
 -\frac{1}{J}\sum_{k=1}^J
 \mathcal B_{-k}\big((0,(\widehat\sigma_{1,-1}^{(k)})^\top)^\top\big)\Big),
\end{equation}
and we set $\widehat\beta_1^d$ to zero if $\widehat\sigma_1<\ell/2$ again.
The following theorem shows the rate of the debiased estimator and its proof is deferred to Appendix~\ref{app:pf:thm:graph-debias}.

\begin{theorem}\label{thm:graph-debias}
      Suppose Assumption~\ref{ass:graph-regularity} holds.
      Choose $J\asymp\log n$, and the
      cross-fitted Bocca debiasing estimator
      \eqref{eq:graph-debiased-estimator} satisfies
      \begin{equation}\label{eq:graph-rate-general}
       \E(\widehat\beta_1^d-\beta_1)^2
       \lesssim
       n^{-1}+p^{-2r}m^{-2/q}
       +n^{-\alpha(\kappa,\gamma)}
       \big(p^{-2r}+{m \log n}/{n}\big).
      \end{equation}      
      In particular, choosing $m=\left\lceil n^{\alpha(\kappa,\gamma)/2}\right\rceil$ yields
      \begin{equation*}
       \E(\widehat\beta_1^d-\beta_1)^2
       \lesssim
       n^{-\min\{1,\,2r+\alpha(\kappa,\gamma)/q\}}.
      \end{equation*}
      \end{theorem}

      \begin{remark}
      \label{rem:graph-chain}
      The choice $m=\lceil n^{\alpha(\kappa,\gamma)/2}\rceil$ does not optimize the rate in \eqref{eq:graph-rate-general}. The optimal choice is $m=\min\big\{p-1,\lceil
      \big({n^{1+\alpha(\kappa,\gamma)-2r}}/{\log n}\big)^{q/(q+2)}
      \rceil\big\}$. However, we still recommend using $m=\lceil n^{\alpha(\kappa,\gamma)/2}\rceil$ to make the procedure adaptive to the regularity parameters $r$ and $q$. Under such a choice of $m$, the debiased estimator
      achieves a strictly better rate than the original Bocca estimator in \eqref{eq:upper} when
      $
       2r>\alpha(\kappa,\gamma)
       (1-1/q),
      $
      and it further achieves the root-$n$ rate
      when $
       2r \ge 1 - \alpha(\kappa,\gamma)/q.
      $
      \end{remark}


\section{Numerical Experiments}\label{sec:simulation}

In this section, we provide numerical simulations to validate the theoretical results. We compare the Bocca estimator with the method-of-moments (MoM) estimator of \citet{ChenLiuMukherjee2024}.
As $\beta_v=v^\top\Sigma^{-1}\E[X_iY_i]$, the  idea of the MoM estimator is to approximate $\Sigma^{-1}$ by a polynomial of $\Sigma$. Let $T_m$ denote the Chebyshev polynomial of degree $m$ and set $y(x)=(L+\ell-2x)/(L-\ell)$. Consider the polynomial of degree $J$
\begin{equation*}
 q_J(x)=\frac1x\left(1-\frac{T_{J+1}(y(x))}{T_{J+1}(y(0))}\right)=\sum_{j=0}^{J}c_j x^j,
\end{equation*}
where $c_j$'s are the corresponding coefficients of the polynomial.
Then $q_J(\Sigma)$ approximates $\Sigma^{-1}$ (see \citet[Lemma~16]{ChenLiuMukherjee2024}).
The MoM estimator is therefore given by
\begin{equation*}
    \widehat\beta_v^{\,\mathrm{MoM}} =\sum_{j=0}^{J}c_j\,\widehat U^{(j)}, \text{ where } \widehat U^{(j)}
 =\frac{(n-j-1)!}{n!}
 \sum_{\substack{i_0,\ldots,i_j\in[n]\\\text{pairwise distinct}}}
 (v^\top X_{i_0})\,(X_{i_0}^\top X_{i_1})\cdots(X_{i_{j-1}}^\top X_{i_j})\,Y_{i_j}.
\end{equation*}
A similar method is also considered in \citet{KongValiant2018} for estimating the learnability of linear regression.

We generate $X=Z\Sigma^{1/2}$ with independent standard Gaussian entries in $Z$, $Y=X\beta+\varepsilon$ with $\varepsilon\sim N(0,I_n)$, and estimate $\beta_1$, that is, $v=e_1$. The covariance is $\Sigma= U^\top\operatorname{diag}(\lambda_1,\ldots,\lambda_p)U$, where $U$ is the orthonormal discrete cosine transform matrix of order $p$. In particular, for any $j \in [p]$, $U_{ij}= p^{-1/2}$ if $i=1$ and $U_{ij}= \sqrt{{2}/{p}}\cos\big({\frac{\pi}{p}(i-1)(j-\tfrac12)}\big)$ if $i>1$, so that $UU^\top = U^\top U=I_p$.
 We consider $L=1$, $\ell=1/\kappa$ and take $\lambda_j=\ell+(L-\ell)(\tfrac{j-1/2}{p})^4$ for all $j \in [p]$. We consider two coefficient vectors of unit norm, where the second one is motivated by the lower bound construction (see Remark~\ref{rem:lower-overview-main}):
\begin{enumerate}
 \item[(i)] $\beta=e_1$, for which $\beta_1=1$ and the response depends on the first covariate only;
 \item[(ii)] $\beta=\Sigma^{-1}e_1/\norm{\Sigma^{-1}e_1}_2$, for which $\beta_{-1}$ is proportional to the direction that confounds $X_1$.
\end{enumerate}
We choose $\delta=1/\log n$ and discretize the contour integral into $M=1024$ nodes following the steps in Section~\ref{sec:computation}. The numerical results are computed over $500$ Monte Carlo replications, except for the QQ experiments, which use $1000$.

\begin{figure}[t]
 \centering
 \includegraphics[width=\textwidth]{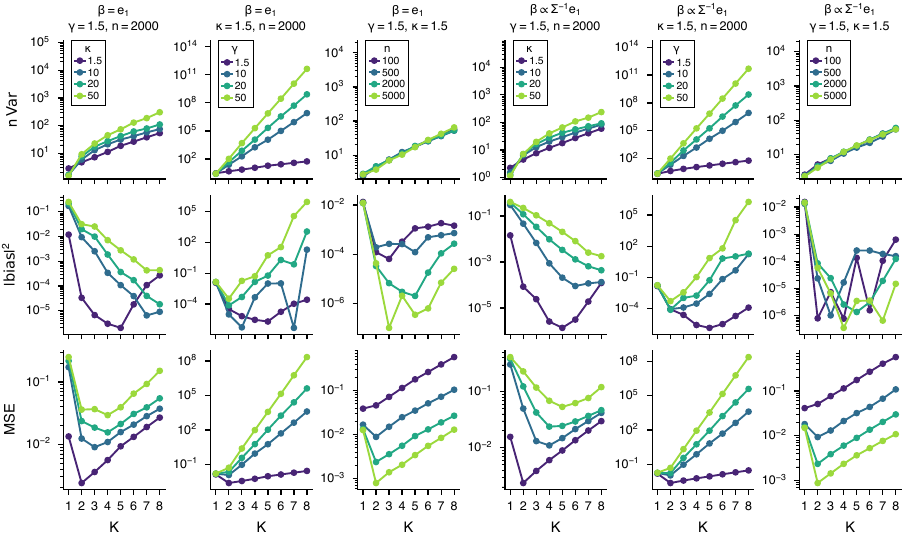}
 \vspace{-25pt}
 \caption{Monte Carlo variance (scaled by $n$), squared bias and mean squared error of $\widehat\beta_1(g_K)$ versus $K$. Columns 1--3 use $\beta=e_1$ and columns 4--6 use $\beta\propto\Sigma^{-1}e_1$; within each triple, the first column fixes $\gamma=1.5$ and $n=2000$ and varies $\kappa$, the second fixes $\kappa=1.5$ and $n=2000$ and varies $\gamma$, and the third fixes $\gamma=\kappa=1.5$ and varies $n$.}
 \label{fig:K}
 \vspace{-10pt}
\end{figure}

\begin{figure}[t]
 \centering
 \includegraphics[width=\textwidth]{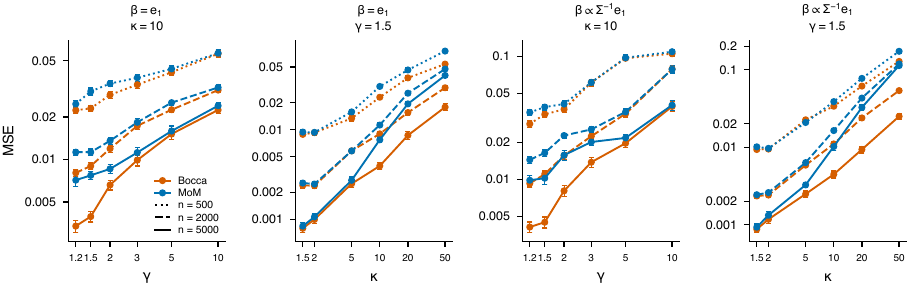}
 \vspace{-25pt}
 \caption{Mean squared error of the Bocca and MoM estimators, as a function of $\gamma$ at $\kappa=10$ and of $\kappa$ at $\gamma=1.5$, for $\beta=e_1$ (left panel) and $\beta\propto\Sigma^{-1}e_1$ (right panel).}
 \label{fig:sweeps}
\end{figure}

\begin{figure}[h]
 \centering
 \includegraphics[width=0.9\textwidth]{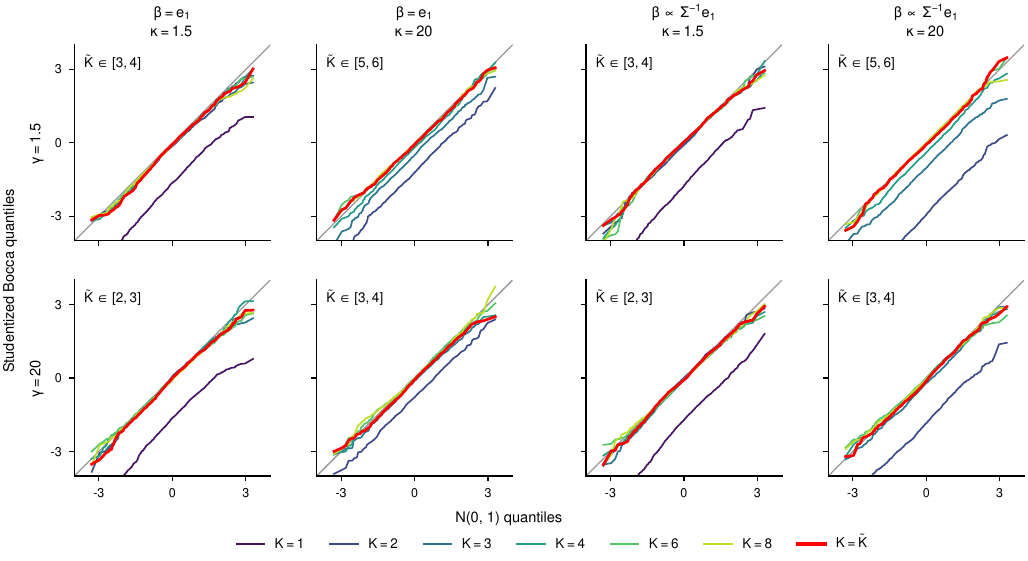}
 \vspace{-13pt}
 \caption{Normal QQ plots of the studentized Bocca estimator $\widehat\sigma(g_K)^{-1}(\widehat\beta_1(g_K)-\beta_1)$ at $n=500$, for $\beta=e_1$ (left panel) and $\beta\propto\Sigma^{-1}e_1$ (right panel); the red lines use $\widetilde K$ in \eqref{eq:normality-degree} with $a=3$.}
 \label{fig:qq}
 \vspace{-10pt}
\end{figure}

Figure~\ref{fig:K} displays the variance-bias tradeoff of the Bocca estimator for different choices of $K$. The empirical variance increases geometrically in $K$, at a rate that grows quickly with $\gamma$ and only mildly with $\kappa$. The squared bias decreases geometrically for small $K$ until it is too small to be distinguished from Monte Carlo error. The mean squared error is unimodal in $K$ and the optimal $K$ tends to be larger as $n$ grows.
Figure~\ref{fig:sweeps} compares the mean squared errors of the Bocca and MoM estimators for different $\gamma$ and $\kappa$ at various sample sizes. For both estimators, we select the best $K$ or $J$ in $\{1, \ldots, 8\}$ that minimizes the mean squared error. The Bocca estimator generally outperforms the MoM estimator. In particular, the Bocca estimator improves on MoM as $\gamma$ decreases, $\kappa$ increases, and $n$ grows. 
Figure~\ref{fig:qq} examines the asymptotic normality in Theorem~\ref{thm:bocca-normality} at $n=500$ over $\gamma,\kappa\in\{1.5,20\}$ with $1000$ Monte Carlo replications, where the variance is estimated by \eqref{eq:normality-estimated-variance} and $\widetilde K$ by \eqref{eq:normality-degree} with $a=3$. For $K=1$ the bias dominates and the quantiles are shifted off the diagonal; the bias then decreases in $K$. The quantiles of the data-driven $\widetilde K$ follow the diagonal in all eight panels, and $\widetilde K$ is larger for $\kappa=20$ and for $\gamma=1.5$, in line with \eqref{eq:normality-degree}.

\vspace{10pt}

\noindent{\bf Use of AI.} The content of this paper and all of the writing were generated by the author first. OpenAI GPT-5.6 and GPT-6 were used for literature research, paper polishing, and identifying mathematical
and typographical errors. Claude Fable-5.1 was used for code generation and visualization.

\bibliographystyle{ims}
\bibliography{ref}


\appendix 
\newpage

\section*{Appendix}

\noindent{\bf Notation in Appendix.}
For real $t,a$, the notation $t\downarrow a$ (respectively, $t\uparrow a$) denotes convergence to $a$ from above (respectively, below); in particular, $t\downarrow0$ means $t\to0$ with $t>0$.
The symbol $0_+$ indicates a limit through positive values; for $x\in\R$, $f(x+\ii0_+)=\lim_{\eta\downarrow0}f(x+\ii\eta)$ denotes the boundary value from the upper half-plane whenever this limit exists.
We write $\lfloor t\rfloor=\max\{k\in\mathbb Z:k\le t\}$ and $\lceil t\rceil=\min\{k\in\mathbb Z:k\ge t\}$.
For a set or event $A$, $\ind_A$ is its indicator, equal to one on $A$ and zero on its complement $A^c$.
For a complex random variable $T$ and $1\le r<\infty$, write $\norm T_{L^r}=(\E|T|^r)^{1/r}$.
For any nonempty set $\mathcal C\subseteq\C$ and bounded function $\psi:\mathcal C\to\C$, write $\norm\psi_{\infty,\mathcal C}=\sup_{z\in\mathcal C}|\psi(z)|$.
In particular, for a real interval $I$, this gives $\norm\psi_{\infty,I}=\sup_{x\in I}|\psi(x)|$.
We use $\norm A_\op$ and $\norm A_F=(\tr A^*A)^{1/2}$ for the operator and Frobenius norms, where $A^*=\overline A^\top$.
The spectrum of $A$ is denoted by $\spec(A)$; $\rank(A)$ is its rank and $\ker A=\{u:Au=0\}$ its null space.
For $z\in\C$ and a nonempty set $E\subseteq\C$, we write $\dist(z,E)=\inf_{w\in E}|z-w|$ for the Euclidean distance from $z$ to $E$; for two nonempty sets, $\dist(E,F)=\inf_{z\in E,w\in F}|z-w|$.
For nonempty compact sets $E,F\subseteq\C$, the Hausdorff distance $d_{\mathrm H}(E,F)=\max\{\sup_{z\in E}\dist(z,F),\sup_{w\in F}\dist(w,E)\}$ is the largest distance from a point in either set to the other set.
For a finite real signed measure $\eta$, $|\eta|$ is its total variation measure, $\supp\eta$ is the closed support of $|\eta|$, and $\norm\eta_\TV=|\eta|(\R)=\int_{\R}|\eta|(\dd x)$ is its total variation mass; thus $\norm{P-Q}_\TV=2\operatorname{TV}(P,Q)$.
Throughout the Appendix, we write $\sup_A f$ for $\sup_{z\in A}f(z)$, and similarly for $\inf$.

\section{Proofs of Upper Bounds}\label{app:mainproofs}

\subsection{The Generic Boundary-Norm Bound}\label{app:pf:thm:generic}
\begin{proof}[Proof of Theorem~\ref{thm:generic}]
Set $\psi(\mu)=(1-g(\mu))/\mu$ and $\mathcal I_\psi H=(2\pi\ii)^{-1}\oint_{\Gamma_\delta}\psi(\mu)H(\mu)\dd\mu$. We have
\begin{align*}
 (\widehat\beta_v(g)-\beta_v)^2
 &\le 2|-\mathcal I_\psi F-\beta_v|^2
       +2|\mathcal I_\psi(\widehat F-F)|^2
 =2|v^\top g(\Sigma)\beta|^2
       +2|\mathcal I_\psi(\widehat F-F)|^2.
\end{align*}
We next bound the first bias term and second variance term to obtain the desired rate.

\proofstep{Bias term.}
Since $g(0)=1$, $\psi$ is holomorphic at zero. The contour identity for \eqref{eq:population-field} gives
\[
 -\mathcal I_\psi F
 =\int_{[\ell,L]}(1-g(\lambda))\,\nu(\dd\lambda)
 \text{ and }
 |-\mathcal I_\psi F-\beta_v|^2
 =|v^\top g(\Sigma)\beta|^2
 \le B_0^2\sup_{\lambda\in [\ell,L]}|g(\lambda)|^2
 =B_0^2 B(g).
\]
The inequality above follows from $\spec(\Sigma)\subseteq[\ell,L]$,
$\norm v_2=1$, and $\norm\beta_2\le B_0$.

\proofstep{Variance term.} We need two lemmas on $\hat{F}$ to bound the variance term.
The following lemma states the rate of $\hat{F}$; its proof is given in Appendix~\ref{sec:pf:lem:field-moments}.

\begin{lemma}[$L^2$-rate]\label{lem:field-moments}
Let $1/\log n\le\delta\le2/\log n$. Given a deterministic and continuous  $\psi$ on $\Gamma_\delta$, there is an $X$-measurable event $\mathcal A$ with $\PP(\mathcal A^c)\le C_Dn^{-D}$, for any prescribed $D>0$, on which the branch in \eqref{eq:empirical-calibration} exists and
\begin{equation}\label{eq:field-moment}
 \E\left[\ind_{\mathcal A}
 \left|\frac1{2\pi\ii}\oint_{\Gamma_\delta}
 \psi(\mu)(\widehat F(\mu)-F(\mu))\dd\mu\right|^2\right]
 \le \frac Cn\norm\psi_{\infty,\Gamma_\delta}^2.
\end{equation}
\end{lemma}

In particular, the event $\mathcal A$ is defined in \eqref{eq:averaged-uniform}.
The following lemma bounds the fourth moment of $\hat{F}$; its proof is given in Appendix~\ref{sec:pf:lem:global-fourth-moment}.

\begin{lemma}[Fourth moment]\label{lem:global-fourth-moment}
      Let $1/\log n\le\delta\le2/\log n$. For a continuous $\psi$ on $\Gamma_\delta$, we have
      \[
       \E\left|\frac1{2\pi\ii}\oint_{\Gamma_\delta}
             \psi(\mu)\widehat F(\mu)\dd\mu\right|^4
       \le Cn^4\norm\psi_{\infty,\Gamma_\delta}^4.
      \]
      \end{lemma}

The maximum modulus principle gives $V(g)\ge|g(0)|^2=1$. Since the contour $\Gamma_\delta$ stays away from zero, choosing $D=8$ in Lemma~\ref{lem:field-moments} gives an $X$-measurable event $\mathcal A$ such that
\[
 \norm\psi_{\infty,\Gamma_\delta}^2
 \lesssim (1+\sqrt{V(g)})^2\lesssim V(g)
 \text{ and thus }
 \E[\ind_{\mathcal A}|\mathcal I_\psi(\widehat F-F)|^2]
 \lesssim \frac{V(g)}n \text{ with } \PP(\mathcal A^c)\lesssim n^{-8}.
\]
As $|\beta_v|\le B_0$ and $V(g)\ge1$, Lemma~\ref{lem:global-fourth-moment} gives $\E|\widehat\beta_v(g)-\beta_v|^4\lesssim n^4V(g)^2$. Therefore, we have
\[
 {\E[(\widehat\beta_v(g)-\beta_v)^2]}
 \le \E [\ind_{\mathcal A}(\widehat\beta_v(g)-\beta_v)^2]
       +(\E|\widehat\beta_v(g)-\beta_v|^4)^{1/2}\PP(\mathcal A^c)^{1/2}
 \lesssim \frac{V(g)}n + B(g),
\]
which proves \eqref{eq:generic-rate}.
\end{proof}

\subsection{Two Constants and Extension Through the Spectrum}\label{app:pf:thm:two}
\begin{proof}[Proof of Theorem~\ref{thm:two}]
There are two parts in Theorem~\ref{thm:two}. We first prove the variance-bias trade-off in \eqref{eq:two-constants} and then prove the variance and bias bounds for $g_K$ in \eqref{eq:gK-bound}.

\proofstep{The variance-bias trade-off in \eqref{eq:two-constants}.} As we explain in Remark~\ref{rem:two-constants}, the variance-bias trade-off in \eqref{eq:two-constants} is essentially due to the two-constant theorem below.

\begin{theorem}[Two-constant theorem
      {\citep[Theorem~4.3.7]{ransford1995potential}}]
    \label{thm:two-constants}
    Let $D\subset\mathbb C$ be a domain with boundary containing a compact, connected subset with more than one point,
    and let $E\subseteq\partial D$ be a Borel set.
    Suppose that $u$ is subharmonic on $D$ and that there
    exist constants $m,M\in\mathbb R$ with $m\le M$ such that
    \[
    u(z) \le M \text{ for any } z\in D
    \text{ and }
    \limsup_{z\to\zeta}u(z)\le m,
    \text{ for any } \zeta\in E.
    \]
    Then, for every $z\in D$, we have
\begin{equation}\label{eq:two-constants-theorem}
      u(z)\le
      m\,\omega_D(z,E)
      +M\bigl(1-\omega_D(z,E)\bigr),
\end{equation}
    where $\omega_D(z,E)$ denotes the harmonic measure
    of $E$ relative to $D$. When $E$ is a union of separated
    boundary components, $\omega_D(z,E)$ is defined as the solution to
\begin{equation}\label{eq:harmonic-measure}
\Delta\omega_D(z,E)=0, \text{ for } z\in D, \text{ and }
\omega_D(z,E)=1, \text{ for }   z\in E, 
\omega_D(z,E)=0, \text{ for }  z\in\partial D\setminus E.
\end{equation} 
\end{theorem}

The slit disk
$\mathcal A_\delta=\Ddisk_\delta\setminus[\ell,L]$
contains zero and has boundary components $[\ell,L]$ and $\Gamma_\delta$.
By Lemma~\ref{lem:geometry}\textup{(i)}, $u_\delta$ in \eqref{eq:potential}
is harmonic in $\mathcal A_\delta$ and takes boundary values zero and one
on $[\ell,L]$ and $\Gamma_\delta$, respectively. Let $D = \mathcal A_\delta$ and $E = [\ell,L]$ or $ \Gamma_\delta$. 
The uniqueness of the Dirichlet problem in \eqref{eq:harmonic-measure} gives
\[
\begin{aligned}
\omega_{\mathcal A_\delta}(z,\Gamma_\delta)
=u_\delta(z),
\omega_{\mathcal A_\delta}(0,\Gamma_\delta)=\alpha_\delta,
\omega_{\mathcal A_\delta}(z,[\ell,L])
=1-u_\delta(z),
\omega_{\mathcal A_\delta}(0,[\ell,L])=1-\alpha_\delta,
\end{aligned}
\]
recalling $\alpha_\delta=u_\delta(0)$.


To better understand the essence of the bias--variance trade-off,
we next give a self-contained proof of \eqref{eq:two-constants} without explicitly using the two-constant theorem. In fact, the following proof implicitly proves the two-constant theorem.
For $\varepsilon>0$, define
\[
v_\varepsilon(z)
=\log(|g(z)|^2+\varepsilon)
 -(1-u_\delta(z))\log(B(g)+\varepsilon)
 -u_\delta(z)\log(V(g)+\varepsilon).
\]
Since $\log(|g|^2+\varepsilon)$ is subharmonic
and $u_\delta$ is harmonic, $v_\varepsilon$ is
subharmonic on $\mathcal A_\delta$.
Moreover, $u_\delta=0$ on $[\ell,L]$ and
$u_\delta=1$ on $\Gamma_\delta$. By the definitions
of $B(g)$ and $V(g)$, we have
$v_\varepsilon\le 0$ on both $\Gamma_\delta$ and $[\ell,L]$.
The subharmonic maximum principle \citep[Chapter~6, Sections~4.1--4.2]{Ahlfors1979} therefore gives
$v_\varepsilon\le 0$ throughout $\mathcal A_\delta$.

Evaluating $v_\varepsilon$ at zero, using $g(0)=1$ and
$u_\delta(0)=\alpha_\delta$, yields
\[
\log(1+\varepsilon)
\le (1-\alpha_\delta)\log(B(g)+\varepsilon)
   +\alpha_\delta\log(V(g)+\varepsilon).
\]
Exponentiating and letting $\varepsilon\downarrow0$ yields
\[
1=|g(0)|^2
\le B(g)^{1-\alpha_\delta}V(g)^{\alpha_\delta},
\]
which proves \eqref{eq:two-constants}.
We can also see that equality in \eqref{eq:two-constants} holds if and only if $\log|g|^2$ is
harmonic on $\mathcal A_\delta$ and takes the constant boundary values $\log B(g)$ on $[\ell,L]$ and $\log V(g)$ on $\Gamma_\delta$. The proof above essentially proves the two-constant theorem for the slit disk case.

\proofstep{Variance and bias bounds for $g_K$ in
\eqref{eq:gK-bound}.}
Next we establish the bounds for the filter.
Lemma~\ref{lem:geometry}\textup{(iv)} gives that $g_K$ is holomorphic
on a neighborhood of $\overline{\Ddisk_\delta}$ and $g_K(0)=1$, as
required by Theorem~\ref{thm:generic}. Recall 
$h_\delta=\pi\Kell(k^2)/\Kell(1-k^2)$.
By Lemma~\ref{lem:geometry}\textup{(ii)},
$q_\delta=h_\delta U_\delta$ takes values in $\C/(2\pi\ii\mathbb Z)$,
with $0<\operatorname{Re}q_\delta<h_\delta$ on
$\mathcal A_\delta$. Moreover,
$\operatorname{Re}q_\delta=h_\delta u_\delta$ and its real value at zero is
$q_\delta(0)=h_\delta\alpha_\delta>0$.
Since $|\cosh(x+\ii y)|\le e^{|x|}$ and $\cosh x\ge e^x/2$ for $x\ge0$,
\begin{equation}\label{eq:gK-bound-endpoint}
 |g_K(z)|
 \le2\exp\left(K\bigl(\operatorname{Re}q_\delta(z)
                            -q_\delta(0)\bigr)\right),
 \qquad z\in\Ddisk_\delta\setminus[\ell,L].
\end{equation}
This inequality is independent of the chosen representative of
$q_\delta(z)$ and holds on the boundary by continuity.
On $[\ell,L]$, $\operatorname{Re}q_\delta=0$, and on $\Gamma_\delta$, $\operatorname{Re}q_\delta=h_\delta$.
Thus the definitions in \eqref{eq:B-V} give
\begin{equation}\label{eq:proof-filter-scales}
 \begin{aligned}
 B(g_K)\le4e^{-2Kq_\delta(0)}
        =4e^{-2Kh_\delta\alpha_\delta},
 V(g_K)\le4e^{2K(h_\delta-q_\delta(0))}
        =4e^{2Kh_\delta(1-\alpha_\delta)}.
 \end{aligned}
\end{equation}
These are the bounds in \eqref{eq:gK-bound}. Raising them to powers
$1-\alpha_\delta$ and $\alpha_\delta$ and multiplying cancels the
exponential factors, giving
$B(g_K)^{1-\alpha_\delta}V(g_K)^{\alpha_\delta}\le4$.
\end{proof}

\subsection{The Optimized Estimation Rate}\label{app:pf:thm:upper}
\begin{proof}[Proof of Theorem~\ref{thm:upper}]
The integer degree in \eqref{eq:scalar-degree} satisfies  $2Kh_\delta=\log n+O(1)$, where $h_\delta=\pi\Kell(k^2)/\Kell(1-k^2)$. Using Theorem~\ref{thm:generic} and \eqref{eq:proof-filter-scales}, we have
\begin{align*}
      &\sup_{P\in\mathcal P(n,p,\gamma,\kappa)}
 \E_P(\widehat\beta_v(g_K)-\beta_v)^2
 \lesssim B(g_K)+\frac{V(g_K)}n\\
 &\qquad \le 4e^{-2Kh_\delta\alpha_\delta}
       +\frac4ne^{2Kh_\delta(1-\alpha_\delta)}
\lesssim n^{-\alpha_\delta}
 =n^{-\alpha(\kappa,\gamma)}
   e^{(\alpha(\kappa,\gamma)-\alpha_\delta)\log n}
 \lesssim n^{-\alpha(\kappa,\gamma)}.
\end{align*}
The last step is by Lemma~\ref{lem:geometry}\textup{(iii)} and $\delta\log n\le2$, proving \eqref{eq:upper}.
\end{proof}
\section{Proofs of Lower Bounds}\label{app:mainproofs-lower}

\subsection{The Scalar Minimax Lower Bound}\label{app:mainproofs-lower-main}
\begin{proof}[Proof of Theorem~\ref{thm:lower}]
We first reduce to $v=e_1$ and $L=1$ within the Gaussian submodel, then apply Le Cam's two-point method to the priors constructed below.

\proofstep{Reduction to $v=e_1$ and $L=1$.}
Choose an orthogonal matrix $O$ with $Oe_1=v$ and some scalar $c_0>0$. Given some normalized Gaussian regression variables $(X^\circ,Y^\circ)$ in the reduction regime, we can transform them to the general regime through  
\begin{equation}\label{eq:reduction-transformation}
      X=\sqrt L\,X^\circ O^\top, Y=c_0Y^\circ, \text{ with } \Sigma=LO\Sigma^\circ O^\top,
      \beta=\frac{c_0}{\sqrt L}O\beta^\circ,
      \sigma=c_0\sigma^\circ,
      \beta_v=\frac{c_0}{\sqrt L}\beta^\circ_1.   
\end{equation}

Lemma~\ref{lem:lower-priors} will supply a regression submodel $(X^\circ,Y^\circ)$ with normalized parameters with covariance spectrum in $[\kappa^{-1},1]$, $\norm{\beta^\circ}_2\le C\sqrt\tau$, and $(\sigma^\circ)^2\le(1+\kappa^{-1})/2$. Choose $c_0>0$ so that $c_0\sqrt{(1+\kappa^{-1})/2}\le\bar\sigma$, and then choose the fixed $\tau>0$ small enough that $c_0C\sqrt{\tau/L}\le B_0$. The transform in \eqref{eq:reduction-transformation} then gives $(X,Y)$ with parameters belonging to $\mathcal G(n,p,\gamma,\kappa)$.

For any estimator $T$ in the original experiment in $\mathcal G(n,p,\gamma,\kappa)$, define
\[
      T^\circ(X^\circ,Y^\circ)=\sqrt L\,T(\sqrt L X^\circ O^\top,c_0Y^\circ)/c_0,\quad\text{so }\E_\theta(T-\beta_v)^2
 =\frac{c_0^2}{L}
   \E_{\theta^\circ}(T^\circ-\beta^\circ_1)^2.
\]
The transformation preserves total variation, so it remains to prove the lower bound for $v=e_1$ and $L=1$.

\proofstep{Le Cam's two-point method.}
Let $\Theta^\circ$ be the normalized Gaussian parameter class described above for $v=e_1$ and $L=1$. The priors $\mathbb Q_0,\mathbb Q_1$ in Lemma~\ref{lem:lower-priors} are supported on $\Theta^\circ$ and satisfy $\mathbb Q_j(\beta_1=b_j)=1$. If $P_j$ is the induced data law, Le Cam's two-point method \citep[Lemma~1]{Yu1997} gives
\begin{align}\label{eq:lower-le-cam}
 \inf_T\sup_{\theta\in\Theta^\circ}\E_\theta(T-\beta_1)^2
 &\ge\inf_T\max_{j=0,1}\E_{P_j}(T-b_j)^2
 \ge\frac{(b_1-b_0)^2}{8}
       (1-\operatorname{TV}(P_0,P_1))
       \gtrsim n^{-\alpha(\kappa,p/n)}.
\end{align}
The last line uses the separation and $\operatorname{TV}(P_0,P_1)\le1/3$ in Lemma~\ref{lem:lower-priors}. The reduction above proves \eqref{eq:lower}. Combining it with $\mathcal G(n,p,\gamma,\kappa)\subseteq\mathcal P(n,p,\gamma,\kappa)$ and Theorem~\ref{thm:upper} gives the stated minimax comparisons.
\end{proof}
\subsection{Construction of the Two Priors}\label{app:lower-construction}
We construct the priors for $v=e_1$ and the spectral interval $[\kappa^{-1},1]$. Write the final design as $X=(X_1,X_{-1})$. We obtain $(X,Y)$ from a preliminary regression $(\mathsf X,\mathsf Y)$ with $d=p-1$ covariates.

\paragraph{Dimension and geometry.}
Set $\gamma^\circ=d/n$, $\upsilon=n^{-1/8}$, and $a=\kappa^{-1}+\upsilon$. Define
\begin{equation}\label{eq:lower-rho-h}
      \rho=\frac{\sqrt{\gamma^\circ}-1+a}{\sqrt{\gamma^\circ}+1-a},
      \qquad h =\frac{\pi\Kell(\rho^2)}{\Kell(1-\rho^2)}.    
\end{equation}
Let $\mathfrak U$ denote the map $U_\delta$ in
\eqref{eq:normalized-U} with $\delta=0$, $\ell=a$, $L=1$,  $\gamma=\gamma^\circ$, and $R_\delta = \sqrt{\gamma^\circ}$.
Apply the construction of
Lemma~\ref{lem:geometry}\textup{(i)--(ii)} and set
$\mathfrak u=\operatorname{Re}\mathfrak U$ and
$\mathfrak q=h\mathfrak U\in\C/(2\pi\ii\mathbb Z)$.
Then $\mathfrak u$ is zero on $[a,1]$ and one on $\Gamma_0$. By \eqref{eq:explicit-potential} and \eqref{eq:k-T-h}, we have
\[
 \mathfrak q(0)=\frac{\pi}{\Kell(1-\rho^2)}
 \left[
 \Kell(\rho^2)-
 \Fell\left(
 \arcsin\frac{\sqrt{\gamma^\circ}-1}
                  {\rho(\sqrt{\gamma^\circ}+1)}
 \,\middle|\,\rho^2\right)\right].
\]
We use the integer degree
\begin{equation}\label{eq:lower-degree}
 K=\left\lceil{\log n}/(2h)\right\rceil = \left\lceil\frac{\log n \Kell(1-\rho^2)}{2\pi\Kell(\rho^2)}\right\rceil.
\end{equation}

\paragraph{Spectral measures.}
For $x>0$, define
\begin{equation}\label{eq:lower-kernel}
 \mathsf K_x(\lambda)=
 \frac{\lambda^2(\gamma^\circ x-\gamma^\circ+1-\lambda)}
 {\lambda(\gamma^\circ-1+\lambda)+\gamma^\circ x(1-\lambda)}.
\end{equation}
Let $x_\pm=(1\pm(\gamma^\circ)^{-1/2})^2$, and set
\begin{equation}\label{eq:edge-weight}
 \mathcal W(x)=
 \frac1{\sqrt{(x-x_-)^2+K^{-4}}}
 +\frac1{\sqrt{(x-x_+)^2+K^{-4}}}.
\end{equation}
We then need two covariance spectrum measures $\nu_0,\nu_1$ on $(a,1)$ supplied by the following lemma, proved in Appendix~\ref{app:pf:lem:lower-measure}.

\begin{lemma}[Covariance spectrum priors]\label{lem:lower-measure}
For every integer $K\asymp\log n$, there are positive finite atomic measures $\nu_0,\nu_1$ on $(a,1)$ such that
\begin{align}
 &|\operatorname{supp}(\nu_0)\cup\operatorname{supp}(\nu_1)|\le CK,
 \qquad \nu_0(\R)+\nu_1(\R)\le C,
 \label{eq:lower-measure-size}\\
 &\int_a^1\left(\frac{\lambda}{1-\lambda}
           +\frac{\lambda}{\lambda-a}\right)\nu_j(\dd\lambda)\le C,
 \text{ for } j=0,1,
 \label{eq:lower-measure-endpoints}\\
 &\int_a^1\lambda\,(\nu_1-\nu_0)(\dd\lambda)
       \ge c e^{-K\mathfrak q(0)}.
 \label{eq:lower-measure-separation}
\end{align}
For $x\in[x_--\upsilon,x_++\upsilon]$, they also satisfy
\begin{equation}\label{eq:lower-transform}
 \left|\int\mathsf K_x(\lambda)(\nu_1-\nu_0)(\dd\lambda)\right|^2
 \le C e^{-2Kh}\mathcal W(x).
\end{equation}
\end{lemma}

\begin{remark}
      The size bound \eqref{eq:lower-measure-size} limits the number of mixing
      directions to $O(\log n)$. The endpoint
      bound \eqref{eq:lower-measure-endpoints} keeps the final covariance spectrum in $[\kappa^{-1},1]$; see the Schur complements
      in \eqref{eq:lower-prior-upper-schur}--\eqref{eq:lower-prior-lower-schur} for details. The vital trade-off is between \eqref{eq:lower-measure-separation} and \eqref{eq:lower-transform}. The lower bound \eqref{eq:lower-measure-separation} ensures that the values of $\beta_1$ under the two mixed hypotheses $P_0$ and $P_1$ are separated by the desired rate. The upper bound  \eqref{eq:lower-transform} controls $\operatorname{TV}(P_0,P_1)$, needed for Le Cam's two-point method. See Lemma~\ref{lem:lower-priors} for the precise statement and Remark~\ref{rem:lower-measure-trade-off} for more details on the interpretation of \eqref{eq:lower-transform} and an overview of the construction of the two priors.
\end{remark}

Take these measures $\nu_0,\nu_1$ for $K$ in \eqref{eq:lower-degree}, and enumerate their joint support as $\{z_1,\ldots,z_s\}\subset(a,1)$.
Write
\begin{equation}\label{eq:lower-atomic-weights}
 \nu_j=\sum_{k=1}^s w_{jk}\delta_{z_k},\qquad
 w_{jk}=\nu_j(\{z_k\}),\qquad j=0,1.
\end{equation}
In fact, in the proof of Lemma~\ref{lem:lower-measure}, we construct $z_1, \ldots, z_s$ as the $s=K$ simple zeros of $f=\cosh(K\mathfrak q)$ on $(a,1)$. Moreover, the signed weights $\eta_k = -1/f'(z_k)$ give the atomic weights: $w_{jk} = {w(z_k)}\eta_k$ if $\eta_k > 0$ and $j = 0$, $w_{jk} = -{w(z_k)}\eta_k$ if $\eta_k < 0$ and $j = 1$, and all other $w_{jk} = 0$, {where $w(z)=(z-a)(1-z)/z^2$}. Also see Remark~\ref{rem:lower-measure-trade-off} for the high-level explanation of this construction.
Define
\begin{equation}\label{eq:lower-prior-precision-parameters}
 \mathsf D=\diag\left(\frac{1-z_1}{z_1},\ldots,
                     \frac{1-z_s}{z_s}\right),
 \qquad
 t_k=\frac{d-s-1-n(1-z_k)}d.
\end{equation}
Here $w_{jk}=0$ if $z_k$ belongs only to the other measure. By
\eqref{eq:lower-measure-size}, $s=O(K)=O(\log n)$; since
$\gamma^\circ>1$ stays away from one under \eqref{eq:aspect-regime},
$\min_{1\le k\le s}t_k\ge c>0$ for large $n$.

\paragraph{Mixing distribution.}
On $V\in\R^{d\times s}$, consider the density of the mixing law of $V$
\begin{equation}\label{eq:Vprior}
 \pi(V)=\frac1{Z_\pi}
 \det(I_d+V\mathsf D V^\top)^{-n/2}
 \exp\left(-\frac d2
       \tr\bigl(\diag(t_1,\ldots,t_s)V^\top V\bigr)\right).
\end{equation}
The normalizer $Z_\pi$ is finite because $t_k>0$. The determinant factor in \eqref{eq:Vprior} follows the Gaussian mixing construction of \citet[Appendix~D.2.1]{KongValiant2018} and \citet[proof of Theorem~4.5]{VerzelenGassiat2018}.
Put
\[
 \Sigma_-(V)=(I_d+V\mathsf D V^\top)^{-1},
 \qquad
 \mathsf Q=V^\top V,\qquad
 \mathcal C=\{\|\mathsf Q-I_s\|_\op\le\upsilon^2\},
 \qquad
 \pi_{\mathcal C}=\pi(\,\cdot\mid\mathcal C).
\]
Lemma~\ref{lem:prior-concentration} bounds $\pi(\mathcal C^c) = o(n^{-D_0})$ for every fixed $D_0>0$. Thus the columns of $V$ are nearly orthonormal with high probability. Notice that the mixing law of $\Sigma_-(V)$ is common to both hypotheses.

\paragraph{The regression experiments.}
Let $s_0=(1+\kappa^{-1})/4$. For a small enough $\tau>0$, set
\begin{equation}\label{eq:lower-auxiliary-parameters}
 \vartheta_j=(\sqrt{\tau w_{j1}},\ldots,\sqrt{\tau w_{js}})^\top,
 \quad
 \varpi_j=s_0-\tau\int\lambda\,\nu_j(\dd\lambda),
 \quad
 b_j=\frac{\varpi_j-s_0}{\varpi_j+s_0}, \quad \text{ for } j=0,1.
\end{equation}
The mass bound \eqref{eq:lower-measure-size} permits a choice of $\tau$ for which $s_0/2\le\varpi_j\le s_0$. Given $V$, generate independent observations
 \begin{equation}\label{eq:auxiliary-regression-observations}
      \mathsf X_i\sim N_d(0,\Sigma_-(V)),\qquad
      \mathsf Y_i=\mathsf X_i^\top V\vartheta_j+\varepsilon_i^\circ,
      \qquad
      \varepsilon_i^\circ\sim N(0,\varpi_j),
 \end{equation}
with noise independent of $\mathsf X$. Denote the resulting data laws mixed over $\pi$ and $\pi_{\mathcal C}$ by $\widetilde P_j$ and $\widetilde P_j^{\mathcal C}$.

Generate $\xi\sim N(0,s_0I_n)$ independently and, under both hypotheses, apply
\begin{equation}\label{eq:final-transform}
 X_1=\mathsf Y+\xi,\qquad X_{-1}=\mathsf X,\qquad Y=\mathsf Y-\xi.
\end{equation}
Conditional on $V$, the final Gaussian regression has parameters, for $j=0,1$,
\begin{align}
 \Sigma_j(V)&=
 \begin{pmatrix}
 \vartheta_j^\top V^\top\Sigma_-(V)V\vartheta_j+\varpi_j+s_0
       &\vartheta_j^\top V^\top\Sigma_-(V) \notag\\
 \Sigma_-(V)V\vartheta_j&\Sigma_-(V)
 \end{pmatrix},\notag\\
 \beta^{(j)}(V)&=
 \bigl(b_j,(1-b_j)(V\vartheta_j)^\top\bigr)^\top,
 \qquad
 \sigma_j^2=\frac{4\varpi_js_0}{\varpi_j+s_0}. \label{eq:final-parameters}
\end{align}
Let $\mathbb Q_j$ be the law of these parameters under $V\sim\pi_{\mathcal C}$, and let $P_j$ be the induced law of the full data $(X,Y)$.

\begin{remark}\label{rem:regression-construction}
The regression construction in \eqref{eq:auxiliary-regression-observations} and \eqref{eq:final-transform} can be summarized in the following diagram:
\begin{center}
      \begin{tikzpicture}[>=Latex,
        var/.style={inner sep=2pt},
        lab/.style={font=\small, inner sep=1.5pt},
        law/.style={anchor=west, inner sep=2pt}]
       \def\rowsep{1.3}
       \def\lawx{3.0}
       \node[var] (Xm) at (0,0)           {$X_{-1}=\mathsf X$};
       \node[var] (Ys) at (0,-\rowsep)    {$\mathsf Y$};
       \node[var] (X1) at (-1.1,-2*\rowsep) {$X_1$};
       \node[var] (Y)  at (1.1,-2*\rowsep)  {$Y$};
       \draw[->] (Xm) -- node[lab, left, xshift=-1pt]  {$\mathsf X_i^\top V\vartheta_j$}
                         node[lab, right, xshift=1pt] {$+\,\varepsilon_i^\circ$} (Ys);
       \draw[->] (Ys) -- node[lab, above left]  {$+\,\xi$} (X1);
       \draw[->] (Ys) -- node[lab, above right] {$-\,\xi$} (Y);
       \node[law] at (\lawx,0)
         {$\mathsf X_i\sim N_d(0,\Sigma_-(V))$, $\varepsilon_i^\circ\sim N(0,\varpi_j)$;};
       \node[law] at (\lawx,-\rowsep)
         {$\mathsf Y\mid X_{-1}\sim N(\mathsf XV\vartheta_j,\ \varpi_jI_n)$,
          $X_1\mid X_{-1}\sim N(\mathsf XV\vartheta_j,\ (\varpi_j+s_0)I_n)$;};
       \node[law] at (\lawx,-2*\rowsep)
         {$Y\mid X\sim N\bigl(\mathsf XV\vartheta_j+b_j(X_1-\mathsf XV\vartheta_j),\
           \sigma_j^2I_n\bigr)$, $\sigma_j^2=\frac{4\varpi_js_0}{\varpi_j+s_0}$.};
      \end{tikzpicture}
      \end{center}
Therefore, the formula for $b_j$ in \eqref{eq:lower-auxiliary-parameters} comes from the residuals $\varepsilon^\circ+\xi$ in $X_1$ and
$\varepsilon^\circ-\xi$ in $Y$: $b_j=
\frac{\Cov(\varepsilon_i^\circ+\xi_i,\varepsilon_i^\circ-\xi_i)}
     {\Var(\varepsilon_i^\circ+\xi_i)}
=\frac{\varpi_j-s_0}{\varpi_j+s_0}$. We construct $\varpi_j=s_0-\tau\int\lambda\,\nu_j(\dd\lambda)$ with $s_0 = (1+\kappa^{-1})/4$ to ensure that, for any row $i$, the conditional variance $\Var({X_{i1}}\mid V) = 2s_0 + o(1)$ under $\mathcal C$ equals the midpoint of the spectrum interval $[\kappa^{-1},1]$ under both hypotheses, up to a small error; see \eqref{eq:lower-X1-variance} for details. On the other side, we perturb $s_0$ by the average of the spectral priors $\nu_0, \nu_1$ in
Lemma~\ref{lem:lower-measure} to create the
discrepancy between $b_0$ and $b_1$.

The regression parameters in \eqref{eq:final-parameters} can also be computed from the above diagram
\[
 \E[Y\mid X,V]
 =\mathsf XV\vartheta_j+b_j(X_1-\mathsf XV\vartheta_j)
 =b_jX_1+(1-b_j)X_{-1}V\vartheta_j,
\]
and the covariance $\Sigma_j(V)$ in \eqref{eq:final-parameters} is from its inverse
\[
 \Sigma_j(V)^{-1}
 =\begin{pmatrix}0&0\\0&I_d+V\mathsf D V^\top\end{pmatrix}
 +\frac1{\varpi_j+s_0}
   \begin{pmatrix}1\\-V\vartheta_j\end{pmatrix}
   \begin{pmatrix}1&-\vartheta_j^\top V^\top\end{pmatrix}.
\]
We can see from the above two equations that our construction makes the nuisance linear coefficient $\beta^{(j)}_{-1}(V)$ proportional to the corresponding part of the precision matrix in the direction $V\vartheta_j$, which also gives the coefficients in the regression of $X_1$ on $X_{-1}$.
\end{remark}

 The following lemma, proved in Appendix~\ref{app:pf:lem:lower-priors}, completes the construction.

\begin{lemma}[Lower bound priors]\label{lem:lower-priors}
For each sufficiently small fixed $\tau>0$, the construction with degree $K$ in \eqref{eq:lower-degree} satisfies, for all sufficiently large $n$, every $V\in\mathcal C$, and $j=0,1$,
\[
 \kappa^{-1}I_p\preceq\Sigma_j(V)\preceq I_p,\qquad
 \|\beta^{(j)}(V)\|_2\le C\sqrt\tau,\qquad
 \frac{1+\kappa^{-1}}3\le\sigma_j^2\le\frac{1+\kappa^{-1}}2.
\]
Moreover, we have
\[
 \operatorname{TV}(P_0,P_1)\le\frac13,\quad
 \mathbb Q_j(\beta_1=b_j)=1,\quad
 |b_1-b_0|\ge c\tau e^{-K\mathfrak q(0)},
 \quad
 e^{-K\mathfrak q(0)}\asymp n^{-\alpha(\kappa,\gamma)/2}.
\]
\end{lemma}

\begin{remark}\label{rem:lower-measure-trade-off}
 We provide more insights on the trade-off between \eqref{eq:lower-measure-separation} and \eqref{eq:lower-transform} in Lemma~{\ref{lem:lower-measure}}. The lower bound \eqref{eq:lower-measure-separation} ensures $|b_1-b_0|\ge c\tau e^{-K\mathfrak q(0)}$ and the upper bound \eqref{eq:lower-transform} ensures $\operatorname{TV}(P_0,P_1)\le\frac13$. To control the total variation,  we first show that
 $\mathsf Y\mid\mathsf X\sim N(0,\mathcal V_j(\mathsf X))$ and aim to make $\mathcal V_0(\mathsf X)$ and $\mathcal V_1(\mathsf X)$ close; see Lemma~\ref{lem:conditional-comparison} with $\mathcal V_j$ in \eqref{eq:conditional-variance-Vj}. We decompose $\mathcal V_j(\mathsf X)$ along the eigenvectors of the Gram matrix $\mathsf X\mathsf X^\top/d$.  Consider an eigenvalue/eigenvector pair $(x,e)$ of $\mathsf X\mathsf X^\top/d$ such that $(\mathsf X\mathsf X^\top/d)e=xe$ and define $v_j(x)=e^\top\mathcal V_j(\mathsf X)e
       =\Var(e^\top\mathsf Y\mid\mathsf X)$. So the kernel $\mathsf K_x$ in \eqref{eq:lower-kernel} is derived from the difference $v_1(x)-v_0(x)$. In particular, \eqref{eq:lower-response-variance-difference}
       gives, uniformly for
       $x\in[x_--\upsilon,x_++\upsilon]$,
\begin{equation}\label{eq:lower-response-variance-difference-rm}
              v_1(x)-v_0(x)
              =\tau\int_a^1\mathsf K_x(\lambda)(\nu_1-\nu_0)(\dd\lambda)
                +O(\tau K/n).
\end{equation}
       Thus \eqref{eq:lower-transform} controls the differences $v_1(x)-v_0(x)$ and further leads to the bound $\operatorname{TV}(P_0,P_1)\le\frac13$; see \eqref{eq:TVaux} in Lemma~\ref{lem:conditional-comparison}.

We next give a high-level overview of the construction of $\nu_0$ and $\nu_1$ in Lemma~\ref{lem:lower-measure}. We aim to find a signed measure $\eta$  and a weight function $w$ on $[a,1]$ such that $\nu_0=w\eta^+$, $\nu_1=w\eta^-$ and $\nu_1-\nu_0=-w\eta$, where $\eta^+$ and $\eta^-$ are the positive and negative parts of $\eta$ respectively. We choose $w(\lambda)=\frac{(\lambda-a)(1-\lambda)}{\lambda^2}$ so its denominator cancels the factor $\lambda^2$ of $\mathsf K_x$  in \eqref{eq:lower-kernel} and its numerator ensures $w(\lambda)$ vanishes at the endpoints $a$ and $1$. See \eqref{eq:lower-response-variance-difference} in the proof of Lemma~\ref{lem:conditional-comparison} for the details. We construct the signed measure $\eta$ from the zeros of some holomorphic $f$ on $[a,1]$. Put
$\Gamma_0=\{z:|z-1|=\sqrt{\gamma^\circ}\}$.
Let $f$ be holomorphic on $\{z:|z-1|\le \sqrt{\gamma^\circ}\}$ and across $\Gamma_0$ with $f(\bar z)=\overline{f(z)}$, and suppose its zeros in the closed disk are simple and lie in $(a,1)$, say $f^{-1}(0)=\{z_1,\ldots,z_s\}$. We then define the signed measure $\eta = -\sum_{j=1}^s(f'(z_j))^{-1}\delta_{z_j}$, so $\eta$ is the sum of the negative residues of $1/f$ at the zeros of $f$. Since
$\nu_0(\R)+\nu_1(\R)=\int_a^1w\,\dd|\eta|\lesssim \|\eta\|_{\TV}$, as in \eqref{eq:lower-weighted-measure-bounds}, we need
\begin{equation}\label{eq:lower-eta-TV-f}
 \|\eta\|_{\TV}=\sum_{j=1}^s\frac1{|f'(z_j)|}
 =\frac1{2\pi}\sum_{j=1}^s\biggl|\oint_{|\xi-z_j|=\rho}\frac{\dd\xi}{f(\xi)}\biggr|
 {\le C},
\end{equation}
{Here $\rho>0$ is small enough that each circle encloses only its indicated zero. For the choice $f=\cosh(K\mathfrak q)$ below, the bound follows from \eqref{eq:lower-slit-phase} and \eqref{eq:lower-cosh-zeros}.} We normalize $f$ by $\|f\|_{\infty,[a,1]}=1$. We then show the trade-off requirements on $f$ implied by \eqref{eq:lower-measure-separation} and \eqref{eq:lower-transform}. On the one hand, since $-\lambda w(\lambda)=\lambda-(1+a)+a/\lambda$, we have
\begin{equation}\label{eq:lower-first-moment-separation-f0}
         \int_a^1\lambda\,(\nu_1-\nu_0)(\dd\lambda)
         =\frac a{f(0)}
          -\frac1{2\pi\ii}\oint_{\Gamma_0}
           \biggl(\xi-(1+a)+\frac a\xi\biggr)\frac{\dd\xi}{f(\xi)}
         =\frac a{f(0)}+O\bigl(\|1/f\|_{\infty,\Gamma_0}\bigr).
\end{equation}
On the other hand, {for $x\in[x_-,x_+]$}, the poles of $\mathsf K_x$ in \eqref{eq:lower-kernel} are on $\Gamma_0$, so to control $\operatorname{TV}(P_0,P_1)\le\frac13$ through \eqref{eq:lower-response-variance-difference-rm}, we need $n\|1/f\|^2_{\infty,\Gamma_0} = O(1)$; see \eqref{eq:Gaussian-comparison-total-variation} and \eqref{eq:lower-covariance-difference-proof}. {The additional ranges $[x_--\upsilon,x_-]\cup[x_+,x_++\upsilon]$ are handled by the nearby-root estimates in \eqref{eq:lower-exterior-root-distance}.}
Similar to the upper bound discussion in Remark~\ref{rem:two-constants},
the two-constants theorem (Theorem~\ref{thm:two-constants}) applied to
$\log|f|$, which is at most $0$ on $[a,1]$, gives
$f(0)\le\|f\|_{\infty,\Gamma_0}^{\mathfrak u(0)}$. We choose $f = \cosh(K\mathfrak q)$ with $K$ in \eqref{eq:lower-degree} to attain the tight scale: $1/f(0)\asymp e^{-K\mathfrak q(0)}
 \asymp n^{-\alpha(\kappa,\gamma)/2}$, which implies the desired separation rate by \eqref{eq:lower-first-moment-separation-f0}. On the other side, we have $\|1/f\|_{\infty,\Gamma_0} {\lesssim} e^{-Kh} \le n^{-1/2}$ so $\operatorname{TV}(P_0,P_1)\le\frac13$. We refer to the proof of Lemma~\ref{lem:lower-priors} in Appendix~\ref{app:pf:lem:lower-priors} and the related lemmas in Appendix~\ref{app:lower} for more details.
\end{remark}

\section{Proofs of Debiasing Methods}\label{app:pf:thm:graph-debias}

\begin{proof}[Proof of Theorem~\ref{thm:graph-debias}]
Write $\alpha=\alpha(\kappa,\gamma)$, $\alpha_\tau=\alpha(\kappa,\gamma/\tau)$, and $\mathcal A=\{\widehat\sigma_1\ge\ell/2\}$. For each fold, set
\[
 w_k=(0,(\widehat\sigma_{1,-1}^{(k)})^\top)^\top,
 \qquad
 \Delta_k=\mathcal B_{-k}(w_k)-w_k^\top\beta.
\]
Since $J^{-1}\sum_{k=1}^J\widehat\sigma_{1,-1}^{(k)}
=n^{-1}\sum_{i=1}^nP_mX_{i1}X_{i,-1}$, on $\mathcal A$, the estimator satisfies
\begin{equation}\label{eq:app-graph-decomposition}
 \widehat\beta_1^d-\beta_1
 =\frac{N+R-J^{-1}\sum_{k=1}^J\Delta_k}{\widehat\sigma_1}, \text{ where }
 N=\frac1n\sum_{i=1}^nX_{i1}\varepsilon_i,
 R=\frac1n\sum_{i=1}^nX_{i1}X_{i,-1}^\top(I-P_m)\beta_{-1}.
\end{equation}
The three numerator terms are from the noise, covariance approximation, and fold-estimation error. Since the estimator is zero on $\mathcal A^c$, we have
\[
 \E(\widehat\beta_1^d-\beta_1)^2
 \lesssim
 \PP(\mathcal A^c)
 +\E\!\left[\ind_{\mathcal A}\left(\frac N{\widehat\sigma_1}\right)^2\right]
 +\E R^2
 +\E\left(\frac1J\sum_{k=1}^J\Delta_k\right)^2.
\]

\proofstep{Noise term.} 
For $a\in\R^p$, as $X = Z\Sigma^{1/2}$ for $Z$ with independent entries having bounded moments, we have
\[
 \E(a^\top X_i)^4
 =3(a^\top\Sigma a)^2
  +(\E Z_{11}^4-3)\sum_{j=1}^p((\Sigma^{1/2}a)_j)^4
 \lesssim\norm a_2^4.
\]
Thus $\Var(\widehat\sigma_1)\lesssim n^{-1}$ and, since $\E\widehat\sigma_1=\Sigma_{11}\ge\ell$,
$
 \PP(\mathcal A^c)\le4\ell^{-2}\Var(\widehat\sigma_1)\lesssim n^{-1}.
$
We have
$\E(N^2\mid X)=\sigma^2\widehat\sigma_1/n$, whence
$
 \E[\ind_{\mathcal A}(N/(\widehat\sigma_1))^2]
 =\E[\ind_{\mathcal A}{\sigma^2/(n\widehat\sigma_1)}]
 \le{2\bar\sigma^2}/(\ell n)$.

\proofstep{Covariance approximation.}
Assumption~\ref{ass:graph-regularity} gives, for $m<p-1$,
\[
 \norm{(I-P_m)\Sigma_{1,-1}}_2^2
 \le\frac{\Sigma_{1,-1}^\top L_G\Sigma_{1,-1}}
          {\lambda_{m+1}(L_G)}
 \lesssim p^{-2r}m^{-2/q};
\]
the left side is zero for $m=p-1$. Therefore we have
\[
 |\E R|^2
 =|\Sigma_{1,-1}^\top(I-P_m)\beta_{-1}|^2
 \lesssim p^{-2r}m^{-2/q}.
\]
The fourth-moment estimate above and Cauchy--Schwarz give
$\E[X_{i1}^2(X_{i,-1}^\top(I-P_m)\beta_{-1})^2]\lesssim1$.
Independence of the $n$ rows now yields
\[
 \E R^2=|\E R|^2+\Var(R)
 \lesssim p^{-2r}m^{-2/q}+n^{-1}.
\]

\proofstep{Fold-estimation error.}
The same fourth-moment estimate implies
$\norm{\Cov(X_{i1}X_{i,-1})}_\op\lesssim1$. Since $P_m$ has rank $m$ and each validation fold has $n/J$ observations,
\[
 \E\norm{w_k}_2^2
 =\norm{P_m\Sigma_{1,-1}}_2^2
   +\frac Jn\tr(P_m\Cov(X_{i1}X_{i,-1})P_m)
 \lesssim p^{-2r}+\frac{Jm}{n}.
\]
Conditional on the validation design $X_{\mathcal I_k}$, the loading $w_k$ is fixed and independent of the training data defining $\mathcal B_{-k}$. Homogeneity in the loading and Theorem~\ref{thm:upper} give
\[
 \E(\Delta_k^2\mid X_{\mathcal I_k})
 \lesssim(n\tau)^{-\alpha_\tau}\norm{w_k}_2^2.
\]
Taking expectations and applying Jensen's inequality,
\[
 \E\left(\frac1J\sum_{k=1}^J\Delta_k\right)^2
 \le\frac1J\sum_{k=1}^J\E\Delta_k^2
 \lesssim(n\tau)^{-\alpha_\tau}
          \left(p^{-2r}+\frac{Jm}{n}\right).
\]

\proofstep{Risk and projection dimension.}
Combining the bounds with \eqref{eq:app-graph-decomposition} gives
\[
 \E(\widehat\beta_1^d-\beta_1)^2
 \lesssim n^{-1}+p^{-2r}m^{-2/q}
 +(n\tau)^{-\alpha_\tau}\left(p^{-2r}+\frac{Jm}{n}\right).
\]
Under \eqref{eq:aspect-regime}, the derivative of $\alpha(\kappa,\cdot)$ is bounded on the relevant ratio interval. Since $\tau=1-J^{-1}$ and $J\asymp\log n$,
\[
 0\le\alpha-\alpha_\tau\lesssim J^{-1},
 \qquad
 (n\tau)^{-\alpha_\tau}
 =n^{-\alpha}\tau^{-\alpha_\tau}
   e^{(\alpha-\alpha_\tau)\log n}
 \lesssim n^{-\alpha}.
\]
This proves \eqref{eq:graph-rate-general}. Finally, $m=\lceil n^{\alpha/2}\rceil\le p-1$ for all sufficiently large $n$, and
\[
 p^{-2r}m^{-2/q}\lesssim n^{-2r-\alpha/q},\qquad
 n^{-\alpha}\frac{Jm}{n}\lesssim(\log n)n^{-1-\alpha/2}=o(n^{-1}),\qquad
 p^{-2r}n^{-\alpha}\lesssim n^{-2r-\alpha/q},
\]
where the last inequality uses $q\ge1$. Hence we have
$\E(\widehat\beta_1^d-\beta_1)^2\lesssim n^{-1}+n^{-2r-\alpha/q}
\lesssim n^{-\min\{1,\,2r+\alpha/q\}}$.
\end{proof}
\section{Calibration, Spectral Control, and the Empirical Inverse}\label{app:calibration}
Let $\varrho$ be the probability measure whose Stieltjes transform solves \eqref{eq:MP}, i.e., 
\begin{equation}\label{eq:stieltjes-transform-definition}
      \varrho((a,b)) = \lim_{\eta\downarrow0}\frac1\pi
      \int_a^b\operatorname{Im}m(x+\ii\eta)\,\dd x \text{ for any real  $a<b$}.
\end{equation}
Thus we have
$m(\xi)=\int (x-\xi)^{-1}\varrho(\dd x)$ on $\C\setminus\supp\varrho$.
We first prove the deterministic calibration, then control its empirical inverse.

\subsection{The Deterministic Calibration and Quantitative Separation}\label{app:pf:prop:calibration}
\begin{proof}[Proof of Proposition~\ref{prop:calibration}]
Define the right-hand side of \eqref{eq:MP} as
\begin{equation}\label{eq:f-definition}
      f(s)=-\frac1s+\frac1n\tr\bigl(\Sigma(I_p+s\Sigma)^{-1}\bigr), \text{ for $s\ne0$ and $-1/s\notin\spec(\Sigma)$},    
\end{equation}
so that \eqref{eq:MP} reads $f(m(z))=z$. Fix $\mu$ with $|\mu-L|>\sqrt\gamma L$. Since $\gamma>1$ and $\spec(\Sigma)\subset[\ell,L]$, we have $\mu\ne0$ and $\mu\notin\spec(\Sigma)$. Substitution gives 
\[
f(-1/\mu)=\mu-(\mu/n)\tr\bigl(\Sigma(\Sigma-\mu I_p)^{-1}\bigr)=\zeta(\mu),
\] 
where the last equality is \eqref{eq:zeta}. Thus we have $f(-1/\mu)=\zeta(\mu)$ and $f(m(\zeta(\mu)))=\zeta(\mu)$; we next identify $m(\zeta(\mu)) = -1/\mu$.

Diagonalizing $\Sigma$ in \eqref{eq:zeta} gives
\begin{equation}\label{eq:zeta-formula}
      \zeta(\mu)=\mu+\frac1n\sum_{j=1}^p\frac{\mu\lambda_j}{\mu-\lambda_j}.
\end{equation}
For each $j$, we have $\mu\lambda_j/(\mu-\lambda_j)=\lambda_j+\lambda_j^2/(\mu-\lambda_j)$. Since $\lambda_j$ is real and $(\mu-\lambda_j)^{-1}=(\bar\mu-\lambda_j)/|\mu-\lambda_j|^2$, the imaginary part of this summand is $-\lambda_j^2\operatorname{Im}\mu/|\mu-\lambda_j|^2$. Taking imaginary parts term by term therefore yields
\begin{equation}\label{eq:imaginary-part-zeta}
      \operatorname{Im}\zeta(\mu)
      =\operatorname{Im}\mu-\frac1n\sum_{j=1}^p
        \frac{\lambda_j^2\operatorname{Im}\mu}{|\lambda_j-\mu|^2}
      =\operatorname{Im}\mu
      \biggl(1-\frac1n\sum_{j=1}^p\frac{\lambda_j^2}{|\lambda_j-\mu|^2}\biggr).
\end{equation}
For $\lambda\in[\ell,L]$, we have 
\[
      |\lambda-\mu|\ge|\mu-L|-(L-\lambda)>\sqrt\gamma L-(L-\lambda)\ge\sqrt\gamma\,\lambda \text{ and thus } \frac1n\sum_{j=1}^p\frac{\lambda_j^2}{|\lambda_j-\mu|^2}<\frac{p}{n\gamma}=1
\]
where the last inequality uses $\gamma>1$ and $\lambda\le L$. Combining the above with \eqref{eq:imaginary-part-zeta}, $\operatorname{Im}\zeta(\mu)$ has the same sign as $\operatorname{Im}\mu$. As $\supp\varrho\subset\R$, we have $\zeta(\mu) \not\in \supp\varrho$ when $\mu$ is not real.

For each fixed $z$ in the upper half-plane, \citet[Eqs.~(1.3)--(1.4)]{SilversteinChoi1995} state that $f(s)=z$ has exactly one solution $s$ in the upper half-plane, namely $m(z)$.  If $\operatorname{Im}\mu>0$, we have $\operatorname{Im}\zeta(\mu)>0$ and $\operatorname{Im}(-1/\mu)=\operatorname{Im}\mu/|\mu|^2>0$. Since $f(-1/\mu)=\zeta(\mu)$, uniqueness at this $z$ gives $m(\zeta(\mu))=-1/\mu$.
If $\operatorname{Im}\mu<0$, we utilize conjugation and have $m(\bar z)=\overline{m(z)}$ and $\zeta(\bar\mu)=\overline{\zeta(\mu)}$, and thus $m(\zeta(\mu))=\overline{m(\zeta(\bar\mu))}=\overline{-1/\bar\mu}=-1/\mu$.

For real $\mu$ with $|\mu-L|>\sqrt\gamma L$, differentiating $f(-1/\mu)=\zeta(\mu)$ gives 
\begin{equation}\label{eq:xi-prime}
      f'(-1/\mu)=\mu^2\zeta'(\mu) = \mu^2\Big(1-\frac1n\sum_{j=1}^p\frac{\lambda_j^2}{(\lambda_j-\mu)^2}\Big) >0.
\end{equation}
Together with $\mu\ne0$ and $\mu\notin\spec(\Sigma)$, Theorem~4.2 of \citet{SilversteinChoi1995} implies that $\zeta(\mu)\notin\supp\varrho$ and $m(\zeta(\mu))=-1/\mu$.

Finally, as we argued above for both nonreal and real cases, $\zeta(\mu) \not \in \supp\varrho$. Therefore, in a neighborhood of $\zeta(\mu)$, we have $m'(z)=\int(x-z)^{-2}\varrho(\dd x)$. This proves that the extension of $m$ is analytic at $\zeta(\mu)$.
\end{proof}

The following lemma further quantifies how the margin $\delta$ quantifies the separation of $\zeta(\mu)$ for $\mu \in \Gamma_\delta$ from $\supp\varrho$.

\begin{lemma}\label{lem:calibration-control}
For $\delta>0$ and $|\mu-L|\ge\sqrt\gamma(1+\delta)L$, we have
\begin{equation}\label{eq:stability-main}
 \operatorname{Re}\zeta'(\mu)\ge1-(1+\delta)^{-2}>0.
\end{equation}
The map $\zeta(\mu)$ is injective on $\{\mu: |\mu-L|\ge\sqrt\gamma(1+\delta)L\}$. There is a constant $c>0$ such that, for all sufficiently small $\delta>0$ and every $\mu\in\Gamma_\delta$,
\begin{equation}\label{eq:calibration-separation}
 \dist(\zeta(\mu),\supp\varrho)\ge c\delta^2,
 \qquad |m'(\zeta(\mu))|\ge c.
\end{equation}
Moreover, for each integer $k\ge0$, there is a constant $C_k$ such that, uniformly for $\mu\in\Gamma_\delta$ and $|\xi-\zeta(\mu)|\le c\delta^2/2$,
\begin{equation}\label{eq:calibration-derivatives}
 |m^{(k)}(\xi)|\le C_k\delta^{-2(k+1)},
 \qquad
 \bigl\|\partial_\xi^k(I_p+m(\xi)\Sigma)^{-1}\bigr\|_\op
 \le C_k\delta^{-2k}.
\end{equation}
\end{lemma}
\begin{proof}
The proof is similar to the proof of Proposition~\ref{prop:calibration} above but adds the margin $\delta$ to the analysis.
\proofstep{Stability and injectivity.}
For $\mu \in \Ddisk_\delta^c = \{z: |z-L|\ge R_\delta\}$ and $\lambda\in[\ell,L]$, we have
\begin{equation}\label{eq:calibration-denominator-bound}
 |\lambda-\mu|\ge\sqrt\gamma(1+\delta)L-(L-\lambda)
 \ge\sqrt\gamma(1+\delta)\lambda,
 \qquad
 \frac1n\sum_{j=1}^p\frac{\lambda_j^2}{|\lambda_j-\mu|^2}
 \le\frac{p}{n\gamma(1+\delta)^2}=(1+\delta)^{-2}.
\end{equation}
By \eqref{eq:zeta-formula}, we have
\begin{equation}\label{eq:zeta-calibration-stability}
      \zeta'(\mu)=1-\frac1n\sum_{j=1}^p\frac{\lambda_j^2}{(\lambda_j-\mu)^2},
      \qquad
      \operatorname{Re}\zeta'(\mu)
      \ge1-\frac1n\sum_{j=1}^p\frac{\lambda_j^2}{|\lambda_j-\mu|^2}
      \ge1-(1+\delta)^{-2},    
\end{equation}
which proves \eqref{eq:stability-main}. For $\mu,\nu$ in $\Ddisk_\delta^c$, the Cauchy--Schwarz inequality gives
\[
 \biggl|\frac1n\sum_{j=1}^p
 \frac{\lambda_j^2}{(\lambda_j-\mu)(\lambda_j-\nu)}\biggr|
 \le
 \biggl(\frac1n\sum_{j=1}^p\frac{\lambda_j^2}{|\lambda_j-\mu|^2}\biggr)^{1/2}
 \biggl(\frac1n\sum_{j=1}^p\frac{\lambda_j^2}{|\lambda_j-\nu|^2}\biggr)^{1/2}
 \le(1+\delta)^{-2}.
\]
Subtracting the two expressions for $\zeta$ therefore yields
\begin{equation}\label{eq:zeta-difference}
 \begin{aligned}
 \zeta(\mu)-\zeta(\nu)
 &=(\mu-\nu)\biggl(1-\frac1n\sum_{j=1}^p
 \frac{\lambda_j^2}{(\lambda_j-\mu)(\lambda_j-\nu)}\biggr),\\
 |\zeta(\mu)-\zeta(\nu)|
 &\ge\bigl(1-(1+\delta)^{-2}\bigr)|\mu-\nu|.
 \end{aligned}
\end{equation}
In particular, $\zeta(\mu)=\zeta(\nu)$ implies $\mu=\nu$.

\proofstep{Image separation.}
Fix $\mu\in\Gamma_\delta$ and choose $0<c_1<L/4$. For $\nu$ in the disk $|\nu-\mu|\le c_1\delta$, we have
\[
 |\nu-L|\ge R_\delta-c_1\delta
 >\sqrt\gamma(1+\delta/2)L.
\]
Thus $\zeta$ is holomorphic on a neighborhood of this closed disk $|\nu-\mu|\le c_1\delta$. Applying \eqref{eq:zeta-difference} with $\delta/2$ in place of $\delta$ gives, for some $c_2>0$ and all sufficiently small $\delta$,
\[
 \inf_{|\nu-\mu|=c_1\delta}|\zeta(\nu)-\zeta(\mu)|
 \ge\bigl(1-(1+\delta/2)^{-2}\bigr)c_1\delta
 \ge c_2\delta^2.
\]
If $|\xi-\zeta(\mu)|<c_2\delta^2$, then on $|\nu-\mu|=c_1\delta$,
\[
 |(\zeta(\nu)-\xi)-(\zeta(\nu)-\zeta(\mu))|
 =|\xi-\zeta(\mu)|<|\zeta(\nu)-\zeta(\mu)|.
\]
The function $\nu\mapsto\zeta(\nu)-\zeta(\mu)$ has exactly one zero in the disk, at $\nu=\mu$, by injectivity; this zero is simple because \eqref{eq:stability-main} gives $\zeta'(\mu)\ne0$. Rouch\'e's theorem \citep[Chapter~4, Section~5.2]{Ahlfors1979} therefore gives exactly one zero of $\nu\mapsto\zeta(\nu)-\xi$ in the same disk. Consequently,
\begin{equation}\label{eq:calibration-image-inclusion}
 \{\xi:|\xi-\zeta(\mu)|<c_2\delta^2\}
 \subset
 \zeta\bigl(\{\nu:|\nu-\mu|<c_1\delta\}\bigr).
\end{equation}
For a real point $\xi=\zeta(\nu)$ in this image,
\[
 \zeta(\bar\nu)=\overline{\zeta(\nu)}=\xi=\zeta(\nu),
 \qquad
 |\bar\nu-L|=|\nu-L|>\sqrt\gamma(1+\delta/2)L.
\]
Injectivity on this exterior region implies $\bar\nu=\nu$, so the real case in the proof of Proposition~\ref{prop:calibration} gives $\xi=\zeta(\nu)\notin\supp\varrho$. Nonreal image points also lie outside $\supp\varrho\subset\R$. Hence
\[
 \{\xi:|\xi-\zeta(\mu)|<c_2\delta^2\}\cap\supp\varrho=\varnothing,
 \qquad
 \dist(\zeta(\mu),\supp\varrho)\ge c_2\delta^2.
\]
Proposition~\ref{prop:calibration} also gives $m(\xi)=-1/\nu$ whenever $\xi=\zeta(\nu)$ lies in this disk.

\proofstep{Derivatives.}
Differentiating $m(\zeta(\mu))=-1/\mu$ gives $m'(\zeta(\mu))\zeta'(\mu)=\mu^{-2}$, and hence
\begin{equation}\label{eq:mprime}
 m'(\zeta(\mu))=\frac1{\mu^2\zeta'(\mu)}.
\end{equation}
On $\Gamma_\delta$, we have $0<R_\delta-L\le|\mu|\le R_\delta+L$ and $|\zeta'(\mu)|\le1+(1+\delta)^{-2}\le2$, so
\[
 |m'(\zeta(\mu))|\ge\frac1{2(R_\delta+L)^2}\ge c.
\]
Choosing $c\le c_2$ small enough proves both assertions in \eqref{eq:calibration-separation}.

For $|\xi-\zeta(\mu)|\le c\delta^2/2$, the triangle inequality gives $\dist(\xi,\supp\varrho)\ge c\delta^2/2$. Differentiating the Stieltjes transform therefore yields
\[
 m^{(k)}(\xi)=k!\int\frac{\varrho(\dd x)}{(x-\xi)^{k+1}},
 \qquad
 |m^{(k)}(\xi)|\le\frac{k!}{\dist(\xi,\supp\varrho)^{k+1}}
 \le k!\biggl(\frac2c\biggr)^{k+1}\delta^{-2(k+1)},
\]
where the inequality uses $\varrho(\R)=1$.
For every $\omega$ with $|\omega-\zeta(\mu)|<c\delta^2$, \eqref{eq:calibration-image-inclusion} supplies $\nu$ with $|\nu-\mu|<c_1\delta$ and $\omega=\zeta(\nu)$, and Proposition~\ref{prop:calibration} gives $m(\omega)=-1/\nu$. Thus
\[
 (I_p+m(\omega)\Sigma)^{-1}=\nu(\nu I_p-\Sigma)^{-1},
 \quad
 \bigl\|(I_p+m(\omega)\Sigma)^{-1}\bigr\|_\op
 =\max_j\frac{|\nu|}{|\nu-\lambda_j|}
 \le1+\max_j\frac{\lambda_j}{|\nu-\lambda_j|}\le2,
\]
where the last inequality follows from \eqref{eq:calibration-denominator-bound} with $\delta/2$ in place of $\delta$. This matrix function  $\omega\mapsto (I_p+m(\omega)\Sigma)^{-1}$ is holomorphic on $|\omega-\zeta(\mu)|<c\delta^2$. For $|\xi-\zeta(\mu)|\le c\delta^2/2$, the circle $|\omega-\xi|=c\delta^2/4$ is contained in that disk, so Cauchy's derivative estimate \citep[Chapter~4, Eq.~(25)]{Ahlfors1979} gives
\[
 \bigl\|\partial_\xi^k(I_p+m(\xi)\Sigma)^{-1}\bigr\|_\op
 \le k!\biggl(\frac4{c\delta^2}\biggr)^k
 \sup_{|\omega-\xi|=c\delta^2/4}
 \bigl\|(I_p+m(\omega)\Sigma)^{-1}\bigr\|_\op
 \le2k!\biggl(\frac4c\biggr)^k\delta^{-2k}.
\]
This proves \eqref{eq:calibration-derivatives}.
\end{proof}

\subsection{A Spectral Event and Preliminary Moments}
Unless otherwise specified, throughout this section and Appendix~\ref{app:moments}, assume that $1/\log n\le\delta\le2/\log n$.
We use $\mu$ as the spectral parameter in the population resolvent $(\Sigma-\mu I_p)^{-1}$ and $\xi$ as the spectral parameter in the sample resolvent $(\widehat S-\xi I_n)^{-1}$. For fixed $b\ge2$ and sufficiently small $c_1>0$, set
\begin{equation}\label{eq:spectral-neighborhood}
 \mathcal T=\{\xi\in\C:\dist(\xi,\zeta(\Gamma_\delta))<c_1\delta^b\},
 \qquad
 \mathcal T^+=\{\xi\in\C:\dist(\xi,\zeta(\Gamma_\delta))<3c_1\delta^b\}.
\end{equation}
The constants are chosen so that Lemma~\ref{lem:calibration-control} applies throughout $\mathcal T^+$. In Lemma~\ref{lem:empirical-control}, we will pick $b > 6$. When proving estimates on this set, we first enlarge its defining radius by a fixed factor and then choose $c_1$ sufficiently small.

Set $\mathcal X=X^\top/\sqrt n=\Sigma^{1/2}Z^\top/\sqrt n$, so that $\mathcal X^\top\mathcal X=\widehat S$, and write $Q(\xi)=(\widehat S-\xi I_n)^{-1}$ for $\xi\notin\spec(\widehat S)$. We define the Green matrix $G(\xi)$ and the population resolvent $\Pi(\xi)$ as follows:
\begin{equation}\label{eq:linearization}
 \begin{aligned}
 G(\xi)
 &=\begin{pmatrix}-I_p&\mathcal X\\ \mathcal X^\top&-\xi I_n\end{pmatrix}^{-1}
 =\begin{pmatrix}
   \mathcal XQ(\xi)\mathcal X^\top-I_p&\mathcal XQ(\xi)\\
   Q(\xi)\mathcal X^\top&Q(\xi)
  \end{pmatrix},\\
 \Pi(\xi)
 &=\begin{pmatrix}-(I_p+m(\xi)\Sigma)^{-1}&0\\0&m(\xi)I_n\end{pmatrix}.
 \end{aligned}
\end{equation}
We write
$\mathsf u_k=\Sigma^{1/2}e_k$ and $\mathsf e_i=(0,e_i)$, and set
$\mathsf Z_{ki}=Z_{ik}/\sqrt n$.
We write $\partial_{ki}=\partial/\partial\mathsf Z_{ki}$ for entry differentiation and $\partial_\xi=\partial/\partial\xi$ for spectral differentiation.

For simplicity of notation, we use the decomposition $\C^{p+n}=\C^p\oplus\C^n$, with feature
coordinates in the first block and sample coordinates in the second.
In products with matrices or vectors in this full block space, we use
the embeddings
\begin{equation}\label{eq:recursion-vector-embedding}
 u\in\C^p\longmapsto\begin{pmatrix}u\\0_n\end{pmatrix},
 \qquad
 v\in\C^n\longmapsto\begin{pmatrix}0_p\\v\end{pmatrix},
\end{equation}
without changing the vector symbols. The same convention applies to
$w$ and to indexed vectors; their spaces are specified at each use.
Vectors in $\C^{p+n}$ are used directly.
For a $(p+n)\times(p+n)$ matrix $B=(B_{ab})_{a,b=1}^2$ partitioned
according to these two blocks,
\begin{align*}
 B u=\begin{pmatrix}B_{11}u\\B_{21}u\end{pmatrix}
       \quad(u\in\C^p), \qquad
 B v=\begin{pmatrix}B_{12}v\\B_{22}v\end{pmatrix}
       \quad(v\in\C^n).
\end{align*}
Likewise, $u^\top B w=u^\top B_{ab}w$ when $u$ and $w$ belong to
blocks $a$ and $b$, respectively. This convention applies to $G$, $\Pi$,
their derivatives, and products of block matrices.
Products with individual blocks or with $\Sigma,D,M,Q,\mathcal X$ use
the original vector dimensions.

\begin{lemma}[Localized resolvent estimates]\label{lem:preliminary}
Let $1/\log n\le\delta\le2/\log n$.
There is an $X$-measurable event $\mathcal A_1$ such that, for every $D>0$ and all sufficiently large $n$,
$
 \PP(\mathcal A_1^c)\le C_Dn^{-D}.
$
On $\mathcal A_1$, we have
\begin{equation}\label{eq:preliminary-spectral}
 \begin{gathered}
 cI_n\preceq\widehat S\preceq CI_n,\qquad
 \inf_{\xi\in\mathcal T^+}\dist(\xi,\spec(\widehat S))\ge c\delta^2,
 \qquad \sup_{\xi\in\mathcal T^+}\|G(\xi)\|_\op\le C\delta^{-2},\\
 \zeta(L-R_\delta)\le\lambda_{\min}(\widehat S)-c\delta^2,
 \qquad
 \zeta(L+R_\delta)\ge\lambda_{\max}(\widehat S)+c\delta^2.
 \end{gathered}
\end{equation}
Put $\mathcal T_0=\mathcal T^+$ and $\mathcal T_j=\mathcal T$ for $j\ge1$.
For every fixed $r\ge2$, $\epsilon>0$, integer $j\ge0$, and deterministic $v,w\in\C^{p+n}$,
\begin{equation}\label{eq:coarse-moments}
 \begin{aligned}
 \Big\|\sup_{\xi\in\mathcal T_j}
   \big|\ind_{\mathcal A_1}\partial_\xi^j
     \bigl[v^\top(G(\xi)-\Pi(\xi))w\bigr]\big|\Big\|_{L^r}
 &\le C n^{-1/2+\epsilon}\|v\|_2\|w\|_2,\\
 \Big\|\sup_{\xi\in\mathcal T_j}
   \big|\ind_{\mathcal A_1}\partial_\xi^j
     (\widehat m(\xi)-m(\xi))\big|\Big\|_{L^r}
 &\le C n^{-1+\epsilon}.
 \end{aligned}
\end{equation}
Moreover, taking $\epsilon = 1/4$, we obtain the event $\mathcal A$ defined by
\begin{equation}\label{eq:averaged-uniform}
\mathcal A = \mathcal A_1\cap
\Big\{\sup_{\xi\in\mathcal T}|\widehat m(\xi)-m(\xi)|
               \le n^{-3/4}\Big\} \text{ and for every $D>0$, } \PP(\mathcal A^c)
 \le C_{D}n^{-D}.
\end{equation}
\end{lemma}
The proof is given in Appendix~\ref{app:calibration-comparison}.
The constants in \eqref{eq:coarse-moments} may depend on $r, \epsilon,$ and the fixed derivative order $j$.

We can further have the following lemma on the bounds of the resolvent and its derivatives at nonreal parameters.

\begin{lemma}\label{lem:nonreal-resolvent}
Fix $K<\infty$. For every sample and every $\xi\notin\R$ with $|\xi|\le K$,
\begin{equation}\label{eq:nonreal-resolvent-bound}
 \|G(\xi)\|_\op+|\widehat m(\xi)|
 \le C_K(1+|\operatorname{Im}\xi|^{-1}).
\end{equation}
For each fixed $j\ge1$, every entry derivative of order $j$ satisfies
\begin{equation}\label{eq:nonreal-resolvent-derivatives}
 \|\partial_{k_1i_1}\cdots\partial_{k_ji_j}G(\xi)\|_\op
 \le C_{K,j}(1+|\operatorname{Im}\xi|^{-1})^{j+1}.
\end{equation}
\end{lemma}
\begin{proof}
Write $\eta=|\operatorname{Im}\xi|$, let $r=\operatorname{rank}\mathcal X$, and denote the nonzero singular values of $\mathcal X$ by $t_1,\ldots,t_r$.
Since $\widehat S=\mathcal X^\top\mathcal X\succeq0$, we have $s_1 = t_1^2,\ldots, s_r = t_r^2$ and the remaining eigenvalues are zero.
By \eqref{eq:linearization}, we can block diagonalize the Green matrix $G(\xi)$ as
\begin{align}\label{eq:G-block-diagonalization}
 \mathcal U^\top G(\xi)\mathcal U
 &= \diag\Big(\frac1{t_1^2-\xi}
 \begin{pmatrix}\xi&t_1\\t_1&1\end{pmatrix}, \ldots, \frac1{t_r^2-\xi}
 \begin{pmatrix}\xi&t_r\\t_r&1\end{pmatrix}, -I_{p-r}, -\xi^{-1}I_{n-r} \Big),
\end{align}
where the orthogonal matrix $\mathcal U$ is constructed by pairing the corresponding left and right singular vectors and completing with orthonormal bases of $\ker\mathcal X^\top$ and $\ker\mathcal X$.

We first bound the operator norm of the $2\times2$ blocks in \eqref{eq:G-block-diagonalization} for $t = t_1, \ldots, t_r$.
For every $t\ge0$, we have
\begin{align*}
 |t^2-\xi|^2
 =(t^2-\operatorname{Re}\xi)^2+\eta^2\ge\eta^2,
 \qquad
 \left|\frac1{t^2-\xi}\right|
 \le\eta^{-1},
 \qquad
 \left|\frac{\xi}{t^2-\xi}\right|
 \le K\eta^{-1}.
\end{align*}
Moreover, $t^2\le|t^2-\xi|+|\xi|\le|t^2-\xi|+K$ gives
\begin{align*}
 \frac t{|t^2-\xi|}
 &\le\frac{\sqrt{|t^2-\xi|+K}}{|t^2-\xi|}
 \le |t^2-\xi|^{-1/2}+\sqrt K\,|t^2-\xi|^{-1}
 \le\eta^{-1/2}+\sqrt K\,\eta^{-1}.
\end{align*}
Applying the above bounds to the $2\times2$ block in \eqref{eq:G-block-diagonalization} gives
\begin{align*}
 \biggl\|\frac1{t^2-\xi}
  \begin{pmatrix}\xi&t\\t&1\end{pmatrix}\biggr\|_\op
 \le\frac{|\xi|+1+2t}{|t^2-\xi|}
 \le C_K\eta^{-1}+2\eta^{-1/2}.
\end{align*}
Combining with $|\xi|^{-1}\le\eta^{-1}$, we obtain
\begin{align*}
 |\widehat m(\xi)|
 \le\|Q(\xi)\|_\op\le\eta^{-1}, \qquad 
 \|G(\xi)\|_\op
  \le C_K(1+\eta^{-1}),
\end{align*}
which proves \eqref{eq:nonreal-resolvent-bound}.
Write $(\mathsf u_k,0)\in\R^{p+n}$ as a column vector and put $\Delta_{ki}=(\mathsf u_k,0)\mathsf e_i^\top+\mathsf e_i (\mathsf u_k,0)^\top$.
The normalization $\mathsf Z_{ki}=Z_{ik}/\sqrt n$ and \eqref{eq:linearization} give
\begin{align}\label{eq:nonreal-resolvent-derivatives-identity}
 \partial_{ki}\mathcal X
 =\Sigma^{1/2}e_ke_i^\top, \qquad
 \partial_{ki}\bigl(G(\xi)^{-1}\bigr)
 =\begin{pmatrix}
    0&\Sigma^{1/2}e_ke_i^\top\\
    e_ie_k^\top\Sigma^{1/2}&0
   \end{pmatrix}
 =\Delta_{ki}.
\end{align}
Differentiating $G(\xi)^{-1}G(\xi)=I_{p+n}$ and plugging in the above derivative identity yields $\partial_{ki}G(\xi)=-G(\xi)\Delta_{ki}G(\xi)$. Differentiating in the same way shows that an order-$j$ derivative of $G(\xi)$ is a sum of at most $j!$ products containing $j+1$ Green matrices $G(\xi)$ and $j$ deterministic matrices of the form $\Delta_{ki}$.
As $\|\Delta_{ki}\|_\op\le\sqrt L$ and $\|G(\xi)\|_\op
\le C_K(1+\eta^{-1})$, we have \eqref{eq:nonreal-resolvent-derivatives}.
\end{proof}

\subsection{Constructing the Sample Inverse}
\begin{lemma}\label{lem:empirical-control}
Let $1/\log n\le\delta\le2/\log n$. For all sufficiently large $n$, on the $X$-measurable event $\mathcal A$ in \eqref{eq:averaged-uniform} using fixed $b>6$ in \eqref{eq:spectral-neighborhood}, the inverse in \eqref{eq:empirical-calibration} exists and
\begin{equation}\label{eq:empirical-control}
 \sup_{\mu\in\Gamma_\delta}|\widehat\zeta(\mu)-\zeta(\mu)|
 \le Cn^{-3/4}\delta^{-c},
 \qquad
 \sup_{\mu\in\Gamma_\delta}
 \|(\widehat S-\widehat\zeta(\mu)I_n)^{-1}\|_\op
 \le C\delta^{-c}.
\end{equation}
On $\mathcal A$, we further have that $\widehat\zeta$ is locally analytic with
$\widehat\zeta(\bar\mu)=\overline{\widehat\zeta(\mu)}$ on $\Gamma_\delta$.
\end{lemma}
Together with Lemma~\ref{lem:preliminary}, this shows that all the results in Lemma~\ref{lem:empirical-control} hold with probability at least $1-C_D n^{-D}$ for any $D>0$.
\begin{proof}[Proof of Lemma~\ref{lem:empirical-control}] Recall $\mathcal T=\{\xi\in\C:\dist(\xi,\zeta(\Gamma_\delta))<c_1\delta^b\}$ in \eqref{eq:spectral-neighborhood}. We take $b > 6$ and work on the event $\mathcal A$ in \eqref{eq:averaged-uniform} using this $b$ for the rest of the proof. The arguments in the rest of the proof are for sufficiently large $n$.
Put
$\mathcal D_\mu=\{\xi:|\xi-\zeta(\mu)|<c_1\delta^b/2\}$ for $\mu\in\Gamma_\delta$ and we have $\overline{\mathcal D_\mu}\subset\mathcal T\subset\mathcal T^+$.
For all sufficiently large $n$, $c_1\delta^b\le c\delta^2$; hence \eqref{eq:calibration-separation} and \eqref{eq:preliminary-spectral} give
\begin{align*}
 \dist(\overline{\mathcal D_\mu},\supp\varrho)
 &\ge\dist(\zeta(\mu),\supp\varrho)-\frac{c_1}{2}\delta^b
 \ge c\delta^2-\frac{c_1}{2}\delta^b\ge\frac c2\delta^2>0,\\
 \dist(\overline{\mathcal D_\mu},\spec(\widehat S))
 &\ge\inf_{\xi\in\mathcal T^+}\dist(\xi,\spec(\widehat S))
 \ge c\delta^2>0.
\end{align*}
The Stieltjes transform $m$ is holomorphic on $\C\setminus\supp\varrho$, and $\widehat m$ is holomorphic on $\C\setminus\spec(\widehat S)$.
The positive distances above therefore give a neighborhood of $\overline{\mathcal D_\mu}$ on which both functions are holomorphic.

\proofstep{Existence and analyticity.} We next expand $m$ and control its Taylor remainder. Equations \eqref{eq:calibration-separation} and \eqref{eq:calibration-derivatives} give, uniformly in $\mu\in\Gamma_\delta$,
\begin{align*}
 |m'(\zeta(\mu))|&\ge c_m>0,
 \qquad \sup_{\xi\in\overline{\mathcal D_\mu}}|m''(\xi)|\le C\delta^{-6}.
\end{align*}
Using $m(\zeta(\mu))=-1/\mu$ from Proposition~\ref{prop:calibration}, Taylor's formula yields, for $\xi\in\overline{\mathcal D_\mu}$,
\begin{align}
 \bigl|m(\xi)+1/\mu-m'(\zeta(\mu))(\xi-\zeta(\mu))\bigr|
 &\le\frac C2\delta^{-6}|\xi-\zeta(\mu)|^2
 \le\frac{Cc_1}{4}\delta^{b-6}|\xi-\zeta(\mu)|\le\frac{c_m}{2}|\xi-\zeta(\mu)|,\notag\\
 |m(\xi)+1/\mu|
 &\ge|m'(\zeta(\mu))|\,|\xi-\zeta(\mu)|-\frac{c_m}{2}|\xi-\zeta(\mu)|
 \ge\frac{c_m}{2}|\xi-\zeta(\mu)|.
 \label{eq:empirical-local-lower-bound}
\end{align}
Therefore, for $\xi\in\mathcal D_\mu$, $m(\xi)+1/\mu=0$ has a unique solution $\xi=\zeta(\mu)$. 
This zero is simple because $|m'(\zeta(\mu))|\ge c_m>0$.
For $\xi \in \partial\mathcal D_\mu$, \eqref{eq:averaged-uniform} and \eqref{eq:empirical-local-lower-bound} give
\begin{align*}
 \bigl|\bigl(\widehat m(\xi)+1/\mu\bigr)
          -\bigl(m(\xi)+1/\mu\bigr)\bigr|
 &\le n^{-3/4}<\frac{c_mc_1}{4}\delta^b\le|m(\xi)+1/\mu|.
\end{align*}
Rouch\'e's theorem \citep[Chapter~4, Section~5.2]{Ahlfors1979} therefore gives exactly one zero $\widehat\zeta(\mu)\in\mathcal D_\mu$ of $\widehat m(\xi)+1/\mu$, counted with multiplicity one. In particular, $\widehat m'(\widehat\zeta(\mu))\ne0$.
Evaluating \eqref{eq:empirical-local-lower-bound} at $\xi = \hat{\zeta}(\mu)$ gives
\begin{align}
 \frac{c_m}{2}|\widehat\zeta(\mu)-\zeta(\mu)|
 &\le|m(\widehat\zeta(\mu))+1/\mu|
 =|m(\widehat\zeta(\mu))-\widehat m(\widehat\zeta(\mu))|
 \le n^{-3/4},\notag\\
 \sup_{\mu\in\Gamma_\delta}|\widehat\zeta(\mu)-\zeta(\mu)|
 &\le\frac{2}{c_m}n^{-3/4}=o(\delta^b).
 \label{eq:empirical-root-displacement}
\end{align}

To verify that these roots $\widehat\zeta(\mu)$ form a locally analytic function, fix $\mu_0\in\Gamma_\delta$.
Since $\widehat m'(\widehat\zeta(\mu_0))\ne0$, the holomorphic inverse-function theorem \citep[Chapter~4, Section~3.3, Corollary~2]{Ahlfors1979} gives a local holomorphic inverse $h_{\mu_0}$ of $\widehat m$ near $-1/\mu_0$, with $h_{\mu_0}(-1/\mu_0)=\widehat\zeta(\mu_0)$.
Because $\mu_0\ne0$, the composition $\phi_{\mu_0}(\nu):=h_{\mu_0}(-1/\nu)$ is holomorphic on a sufficiently small neighborhood $U_{\mu_0}$ of $\mu_0$ and satisfies
$
 \phi_{\mu_0}(\mu_0)=\widehat\zeta(\mu_0),
 \widehat m(\phi_{\mu_0}(\nu))=-1/\nu.
$
By continuity of $\phi_{\mu_0}$ and $\zeta$, shrink $U_{\mu_0}$ so that, for $\nu\in U_{\mu_0}\cap\Gamma_\delta$,
\begin{align*}
 |\phi_{\mu_0}(\nu)-\zeta(\nu)|
 &\le|\phi_{\mu_0}(\nu)-\widehat\zeta(\mu_0)|
     +|\widehat\zeta(\mu_0)-\zeta(\mu_0)|
     +|\zeta(\mu_0)-\zeta(\nu)|\\
 &<\frac{c_1}{8}\delta^b+\frac{2}{c_m}n^{-3/4}+\frac{c_1}{8}\delta^b
 <\frac{c_1}{2}\delta^b.
\end{align*}
Thus $\phi_{\mu_0}(\nu)\in\mathcal D_\nu$, and uniqueness of the zero in $\mathcal D_\nu$ gives $\phi_{\mu_0}(\nu)=\widehat\zeta(\nu)$.
This proves local analyticity and continuity of  $\widehat\zeta$ on $\Gamma_\delta$.

\proofstep{Branch conditions.} We next verify the remaining branch conditions in \eqref{eq:empirical-calibration}.
Recall that $s_1,\ldots,s_n$ are the eigenvalues of $\widehat S$.
The identities $\widehat m(\bar\xi)=\overline{\widehat m(\xi)}$ and $\zeta(\bar\mu)=\overline{\zeta(\mu)}$ imply
\begin{align*}
 \widehat m\bigl(\overline{\widehat\zeta(\mu)}\bigr)
 =-1/\bar\mu, \quad
 \bigl|\overline{\widehat\zeta(\mu)}-\zeta(\bar\mu)\bigr|
 =|\widehat\zeta(\mu)-\zeta(\mu)|
 \le\frac{2}{c_m}n^{-3/4}<\frac{c_1}{2}\delta^b.
\end{align*}
Uniqueness in $\mathcal D_{\bar\mu}$ therefore gives
$\widehat\zeta(\bar\mu)=\overline{\widehat\zeta(\mu)}$.
In particular, the roots at $\mu_-=L-R_\delta<0$ and $\mu_+=L+R_\delta>0$ are real.
Since $n^{-3/4}\delta^{-2}\le n^{-3/4}(\log n)^2\to0$, \eqref{eq:preliminary-spectral} and \eqref{eq:empirical-root-displacement} give
\begin{align}
 \widehat\zeta(\mu_-)
 &\le\zeta(\mu_-)+\frac{2}{c_m}n^{-3/4}
 \le\lambda_{\min}(\widehat S)-c\delta^2+\frac{2}{c_m}n^{-3/4}
 \le\lambda_{\min}(\widehat S)-\frac c2\delta^2,\notag\\
 \widehat\zeta(\mu_+)
 &\ge\zeta(\mu_+)-\frac{2}{c_m}n^{-3/4}
 \ge\lambda_{\max}(\widehat S)+c\delta^2-\frac{2}{c_m}n^{-3/4}
 \ge\lambda_{\max}(\widehat S)+\frac c2\delta^2.
 \label{eq:empirical-real-crossings}
\end{align}
Since
\begin{align}
 \widehat m'(x)
 =\frac1n\sum_{i=1}^n\frac1{(s_i-x)^2}>0,
 \text{ for } x<\lambda_{\min}(\widehat S), \quad
 \lim_{x\to-\infty}\widehat m(x)=0, \quad
 \lim_{x\uparrow\lambda_{\min}(\widehat S)}\widehat m(x)=+\infty,
 \label{eq:empirical-stieltjes-monotonicity}
\end{align}
the value $\widehat\zeta(\mu_-)=\widehat\zeta(L-R_\delta)$ is the unique solution of $\widehat m(x)=-1/\mu_->0$ for $x<\lambda_{\min}(\widehat S)$. This checks the initial condition $\widehat\zeta(L-R_\delta)<\lambda_{\min}(\widehat S)$ in \eqref{eq:empirical-calibration}.
For $\mu\in\Gamma_\delta\setminus\R$, as in \eqref{eq:imaginary-part-zeta}, we have
\begin{align}\label{eq:empirical-calibration-nonreal}
 \frac{\operatorname{Im}\widehat\zeta(\mu)}n
 \sum_{i=1}^n\frac1{|s_i-\widehat\zeta(\mu)|^2}
 &=\operatorname{Im}\bigl(-1/\mu\bigr)
 =\frac{\operatorname{Im}\mu}{|\mu|^2}.
\end{align}
So
$\operatorname{Im}\widehat\zeta(\mu)\operatorname{Im}\mu>0$.
Combining this sign relation with \eqref{eq:empirical-real-crossings} and $\Gamma_\delta\cap\R=\{\mu_-,\mu_+\}$ gives
\begin{align*}
 \widehat\zeta(\Gamma_\delta)\cap
 [\lambda_{\min}(\widehat S),\lambda_{\max}(\widehat S)]
 =\varnothing.
\end{align*}
Hence all conditions in \eqref{eq:empirical-calibration} hold.
Finally, 
$\widehat\zeta(\mu)\in\mathcal D_\mu\subset\mathcal T^+$ allows us to apply \eqref{eq:preliminary-spectral} and get
\begin{align*}
 \sup_{\mu\in\Gamma_\delta}
 \|(\widehat S-\widehat\zeta(\mu)I_n)^{-1}\|_\op
 &=\sup_{\mu\in\Gamma_\delta}
   \frac1{\dist(\widehat\zeta(\mu),\spec(\widehat S))}
 \le C\delta^{-2}.
\end{align*}
Together with \eqref{eq:empirical-root-displacement} and $\delta\le1$, this proves \eqref{eq:empirical-control}.
\end{proof}

\subsection{Existence and Uniqueness of the Empirical Inverse}\label{app:pf:prop:empirical-inverse}
\begin{proof}[Proof of Proposition~\ref{prop:empirical-inverse}]
Fix $b>6$ in \eqref{eq:spectral-neighborhood} and work on the event $\mathcal A$ in \eqref{eq:averaged-uniform}. Lemma~\ref{lem:empirical-control} gives a continuous branch satisfying \eqref{eq:empirical-calibration}, and \eqref{eq:averaged-uniform} gives $\PP(\mathcal A^c)\le C_Dn^{-D}$ for every fixed $D>0$.

For uniqueness, let $\widehat\zeta_1,\widehat\zeta_2$ be two continuous functions satisfying \eqref{eq:empirical-calibration}. At $\mu_-=L-R_\delta<0$, \eqref{eq:empirical-calibration} gives
\[
 \widehat\zeta_j(\mu_-)<\lambda_{\min}(\widehat S),\qquad
 \widehat m(\widehat\zeta_j(\mu_-))=-1/\mu_-,\qquad j=1,2.
\]
The strict monotonicity in \eqref{eq:empirical-stieltjes-monotonicity} therefore implies $\widehat\zeta_1(\mu_-)=\widehat\zeta_2(\mu_-)$. Thus the set
$
 \mathcal E=\{\mu\in\Gamma_\delta:\widehat\zeta_1(\mu)=\widehat\zeta_2(\mu)\}
$
is nonempty, and continuity of both functions makes it closed in $\Gamma_\delta$.
For $\mu_0\in\mathcal E$, the common value $\xi_0=\widehat\zeta_1(\mu_0)=\widehat\zeta_2(\mu_0)$ satisfies $\widehat m'(\xi_0)\ne0$ by \eqref{eq:empirical-calibration}. The holomorphic inverse-function theorem \citep[Chapter~4, Section~3.3, Corollary~2]{Ahlfors1979} makes $\widehat m$ injective on a neighborhood of $\xi_0$. By continuity, both functions take values in this neighborhood for $\mu\in\Gamma_\delta$ sufficiently close to $\mu_0$, and \eqref{eq:empirical-calibration} gives $ \widehat m(\widehat\zeta_1(\mu))
=-1/\mu
=\widehat m(\widehat\zeta_2(\mu))$ and thus $\widehat\zeta_1(\mu)=\widehat\zeta_2(\mu)$.
Hence $\mathcal E$ is also open in $\Gamma_\delta$. Since $\Gamma_\delta$ is connected, $\mathcal E=\Gamma_\delta$, proving uniqueness.
\end{proof}

\subsection{Proof of Lemma~\ref{lem:preliminary}}\label{app:calibration-comparison}
We use stochastic domination in the sense of \citet[Definition~3.4]{KnowlesYin2017}:
\begin{equation}\label{eq:stochastic-domination}
      U\prec V
      \quad\Longleftrightarrow\quad
      \text{for every }\epsilon,D>0,\quad
      \sup_\theta\PP\{|U(\theta)|>n^\epsilon V(\theta)\}
            \le C_{\epsilon,D}n^{-D},
\end{equation}
where $\theta$ ranges over the indicated deterministic parameters and $n$ is sufficiently large. Most of the results in \citet{KnowlesYin2017} are stated using the above polynomial failure probability for any exponent $D$. To apply them to our setting which requires uniform control when  $\delta$ shrinks as $(\log n)^{-1}$, we give the following lemma which shows when stochastic domination implies the uniform estimate.

\begin{lemma}\label{lem:gap-continuation}
Let $\delta\ge1/\log n$ and let $\alpha,B,M,c,K>0$ be fixed. Put $a=c\delta^B$, and let
$\mathcal R\subset\{\xi:\operatorname{Im}\xi\ge0,\ |\xi|\le K\}$ be compact.
Write $\mathcal R^{(s)}=\{\xi:\dist(\xi,\mathcal R)<sa\}$ for $s=1,2$.
Suppose that, on an event $\mathcal B$, a random function $F$ is holomorphic on $\mathcal R^{(2)}$ and
\begin{equation}\label{eq:continuation-pathwise}
 \sup_{\xi\in\mathcal R^{(2)}}|F(\xi)|\le C\delta^{-M}.
\end{equation}
Suppose also that, uniformly for $\xi\in\mathcal R^{(1)}$ with $\eta=\operatorname{Im}\xi\ge n^{-1/2}$,
\begin{equation}\label{eq:continuation-input}
 \ind_{\mathcal B}|F(\xi)|\prec n^{-\alpha}(1+\eta^{-1}).
\end{equation}
Then, with $F$ defined to be zero on $\mathcal B^c$, we have
\begin{equation}\label{eq:continuation-output}
 \sup_{\xi\in\mathcal R}|\ind_{\mathcal B}F(\xi)|\prec n^{-\alpha},
 \qquad
 \Big\|\sup_{\xi\in\mathcal R}|\ind_{\mathcal B}F(\xi)|\Big\|_{L^r}
       \le C n^{-\alpha+\epsilon}
 \text{ for any } r\ge1,\ \epsilon>0.
\end{equation}
\end{lemma}

\begin{proof}
Fix $\epsilon,D>0$. Choose $0<\theta<\min(\epsilon/4,1/4)$, put $h=n^{-\theta}$, and take an integer $q$ such that $q\theta>\alpha+1$.
For large $n$, $qh<a/4$. Cauchy's estimate \citep[Chapter~4, Eq.~(25)]{Ahlfors1979} and \eqref{eq:continuation-pathwise} give
\begin{equation}\label{eq:continuation-cauchy}
      \sup_{\xi\in\mathcal R^{(1)}}|F^{(k)}(\xi)|
      \le C_k a^{-k}\delta^{-M} \text{on }\mathcal B.   
\end{equation}
The identities
$\sum_{j=1}^q(-1)^{j+1}\binom qj j^k=\ind_{\{k=0\}}$, $0\le k<q$,
and Taylor's expansion imply, on $\mathcal B$,
\begin{equation}\label{eq:continuation-interpolation}
 F(\xi)=\sum_{j=1}^q(-1)^{j+1}\binom qjF(\xi+\ii jh)
             +O_q(h^qa^{-q}\delta^{-M}),\qquad \xi\in\mathcal R.
\end{equation}
The shifted parameters $\xi+\ii jh$ in \eqref{eq:continuation-interpolation} satisfy
\begin{align*}
 \operatorname{Im}(\xi+\ii jh)
 &=\operatorname{Im}\xi+jh
 \ge h=n^{-\theta}\ge n^{-1/4}\ge n^{-1/2},\\
 \dist(\xi+\ii jh,\mathcal R)
 &\le|(\xi+\ii jh)-\xi|=jh\le qh<a/4<a.
\end{align*}
The distance bound uses $\xi\in\mathcal R$ and shows that every shifted parameter lies in $\mathcal R^{(1)}$.
By \eqref{eq:continuation-input} and \eqref{eq:stochastic-domination}, for every $D_0>0$ and all sufficiently large $n$,
\begin{equation}\label{eq:continuation-input-bound}
 \sup_{\xi\in\mathcal R}\max_{1\le j\le q}
 \PP\left\{
  |\ind_{\mathcal B}F(\xi+\ii jh)|
  >n^{-\alpha+\epsilon/4}(1+h^{-1})
 \right\}
 \le C_{\epsilon,D_0}n^{-D_0}.
\end{equation}
The remainder in \eqref{eq:continuation-interpolation} is $o(n^{-\alpha})$, since
$h^q\le n^{-\alpha-1}$ and $a^{-q}\delta^{-M}\le c^{-q}(\log n)^{Bq+M}$.

Fix $A>\alpha+2$ and choose an $n^{-A}$-net denoted as $\mathcal R_A\subset\mathcal R$ such that
\[
 \sup_{\xi\in\mathcal R}\min_{\nu\in\mathcal R_A}|\xi-\nu|\le n^{-A},
 \text{ and } |\mathcal R_A|\le Cn^{2A}.
\]
For $\nu\in\mathcal R_A$, $j=1,\ldots,q$, $\nu+\ii jh$ lies in $\mathcal R^{(1)}$ because $qh<a/4$.
Applying \eqref{eq:continuation-input} with failure exponent $D+2A+1$ and taking a union bound of \eqref{eq:continuation-input-bound} give
\[
 \begin{aligned}
\mathcal{E}_A = \left\{
      \max_{\nu\in\mathcal R_A}\max_{1\le j\le q}
   |\ind_{\mathcal B}F(\nu+\ii jh)|
   \le n^{-\alpha+\epsilon/4}(1+h^{-1})
 \right\} \text{ and }
 \PP(\mathcal{E}_A^c)
   \le Cq\,|\mathcal R_A|n^{-D-2A-1}
 \le C_{\epsilon,D}n^{-D}.
 \end{aligned}
\]
On $\mathcal{E}_A$, \eqref{eq:continuation-interpolation} implies
\[
 \begin{aligned}
 \max_{\nu\in\mathcal R_A}|\ind_{\mathcal B}F(\nu)|
 &\le(2^q-1)n^{-\alpha+\epsilon/4}(1+h^{-1})
       +C_qh^qa^{-q}\delta^{-M}
       \le C_qn^{-\alpha+\epsilon/2}.
 \end{aligned}
\]
For each $\xi\in\mathcal R$, choose $\nu_\xi\in\mathcal R_A$ with $|\xi-\nu_\xi|\le n^{-A}$. For $0\le t\le1$ and all large $n$,
$\dist(\nu_\xi+t(\xi-\nu_\xi),\mathcal R)\le n^{-A}<a$,
so the joining segment $\{\nu_\xi+t(\xi-\nu_\xi):0\le t\le1\}$ lies in $\mathcal R^{(1)}$. On $\mathcal B$,
\begin{align*}
 |F(\xi)-F(\nu_\xi)|
 &=\left|\int_0^1(\xi-\nu_\xi)
       F'(\nu_\xi+t(\xi-\nu_\xi))\dd t\right|
       \le|\xi-\nu_\xi|
       \int_0^1|F'(\nu_\xi+t(\xi-\nu_\xi))|\dd t\\
 &\le n^{-A}\sup_{\omega\in\mathcal R^{(1)}}|F'(\omega)|
 \le Cn^{-A}a^{-1}\delta^{-M}
 =o(n^{-\alpha}).
\end{align*}
The second inequality uses the $n^{-A}$-net property and the third inequality uses \eqref{eq:continuation-cauchy} with $k=1$, and the last equality uses $a=c\delta^B$, $\delta^{-1}\le\log n$, and $A>\alpha+2$.
Consequently, on $\mathcal{E}_A$, we have
\begin{align*}
 \sup_{\xi\in\mathcal R}|\ind_{\mathcal B}F(\xi)|
 &\le\max_{\nu\in\mathcal R_A}|\ind_{\mathcal B}F(\nu)|
   +\sup_{\xi\in\mathcal R}
      |\ind_{\mathcal B}(F(\xi)-F(\nu_\xi))|\\
 &\le C_qn^{-\alpha+\epsilon/2}
       +C'n^{-A}(\log n)^{B+M}
 \le C'_qn^{-\alpha+\epsilon/2}.
\end{align*}
As we define $F$ to be zero on $\mathcal B^c$, this inequality is valid also on $\mathcal B^c$.
Since $C'_qn^{\epsilon/2}\le n^\epsilon$ for all large $n$ and the failure probability is at most $C_{\epsilon,D}n^{-D}$, this proves the first assertion of \eqref{eq:continuation-output}.

For the second assertion of \eqref{eq:continuation-output}, fix $r\ge1$, let $V=\sup_{\xi\in\mathcal R}|\ind_{\mathcal B}F(\xi)|$, and choose $D_0>r\alpha+1$.
The first assertion of \eqref{eq:continuation-output}, with exponent $\epsilon/2$, gives
$\PP(V>n^{-\alpha+\epsilon/2})\le C_{\epsilon,D_0}n^{-D_0}$.
Also, $V\le C\delta^{-M}$ on every sample by \eqref{eq:continuation-pathwise}. Therefore, we have
\begin{align*}
 \E V^r
 &=\E\!\left[V^r\ind_{\{V\le n^{-\alpha+\epsilon/2}\}}\right]
   +\E\!\left[V^r\ind_{\{V>n^{-\alpha+\epsilon/2}\}}\right]\\
 &\le n^{-r\alpha+r\epsilon/2}
       +(C\delta^{-M})^r\PP(V>n^{-\alpha+\epsilon/2})
\le n^{-r\alpha+r\epsilon/2}
       +C_{r,\epsilon,D_0}\delta^{-Mr}n^{-D_0}
 \le C_{r,\epsilon}n^{-r\alpha+r\epsilon}.
\end{align*}
The last inequality uses $(\log n)^{Mr}\le n$ for all large $n$ and $D_0>r\alpha+1$.
Therefore, we have $\|V\|_{L^r}\le C_{r,\epsilon}n^{-\alpha+\epsilon}$.
\end{proof}

\begin{proof}[Proof of Lemma~\ref{lem:preliminary}]
Throughout the proof, $v,w$ are deterministic unit vectors. The proof is based on the results in \citet{KnowlesYin2017}. 

We first compare the differences in notation between our paper and \citet{KnowlesYin2017} so it is easier for the readers to cross-check.
Set $\mathsf Z=Z^\top/\sqrt n$ and $B_\Sigma=\diag(\Sigma^{1/2},I_n)$.
The $G$ and $\Pi$ matrices of \citet[Definitions~3.2--3.3]{KnowlesYin2017}, denoted as $G_{\Sigma}$ and $\Pi_{\Sigma}$ below,  satisfy
$
 G_{\Sigma}=B_\Sigma G B_\Sigma,
 \Pi_{\Sigma}=B_\Sigma\Pi B_\Sigma,$ and their 
 $m_N=\widehat m$ in \citet[Eq.~(3.9)]{KnowlesYin2017}.
Their sample size and dimensions in our setting are $N=n$ and $M=\widehat M=p$. Assumption~2.1 and Eq.~(2.9) of \citet{KnowlesYin2017} hold under our setup in Section~\ref{sec:setup}:
\[
 \E\mathsf Z_{ki}=0,\quad
 \E|\mathsf Z_{ki}|^2=n^{-1},\quad
 \sup_{k,i}\E|\sqrt n\mathsf Z_{ki}|^r\le M_r,\quad
 \ell I_p\preceq\Sigma\preceq LI_p,\quad
 1< \gamma_0\le \gamma \le \gamma_1<\infty.
\]
Here $\gamma_0$ and $\gamma_1$ are fixed constants not changing with $n, p$.

\proofstep{Spectral bounds.} We next aim to show $cI_n\preceq\widehat S\preceq CI_n$ in \eqref{eq:preliminary-spectral} for some constants $c,C$ to be specified later with high probability.
We aim to apply \citet[Theorem~3.12]{KnowlesYin2017} to $W_0=ZZ^\top/n$.
Its limiting spectral distribution is the Marchenko--Pastur law, with parameters and density given by
\[
 a_\pm^0=(\sqrt\gamma\pm1)^2,\qquad
 x_\pm^0=-\frac1{1\pm\sqrt\gamma},\qquad
 \rho_0(x)=\frac{\sqrt{(a_+^0-x)(x-a_-^0)}}{2\pi x}.
\]
Thus $a_-^0\ge c$, $a_+^0-a_-^0\ge c$, $|x_\pm^0+1|\ge c$, and for any fixed $\tau>0$,
$\inf_{[a_-^0+\tau,a_+^0-\tau]}\rho_0\ge c$.
Therefore, the regularity conditions in \citet[Definition~2.7]{KnowlesYin2017} hold for $W_0$.
Since $\ell I_p\preceq\Sigma\preceq LI_p$,
\begin{align*}
 \ell W_0\preceq\widehat S=\frac1nZ\Sigma Z^\top\preceq LW_0,\quad 
 \ell\lambda_{\min}(W_0)
 \le\lambda_{\min}(\widehat S)
 \le\lambda_{\max}(\widehat S)\le L\lambda_{\max}(W_0).
\end{align*}
Set $c_*=\ell^2(\sqrt{\gamma_0}-1)^2/(2L)$ and
$C_*=2L(\sqrt{\gamma_1}+1)^2$. Since $\ell<L$, for all sufficiently large $n$,
\[
 c_*\le\ell(a_-^0-n^{-1/2}),\qquad
 L(a_+^0+n^{-1/2})\le C_*.
\]
Thus, for every $D>0$ and sufficiently large $n$, we have
\begin{align}
\notag &\PP\bigl(\spec(\widehat S)\subset[c_*,C_*]\bigr)
      = \PP\bigl(c_*\le\lambda_{\min}(\widehat S),\,
      \lambda_{\max}(\widehat S)\le C_*\bigr)\\
      &\qquad \ge \PP\bigl(a_-^0-n^{-1/2}\le\lambda_{\min}(W_0),\,
      \lambda_{\max}(W_0)\le a_+^0+n^{-1/2}\bigr) 
      \ge 1-C_Dn^{-D}, \label{eq:KY-spectral-bounds}
\end{align}
where the first inequality follows from the eigenvalue comparison and the bounds on $c_*,C_*$ above, and the second inequality follows from \citet[Theorem~3.12]{KnowlesYin2017}.

We next argue that the same interval also contains $\supp\varrho$.
Indeed, the support characterization in \citet[Lemmas~2.4--2.6]{KnowlesYin2017} identifies its lower edge using the unique $s>0$ satisfying
\[
 \frac1n\sum_{j=1}^p\frac{\lambda_j^2}{(\lambda_j+s)^2}=1,
 \qquad \ell(\sqrt\gamma-1)\le s\le L(\sqrt\gamma-1).
\]
And \citet[Lemmas~2.4--2.6]{KnowlesYin2017} give
\[
 \min\supp\varrho
 =\frac{s^2}{n}\sum_{j=1}^p\frac{\lambda_j}{(\lambda_j+s)^2}
 \ge\frac{\ell^2}{L}(\sqrt{\gamma_0}-1)^2,
 \quad
 \max\supp\varrho\le\zeta\bigl(L(1+\sqrt\gamma)\bigr)
                    \le L(1+\sqrt{\gamma_1})^2.
\]
 By \citet[Lemma~A.1]{KnowlesYin2017}, as $p>n$, there is no atom at zero, i.e., $\varrho(\{0\})=0$.
Thus $\supp\varrho\subset[2c_*,C_*/2]\subset[c_*,C_*]$.

Choose $r_0>0$ with $16r_0<c_*$ and set
\begin{equation}\label{eq:KY-enlarged-domain}
 \mathcal V=\{\xi:\dist(\xi,\zeta(\Gamma_\delta))<6c_1\delta^b\}
                  \cup\{\xi:|\xi|<3r_0\}.
\end{equation}
Comparing \eqref{eq:spectral-neighborhood} and \eqref{eq:KY-enlarged-domain} gives $\mathcal T\subset\mathcal T^+\subset\mathcal V$, since $c_1\delta^b<3c_1\delta^b<6c_1\delta^b$.
Let $c$ be the constant in \eqref{eq:calibration-separation} and choose $c_1$ so that $6c_1\le c/2$.
Since $b\ge2$, for $0<\delta\le1$ we have $6c_1\delta^b\le c\delta^2/2$.
If $\dist(\xi,\zeta(\Gamma_\delta))<6c_1\delta^b$, choose $\mu\in\Gamma_\delta$ with
$|\xi-\zeta(\mu)|<6c_1\delta^b$. Then
\begin{align*}
 \dist(\xi,\supp\varrho)
 &\ge\dist(\zeta(\mu),\supp\varrho)-|\xi-\zeta(\mu)|
 \ge c\delta^2-6c_1\delta^b\ge\frac c2\delta^2,\\
 \|\Pi(\xi)\|_\op
 &=\max\bigl\{|m(\xi)|,\|(I_p+m(\xi)\Sigma)^{-1}\|_\op\bigr\}
 \le\max\{C_0\delta^{-2},C_0\}\le C\delta^{-2},
\end{align*}
where the distance bound follows from \eqref{eq:calibration-separation}, and the operator norm bound follows from
\eqref{eq:linearization} and \eqref{eq:calibration-derivatives} with $k=0$ in Lemma~\ref{lem:calibration-control}.

For the disk $\{\xi\in\C:|\xi|<3r_0\}$ in \eqref{eq:KY-enlarged-domain},
$\supp\varrho\subset[c_*,C_*]$ and $16r_0<c_*$ give
\begin{align*}
 \dist(\xi,\supp\varrho)
 &\ge c_*-|\xi|\ge c_*-3r_0\ge c_*/2,\\
 |m(\xi)|
 &=\biggl|\int\frac{\varrho(\dd x)}{x-\xi}\biggr|
 \le\int\frac{\varrho(\dd x)}{|x-\xi|}
 \le\frac1{c_*-3r_0}\le\frac2{c_*},\\
 \operatorname{Re}m(\xi)
 &=\int\frac{x-\operatorname{Re}\xi}{|x-\xi|^2}\,\varrho(\dd x)>0,\\
 \|(I_p+m(\xi)\Sigma)^{-1}\|_\op
 &=\max_{1\le j\le p}\frac1{|1+\lambda_jm(\xi)|}
 \le\max_{1\le j\le p}\frac1{1+\lambda_j\operatorname{Re}m(\xi)}\le1.
\end{align*}
Thus \eqref{eq:linearization} gives $\|\Pi(\xi)\|_\op\le\max\{1,2/c_*\}$ on this disk $\{\xi\in\C:|\xi|<3r_0\}$.
Combining the bounds on the two sets in \eqref{eq:KY-enlarged-domain}, and using $\delta\le1$, yields 
\begin{equation}\label{eq:comparison-deterministic-gap}
 \dist(\mathcal V,\supp\varrho)\ge c\delta^2,
 \qquad \sup_{\xi\in\mathcal V}\|\Pi(\xi)\|_\op\le C\delta^{-2}.
\end{equation}
Write $a_-=\min\supp\varrho$, $a_+=\max\supp\varrho$, and $\xi_\pm=\zeta(L\pm R_\delta)$.
For real $\mu\in(-\infty,L-R_\delta]\cup[L+R_\delta,\infty)$,
\eqref{eq:zeta-calibration-stability} gives
\[
 \zeta'(\mu)
 =1-\frac1n\sum_{j=1}^p\frac{\lambda_j^2}{(\lambda_j-\mu)^2}
 \ge1-(1+\delta)^{-2}>0.
\]
Thus $\zeta$ is strictly increasing on each of these two rays: for $\mu<\nu$ on the same ray,
$\zeta(\nu)-\zeta(\mu)=\int_\mu^\nu\zeta'(t)\dd t
\ge[1-(1+\delta)^{-2}](\nu-\mu)>0$.
By Proposition~\ref{prop:calibration}, the images of both rays avoid $\supp\varrho$;
also, \eqref{eq:zeta} gives $\zeta(\mu)\to\pm\infty$ as $\mu\to\pm\infty$.
Consequently,
\[
 \zeta\bigl((-\infty,L-R_\delta]\bigr)=(-\infty,\xi_-]\subset(-\infty,a_-),
 \qquad
 \zeta\bigl([L+R_\delta,\infty)\bigr)=[\xi_+,\infty)\subset(a_+,\infty).
\]
Combining these inclusions with \eqref{eq:calibration-separation} at $\mu=L\pm R_\delta$ gives
\begin{equation}\label{eq:comparison-crossings}
 \xi_- = \zeta(L-R_\delta)\le a_--c\delta^2,\qquad \xi_+= \zeta(L+R_\delta)\ge a_++c\delta^2.
\end{equation}

\proofstep{Two important properties.} For the rest of the proof, to establish \eqref{eq:preliminary-spectral} and \eqref{eq:coarse-moments}, we essentially need to prove the following two important properties.

First, we establish the {\it spectral gap} which leads to \eqref{eq:preliminary-spectral}:
\begin{equation}\label{eq:comparison-sample-gap}
 \dist(\mathcal V,\spec(\widehat S))\ge c\delta^2, \quad
      \xi_- \le\lambda_{\min}(\widehat S)-c\delta^2, \quad
      \xi_+\ge\lambda_{\max}(\widehat S)+c\delta^2.
     \end{equation}
This holds with probability at least $1-C_Dn^{-D}$ for every fixed $D>0$ and all sufficiently large $n$.

Second, we need the {\it anisotropic and averaged local laws}, uniformly for $\xi=E+\ii\eta\in\mathcal V$ with $|\xi|\ge r_0/2$ and $\eta\ge n^{-1/2}$, which lead to \eqref{eq:coarse-moments}:
     \begin{equation}\label{eq:comparison-nonreal-laws}
      |v^\top(G(\xi)-\Pi(\xi))w|
            \prec\sqrt{\frac{\operatorname{Im}m(\xi)}{n\eta}}+\frac1{n\eta},
      \qquad
      |\widehat m(\xi)-m(\xi)|\prec\frac1{n\eta}.
\end{equation}
These are the estimates in \citet[Theorem~3.6]{KnowlesYin2017}, which requires the edge and bulk regularity conditions in \citet[Definition~2.7]{KnowlesYin2017}.
For \eqref{eq:coarse-moments}, we need these estimates only on the restricted subset of $\mathcal V$ and thus do not need the edge and bulk regularity conditions in general.
Instead, we will use the other results in \citet{KnowlesYin2017} to establish  \eqref{eq:coarse-moments} for our purpose. Specifically, put
\[
 d=\frac1p\sum_{j=1}^p\left(1-\frac{\lambda_j}{L}\right),\qquad
 d_0=\frac{\gamma_0-1}{8\gamma_0},\qquad
 q(\mu)=\frac1n\sum_{j=1}^p\frac{\lambda_j^2}{|\lambda_j-\mu|^2}.
\]
We will consider two cases: $d\ge d_0$ and $d<d_0$, and prove \eqref{eq:comparison-sample-gap} and \eqref{eq:comparison-nonreal-laws}  separately for these two cases.

\proofstep{Case 1: $d\ge d_0$.}
For $|\mu-L|\ge\sqrt\gamma L$, we have
\[
 q(\mu)\le\frac1p\sum_{j=1}^p
       \biggl(\frac{\sqrt\gamma\,\lambda_j}{\lambda_j+L(\sqrt\gamma-1)}\biggr)^2
 \le1-(1-\gamma_0^{-1/2})^2d
 \le1-(1-\gamma_0^{-1/2})^2d_0.
\]
Set $c=(1-\gamma_0^{-1/2})^2d_0/2\in(0,1/2)$ and
$r=cL(\sqrt{\gamma_0}-1)/2>0$.
For $\mu\in\Gamma_\delta$ and $|\nu-\mu|\le r$, the triangle inequality gives
\begin{align}
\notag \min\bigl\{|\mu|,\min_{1\le j\le p}|\mu-\lambda_j|\bigr\} 
 &\ge R_\delta-L\ge L(\sqrt{\gamma_0}-1)=2r/c, \text{ which further implies}\\
 \notag  |\nu|&\ge2r/c-r>0,\\
\label{eq:zeta-local-bound-e1} |\nu-\lambda_j|&\ge|\mu-\lambda_j|-r
 \ge(1-c/2)|\mu-\lambda_j|\ge(1-c/2)\,2r/c>0,\\
 \label{eq:zeta-local-bound-e2}  q(\nu)&\le\frac{q(\mu)}{(1-c/2)^2}
 \le\frac{1-2c}{(1-c/2)^2}\le1-c.
\end{align}
In particular, \eqref{eq:zeta-local-bound-e1} combined with $\zeta$'s formula in \eqref{eq:zeta-formula} implies that $\zeta$ is holomorphic on a neighborhood of $\{\nu:|\nu-\mu|\le r\}$ for $\mu\in\Gamma_\delta$.
As in \eqref{eq:zeta-difference}, we have, for $|\nu_i-\mu|\le r$, $i=1,2$,
\[
 |\zeta(\nu_1)-\zeta(\nu_2)|
 \ge\bigl(1-\sqrt{q(\nu_1)q(\nu_2)}\bigr)|\nu_1-\nu_2|
 \ge c|\nu_1-\nu_2|,
\]
where the last inequality follows from \eqref{eq:zeta-local-bound-e2}.
Thus $\zeta$ is injective on the disk $\{\nu:|\nu-\mu|\le r\}$, and $\zeta(\nu)-\zeta(\mu)$ has exactly one zero there, at $\nu=\mu$;
the zero is simple by \eqref{eq:stability-main}.
For $|\omega-\zeta(\mu)|<cr$ and $|\nu-\mu|=r$, we have
\[
 |(\zeta(\nu)-\omega)-(\zeta(\nu)-\zeta(\mu))|
 =|\omega-\zeta(\mu)|<cr\le|\zeta(\nu)-\zeta(\mu)|.
\]
Rouch\'e's theorem \citep[Chapter~4, Section~5.2]{Ahlfors1979}, applied as in the proof of \eqref{eq:calibration-image-inclusion}, therefore gives
\begin{equation}\label{eq:fixed-gap-image-inclusion}
 \{\omega:|\omega-\zeta(\mu)|<cr\}
 \subset\zeta\bigl(\{\nu:|\nu-\mu|<r\}\bigr),
 \qquad \mu\in\Gamma_\delta.
\end{equation}
For each such $\omega$, let $\nu$ be the unique solution of $\zeta(\nu)=\omega$ with $|\nu-\mu|<r$.
By \eqref{eq:imaginary-part-zeta}, we have
\[
 |\operatorname{Im}\omega| = |\operatorname{Im}\zeta(\nu)|
 =(1-q(\nu))|\operatorname{Im}\nu|
 \ge c|\operatorname{Im}\nu|,
 \text{ and thus } \omega\in\R \Longrightarrow \nu\in\R.
\]
If $\omega\in\R$, then $\nu\ne0$, $\nu\notin\spec(\Sigma)$, and as in \eqref{eq:xi-prime}, we have
\[
 \partial_u\biggl(-\frac1u+\frac1n\sum_{j=1}^p\frac{\lambda_j}{1+\lambda_ju}\biggr)
 \bigg|_{u=-1/\nu}
 =\nu^2\zeta'(\nu)=\nu^2(1-q(\nu))\ge c\nu^2>0.
\]
The support criterion of \citet[Theorem~4.2]{SilversteinChoi1995} therefore gives $\omega=\zeta(\nu)\notin\supp\varrho$.
If $\omega\notin\R$, the same conclusion follows from $\supp\varrho\subset\R$.
Consequently, for every $\mu\in\Gamma_\delta$,
\begin{equation}\label{eq:fixed-gap-support-criterion}
      \{\omega:|\omega-\zeta(\mu)|<cr\}\cap\supp\varrho=\varnothing,
 \qquad \dist(\zeta(\mu),\supp\varrho)\ge cr. 
\end{equation}
Set $g=\min\{cr/2,c_*/4\}>0$ and take $n$ sufficiently large that $6c_1\delta^b\le cr/2$.
Combining \eqref{eq:fixed-gap-support-criterion}, \eqref{eq:KY-enlarged-domain}, and \eqref{eq:comparison-crossings} gives
\begin{equation}\label{eq:fixed-gap-alternative}
 \begin{aligned}
 \dist(\mathcal V,\supp\varrho)
 &\ge\min\{cr-6c_1\delta^b,c_*-3r_0\}\ge g,\\
 \xi_-&\le a_--cr\le a_--g,\qquad
 \xi_+\ge a_++cr\ge a_++g.
 \end{aligned}
\end{equation}

We further split Case 1 into two subcases: Case 1(a): $n^{-1/2}\le\eta<g/2$ and Case 1(b): $\eta\ge g/2$.

\proofstep{Case 1(a): $d\ge d_0$, $n^{-1/2}\le\eta<g/2$; local law.}
We now apply \citet[Theorem~3.16(i)]{KnowlesYin2017} at $\xi=E+\ii\eta\in\mathcal V$ with $|\xi|\ge r_0/2$ and $n^{-1/2}\le\eta<g/2$.
Define the domains
\begin{align*}
 D(\tau,n)
 &=\{\xi=E+\ii\eta\in\C^+:|\xi|\ge\tau,\ |E|\le\tau^{-1},\ n^{-1+\tau}\le\eta\le\tau^{-1}\},\\
 D^{\mathrm{o}}(\tau,\tau',n)
 &=\{\xi\in D(\tau,n):\dist(\operatorname{Re}\xi,\supp\varrho)\ge\tau'\}.
\end{align*}
Here, $D$ is the same as in \citet[Eq.~(3.7)]{KnowlesYin2017} and $D^{\mathrm{o}}$ is considered in \citet[Section~3.3]{KnowlesYin2017}.
Fix $K\ge\max\{1,C_*\}$ with $\sup_{\xi\in\mathcal V}|\xi|\le K$ for all sufficiently large $n$, and choose
$0<\tau<\min\{r_0/2,g/4,1/4,K^{-1}\}$ small enough such that  Assumption~2.1 of \citet{KnowlesYin2017} holds; set $\tau'=g/2$.
For $\xi=E+\ii\eta\in\mathcal V$ with $|\xi|\ge r_0/2$ and $n^{-1/2}\le\eta<g/2$, we then have
\begin{align*}
 |\xi|&\ge r_0/2>\tau,\quad
 \max\{|E|,\eta\}\le|\xi|\le K<\tau^{-1},\quad
 n^{-1+\tau}\le n^{-3/4}\le n^{-1/2}\le\eta,\\
 \dist(E,\supp\varrho)
 &\ge\dist(\xi,\supp\varrho)-|\xi-E|
 =\dist(\xi,\supp\varrho)-\eta
 \ge g-\eta>g/2=\tau',
\end{align*}
where the last line uses \eqref{eq:fixed-gap-alternative}.
Thus $\xi\in D^{\mathrm{o}}(\tau,g/2,n)$, and \citet[Theorem~3.16(i), Eqs.~(3.10)--(3.11)]{KnowlesYin2017} give \eqref{eq:comparison-nonreal-laws}.

\proofstep{Case 1(b): $d\ge d_0$, $\eta\ge g/2$; local law.}
For the remaining points $\eta\ge g/2$, we use \citet[Theorems~3.21--3.22]{KnowlesYin2017} on
\begin{equation}\label{eq:comparison-high-eta-domain}
 \mathcal S=\{\xi=E+\ii\eta:|E|\le K,\ g/2\le\eta\le\tau^{-1}\}\subset D(\tau,n),
\end{equation}
where the inclusion holds for all sufficiently large $n$.
This is a spectral domain in \citet[Definition~3.20]{KnowlesYin2017}, since
$\xi\in\mathcal S$ and $\eta\le t\le\tau^{-1}$ imply $E+\ii t\in\mathcal S$.
Choose the fixed $\tau$ above small enough also to satisfy
$\tau\le\min\{1/2,\ell g/(2((C_*+K)^2+4L^2))\}$.
We next check the condition in \citet[Eq.~(3.20)]{KnowlesYin2017} that
\begin{equation}\label{eq:comparison-denominator-condition}
 \inf_{\xi\in\mathcal S}\min_{1\le j\le p}|1+\lambda_jm(\xi)|\ge\tau.
\end{equation}
In fact, as $\supp\varrho\subset[c_*,C_*]$, we have for all $\xi\in\mathcal S$ and $j=1,\ldots,p$,
\begin{align*}
 \text{ when }\eta\ge2L:\quad
 |1+\lambda_jm(\xi)|
 &\ge1-\lambda_j|m(\xi)|\ge1-L/\eta\ge1/2\ge\tau,\\
 \text{ when }g/2\le\eta<2L:\quad
 |1+\lambda_jm(\xi)|
 &\ge\lambda_j\operatorname{Im}m(\xi)
 \ge\ell\int\frac{\eta}{(x-E)^2+\eta^2}\,\varrho(\dd x)\\
 &\ge\frac{\ell g}{2((C_*+K)^2+4L^2)}\ge\tau.
\end{align*}

Recall $f$ from \eqref{eq:f-definition} and define $A$ as follows:
\[
 f(w)= -\frac1w+\frac1n\tr\bigl(\Sigma(I_p+w\Sigma)^{-1}\bigr) =-\frac1w+\frac1n\sum_{j=1}^p\frac{\lambda_j}{1+\lambda_jw},
 \text{ and }
 A(w)=\frac1n\sum_{j=1}^p\biggl|\frac{\lambda_jw}{1+\lambda_jw}\biggr|^2,
\]
so that \eqref{eq:MP} is $f(m(\xi))=\xi$.
Let $u$ be a function of the form of the Stieltjes transform of some probability measure $\widetilde\varrho$ supported on a fixed interval $[0,C]$:
\begin{equation}\label{eq:comparison-test-transform}
 u:\C^+\longrightarrow\C^+,\qquad
 u(\xi)=\int_{[0,C]}\frac{\widetilde\varrho(\dd x)}{x-\xi}.
\end{equation}
We next prove the following stability condition introduced in \citet[Definition~5.4]{KnowlesYin2017}, uniformly over $\xi\in\mathcal S$, $0<\epsilon\le\min\{g/4,1\}$, and $u$ in the form of \eqref{eq:comparison-test-transform}:
\begin{equation}\label{eq:comparison-stability-condition}
 |f(u(\xi))-\xi|\le\epsilon
 \quad\Longrightarrow\quad
 |u(\xi)-m(\xi)|\le C\epsilon.
\end{equation}
As in \eqref{eq:imaginary-part-zeta} with $\mu = -1/m(\xi)$, taking imaginary parts of $f(m(\xi))=\zeta(-1/m(\xi)) = \xi$ gives
\begin{align}
 \notag 1-A(m(\xi))
 &=\frac{\eta|m(\xi)|^2}{\operatorname{Im}m(\xi)}
 \ge\eta\operatorname{Im}m(\xi)
 =\int\frac{\eta^2}{(x-E)^2+\eta^2}\,\varrho(\dd x)\\
 &\ge\frac{g^2}{4(C_*+K)^2+g^2}=:c_{\mathrm s}>0. \label{eq:comparison-stability-condition-1}
\end{align}
If $|f(u(\xi))-\xi|\le\epsilon\le g/4$, then $\operatorname{Im}f(u(\xi))\ge\eta-\epsilon\ge g/4>0$ and thus
\begin{align}\label{eq:comparison-stability-condition-2}
 A(u(\xi))
 &=1-\frac{\operatorname{Im}f(u(\xi))\,|u(\xi)|^2}{\operatorname{Im}u(\xi)}\le1,
 \qquad |u(\xi)| \text{ and } |m(\xi)|\le\eta^{-1}\le2/g.
\end{align}
By the definition of $f$ in \eqref{eq:f-definition}, we have
\begin{align*}
 f(u(\xi))-\xi
 &=\frac{u(\xi)-m(\xi)}{u(\xi)m(\xi)}
  \biggl(1-\frac1n\sum_{j=1}^p
       \frac{\lambda_j^2u(\xi)m(\xi)}{(1+\lambda_ju(\xi))(1+\lambda_jm(\xi))}\biggr),\\
 \biggl|\frac1n\sum_{j=1}^p
       \frac{\lambda_j^2u(\xi)m(\xi)}{(1+\lambda_ju(\xi))(1+\lambda_jm(\xi))}\biggr|
 &\le\sqrt{A(u(\xi))A(m(\xi))}\le\sqrt{1-c_{\mathrm s}}<1,\\
 |u(\xi)-m(\xi)|
 &\le\frac{4}{g^2(1-\sqrt{1-c_{\mathrm s}})}|f(u(\xi))-\xi|\le C\epsilon,
\end{align*}
where the second line is by the Cauchy--Schwarz inequality and \eqref{eq:comparison-stability-condition-1}, and the third line is by combining the first line with \eqref{eq:comparison-stability-condition-2}. 
This proves the stability condition in \eqref{eq:comparison-stability-condition}.
Since $\operatorname{Im}m(\xi)\le2/g$ and $\epsilon\le1$, it also gives
\[
 |u(\xi)-m(\xi)|\le C\epsilon
 \le\frac{C(2/g+1)\epsilon}{\operatorname{Im}m(\xi)+\sqrt\epsilon}.
\]
This is the bound in \citet[Eq.~(5.20)]{KnowlesYin2017}.
Every error profile $n^{-2}\le\epsilon(\xi)\le(\log n)^{-1}$ allowed in \citet[Definition~5.4, denoted as $\delta(z)$]{KnowlesYin2017} satisfies $\epsilon(\xi)\le\min\{g/4,1\}$ for all sufficiently large $n$.
Thus the equation $f(m(\xi))=\xi$ is stable on $\mathcal S$ in the sense of \citet[Definition~5.4]{KnowlesYin2017}.

We next introduce the Gaussian matrices required by \citet[Theorem~3.21]{KnowlesYin2017}. Let
$\Lambda=\diag(\lambda_1,\ldots,\lambda_p)$ and let $\tilde{\mathsf Z}$ be a $p\times n$ matrix with independent entries
$\tilde{\mathsf Z}_{ji}\sim N(0,n^{-1})$.
For $\xi\notin\spec((\tilde{\mathsf Z})^\top\Lambda\tilde{\mathsf Z})$, define $
 \tilde Q(\xi)
 =\bigl((\tilde{\mathsf Z})^\top\Lambda\tilde{\mathsf Z}-\xi I_n\bigr)^{-1}$ and $
 \tilde m(\xi)=\frac1n\tr\tilde Q(\xi)$.
As in \eqref{eq:linearization}, for $\xi\in\mathcal S$ define the Green matrix and its deterministic equivalent by
\begin{align*}
 \tilde{G}(\xi)
 &=\begin{pmatrix}
 -I_p&\Lambda^{1/2}\tilde{\mathsf Z}\\
 (\tilde{\mathsf Z})^\top\Lambda^{1/2}&-\xi I_n
 \end{pmatrix}^{-1}
 =\begin{pmatrix}
 \Lambda^{1/2}\tilde{\mathsf Z}\tilde Q(\xi)(\tilde{\mathsf Z})^\top\Lambda^{1/2}-I_p
 &\Lambda^{1/2}\tilde{\mathsf Z}\tilde Q(\xi)\\
 \tilde Q(\xi)(\tilde{\mathsf Z})^\top\Lambda^{1/2}&\tilde Q(\xi)
 \end{pmatrix},\\
 \tilde{\Pi}(\xi)
 &=\diag\bigl(-(I_p+m(\xi)\Lambda)^{-1},\ m(\xi)I_n\bigr).
\end{align*}
The function $m$ is the same as in \eqref{eq:MP}, since $\Sigma$ and $\Lambda$ have the same eigenvalues.
By \eqref{eq:comparison-denominator-condition} and \eqref{eq:comparison-stability-condition},
\citet[Theorem~3.22]{KnowlesYin2017} give, uniformly for $\xi\in\mathcal S$,
\begin{align*}
 \max_{1\le a,b\le p+n}|(\tilde{G}(\xi)-\tilde{\Pi}(\xi))_{ab}|
 &\prec\sqrt{\frac{\operatorname{Im}m(\xi)}{n\eta}}+\frac1{n\eta},\\
 |\tilde m(\xi)-m(\xi)|
 &\prec\frac1{n\eta},
\end{align*}
where the first inequality is the entrywise local law in \citet[Eq.~(3.18)]{KnowlesYin2017}, after absorbing the covariance factors using $\ell\le\lambda_j\le L$, and the second is the averaged local law in \citet[Definition~3.20(iii)]{KnowlesYin2017}.
These are the laws for $(\tilde{\mathsf Z},\Lambda,\mathcal S)$ required in \citet[Theorem~3.21(i)--(ii)]{KnowlesYin2017}.
\citet[Theorem~3.21(i)]{KnowlesYin2017} transfer the Gaussian case to the local law for $(\mathsf Z,\Sigma,\mathcal S)$;
\citet[Theorem~3.21(ii)]{KnowlesYin2017} transfer the averaged law.
Converting their $G_\Sigma$ to $G$ by the definition $G_\Sigma=B_\Sigma G B_\Sigma$ then gives \eqref{eq:comparison-nonreal-laws} on $\mathcal S$ in \eqref{eq:comparison-high-eta-domain}, and hence also for $\xi\in\mathcal V$ in \eqref{eq:KY-enlarged-domain} with $\eta\ge g/2$.

\proofstep{Case 1: $d\ge d_0$; spectral gap.} We now aim to prove the spectral gap \eqref{eq:comparison-sample-gap} for the case $d\ge d_0$.
It remains to exclude sample eigenvalues from
$\{s\in[c_*,C_*]:\dist(s,\supp\varrho)\ge g/2\}$.
As in the proof of Lemma~\ref{lem:gap-continuation}, we apply the $\epsilon$-net trick.
Specifically, choose an $n^{-2}$-net $\mathcal R_2$ of $[c_*,C_*]$ containing both endpoints, and retain the points separated from $\supp\varrho$:
\begin{equation}\label{eq:fixed-gap-net}
 \begin{gathered}
 \{c_*,C_*\}\subset\mathcal R_2\subset[c_*,C_*],\qquad
 \sup_{s\in[c_*,C_*]}\min_{E\in\mathcal R_2}|s-E|\le n^{-2},\qquad
 |\mathcal R_2|\le Cn^2,\\
 \mathcal R_{2,g}=\{E\in\mathcal R_2:\dist(E,\supp\varrho)\ge g/4\}.
 \end{gathered}
\end{equation}
Since $\supp\varrho\subset[2c_*,C_*/2]$ and $g\le c_*/4$, we have $c_*\in\mathcal R_{2,g}$.
Apply \citet[Theorem~3.16(i) and Remark~3.17]{KnowlesYin2017} with $\tau'=g/4$, $\eta=n^{-2}$, and
$v=w=\mathsf e_i=(0_p,e_i)\in\R^{p+n}$ for $i=1,\ldots,n$.
For $E\in\mathcal R_{2,g}$ and all sufficiently large $n$, the conditions required in \citet[Theorem~3.16(i)]{KnowlesYin2017}, extended to $\eta > 0$ by \citet[Remark~3.17]{KnowlesYin2017}, follow from
\[
 \begin{gathered}
 |E+\ii n^{-2}|\ge c_*>\tau,\qquad
 |E|\le C_*\le K<\tau^{-1},\qquad
 0<n^{-2}\le\tau^{-1},\\
 \dist(E,\supp\varrho)\ge g/4\ge n^{-2/3+\tau}.
 \end{gathered}
\]
For $E\in\mathcal R_{2,g}$, the Stieltjes transform formula gives $|m(E+\ii\eta)|\le4/g$ and $\operatorname{Im}m(E+\ii\eta)\le16\eta/g^2$.
The discussion in \citet[Remark~3.17]{KnowlesYin2017} therefore gives
\[
 |Q_{ii}(E+\ii n^{-2})-m(E+\ii n^{-2})|\prec n^{-1/2},
 \qquad E\in\mathcal R_{2,g},\quad i=1,\ldots,n.
\]
Fix $D>0$ and define
\begin{equation}\label{eq:fixed-gap-net-event}
 \mathcal E_{2,g}
 =\biggl\{\max_{E\in\mathcal R_{2,g}}\max_{1\le i\le n}
       |Q_{ii}(E+\ii n^{-2})-m(E+\ii n^{-2})|\le n^{-1/4}\biggr\}.
\end{equation}
Using \eqref{eq:stochastic-domination} with $\epsilon=1/4$ and failure exponent $D+4$, the union bound gives
\begin{align*}
 \PP(\mathcal E_{2,g}^c)
 &\le\sum_{E\in\mathcal R_{2,g}}\sum_{i=1}^n
   \PP\bigl(|Q_{ii}(E+\ii n^{-2})-m(E+\ii n^{-2})|>n^{-1/4}\bigr)\\
 &\le C_Dn|\mathcal R_{2,g}|n^{-D-4}
 \le C_Dn^{-D-1}\le C_Dn^{-D},
\end{align*}
where $|\mathcal R_{2,g}|\le|\mathcal R_2|\le Cn^2$ by \eqref{eq:fixed-gap-net}.
On $\mathcal E_{2,g}$, the identity $\widehat m(\xi)=n^{-1}\sum_{i=1}^nQ_{ii}(\xi)$ yields
\begin{align}
      \notag \max_{E\in\mathcal R_{2,g}}|\widehat m(E+\ii n^{-2})|
 &\le\max_{E\in\mathcal R_{2,g}}|m(E+\ii n^{-2})|
 +\frac1n\sum_{i=1}^n\max_{E\in\mathcal R_{2,g}}
       |Q_{ii}(E+\ii n^{-2})-m(E+\ii n^{-2})|\\
       \label{eq:fixed-gap-net-bound} &\le4/g+n^{-1/4}\le4/g+1=:C_g.
\end{align}

Recall that $s_1,\ldots,s_n$ are the eigenvalues of $\widehat S$. We are going to prove by contradiction that none of $s_1,\ldots,s_n$ lies in
$\{s\in[c_*,C_*]:\dist(s,\supp\varrho)\ge g/2\}$. 
Suppose that, on $\mathcal E_{2,g}$, some $s\in\spec(\widehat S)\cap[c_*,C_*]$ satisfies $\dist(s,\supp\varrho)\ge g/2$.
Choose $E_s\in\mathcal R_2$ with $|s-E_s|\le n^{-2}$ as in \eqref{eq:fixed-gap-net}.
For all sufficiently large $n$, we have
\[
 \dist(E_s,\supp\varrho)
 \ge\dist(s,\supp\varrho)-|s-E_s|
 \ge g/2-n^{-2}\ge g/4,
 \qquad E_s\in\mathcal R_{2,g}.
\]
The spectral representation of $\widehat m$ and \eqref{eq:fixed-gap-net-bound} then imply
\begin{align*}
 C_g
 &\ge|\widehat m(E_s+\ii n^{-2})|
 \ge\operatorname{Im}\widehat m(E_s+\ii n^{-2})
 =\frac1n\sum_{k=1}^n\frac{n^{-2}}{(s_k-E_s)^2+n^{-4}}
 \ge\frac{n^{-2}}{n((s-E_s)^2+n^{-4})}
 \ge \frac n2,
\end{align*}
which is impossible for $n>2C_g$.
Thus, by contradiction, on $\mathcal E_{2,g}$,
\[
 \spec(\widehat S)\cap
 \{s\in[c_*,C_*]:\dist(s,\supp\varrho)\ge g/2\}=\varnothing.
\]
Combining the above exclusion with \eqref{eq:KY-spectral-bounds} gives
\begin{align*}
 \PP\bigl(\exists s\in\spec(\widehat S):\dist(s,\supp\varrho)\ge g/2\bigr)
 &\le\PP(\mathcal E_{2,g}^c)
   +\PP\bigl(\spec(\widehat S)\not\subset[c_*,C_*]\bigr)
 \le C_Dn^{-D}.
\end{align*}
On the event $\{\max_{s\in\spec(\widehat S)}\dist(s,\supp\varrho)<g/2\}$, \eqref{eq:fixed-gap-alternative} gives
\begin{align*}
 \dist(\mathcal V,\spec(\widehat S))
 &\ge\dist(\mathcal V,\supp\varrho)
       -\max_{s\in\spec(\widehat S)}\dist(s,\supp\varrho)\ge g/2,\\
 \lambda_{\min}(\widehat S)-\xi_-
 &\ge(a_--g/2)-(a_--g)=g/2,\\
 \xi_+-\lambda_{\max}(\widehat S)
 &\ge(a_++g)-(a_++g/2)=g/2.
\end{align*}
Since $\delta\to0$, we have $c\delta^2\le g/2$ for all sufficiently large $n$. Hence we prove \eqref{eq:comparison-sample-gap}.
As $\mathcal T^+\subset\mathcal V$, this proves that the three assertions in \eqref{eq:preliminary-spectral} hold with high probability in Case~1 when $d \ge d_0$.

\proofstep{Case 2: $d<d_0$; bulk and edge regularity.} For the case $d<d_0$,
we aim to apply \citet[Theorem~3.6]{KnowlesYin2017} for \eqref{eq:comparison-nonreal-laws} and \citet[Theorem~3.12]{KnowlesYin2017} for \eqref{eq:comparison-sample-gap}.
These results require edge and bulk regularity formulated in \citet[Definition~2.7]{KnowlesYin2017}. We first verify these conditions.
For real $\mu\in[0,L]$ satisfying $\mu\neq \lambda_j$ for every $j=1,\ldots,p$, we have
\begin{equation}\label{eq:one-band-alternative}
 q(\mu)\ge\frac1{nL^2}\sum_{j=1}^p\lambda_j^2
       \ge\gamma_0(1-2d_0)=:K_0>1.
\end{equation}
For real $\mu\notin\spec(\Sigma)$, the derivative of $q(\mu)$ is given by
\[
 q'(\mu)=\frac2n\sum_{j=1}^p\frac{\lambda_j^2}{(\lambda_j-\mu)^3},
 \text{ thus }
 q'(\mu)>0 \text{ when } \mu<0, q'(\mu)<0 \text{ when } \mu>L.
\]
Thus $q$ is strictly increasing on $(-\infty,0)$ and strictly decreasing on $(L,\infty)$.
Moreover,
\[
 \lim_{\mu\to-\infty}q(\mu)=\lim_{\mu\to+\infty}q(\mu)=0,
 \qquad q(0)=\frac pn=\gamma>1,
 \qquad
 \lim_{\mu\downarrow L}q(\mu)
 \ge\frac1{nL^2}\sum_{j=1}^p\lambda_j^2\ge K_0>1.
\]
Since $q$ is continuous on $(-\infty,0]$, there is a root $\mu_-<0$ of $q(\mu)=1$; strict increase makes this root unique on $(-\infty,0)$.
Likewise, continuity on $(L,\infty)$ and the limits at $L$ and $+\infty$ give a root $\mu_+>L$, unique there by strict decrease.
Combining with \eqref{eq:one-band-alternative}, the complete set of real solutions is
\begin{equation}\label{eq:two-root-description}
      \{\mu\in\R\setminus\spec(\Sigma):q(\mu)=1\}
      =\{\mu_-,\mu_+\},\qquad \mu_-<0<L<\mu_+.
\end{equation}

Put $d_+=(\sqrt{K_0}-1)/2>0$.
The bounds $\ell\le\lambda_j\le L$ and $\gamma_0\le\gamma\le\gamma_1$ give
\begin{align*}
 q(-L\sqrt{\gamma_1})
 &\le\gamma/\gamma_1\le1,
 \qquad q\bigl(L(1+\sqrt{\gamma_1})\bigr)\le\gamma/\gamma_1\le1,\\
 q\bigl(-\ell(\sqrt{\gamma_0}-1)/2\bigr)
 &\ge\frac{4\gamma_0}{(1+\sqrt{\gamma_0})^2}>1,\\
 q\bigl(L(1+d_+)\bigr)
 &\ge\frac{\sum_{j=1}^p\lambda_j^2}{nL^2(1+d_+)^2}
 \ge\frac{K_0}{(1+d_+)^2}
 =\frac{4K_0}{(1+\sqrt{K_0})^2}>1.
\end{align*}
Combining the above estimates with the strict monotonicity of $q$ on each exterior interval gives
\begin{equation}\label{eq:comparison-exterior-root-bounds}
 -L\sqrt{\gamma_1}\le\mu_-\le-\frac\ell2(\sqrt{\gamma_0}-1),
 \quad
 L(1+d_+)\le\mu_+\le L(1+\sqrt{\gamma_1}).
\end{equation}
Combining \eqref{eq:two-root-description} with \citet[Lemmas~2.4--2.6]{KnowlesYin2017} gives
\begin{equation}\label{eq:comparison-support-interval}
 \supp\varrho=[a_-,a_+],\qquad a_\pm=\zeta(\mu_\pm).
\end{equation}
The conditions of \citet[Definition~2.7(i)]{KnowlesYin2017} now follow from
\[
 a_-\ge c_*,\qquad
 a_+-a_-\ge\ell,
 \qquad
 \min_{j,\,s\in\{-,+\}}
       \left|-\frac1{\mu_s}+\frac1{\lambda_j}\right|\ge c.
\]
For the second inequality above, \eqref{eq:zeta} and $\mu_-<0<\ell\le\lambda_j\le L<\mu_+$ give
\begin{align*}
 a_+-a_-
 &=\zeta(\mu_+)-\zeta(\mu_-)\\
 &=\mu_+-\mu_-+\frac1n\sum_{j=1}^p\lambda_j^2
   \biggl(\frac1{\mu_+-\lambda_j}-\frac1{\mu_--\lambda_j}\biggr)\ge\mu_+-\mu_->L\ge\ell.
\end{align*}

We next verify the bulk condition in \citet[Definition~2.7(ii)]{KnowlesYin2017}. 
Write $m(x+\ii0_+)=\lim_{\eta\downarrow0}m(x+\ii\eta)$ for the boundary value from the upper half-plane.
For each fixed $\tau>0$ such that $[a_-+\tau,a_+-\tau]\neq\varnothing$, we claim
\begin{equation}\label{eq:comparison-bulk-lower}
 \inf_{x\in[a_-+\tau,a_+-\tau]}
           \operatorname{Im}m(x+\ii0_+)\ge c_{\tau}>0.
\end{equation}

The proof of Lemma~2.6 in
\citet[Appendix~A.1]{KnowlesYin2017}
establishes positivity of $\operatorname{Im}m$ on the interior
of each interval in their Eq.~(2.13).
By \eqref{eq:two-root-description}, the only such interval
in our setting is $[a_-,a_+]$, and hence
\[
 \{x\in(a_-,a_+):\operatorname{Im}m(x+\ii0_+)=0\}=\varnothing.
\]
Fix $x\in[a_-+\tau,a_+-\tau]$ and set $\mu=-1/m(x+\ii0_+)$.
Equations~\eqref{eq:MP} and \eqref{eq:imaginary-part-zeta} then yield
\begin{align}
 \zeta(\mu)&=x,\qquad
 \operatorname{Im}\mu
 =\frac{\operatorname{Im}m(x+\ii0_+)}{|m(x+\ii0_+)|^2}>0,
 \label{eq:comparison-bulk-zeta}\\
 0&=\operatorname{Im}\zeta(\mu)
   =\operatorname{Im}\mu\bigl(1-q(\mu)\bigr)
   \quad\Longrightarrow\quad q(\mu)=1.
 \label{eq:comparison-bulk-q}
\end{align}
The inequalities
\[
 q(\mu)\ge\frac{\gamma_0}{(1+|\mu|/\ell)^2},
 \qquad
 q(\mu)\le\frac{\gamma_1L^2}{(|\mu|-L)^2}\qquad  \text{ when } |\mu|>L 
\]
show that $\ell(\sqrt{\gamma_0}-1)\le|\mu|\le L(1+\sqrt{\gamma_1})$.
Suppose that \eqref{eq:comparison-bulk-lower} fails for some fixed $\tau>0$.
Then there are integers $p_k>n_k\ge1$ and population eigenvalues $\lambda_{1,k},\ldots,\lambda_{p_k,k}$ satisfying
\begin{align*}
 &\gamma_0\le\gamma_k=\frac{p_k}{n_k}\le\gamma_1,\qquad
   \ell\le\lambda_{j,k}\le L\quad(1\le j\le p_k), \qquad d_k=\frac1{p_k}\sum_{j=1}^{p_k}\biggl(1-\frac{\lambda_{j,k}}L\biggr)<d_0,
\end{align*}
with $\ell,L,\gamma_0,\gamma_1$ fixed in $k$.
Let $m_k$ solve \eqref{eq:MP} with $n=n_k$, $p=p_k$, and eigenvalues $\lambda_{j,k}$; let $\varrho_k$ be its associated probability measure and write $\supp\varrho_k=[a_{-,k},a_{+,k}]$.
Write $\zeta_k$ and $q_k$ for $\zeta$ and $q$ formed from the same dimensions and eigenvalues, and let $\mu_{-,k}<0<L<\mu_{+,k}$ be the two real roots of $q_k(\mu)=1$.
These eigenvalue arrays and points $x_k$ can be chosen so that
\begin{equation}\label{eq:comparison-bulk-x-k}
      x_k\in[a_{-,k}+\tau,a_{+,k}-\tau],\qquad
      0<\operatorname{Im}m_k(x_k+\ii0_+)\le k^{-1}.  
\end{equation}
Define
\[
 H_k=\frac1{p_k}\sum_{j=1}^{p_k}\delta_{\lambda_{j,k}},\qquad
 \mu_k=-\frac1{m_k(x_k+\ii0_+)}.
\]
Applying \eqref{eq:comparison-bulk-zeta} and \eqref{eq:comparison-bulk-q} to the $k$th eigenvalue array gives $\zeta_k(\mu_k)=x_k$ and $q_k(\mu_k)=1$.
Since the probability measures $H_k$ are supported on the compact interval $[\ell,L]$, Prokhorov's theorem \citep[Theorem~5.1]{Billingsley1999} gives a weakly convergent subsequence in the space of probability measures on $[\ell,L]$.
Along this subsequence, the bounds on $\gamma_k$ and $|\mu_k|$, together with \eqref{eq:comparison-exterior-root-bounds}, allow us to apply the Bolzano--Weierstrass theorem \citep[Section~2.21]{WhittakerWatson1920} coordinatewise to the bounded tuple $(\gamma_k,\mu_k,\mu_{-,k},\mu_{+,k})$.
Passing to a further subsequence and relabeling its indices by $k$, we obtain
\begin{gather*}
 H_k\Rightarrow H_\infty,\qquad
 \supp H_\infty\subset[\ell,L],\qquad
 \gamma_k\to\gamma_\infty\in[\gamma_0,\gamma_1],\\
 \mu_k\to\mu_\infty,\qquad
 \mu_{-,k}\to\mu_{-,\infty}<0,\qquad
 \mu_{+,k}\to\mu_{+,\infty}>L.
\end{gather*}
Moreover, \eqref{eq:comparison-bulk-x-k} implies
\[
 0<\operatorname{Im}\mu_k
 =|\mu_k|^2\operatorname{Im}m_k(x_k+\ii0_+)
 \le\frac{L^2(1+\sqrt{\gamma_1})^2}{k}\longrightarrow0,
 \text{ thus } \mu_\infty\in\R.
\]
If $\mu_\infty\in[0,L]$, then $|\lambda-\mu_k|\le L+|\mu_k-\mu_\infty|$ for every $\lambda\in[\ell,L]$, and
\begin{align*}
 1=q_k(\mu_k)
 &=\gamma_k\int\frac{\lambda^2}{|\lambda-\mu_k|^2}\,H_k(\dd\lambda)
 \ge\frac{\gamma_0L^2}{(L+|\mu_k-\mu_\infty|)^2}
       \int\bigl(2\lambda/L-1\bigr)\,H_k(\dd\lambda)\\
 &=\frac{\gamma_0L^2(1-2d_k)}{(L+|\mu_k-\mu_\infty|)^2}
 \ge\frac{\gamma_0L^2(1-2d_0)}{(L+|\mu_k-\mu_\infty|)^2}
 \longrightarrow K_0>1.
\end{align*}
The last limit contradicts $q_k(\mu_k)=1$.
Hence $\mu_\infty\in(-\infty,0)\cup(L,\infty)$.

Define
\[
 q_\infty(z)=\gamma_\infty\int\frac{\lambda^2}{|\lambda-z|^2}\,H_\infty(\dd\lambda),
 \qquad z\in\C\setminus[\ell,L].
\]
For any sequence $z_k\to z\notin[\ell,L]$, for all sufficiently large $k$,
\begin{gather*}
 \dist(z_k,[\ell,L])\ge\tfrac12\dist(z,[\ell,L])>0,
 \sup_{\lambda\in[\ell,L]}
 \biggl|\frac{\lambda^2}{|\lambda-z_k|^2}
       -\frac{\lambda^2}{|\lambda-z|^2}\biggr|\longrightarrow0.
\end{gather*}
Together with $H_k\Rightarrow H_\infty$ and $\gamma_k\to\gamma_\infty$, this gives $q_k(z_k)\to q_\infty(z)$.
Applying this convergence to the three sequences $\mu_k$, $\mu_{-,k}$, and $\mu_{+,k}$ yields
\[
 q_\infty(\mu_\infty)
 =\lim_{k\to\infty}q_k(\mu_k)=1,\qquad
 q_\infty(\mu_{\pm,\infty})
 =\lim_{k\to\infty}q_k(\mu_{\pm,k})=1.
\]
For real $t\in(-\infty,0)\cup(L,\infty)$,
\[
 q_\infty'(t)
 =2\gamma_\infty\int\frac{\lambda^2}{(\lambda-t)^3}\,H_\infty(\dd\lambda),
\text{ thus } q_\infty'(t)>0, \text{ when } t<0, \text{ and } q_\infty'(t)<0, \text{ when } t>L.
\]
Thus each exterior interval $(-\infty,0)$ and $(L,\infty)$ contains exactly one solution of $q_\infty(t)=1$, so
$\mu_\infty=\mu_{\sigma,\infty}$ for some $\sigma\in\{-,+\}$.
As in \eqref{eq:zeta-difference}, the Cauchy--Schwarz inequality now gives
\begin{align*}
 \tau
 \le|x_k-a_{\sigma,k}|
 =|\zeta_k(\mu_k)-\zeta_k(\mu_{\sigma,k})|
 \le\bigl(1+\sqrt{q_k(\mu_k)q_k(\mu_{\sigma,k})}\bigr)
       |\mu_k-\mu_{\sigma,k}|
       =2|\mu_k-\mu_{\sigma,k}|\longrightarrow0,
\end{align*}
where the first inequality uses the choice of $x_k$, and the last limit uses
$\mu_k-\mu_{\sigma,k}\to\mu_\infty-\mu_{\sigma,\infty}=0$.
This contradicts $\tau>0$.
This proves \eqref{eq:comparison-bulk-lower} and we check all conditions of \citet[Definition~2.7]{KnowlesYin2017}. Then we have the spectral gap and local law for Case 2.

\proofstep{Case 2(a): $d<d_0$; local law.}
\citet[Theorem~3.6]{KnowlesYin2017} now give \eqref{eq:comparison-nonreal-laws} uniformly on $\mathcal V$ for $|\xi|\ge r_0/2$ and $\eta\ge n^{-1/2}$.

\proofstep{Case 2(b): $d<d_0$; spectral gap.}
\citet[Theorem~3.12]{KnowlesYin2017} give
\[
 \PP(\spec(\widehat S)\subset[a_--n^{-1/2},a_++n^{-1/2}])\ge1-C_Dn^{-D}.
\]
Since $n^{-1/2}=o(\delta^2)$, \eqref{eq:comparison-deterministic-gap}, \eqref{eq:comparison-crossings}, and the support identity \eqref{eq:comparison-support-interval} give \eqref{eq:comparison-sample-gap} for Case~2 when $d < d_0$.

\proofstep{Conclusion of the proof of \eqref{eq:preliminary-spectral}.}
For both cases $d\ge d_0$ and $d<d_0$, define $\mathcal A_1$ by
\[
 \mathcal A_1=\left\{
 c_*I_n\preceq\widehat S\preceq C_*I_n,
 \dist(\mathcal V,\spec(\widehat S))\ge c\delta^2,
 \xi_-\le\lambda_{\min}(\widehat S)-c\delta^2,
 \xi_+\ge\lambda_{\max}(\widehat S)+c\delta^2\right\}.
\]
Combining~\eqref{eq:KY-spectral-bounds} and \eqref{eq:comparison-sample-gap} with the proof for Case~2 when $d<d_0$ above gives $\PP(\mathcal A_1^c)\le C_Dn^{-D}$ for every $D>0$.
On $\mathcal A_1$, $\|\mathcal X\|_\op\le\sqrt{C_*}$, and the resolvent identity in \eqref{eq:linearization} gives, for every $\xi\in\mathcal V$,
\begin{align}
 \|Q(\xi)\|_\op
 &=\frac1{\dist(\xi,\spec(\widehat S))}\le C\delta^{-2},\notag\\
 \|G(\xi)\|_\op
 &\le1+(1+\|\mathcal X\|_\op^2)\|Q(\xi)\|_\op
 \le1+C(1+C_*)\delta^{-2}\le C\delta^{-2}.
 \label{eq:localized-green-bound}
\end{align}
Since $\mathcal T^+\subset\mathcal V$, on $\mathcal A_1$, we have all the assertions of \eqref{eq:preliminary-spectral}, and thus conclude the proof of the first part of Lemma~\ref{lem:preliminary}.

\proofstep{Proof of \eqref{eq:coarse-moments}. Step 1: $j=0$.} We now aim to prove \eqref{eq:coarse-moments} starting with the original function, i.e., $j=0$.  Then in Step 2, we prove the case for $j \ge 1$. For $F_v(\xi)=v^\top(G(\xi)-\Pi(\xi))w$ and $F_0(\xi)=\widehat m(\xi)-m(\xi)$,
the unit-norm assumption on $v,w$ and the lower-right blocks in \eqref{eq:linearization} give, on $\mathcal A_1$,
\begin{align*}
 |F_v(\xi)|
 &\le\|v\|_2\|G(\xi)-\Pi(\xi)\|_\op\|w\|_2
 =\|G(\xi)-\Pi(\xi)\|_\op,\\
 |F_0(\xi)|
 &=\left|\frac1n\tr\bigl(Q(\xi)-m(\xi)I_n\bigr)\right|
 \le\|Q(\xi)-m(\xi)I_n\|_\op
 \le\|G(\xi)-\Pi(\xi)\|_\op.
\end{align*}
Thus \eqref{eq:localized-green-bound} and \eqref{eq:comparison-deterministic-gap} imply
\begin{equation}\label{eq:localized-pathwise}
 \begin{aligned}
 \sup_{\xi\in\mathcal V}(|F_v(\xi)|+|F_0(\xi)|)
 &\le2\sup_{\xi\in\mathcal V}\|G(\xi)-\Pi(\xi)\|_\op
 \le2\sup_{\xi\in\mathcal V}
       \bigl(\|G(\xi)\|_\op+\|\Pi(\xi)\|_\op\bigr)
 \le C\delta^{-2}.
 \end{aligned}
\end{equation}
For $\xi=E+\ii\eta\in\mathcal V\cap\C^+$, \eqref{eq:comparison-deterministic-gap} and $\varrho(\R)=1$ give
\begin{equation}\label{eq:localized-imaginary-part-bound}
 \frac{\operatorname{Im}m(\xi)}{\eta}
 =\int\frac{\varrho(\dd x)}{|x-\xi|^2}
 \le\frac1{\dist(\xi,\supp\varrho)^2}
 \le C\delta^{-4}.
\end{equation}
Hence, for $\xi\in\mathcal V$ with $|\xi|\ge r_0/2$ and $\eta\ge n^{-1/2}$, the error bounds in \eqref{eq:comparison-nonreal-laws} satisfy
\begin{align*}
 \sqrt{\frac{\operatorname{Im}m(\xi)}{n\eta}}+\frac1{n\eta}
 \le Cn^{-1/2}\delta^{-2}+n^{-1}\eta^{-1}
 \le C\delta^{-2}n^{-1/2}(1+\eta^{-1}),
 \frac1{n\eta}
 \le n^{-1}(1+\eta^{-1}).
\end{align*}
As $\delta\ge1/\log n$, \eqref{eq:comparison-nonreal-laws} yields uniformly for $\xi\in\mathcal V$,
\begin{equation}\label{eq:localized-stochastic-bounds}
 |\ind_{\mathcal A_1}F_v(\xi)|\prec n^{-1/2}(1+\eta^{-1}),
 |\ind_{\mathcal A_1}F_0(\xi)|\prec n^{-1}(1+\eta^{-1}), \text{ if $|\xi|\ge r_0/2$ and $\eta\ge n^{-1/2}$.}
\end{equation}

To prove \eqref{eq:coarse-moments} for $j=0$, we need to split into two cases: (a) $|\xi|\ge r_0$ where we can apply  Lemma~\ref{lem:gap-continuation} and (b) $|\xi| < r_0$ where we need to use Cauchy's integral formula \citep[Chapter~4, Eq.~(22)]{Ahlfors1979} on $F_0, F_v$ to reduce back to case (a).

\proofstep{Proof of \eqref{eq:coarse-moments}. Step 1(a) $|\xi|\ge r_0$.} Suppose $\xi\in\mathcal T^+$, $|\xi|\ge r_0$, $\operatorname{Im}\xi\ge0$. 
Set $a=c_1\delta^b$ and
\[
 \mathcal R=
 \{\xi\in\overline{\mathcal T^+}:|\xi|\ge r_0,\ \operatorname{Im}\xi\ge0\}
 \cup\{\xi:|\xi|=2r_0,\ \operatorname{Im}\xi\ge0\}.
\]
The set $\mathcal R$ is compact, is bounded uniformly in $n$, and lies in the closed upper half-plane.
For large $n$, $2a\le r_0/2$. Recall $\mathcal R^{(1)}, \mathcal R^{(2)}$ from Lemma~\ref{lem:gap-continuation}. 
For $\omega\in\mathcal R^{(2)}$, choose $z\in\mathcal R$ with $|\omega-z|<2a$.
By \eqref{eq:spectral-neighborhood},
\begin{align*}
 |\omega|
 &\ge |z|-|\omega-z|>r_0-2a\ge r_0/2,\\
 \dist(\omega,\zeta(\Gamma_\delta))
 &\le\dist(z,\zeta(\Gamma_\delta))+|\omega-z|
 <3a+2a<6a
 \text{ for } z\in\overline{\mathcal T^+},\\
 |\omega|
 &\le |z|+|\omega-z|<2r_0+2a\le5r_0/2<3r_0
 \text{ for } |z|=2r_0.
\end{align*}
Thus \eqref{eq:KY-enlarged-domain} gives
$
 \mathcal R^{(1)}\subset\mathcal R^{(2)}
 \subset\mathcal V\cap\{\omega:|\omega|\ge r_0/2\}.
$
On $\mathcal A_1$, \eqref{eq:comparison-sample-gap}, \eqref{eq:comparison-deterministic-gap}, and \eqref{eq:linearization} imply that $F_v$ and $F_0$ are holomorphic on $\mathcal V$, hence on $\mathcal R^{(2)}$.
Equation~\eqref{eq:localized-pathwise} verifies \eqref{eq:continuation-pathwise} with $M=2$, and \eqref{eq:localized-stochastic-bounds} verifies \eqref{eq:continuation-input} on $\mathcal R^{(1)}$ for $\operatorname{Im}\xi\ge n^{-1/2}$.
Lemma~\ref{lem:gap-continuation} therefore gives \eqref{eq:continuation-output} with $\mathcal B=\mathcal A_1$, $B=b$, $c=c_1$, and $\alpha=1/2$ for $F_v$ and $\alpha=1$ for $F_0$.

Conjugation extends both bounds in \eqref{eq:continuation-output} from $\mathcal R$ to $\{\bar\xi:\xi\in\mathcal R\}$, including the lower semicircle $|\xi|=2r_0$; for complex $v,w$, apply \eqref{eq:localized-stochastic-bounds} and \eqref{eq:continuation-output} with $\bar v,\bar w$ in place of $v,w$.

\proofstep{Proof of \eqref{eq:coarse-moments}. Step 1(b) $|\xi|\le r_0$.}
On $\mathcal A_1$, $F_v$ and $F_0$ are holomorphic on $|\xi|<3r_0$, and Cauchy's integral formula gives
\begin{equation}\label{eq:localized-cauchy-zero}
 \sup_{|\xi|\le r_0}|F(\xi)|
       \le2\sup_{|\omega|=2r_0}|F(\omega)|,
       \qquad F\in\{F_v,F_0\}.
\end{equation}
Consequently, we reduce back to Step 1(a) and get
\begin{equation}\label{eq:localized-uniform-bounds}
 \sup_{\xi\in\mathcal T^+}|\ind_{\mathcal A_1}F_v(\xi)|\prec n^{-1/2},
 \qquad
 \sup_{\xi\in\mathcal T^+}|\ind_{\mathcal A_1}F_0(\xi)|\prec n^{-1}.
\end{equation}
The $L^r$ estimate in \eqref{eq:continuation-output}, together with \eqref{eq:localized-cauchy-zero}, proves \eqref{eq:coarse-moments} for $j=0$.
The $F_0$ bound in \eqref{eq:localized-uniform-bounds} proves \eqref{eq:averaged-uniform}.

\proofstep{Proof of \eqref{eq:coarse-moments}. Step 2: $j\ge1$.}
Finally, fix $F\in\{F_v,F_0\}$ and an integer $j\ge1$.
By \eqref{eq:spectral-neighborhood}, the closed disk $\{|\omega-\xi|\le a\}$ lies in $\mathcal T^+$ for every $\xi\in\mathcal T$.
For each realization in $\mathcal A_1$, the function $\omega\mapsto F(\omega)$ is holomorphic on $\mathcal T^+$.
Applying Cauchy's integral formula to $F$ on $|\omega-\xi|=a$ and differentiating $j$ times gives
\begin{align*}
 F^{(j)}(\xi)
 =\frac{j!}{2\pi\ii}\oint_{|\omega-\xi|=a}
       \frac{F(\omega)}{(\omega-\xi)^{j+1}}\,\dd\omega, \quad
 |F^{(j)}(\xi)|
 \le\frac{j!}{2\pi}\frac{2\pi a}{a^{j+1}}
       \sup_{|\omega-\xi|=a}|F(\omega)|
 \le j!a^{-j}\sup_{\omega\in\mathcal T^+}|F(\omega)|.
\end{align*}
Define the random functions $\ind_{\mathcal A_1}F(\omega)$ and $\ind_{\mathcal A_1}F^{(j)}(\xi)$ to equal $F(\omega)$ and $F^{(j)}(\xi)$, respectively, on $\mathcal A_1$, and to equal zero on $\mathcal A_1^c$.
Taking the supremum over $\xi\in\mathcal T$ and then the $L^r$ norm yields
\begin{equation}\label{eq:Cauchy-derivative}
 \left\|\sup_{\xi\in\mathcal T}|\ind_{\mathcal A_1}F^{(j)}(\xi)|\right\|_{L^r}
 \le j!a^{-j}
       \left\|\sup_{\omega\in\mathcal T^+}|\ind_{\mathcal A_1}F(\omega)|\right\|_{L^r}.
\end{equation}
Use the $j=0$ bound with $\epsilon/2$ and $a^{-j}=O((\log n)^{bj})=o(n^{\epsilon/2})$.
This proves \eqref{eq:coarse-moments} for $j\ge1$.
Homogeneity restores the factors $\|v\|_2\|w\|_2$, completing the proof.
\end{proof}

\section{Moments of the Integrated Calibrated Field}\label{app:moments}
We study the rate of $\hat{F}$ in \eqref{eq:empirical-field} and prove Lemmas~\ref{lem:field-moments} and~\ref{lem:global-fourth-moment} in this section. 
For a deterministic function $\psi$, recall that
$\mathcal I_\psi H=(2\pi\ii)^{-1}\oint_{\Gamma_\delta}\psi(\mu)H(\mu)\dd\mu$
and $\norm\psi_{\infty,\Gamma_\delta} = \sup_{\mu \in \Gamma_\delta} |\psi(\mu)|$.

\subsection{Covariance Inputs}

Recall the Green matrix $G(\xi)$ and its deterministic equivalent $\Pi(\xi)$ in \eqref{eq:linearization}, the Stieltjes transform $m(\xi)$ solving \eqref{eq:MP}, and the neighborhood $\mathcal T$ in \eqref{eq:spectral-neighborhood}. Vectors are embedded in the top block, so that $v^\top G(\xi)w=v^\top G_{11}(\xi)w$ for $v,w\in\C^p$, and $\mathsf u_k=\Sigma^{1/2}e_k$ as before. For $\xi,\xi_1,\xi_2\in\mathcal T$, define
\begin{equation}\label{eq:covariance-functions}
 \begin{gathered}
 D(\xi)=I_p+m(\xi)\Sigma,\qquad
 M(\xi_1,\xi_2)=\Sigma D(\xi_1)^{-1}D(\xi_2)^{-1},\\
 \mathcal Q(\xi_1,\xi_2)
 =\frac{m(\xi_1)m(\xi_2)}{n}
  \tr\bigl(\Sigma^2D(\xi_1)^{-1}D(\xi_2)^{-1}\bigr),\qquad
 S(\xi_1,\xi_2)=\frac{m(\xi_1)m(\xi_2)}{1-\mathcal Q(\xi_1,\xi_2)}.
 \end{gathered}
\end{equation}
Since $\Sigma$ and $D(\xi)$ commute, $M(\xi_1,\xi_2)$ is symmetric and $M(\xi_1,\xi_2)=M(\xi_2,\xi_1)$. Note that $\Pi_{11}(\xi)=-D(\xi)^{-1}$. The deterministic kernels below are defined on the full set $\mathcal T\times\mathcal T$, including the real crossings of the contour. The stochastic moment estimates will only be used at nonreal parameters.
Write
\begin{equation}\label{eq:nonreal-contour}
 \mu(t)=L+R_\delta e^{\ii t},\qquad
 \Gamma_\delta^\circ
 =\{\mu(t):0\le t\le2\pi,\ |\sin t|\ge n^{-2}\},\qquad
 B_n=\Gamma_\delta\setminus\Gamma_\delta^\circ,
\end{equation}
where $\Gamma_\delta^\circ$ omits two short arcs $B_n$ around the real points of $\Gamma_\delta$.
 By \eqref{eq:imaginary-part-zeta} and
\eqref{eq:calibration-denominator-bound}, we have
\[
 |\operatorname{Im}\zeta(\mu)|
 \ge\bigl(1-(1+\delta)^{-2}\bigr)|\operatorname{Im}\mu|
 \ge c\delta n^{-2},\qquad \mu\in\Gamma_\delta^\circ.
\]
Consequently, the Green matrix $G(\zeta(\mu))$ in \eqref{eq:linearization} is well-defined for every realization of $X$ and every $\mu\in\Gamma_\delta^\circ$.
 Lemma~\ref{lem:nonreal-resolvent} gives
\begin{equation}\label{eq:raw-polynomial-bound}
 \|G(\zeta(\mu))\|_\op+|\widehat m(\zeta(\mu))|
 \le Cn^2\delta^{-1},\qquad \mu\in\Gamma_\delta^\circ.
\end{equation}
The omitted arcs $ {B_n}$ are handled in the proof of
Lemma~\ref{lem:field-moments}. The following lemma gives the leading term of the covariance.

\begin{lemma}\label{lem:recursion}
Let $1/\log n\le\delta\le2/\log n$, let $\xi_j=\zeta(\mu_j)$ with $\mu_j\in\Gamma_\delta^\circ$ for $j=1,2$, and let $v_1,w_1,v_2,w_2,v\in\C^p$ be deterministic vectors with Euclidean norms bounded by a constant. Define
\[
 H_j=v_j^\top\bigl(G(\xi_j)-\Pi(\xi_j)\bigr)w_j,
 \qquad
 \eta(\xi)=\frac1{\sqrt n}v^\top G_{12}(\xi)\varepsilon,
 \qquad \xi\in\zeta(\Gamma_\delta^\circ).
\]
With $M$ and $S$ as in \eqref{eq:covariance-functions}, define the Gaussian kernel
\begin{equation}\label{eq:leading-covariance}
 \begin{aligned}
 \Xi^{\rm g}(\xi_1,\xi_2)=\frac{S(\xi_1,\xi_2)}n\Bigl[&
 \bigl(v_1^\top M(\xi_1,\xi_2)v_2\bigr)
 \bigl(w_1^\top M(\xi_1,\xi_2)w_2\bigr)\\
 &\quad+\bigl(v_1^\top M(\xi_1,\xi_2)w_2\bigr)
 \bigl(w_1^\top M(\xi_1,\xi_2)v_2\bigr)\Bigr],
 \end{aligned}
\end{equation}
and the noise kernel
\begin{equation}\label{eq:noise-kernel}
 \Xi^{\rm n}(\xi_1,\xi_2)
 =\frac{\sigma^2}{n}S(\xi_1,\xi_2)\,v^\top M(\xi_1,\xi_2)v.
\end{equation}
There is a constant $C$, independent of $n$ and $\delta$, such that, uniformly in the contour points and in the indicated vectors,
\begin{align}\label{eq:moment-recursion}
 \bigl|\E(H_1H_2)-\Xi^{\rm g}(\xi_1,\xi_2)\bigr|\le\frac Cn, \quad
 \bigl|\E[\eta(\xi_1)\eta(\xi_2)]-\Xi^{\rm n}(\xi_1,\xi_2)\bigr|\le\frac Cn.
\end{align}
If $v_2,w_2, v$ are real, then $\E(H_1\overline{H_2})$ and $\E[\eta(\xi_1)\overline{\eta(\xi_2)}]$ have the same bounds as above, {with leading kernels $\Xi^{\rm g}(\xi_1,\bar\xi_2)$ and $\Xi^{\rm n}(\xi_1,\bar\xi_2)$, respectively}.
\end{lemma}

The lemma above is essentially based on 
\citet[Lemma~5.2]{Yang2024}. We present the proof of  Lemma~\ref{lem:recursion} in Appendix~\ref{app:pf:lem:recursion}  by controlling the bounds uniformly in $\delta$.

The following lemma bounds the two kernels above, with proof in
Appendix~\ref{app:pf:lem:integrated-kernel}.

\begin{lemma}[Integrated Gaussian and noise kernels]\label{lem:integrated-kernel}
Let $0<\delta\le1$, let $v,\beta\in\R^p$ with $\norm v_2\le1$ and $\norm\beta_2\le B_0$, and let $\psi$ be bounded and measurable on $\Gamma_\delta$. For a function $\Theta$ on $\mathcal T\times\mathcal T$, define
\begin{equation}\label{eq:hermitian-contour}
 \mathcal I^{(2)}_\psi\Theta
 =\frac1{(2\pi)^2}\oint_{\Gamma_\delta}\oint_{\Gamma_\delta}
 \psi(\mu)\overline{\psi(\nu)}\,
 \Theta\bigl(\zeta(\mu),\overline{\zeta(\nu)}\bigr)
 \dd\mu \dd\bar\nu.
\end{equation}
Let $\Xi^{\rm g}$ be the Gaussian kernel in \eqref{eq:leading-covariance} with $v_1=v_2=v$ and $w_1=w_2=\beta$, and let $\Xi^{\rm n}$ be the noise kernel in \eqref{eq:noise-kernel}. We then have
\begin{equation}\label{eq:integrated-kernel}
 \bigl|\mathcal I^{(2)}_\psi\Xi^{\rm g}\bigr|
 +\bigl|\mathcal I^{(2)}_\psi\Xi^{\rm n}\bigr|
 \le\frac Cn\norm\psi_{\infty,\Gamma_\delta}^2.
\end{equation}
\end{lemma}

\subsection{Proof of Lemma~\ref{lem:field-moments}}\label{sec:pf:lem:field-moments}
\begin{proof}
Fix $b>6$ in \eqref{eq:spectral-neighborhood}.
 For the $X$-measurable event $\mathcal A$ in \eqref{eq:averaged-uniform}, Lemma~\ref{lem:preliminary} shows $\PP(\mathcal A^c)\le C_Dn^{-D}$, for every $D>0$, and the empirical branch exists on $\mathcal A$ by Lemma~\ref{lem:empirical-control}. We choose $D$ sufficiently large.
Recall $\Gamma_\delta^\circ, B_n$ in \eqref{eq:nonreal-contour}, and put
\[
 \psi_\circ=\psi\ind_{\Gamma_\delta^\circ},
 \qquad
 \mathcal I_\psi^\circ H
 =\frac1{2\pi\ii}\int_{\Gamma_\delta^\circ}\psi(\mu)H(\mu)\dd\mu.
\]

\proofstep{Bound on $B_n$.} We first bound the contribution of $B_n$.
On $\mathcal A$, Lemmas~\ref{lem:preliminary} and~\ref{lem:empirical-control} give $\norm{\mathcal X}_\op\le C$ and $\sup_{\Gamma_\delta}\norm{Q(\widehat\zeta(\mu))}_\op\le C\delta^{-C}$. The representation
\[
 \widehat F(\mu)
 =v^\top\mathcal XQ(\widehat\zeta(\mu))\mathcal X^\top\beta
   +\frac1{\sqrt n}v^\top\mathcal XQ(\widehat\zeta(\mu))\varepsilon
\]
and the conditional covariance $\E[\varepsilon\varepsilon^\top|X]=\sigma^2I_n$ imply
$\sup_{\mu\in\Gamma_\delta}\norm{\ind_{\mathcal A}(\widehat F(\mu)-F(\mu))}_{L^2}\le C\delta^{-C}$.
Here $F$ is uniformly bounded by \eqref{eq:population-field}. As $|B_n|\le Cn^{-2}$ by definition \eqref{eq:nonreal-contour}, Minkowski's inequality therefore yields
\begin{equation}\label{eq:omitted-arcs}
 \left\|\ind_{\mathcal A}\frac1{2\pi\ii}
       \int_{B_n}\psi(\mu)(\widehat F(\mu)-F(\mu))\dd\mu\right\|_{L^2}
 \le Cn^{-2}\delta^{-C}\norm\psi_{\infty,\Gamma_\delta}
 =o(n^{-1/2})\norm\psi_{\infty,\Gamma_\delta}.
\end{equation}

\proofstep{Decomposition on $\Gamma_\delta^\circ$.}
Put $\Delta Q(\mu)=Q(\widehat\zeta(\mu))-Q(\zeta(\mu))$. By \eqref{eq:linearization},
$G_{11}=\mathcal XQ\mathcal X^\top-I_p$ and $G_{12}=\mathcal XQ$. Moreover, Proposition~\ref{prop:calibration} gives $m(\zeta(\mu))=-1/\mu$, so
$F(\mu)=v^\top\beta+v^\top\Pi_{11}(\zeta(\mu))\beta$.
Consequently, on $\mathcal A$ and for $\mu\in\Gamma_\delta^\circ$,
\begin{align}\label{eq:Fhat-F-decomposition}
 \widehat F(\mu)-F(\mu)
 ={}&v^\top\bigl(G_{11}(\zeta(\mu))-\Pi_{11}(\zeta(\mu))\bigr)\beta
       +\frac1{\sqrt n}v^\top G_{12}(\zeta(\mu))\varepsilon\notag\\
    &+v^\top\mathcal X\Delta Q(\mu)\mathcal X^\top\beta
       +\frac1{\sqrt n}v^\top\mathcal X\Delta Q(\mu)\varepsilon.
\end{align}
On $\mathcal A$, applying $\mathcal I_\psi^\circ$ to \eqref{eq:Fhat-F-decomposition} gives the decomposition
\begin{equation}\label{eq:Ipsi_decompose}
 \begin{aligned}
 \mathcal I_\psi^\circ(\widehat F-F)&=D_\psi+N_\psi+R_\psi, \text{ where }\\
 D_\psi&=\mathcal I_\psi^\circ\!
   \left[v^\top(G_{11}(\zeta(\cdot))-\Pi_{11}(\zeta(\cdot)))\beta\right],\\
 N_\psi&=\frac1{\sqrt n}\mathcal I_\psi^\circ\!
   \left[v^\top G_{12}(\zeta(\cdot))\varepsilon\right],\\
 R_\psi&=\mathcal I_\psi^\circ
   [v^\top\mathcal X\Delta Q\,\mathcal X^\top\beta]
   +\frac1{\sqrt n}\mathcal I_\psi^\circ
   [v^\top\mathcal X\Delta Q\,\varepsilon].
 \end{aligned}
\end{equation}
The variables $D_\psi$ and $N_\psi$ are defined on every sample, since their spectral parameters are nonreal. Only $R_\psi$ is restricted to $\mathcal A$, and we set it to zero elsewhere. In view of \eqref{eq:omitted-arcs}, it suffices to show
\begin{equation}\label{eq:F-rate-goal}
 \norm{D_\psi}_{L^2}+\norm{N_\psi}_{L^2}
 \le Cn^{-1/2}\norm\psi_{\infty,\Gamma_\delta},
 \qquad
 \norm{\ind_{\mathcal A}R_\psi}_{L^2}
 =o(n^{-1/2})\norm\psi_{\infty,\Gamma_\delta}.
\end{equation}

\proofstep{Design term $D_\psi$.}
Let $H(\xi)=v^\top(G_{11}(\xi)-\Pi_{11}(\xi))\beta$. Fubini's theorem applies by the polynomial resolvent bounds on the retained contour, and hence
\[
 \E|D_\psi|^2
 =\frac1{(2\pi)^2}\int_{\Gamma_\delta^\circ}\int_{\Gamma_\delta^\circ}
 \psi(\mu)\overline{\psi(\nu)}\,
 \E\!\left[H(\zeta(\mu))\overline{H(\zeta(\nu))}\right]
 \dd\mu\,\dd\bar\nu.
\]
Since $v,\beta$ are real, Lemma~\ref{lem:recursion}, with $v_1=v_2=v$ and $w_1=w_2=\beta$, gives
\[
 \E|D_\psi|^2
 \le\bigl|\mathcal I^{(2)}_{\psi_\circ}\Xi^{\rm g}\bigr|
    +\frac Cn\norm\psi_{\infty,\Gamma_\delta}^2
 \le\frac Cn\norm\psi_{\infty,\Gamma_\delta}^2.
\]
The last inequality is Lemma~\ref{lem:integrated-kernel}, applied to $\psi_\circ$.

\proofstep{Noise term $N_\psi$.}
For $\eta(\xi)=n^{-1/2}v^\top G_{12}(\xi)\varepsilon$, the same calculation, using \eqref{eq:moment-recursion} and Lemma~\ref{lem:integrated-kernel}, gives
\[
 \E|N_\psi|^2
 \le\bigl|\mathcal I^{(2)}_{\psi_\circ}\Xi^{\rm n}\bigr|
    +\frac Cn\norm\psi_{\infty,\Gamma_\delta}^2
 \le\frac Cn\norm\psi_{\infty,\Gamma_\delta}^2.
\]
This proves the first part of \eqref{eq:F-rate-goal}.

\proofstep{Calibration remainder $R_\psi$.}
Equation \eqref{eq:empirical-control} gives, on $\mathcal A$,
\[
 \Delta Q=(\widehat\zeta-\zeta)Q(\widehat\zeta)Q(\zeta),
 \qquad
 \sup_{\mu\in\Gamma_\delta}\norm{\Delta Q(\mu)}_\op
 \le Cn^{-3/4}\delta^{-C}.
\]
Denote the two terms of $R_\psi$ in \eqref{eq:Ipsi_decompose} by
\begin{align*}
      R_\psi^{(1)}=\mathcal I_\psi^\circ
      [v^\top\mathcal X\Delta Q\,\mathcal X^\top\beta] \text{ and }
      R_\psi^{(2)} = \frac1{\sqrt n}\mathcal I_\psi^\circ
      [v^\top\mathcal X\Delta Q\,\varepsilon] \text{ such that } R_\psi  = R_\psi^{(1)}+ R_\psi^{(2)}. 
\end{align*}

Since $\norm{\mathcal X}_\op\le C$ on $\mathcal A$, 
$|R_\psi^{(1)}|\le Cn^{-3/4}\delta^{-C}\norm\psi_{\infty,\Gamma_\delta}$ and $
\EE[|R_\psi^{(2)}|^2| \mathcal X, \cA] \le Cn^{-5/2}\delta^{-C}\norm\psi_{\infty,\Gamma_\delta}^2$.
Thus we have
\[
 \norm{\ind_{\mathcal A}R_\psi}_{L^2}
 \le C(n^{-3/4}+n^{-5/4})\delta^{-C}
       \norm\psi_{\infty,\Gamma_\delta}
 =o(n^{-1/2})\norm\psi_{\infty,\Gamma_\delta}.
\]
Combining this with all the above bounds proves \eqref{eq:field-moment}.
\end{proof}

\subsection{Proof of Lemma~\ref{lem:global-fourth-moment}}\label{sec:pf:lem:global-fourth-moment}

We first record the Ward identity used in the proof. For completeness, we provide the proof as well.

\begin{lemma}[Ward identity
 {\citep[Eq.~(3.6)]{BenaychGeorgesKnowles2016}}]
\label{lem:ward-identity}
Let $S$ be a real symmetric $d\times d$ matrix and set
$Q(\xi)=(S-\xi I_d)^{-1}$ for $\xi\notin\spec(S)$. Then
\[
 \operatorname{Im}Q(\xi)
 =\frac{Q(\xi)-Q(\xi)^*}{2\ii}
 =(\operatorname{Im}\xi)Q(\xi)^*Q(\xi).
\]
Consequently, for $a,b\in\R^d$,
$
 \bigl|\operatorname{Im}(a^\top Q(\xi)b)\bigr|
 \le \norm a_2\norm b_2
       \bigl|\operatorname{Im}\tr Q(\xi)\bigr|.
$
\end{lemma}

\begin{proof}
We have
$
 Q(\xi)-Q(\xi)^*
 =(\xi-\bar\xi)Q(\xi)^*Q(\xi).
$
Thus $\operatorname{Im}Q(\xi)$ is positive or negative semidefinite,
and
$
 \norm{\operatorname{Im}Q(\xi)}_\op
 \le \bigl|\tr\operatorname{Im}Q(\xi)\bigr|.
$
Since $S,a,b$ are real,
$\operatorname{Im}(a^\top Q(\xi)b)
=a^\top\operatorname{Im}Q(\xi)b$, which proves the bound.
\end{proof}

\begin{proof}[Proof of Lemma~\ref{lem:global-fourth-moment}]
We first prove the deterministic bound
\begin{equation}\label{eq:deterministic-field-bound}
 |\mathcal I_\psi\widehat F|
 \le C\norm\psi_{\infty,\Gamma_\delta}\norm{Xv}_2\norm Y_2
\end{equation}
and then take moments. Recall that
$\mathcal I_\psi H=(2\pi\ii)^{-1}
\oint_{\Gamma_\delta}\psi(\mu)H(\mu)\dd\mu$, and write
\[
 \norm H_\Gamma^2
 =\frac1{2\pi}\int_0^{2\pi}
       |H(L+R_\delta e^{\ii t})|^2\,\dd t.
\]
If the branch in \eqref{eq:empirical-calibration} does not exist,
then $\widehat F=0$ and the bound is immediate. Otherwise, fix a data
realization for which the branch exists and set
\[
 Q(\xi)=(\widehat S-\xi I_n)^{-1},\qquad
 h(\xi)=-\frac1{\widehat m(\xi)},\qquad
 F_0(\xi)=\frac1n(Xv)^\top Q(\xi)Y.
\]
By \eqref{eq:empirical-calibration} and \eqref{eq:empirical-field},
$
 h(\widehat\zeta(\mu))=\mu,
 \widehat F(\mu)=F_0(\widehat\zeta(\mu))$, for 
 $
 \mu\in\Gamma_\delta.
$

\proofstep{The boundary curve and poles.}
Let $s_1\le\cdots\le s_n$ be the eigenvalues of $\widehat S$ and put
\[
 \widehat\Gamma=\widehat\zeta(\Gamma_\delta),\qquad
 \xi_-=\widehat\zeta(L-R_\delta),\qquad
 \xi_+=\widehat\zeta(L+R_\delta).
\]
The identity $h\circ\widehat\zeta=\mathrm{id}$ makes
$\widehat\zeta$ injective on $\Gamma_\delta$. Moreover,
$
 h'(\widehat\zeta(\mu))
 =\mu^2\widehat m'(\widehat\zeta(\mu))\ne0
$
by \eqref{eq:empirical-calibration}, so the local inverse function
theorem \citep[Chapter~4, Section~3.3]{Ahlfors1979} makes
$\widehat\zeta$ smooth along $\Gamma_\delta$. The Ward identity in 
Lemma~\ref{lem:ward-identity} gives
\[
 \operatorname{Im}\widehat m(\xi)
 =\frac{\operatorname{Im}\xi}{n}
       \tr\bigl(Q(\xi)^*Q(\xi)\bigr),
\]
so $\operatorname{Im}\widehat m(\xi)$ has the sign of
$\operatorname{Im}\xi$. Since
$\widehat m(\widehat\zeta(\mu))=-1/\mu$,
$\operatorname{Im}\widehat\zeta(\mu)$ has the sign of
$\operatorname{Im}\mu$. Since $\widehat m(\xi)=n^{-1}\sum_{i=1}^n(s_i-\xi)^{-1}$ is positive for real $\xi<s_1$ and negative for real $\xi>s_n$ while $L-R_\delta<0<L+R_\delta$, the interval exclusion $\xi_-, \xi_+ \not\in [s_1,s_n]$ in \eqref{eq:empirical-calibration}  implied by Proposition~\ref{prop:empirical-inverse} 
 gives
$
 \xi_-<s_1\le s_n<\xi_+.
$
Thus $\widehat\Gamma$ is a positively oriented simple closed curve
whose interior contains $(\xi_-,\xi_+)$.

The poles of $h$ are the zeros of $\widehat m$, all of which lie in
$[s_1,s_n]$ as $\operatorname{Im}\widehat m(\xi)$ has the sign of
$\operatorname{Im}\xi$; the eigenvalues are removable
singularities of $h$, with $h(s_i)=0$. Also,
$
 \operatorname{Im}h(\xi)
 =\frac{\operatorname{Im}\widehat m(\xi)}
        {|\widehat m(\xi)|^2},
$
so every zero of $h-L$ is real. For $x \in \R\setminus[s_1,s_n]$,
$h'(x)=\widehat m'(x)/\widehat m(x)^2>0$.
Since $h(\xi_-)=L-R_\delta$ and $h(\xi_+)=L+R_\delta$,
every zero of $h-L$ lies in $(\xi_-,\xi_+)$.

\proofstep{The circular $L^2$ identity.}
Define the rational function
\[
 H_F(\xi)=F_0(\xi)^2\frac{h'(\xi)}{h(\xi)-L}
\]
which has possible finite poles only at $s_1, \ldots, s_n$, the poles
of $h$, and the zeros of $h-L$. As we discussed above, all these points lie inside
$\widehat\Gamma$. At infinity, we have
\begin{equation}\label{eq:HF-at-infinity}
      F_0(\xi)=O(\xi^{-1}),\quad
      h(\xi)=\xi+O(1),\quad h'(\xi)=1+O(\xi^{-2}), \text{ and hence } H_F(\xi)=O(\xi^{-3}).
\end{equation}
Changing variables $\mu = L+R_\delta e^{\ii t}$ and $\xi = \widehat\zeta(\mu)$, we obtain
\begin{align}\label{eq:field-square-cancellation}
 \frac1{2\pi}\int_0^{2\pi}
       \widehat F(L+R_\delta e^{\ii t})^2\,\dd t
 =\frac1{2\pi\ii}\oint_{\Gamma_\delta}
       \frac{\widehat F(\mu)^2}{\mu-L}\,\dd\mu
       =\frac1{2\pi\ii}\oint_{\widehat\Gamma}
       F_0(\xi)^2\frac{h'(\xi)}{h(\xi)-L}\,\dd\xi
 =0,
\end{align}
where the last equality follows from Cauchy's theorem
\citep[Chapter~4, Theorem~15]{Ahlfors1979}, which gives $ \oint_{\widehat\Gamma}
H_F(\xi)\dd\xi =  \oint_{|\xi| = T}H_F(\xi)\dd\xi$
for any sufficiently large $T$ as the poles of $H_F$ are inside $\widehat\Gamma$. By \eqref{eq:HF-at-infinity},  $\oint_{\widehat\Gamma}
H_F(\xi)\dd\xi =\lim_{T\to\infty}\oint_{|\xi| = T}
H_F(\xi)\dd\xi = 0$.
As $\operatorname{Re}(z^2)=|z|^2-2(\operatorname{Im}z)^2$, we obtain
\begin{equation}\label{eq:field-L2-balance}
 \norm{\widehat F}_\Gamma^2
 -2\norm{\operatorname{Im}\widehat F}_\Gamma^2
 =
 \frac1{2\pi}\int_0^{2\pi}
 \operatorname{Re}\left(
 \widehat F(L+R_\delta e^{\ii t})^2
 \right)\dd t
 =0.
\end{equation}
Similarly, applying the same argument to $-1/\mu$ in place of $\hat F(\mu)$, we have
\begin{align*}
      \left\|\frac1\mu\right\|_\Gamma^2 - 2\left\|\operatorname{Im}\left(-\frac1\mu\right)\right\|_\Gamma^2
 &=
 \frac1{2\pi}\int_0^{2\pi}
 \operatorname{Re}\left(
  1/(L+R_\delta e^{\ii t})^2
 \right)\dd t \\
 &= \operatorname{Re}
 \frac1{2\pi}\int_0^{2\pi}
 \frac{e^{-2\ii t}}
 {(R_\delta+Le^{-\ii t})^2}\,\dd t =
 \operatorname{Re}\frac1{2\pi\ii}\oint_{|w|=1}
 \frac{w^2}{(R_\delta+Lw)^2}\,\frac{\dd w}{w}
=0.
\end{align*}
Here, the last equality follows from Cauchy's integral formula for $f(w)=w^2/(R_\delta+Lw)^2$. Then we have
\begin{equation}\label{eq:poisson-kernel-balance}
      2\left\|\operatorname{Im}\left(-\frac1\mu\right)\right\|_\Gamma^2
 =\left\|\frac1\mu\right\|_\Gamma^2
 =\frac1{2\pi}\int_0^{2\pi}
       \frac{\dd t}{|L+R_\delta e^{\ii t}|^2}
 =\frac1{R_\delta^2-L^2}. 
\end{equation}

\proofstep{The deterministic bound.}
Apply Lemma~\ref{lem:ward-identity} with
$a=Xv$, $b=Y$, and $\xi=\widehat\zeta(\mu)$.
Since
$\tr Q(\widehat\zeta(\mu))
=n\widehat m(\widehat\zeta(\mu))=-n/\mu$, we have
\begin{equation}\label{eq:deterministic-field-bound-ImF}
      |\operatorname{Im}\widehat F(\mu)|
      \le \frac{\norm{Xv}_2\norm Y_2}{n}
           \bigl|\operatorname{Im}\tr Q(\widehat\zeta(\mu))\bigr|
      =\norm{Xv}_2\norm Y_2
           \left|\operatorname{Im}\left(-\frac1\mu\right)\right|.  
\end{equation}
We therefore have
\[
 |\mathcal I_\psi\widehat F|
 \le R_\delta\norm\psi_{\infty,\Gamma_\delta}
       \norm{\widehat F}_\Gamma
 \le \frac{R_\delta}{\sqrt{R_\delta^2-L^2}}\,
       \norm\psi_{\infty,\Gamma_\delta}\norm{Xv}_2\norm Y_2,
\]
where the first inequality is by the Cauchy--Schwarz inequality and the second inequality is by combining \eqref{eq:field-L2-balance} with \eqref{eq:deterministic-field-bound-ImF} and \eqref{eq:poisson-kernel-balance}.
Since $R_\delta^2/(R_\delta^2-L^2)
=1+(\gamma(1+\delta)^2-1)^{-1}\le C$
by \eqref{eq:aspect-regime}, this proves
\eqref{eq:deterministic-field-bound}.

\proofstep{Moment bound.}
By the bounded moments of $Z_{ij}$, Rosenthal's inequality
\citep[Theorem~9]{Boucheron2005}  gives,
for $c\in\R^p$,
\[
 \Big\|\sum_{j=1}^pc_jZ_{ij}\Big\|_{L^8}
 \le C\Big(
       \Bigl(\sum_{j=1}^pc_j^2\Bigr)^{1/2}
       +\Bigl(\sum_{j=1}^p|c_j|^8\Bigr)^{1/8}
       \Big)
 \le C\norm c_2.
\]
Apply this to $c=\Sigma^{1/2}v$ and $c=\Sigma^{1/2}\beta$,
and use $\E|\varepsilon_i|^8\le C$, to get
$\E|(Xv)_i|^8+\E|Y_i|^8\le C$. Hence
\[
 \E\norm{Xv}_2^8+\E\norm Y_2^8
 \le n^3\sum_{i=1}^n
       \bigl(\E|(Xv)_i|^8+\E|Y_i|^8\bigr)
 \le Cn^4.
\]
Finally, \eqref{eq:deterministic-field-bound} and
the Cauchy--Schwarz inequality give
\[
 \E|\mathcal I_\psi\widehat F|^4
 \le C\norm\psi_{\infty,\Gamma_\delta}^4
       (\E\norm{Xv}_2^8)^{1/2}
       (\E\norm Y_2^8)^{1/2}
 \le Cn^4\norm\psi_{\infty,\Gamma_\delta}^4.
\]
\end{proof}

\subsection{Integration of the Leading Covariance}\label{app:pf:lem:integrated-kernel}

Here we prove Lemma~\ref{lem:integrated-kernel}. We first collect the facts related to positive kernels used in the proof. Throughout this subsection, $w$ and $u$ range over the open unit disk $\{w\in\C:|w|<1\}$, and a kernel is a continuous function $K(w,u)$ with $w,u$ in the open unit disk. A kernel $K$ is positive, written $K\succeq0$, if
\[
 \sum_{i,j=1}^Nc_i\bar c_jK(w_i,w_j)\ge0
 \text{ for all }N\ge1,\ |w_1|,\ldots,|w_N|<1,\ c\in\C^N,
\]
and we write $K_1\succeq K_2$ if $K_1-K_2\succeq0$. If $K_1\succeq0$ and $K_2\succeq0$, then $K_1K_2\succeq0$ by the Schur product theorem \citep[Theorem~7.5.3]{HornJohnson2013}. The Szeg\H{o} kernel 
\begin{equation}\label{eq:hardy-kernel}
 K_S(w,u)=\frac1{1-w\bar u}=\sum_{a=0}^\infty w^a\bar u^a,
\end{equation}
is the positive definite reproducing kernel of the Hardy space of the unit disk \citep[Section~1.4.1]{PaulsenRaghupathi2016}. 
The proof of Lemma~\ref{lem:integrated-kernel} rests on the following bound.

\begin{lemma}[Hardy kernel bound]\label{lem:hardy-kernel}
Let $K$ be a kernel with $K_S\succeq  K\succeq0$, let $0<\rho<1$, and let $b \in L^2([0,2\pi])$. Then
\begin{equation}\label{eq:hardy-kernel-bound}
 0\le\int_0^{2\pi}\!\!\int_0^{2\pi}
 b(t)\overline{b(s)}\,K(\rho e^{-\ii t},\rho e^{-\ii s})\,
 \frac{\dd t\dd s}{(2\pi)^2}
 \le\frac1{2\pi}\int_0^{2\pi}|b(t)|^2\dd t.
\end{equation}
\end{lemma}
\begin{proof}
For a kernel $K_0\succeq0$, the double integral of $b(t)\overline{b(s)}K_0(\rho e^{-\ii t},\rho e^{-\ii s})$ over $[0,2\pi]^2$ is a limit of Riemann sums of the form $\sum_{i,j}c_i\bar c_jK_0(w_i,w_j)$ with $w_i=\rho e^{-\ii t_i}$ and $c_i=b(t_i)\Delta t$, hence nonnegative. Taking $K_0=K$ gives the first inequality in \eqref{eq:hardy-kernel-bound}, and taking $K_0=K_S-K$ bounds the integral for $K$  in \eqref{eq:hardy-kernel-bound} by the integral for $K_S$. Given $a\in\mathbb Z$, let $\widehat b(a)=\frac1{2\pi}\int_0^{2\pi}b(t)e^{-\ii at}\dd t$ be the $a$th Fourier coefficient of $b$. 
Since $K_S(\rho e^{-\ii t},\rho e^{-\ii s})=\sum_{a\ge0}\rho^{2a}e^{-\ii a(t-s)}$ converges uniformly in $(t,s)$, integrating termwise gives
\[
 \int_0^{2\pi}\!\!\int_0^{2\pi}
 b(t)\overline{b(s)}\,K_S(\rho e^{-\ii t},\rho e^{-\ii s})\,
 \frac{\dd t\dd s}{(2\pi)^2}
 =\sum_{a=0}^\infty\rho^{2a}\,\widehat b(a)\overline{\widehat b(a)}
 =\sum_{a=0}^\infty\rho^{2a}|\widehat b(a)|^2 
 \le \frac1{2\pi}\int_0^{2\pi}|b(t)|^2\dd t,
\]
where the last inequality follows from $\rho<1$ and Parseval's identity \citep[Section~9.5]{WhittakerWatson1920}.
\end{proof}

\begin{proof}[Proof of Lemma~\ref{lem:integrated-kernel}] We aim to apply Lemma~\ref{lem:hardy-kernel} to $\mathcal I^{(2)}_\psi\Theta$ in \eqref{eq:hermitian-contour} for each of the kernels $\Theta = \Xi^{\rm g}$ and $\Theta = \Xi^{\rm n}$. We first reformulate those values to align with the formulation in Lemma~\ref{lem:hardy-kernel}.

\proofstep{Reformulating the kernels.}
Fix $\mu,\nu\in\Gamma_\delta$ and put $\xi_1=\zeta(\mu)$, $\xi_2=\overline{\zeta(\nu)}$. By Proposition~\ref{prop:calibration} and $m(\bar\xi)=\overline{m(\xi)}$, we have $m(\xi_1)=-1/\mu$ and $m(\xi_2)=-1/\bar\nu$, so that
\begin{equation}\label{eq:calibrated-values}
 D(\xi_1)^{-1}=\mu(\mu I_p-\Sigma)^{-1},\qquad
 D(\xi_2)^{-1}=\bar\nu(\bar\nu I_p-\Sigma)^{-1},\qquad
 m(\xi_1)m(\xi_2)=1/(\mu\bar\nu).
\end{equation}
For $\lambda\in[\ell,L]$ and $\mu\in\Gamma_\delta$, we have $|\lambda-\mu|\ge R_\delta-(L-\lambda)\ge\ell$, so $\norm{(\Sigma-\mu I_p)^{-1}}_\op\le1/\ell$, while $|\mu|\le L+R_\delta\le C$.
Write the eigendecomposition $\Sigma=\sum_{j=1}^p\lambda_jo_jo_j^\top$ and use the coordinate $w(\mu)=\sqrt\gamma L/(L-\mu)$, which maps $\Gamma_\delta$ onto the circle $|w|=(1+\delta)^{-1}$. Set
\[
 f_j(w)=\frac1{\sqrt n}\,\frac{\lambda_jw}{\sqrt\gamma L-(L-\lambda_j)w},
 \qquad
 K_f(w,u)=\frac1{1-\sum_{j=1}^pf_j(w)\overline{f_j(u)}}.
\]
Since $m(\xi_1)=-1/\mu$ and $\lambda_j-\mu=\bigl(\sqrt\gamma L-(L-\lambda_j)w(\mu)\bigr)/w(\mu)$, we have
\[
 \frac{\lambda_jm(\xi_1)}{1+\lambda_jm(\xi_1)}
 =\frac{\lambda_j}{\lambda_j-\mu}
 =\sqrt n\,f_j(w(\mu)),
 \qquad
 \frac{\lambda_jm(\xi_2)}{1+\lambda_jm(\xi_2)}
 =\frac{\lambda_j}{\lambda_j-\bar\nu}
 =\sqrt n\,\overline{f_j(w(\nu))}.
\]
Expanding the trace in \eqref{eq:covariance-functions} in the eigenbasis of $\Sigma$ therefore gives
\[
 \mathcal Q(\xi_1,\xi_2)
 =\sum_{j=1}^pf_j(w(\mu))\overline{f_j(w(\nu))},
 \qquad
 \frac1{1-\mathcal Q(\xi_1,\xi_2)}=K_f(w(\mu),w(\nu)).
\]
Expanding $M(\xi_1,\xi_2)=\Sigma D(\xi_1)^{-1}D(\xi_2)^{-1}$ in the same eigenbasis and using \eqref{eq:calibrated-values}, we have, for $x,y\in\{v,\beta\}$,
\begin{equation}\label{eq:calibrated-bilinear}
 x^\top M(\xi_1,\xi_2)y
 =\mu\bar\nu\sum_{j=1}^p(o_j^\top x)(o_j^\top y)A_j(\mu)\overline{A_j(\nu)},
 \text{ where }
 A_j(\mu)=\frac{\sqrt{\lambda_j}}{\lambda_j-\mu},
\end{equation}
and $S(\xi_1,\xi_2)=m(\xi_1)m(\xi_2)/(1-\mathcal Q(\xi_1,\xi_2))=K_f(w(\mu),w(\nu))/(\mu\bar\nu)$. Hence, with $A_{jk}(\mu)=\mu A_j(\mu)A_k(\mu)$, the Gaussian term $\Xi^{\rm g}$ in \eqref{eq:leading-covariance} with $v_1=v_2=v$ and $w_1=w_2=\beta$ becomes
\begin{equation}\label{eq:gaussian-kernel-expansion}
 \Xi^{\rm g}(\xi_1,\xi_2)
 =\frac1n\sum_{j\le k}\varkappa_{jk}\,
  K_f(w(\mu),w(\nu))A_{jk}(\mu)\overline{A_{jk}(\nu)},
 \text{ where }
 A_{jk}(\mu)=\frac{\mu\sqrt{\lambda_j\lambda_k}}{(\lambda_j-\mu)(\lambda_k-\mu)},
\end{equation}
with the nonnegative coefficients
\[
 \varkappa_{jk}=\bigl((o_j^\top v)(o_k^\top\beta)
       +(o_k^\top v)(o_j^\top\beta)\bigr)^2,
 \text{ for $j<k$, and }
 \varkappa_{jj}=2\bigl((o_j^\top v)(o_j^\top\beta)\bigr)^2,
\]
whose sum is $\sum_{j\le k}\varkappa_{jk}=\norm v_2^2\norm\beta_2^2+(v^\top\beta)^2\le2B_0^2$.
In the same way, the noise kernel \eqref{eq:noise-kernel} becomes
\begin{equation}\label{eq:noise-kernel-expansion}
 \Xi^{\rm n}(\xi_1,\xi_2)
 =\frac1n\sum_{j=1}^p\sigma^2(o_j^\top v)^2\,
  K_f(w(\mu),w(\nu))A_j(\mu)\overline{A_j(\nu)},
\end{equation}
with coefficients $\sigma^2(o_j^\top v)^2$ summing to at most $\bar\sigma^2$. Since $|\lambda_j-\mu|\ge\ell$ and $|\mu|\le C$ on $\Gamma_\delta$, the multipliers satisfy
\begin{equation}\label{eq:multiplier-bounds}
      |A_j(\mu)|\le\sqrt L/\ell \text{ and } |A_{jk}(\mu)|\le CL/\ell^2 \text{ for $\mu\in\Gamma_\delta$.}
\end{equation}

\proofstep{Implementing the Hardy kernel bound.}
We verify the conditions of Lemma~\ref{lem:hardy-kernel} for $K=K_f$. Consider the function $\Phi(w)=(f_1(w)/w,\ldots,f_p(w)/w)$, which is analytic for $|w|<\sqrt\gamma L/(L-\ell)$, a disk of radius larger than one. For $|w|<1$, we have $|\sqrt\gamma L-(L-\lambda_j)w|\ge\sqrt\gamma L-L+\lambda_j\ge\sqrt\gamma\lambda_j$, where the last step uses $\lambda_j\le L$ and $\gamma\ge1$. Hence
\[
 \sup_{|w|<1}\norm{\Phi(w)}_2^2
 =\sup_{|w|<1}\frac1n\sum_{j=1}^p
 \frac{\lambda_j^2}{|\sqrt\gamma L-(L-\lambda_j)w|^2}
 \le\frac p{n\gamma}=1.
\]
Therefore $F(w,u)=\sum_{j=1}^pf_j(w)\overline{f_j(u)}=w\bar u\,\Phi(w)\Phi(u)^*$ satisfies $|F(w,u)|\le|w||u|<1$, so $K_f=1/(1-F)=\sum_{a\ge0}F^a$ is a kernel. The Gram form $\sum_{i,j}c_i\bar c_jF(w_i,w_j)=\sum_{j=1}^p|\sum_ic_if_j(w_i)|^2$ shows $F\succeq0$, and the Schur product theorem \citep[Theorem~7.5.3]{HornJohnson2013} applied to each power $F^a$ gives $K_f\succeq0$. To verify $K_S\succeq K_f$ for the Szeg\H{o} kernel $K_S$ in \eqref{eq:hardy-kernel}, we construct another kernel 
\[
 K_\Phi(w,u)=
 \frac{1-\Phi(w)\Phi(u)^*}{1-w\bar u}\succeq0,
\]
where the last positivity is by applying 
\citet[Theorem~1.1]{BallBiswasFangTerHorst2008} to $\Phi$. Then we have the following identity
\[
      K_S(w,u)-K_f(w,u)
 =w\bar u\,K_\Phi(w,u)K_f(w,u)\succeq0,
\]
where the positivity is by the Schur product theorem since $w\bar u\succeq0$ as well. Thus $K_S\succeq K_f \succeq  0$ satisfies the conditions of Lemma~\ref{lem:hardy-kernel}.

Parameterize $\Gamma_\delta$ by $\mu(t)=L-R_\delta e^{\ii t}$, $t\in[0,2\pi]$, so that
$w(\mu(t))=\rho e^{-\ii t}$ with $\rho=(1+\delta)^{-1}$, and $\dd\mu \dd\bar\nu=\mu'(t)\overline{\mu'(s)}\dd t\dd s$ for $\mu=\mu(t)$, $\nu=\mu(s)$. For any one multiplier $A$ among $A_j$ in \eqref{eq:calibrated-bilinear} and $A_{jk}$ in \eqref{eq:gaussian-kernel-expansion}, define the kernel $\Theta_A$ on $\zeta(\Gamma_\delta)\times\zeta(\Gamma_\delta)$ by
\[
 \Theta_A\bigl(\zeta(\mu),\overline{\zeta(\nu)}\bigr)
 =K_f(w(\mu),w(\nu))A(\mu)\overline{A(\nu)},
 \qquad\mu,\nu\in\Gamma_\delta,
\]
and put $b(t)=\mu'(t)\psi(\mu(t))A(\mu(t))$, which belongs to $L^2([0,2\pi])$. Since $|\mu'(t)|=R_\delta\le C$, $\mu(t)\in\Gamma_\delta$, and $|A(\mu)|\le\sqrt L/\ell+CL/\ell^2$ on $\Gamma_\delta$ by \eqref{eq:multiplier-bounds}, we have
\begin{equation*}
 |b(t)|
 =|\mu'(t)|\,|\psi(\mu(t))|\,|A(\mu(t))|
 \le R_\delta\,\norm\psi_{\infty,\Gamma_\delta}
   \Bigl(\frac{\sqrt L}{\ell}+\frac{CL}{\ell^2}\Bigr)
 \le C\norm\psi_{\infty,\Gamma_\delta},
 \qquad t\in[0,2\pi].
\end{equation*}
By \eqref{eq:hermitian-contour}, we have
\[
 \mathcal I^{(2)}_\psi\Theta_A  =\frac1{(2\pi)^2}\oint_{\Gamma_\delta}\oint_{\Gamma_\delta}
 \psi(\mu)\overline{\psi(\nu)}\,
 \Theta_A\bigl(\zeta(\mu),\overline{\zeta(\nu)}\bigr)
 \dd\mu\,\dd\bar\nu
 =\int_0^{2\pi}\!\!\int_0^{2\pi}
 b(t)\overline{b(s)}K_f(\rho e^{-\ii t},\rho e^{-\ii s})\,\frac{\dd t\dd s}{(2\pi)^2}.
\]
Lemma~\ref{lem:hardy-kernel} with $K=K_f$ therefore gives
\[
 0\le\mathcal I^{(2)}_\psi\Theta_A
 \le\frac1{2\pi}\int_0^{2\pi}|b(t)|^2\dd t
 \le C\norm\psi_{\infty,\Gamma_\delta}^2.
\]
Applying this bound termwise in \eqref{eq:gaussian-kernel-expansion} and \eqref{eq:noise-kernel-expansion}, we have
\[
 \bigl|\mathcal I^{(2)}_\psi\Xi^{\rm g}\bigr|
 \le\frac{2B_0^2C}{n}\norm\psi_{\infty,\Gamma_\delta}^2,
 \qquad
 \bigl|\mathcal I^{(2)}_\psi\Xi^{\rm n}\bigr|
 \le\frac{\bar\sigma^2C}{n}\norm\psi_{\infty,\Gamma_\delta}^2.
\]
\end{proof}

\subsection{Proof of Lemma~\ref{lem:recursion}}\label{app:pf:lem:recursion}

To prove Lemma~\ref{lem:recursion}, we first introduce the cumulant and the cumulant expansion.

\begin{definition}[Cumulants]\label{def:cumulants}
Let $Z$ be a real random variable with $\E|Z|^q<\infty$ for an integer
$q\ge1$. Its cumulant of order $j$, for $1\le j\le q$, is
\[
 \chi_j(Z)=\ii^{-j}
   \left.\frac{\dd^j}{\dd t^j}\log\E e^{\ii tZ}\right|_{t=0}.
\]
\end{definition}
When the respective moments exist, differentiation gives
\begin{align*}
 \chi_1(Z)&=\E Z,
 &\chi_2(Z)&=\Var(Z),\\
 \chi_3(Z)&=\E[(Z-\E Z)^3],
 &\chi_4(Z)&=\E[(Z-\E Z)^4]-3\Var(Z)^2.
\end{align*}
The chain rule also gives $\chi_j(aZ)=a^j\chi_j(Z)$ for every
$a\in\R$ and $1\le j\le q$.

In our proof, we only use the cumulants of $Z_{11}$ so we abbreviate $\chi_j = \chi_j(Z_{11})$ when it is clear from the context.
Write $C^q(\R;\C)$ for the space of complex-valued functions on
$\R$ with continuous derivatives through order $q$. We have the following lemma on the cumulant expansion for smooth functions.

\begin{lemma}[Cumulant expansion]\label{lem:cumulant-expansion}
Let $\mathsf Z$ be a centered real random variable with $\E \mathsf Z^2=n^{-1}$.
Fix an integer $q\ge2$ and suppose $\E|\mathsf Z|^{q+1}<\infty$.
For $f\in C^q(\R;\C)$ with $\|f^{(q)}\|_\infty<\infty$,
\begin{equation}\label{eq:cumulant-expansion-formula}
 \E[\mathsf Z f(\mathsf Z)]
 =\sum_{j=2}^{q}
   \frac{\chi_j(\mathsf Z)}{(j-1)!}
   \E f^{(j-1)}(\mathsf Z)+R_{q+1}.
\end{equation}
For every fixed $\rho>0$, the remainder satisfies
\begin{equation}\label{eq:cumulant-expansion-remainder}
 \begin{aligned}
 |R_{q+1}|\le{}&C_q\E|\mathsf Z|^{q+1}
       \sup_{|t|\le n^{-1/2+\rho}}|f^{(q)}(t)|
       +C_q\E\bigl[|\mathsf Z|^{q+1}
       \ind_{\{|\mathsf Z|>n^{-1/2+\rho}\}}\bigr]
       \sup_{t\in\R}|f^{(q)}(t)|.
 \end{aligned}
\end{equation}
\end{lemma}
\begin{proof}
This is a direct application of \citet[Lemma~3.7]{Yang2024} with $l=q-1$ to the real and
imaginary parts of $f$.
\end{proof}

\begin{proof}[Proof of Lemma~\ref{lem:recursion}]
Fix a sufficiently small $0<\epsilon_0<1/32$. By homogeneity, we can first assume that the
deterministic vectors $v_1,w_1,v_2,w_2,v\in\C^p$ have norm at most one. All estimates below are
uniform for $\xi_j=\zeta(\mu_j)$, $\mu_j\in\Gamma_\delta^\circ$ in \eqref{eq:nonreal-contour}.
We first establish the resolvent moment bounds, then compute the noise covariance
by conditioning on $X$, and finally evaluate the cumulant expansion of $\E(H_1H_2)$.

\proofstep{Uniform local-law inputs.}
Calibration $m(\zeta(\mu)) = - 1/\mu$ and \eqref{eq:calibration-denominator-bound} give
\begin{align*}
 |m(\xi_j)|+\|D(\xi_j)^{-1}\|_\op
 &=\frac1{|\mu_j|}
   +\max_{1\le k\le p}\frac{|\mu_j|}{|\mu_j-\lambda_k|}
 \le C,\\
 |\mathcal Q(\xi_1,\xi_2)|
 &=\left|\frac1n\sum_{k=1}^p
       \frac{\lambda_k^2}{(\lambda_k-\mu_1)(\lambda_k-\mu_2)}\right|
       \le\prod_{j=1}^2
       \biggl(\frac1n\sum_{k=1}^p
       \frac{\lambda_k^2}{|\lambda_k-\mu_j|^2}\biggr)^{1/2}
 \le(1+\delta)^{-2}.
\end{align*}
Consequently, for all sufficiently large $n$,
\begin{equation}\label{eq:recursion-stability}
 \begin{aligned}
 |1-\mathcal Q(\xi_1,\xi_2)|
 &\ge1-|\mathcal Q(\xi_1,\xi_2)|
 \ge\frac{\delta(2+\delta)}{(1+\delta)^2}\ge c\delta,\\
 |S(\xi_1,\xi_2)|&\le C\delta^{-1},
 \qquad \|M(\xi_1,\xi_2)\|_\op\le C.
 \end{aligned}
\end{equation}
This controls the inverse stability factors in the covariance calculation.

For every fixed $r\ge2$ and deterministic unit vectors
$u,w\in\C^{p+n}$, since $\xi_j\in\mathcal T^+$ in \eqref{eq:spectral-neighborhood}, \eqref{eq:coarse-moments}
and \eqref{eq:preliminary-spectral} give
\begin{equation}\label{eq:recursion-raw-moments-on-A1}
\begin{aligned}
 \bigl\|\ind_{\mathcal A_1}u^\top(G(\xi_j)-\Pi(\xi_j))w\bigr\|_{L^r}
 &\le Cn^{-1/2+\epsilon_0},\quad 
 \bigl\|\ind_{\mathcal A_1}(\widehat m(\xi_j)-m(\xi_j))\bigr\|_{L^r}
 \le Cn^{-1+\epsilon_0},\\
 \bigl\|\ind_{\mathcal A_1}\|G(\xi_j)\|_\op\bigr\|_{L^r}
 &\le C\delta^{-2}\PP(\mathcal A_1)^{1/r}\le C\delta^{-2}.
\end{aligned}
\end{equation}
On $\mathcal A_1^c$, \eqref{eq:linearization}
gives $\|\Pi(\xi_j)\|_\op+|m(\xi_j)|\le C$.
Thus \eqref{eq:raw-polynomial-bound} yields
\begin{equation}\label{eq:recursion-raw-moments-complementary}
      \begin{aligned}
            \bigl\|\ind_{\mathcal A_1^c}u^\top(G(\xi_j)-\Pi(\xi_j))w\bigr\|_{L^r}
            &\le\bigl\|\ind_{\mathcal A_1^c}(\|G(\xi_j)\|_\op+C)\bigr\|_{L^r}
            \le Cn^2\delta^{-1}\PP(\mathcal A_1^c)^{1/r},\\
            \bigl\|\ind_{\mathcal A_1^c}(\widehat m(\xi_j)-m(\xi_j))\bigr\|_{L^r}
            &\le\bigl\|\ind_{\mathcal A_1^c}(|\widehat m(\xi_j)|+C)\bigr\|_{L^r}
            \le Cn^2\delta^{-1}\PP(\mathcal A_1^c)^{1/r},\\
            \bigl\|\ind_{\mathcal A_1^c}\|G(\xi_j)\|_\op\bigr\|_{L^r}
            &\le\bigl(\E[\ind_{\mathcal A_1^c}(Cn^2\delta^{-1})^r]\bigr)^{1/r}
            =Cn^2\delta^{-1}\PP(\mathcal A_1^c)^{1/r}.
           \end{aligned}    
\end{equation}

For any fixed $A>0$, choose $D\ge r(A+3)$ in
$\PP(\mathcal A_1^c)\le C_Dn^{-D}$. Since $\delta^{-1}\le\log n$,
the right-hand sides of the three inequalities in \eqref{eq:recursion-raw-moments-complementary} satisfy
\[
 Cn^2\delta^{-1}\PP(\mathcal A_1^c)^{1/r}
 \le C_Dn^{2-D/r}\log n
 \le C_Dn^{-A-1}\log n\le C_{A,r}n^{-A}.
\]
Taking $A=2$ and combining \eqref{eq:recursion-raw-moments-on-A1} and \eqref{eq:recursion-raw-moments-complementary}, we have
\begin{equation}\label{eq:recursion-raw-moments}
      \begin{aligned}
      \|u^\top(G(\xi_j)-\Pi(\xi_j))w\|_{L^r}
      &\le Cn^{-1/2+\epsilon_0},\\
      \|\widehat m(\xi_j)-m(\xi_j)\|_{L^r}
      &\le Cn^{-1+\epsilon_0},
      \qquad
      \bigl\|\|G(\xi_j)\|_\op\bigr\|_{L^r}\le C\delta^{-2}.
      \end{aligned}
     \end{equation}

For a holomorphic function $h$, write
$h[\xi_1,\xi_2]=(h(\xi_1)-h(\xi_2))/(\xi_1-\xi_2)$, with
$h[\xi,\xi]=h'(\xi)$, and apply this notation entrywise to matrices. 
By definitions \eqref{eq:MP} and \eqref{eq:linearization}, we have the identities
\begin{equation}\label{eq:recursion-divided-difference-identities}
 \begin{gathered}
 G(\xi_1)\diag(0,I_n)G(\xi_2)=G[\xi_1,\xi_2],\qquad
 m[\xi_1,\xi_2]
 =\frac{m(\xi_1)m(\xi_2)}{1-\mathcal Q(\xi_1,\xi_2)}
 =S(\xi_1,\xi_2),\\
 \Pi_{11}[\xi_1,\xi_2]
 =m[\xi_1,\xi_2]\Sigma D(\xi_1)^{-1}D(\xi_2)^{-1}
 =S(\xi_1,\xi_2)M(\xi_1,\xi_2).
 \end{gathered}
\end{equation}
When $\xi_1=\xi_2$, the identities above follow by differentiation.

For deterministic feature vectors $u,w\in\C^p$ of norm at most one,
let $U(\xi)=u^\top(G(\xi)-\Pi(\xi))w$ on $\mathcal A_1$ and
$U(\xi)=0$ on $\mathcal A_1^c$.
Then $U$ is holomorphic on $\mathcal T^+$, and
\eqref{eq:coarse-moments} gives
\[
 \Bigl\|\sup_{\xi\in\mathcal T^+}|U(\xi)|\Bigr\|_{L^r}
 +\Bigl\|\sup_{\xi\in\mathcal T}|U'(\xi)|\Bigr\|_{L^r}
 \le Cn^{-1/2+\epsilon_0}.
\]
Recall that $\mathcal T=\{\xi\in\C:\dist(\xi,\zeta(\Gamma_\delta))<c_1\delta^b\}$ in \eqref{eq:spectral-neighborhood}.
If $|\xi_1-\xi_2|<c_1\delta^b$, then
$\xi_2+t(\xi_1-\xi_2)\in\mathcal T$ for every $t\in[0,1]$.
Hence, including $\xi_1=\xi_2$,
\begin{align*}
 U[\xi_1,\xi_2]
 =\int_0^1U'\bigl(\xi_2+t(\xi_1-\xi_2)\bigr)\dd t,
 \|U[\xi_1,\xi_2]\|_{L^r}
 \le\int_0^1
   \bigl\|U'\bigl(\xi_2+t(\xi_1-\xi_2)\bigr)\bigr\|_{L^r}\dd t
 \le Cn^{-1/2+\epsilon_0}.
\end{align*}
If $|\xi_1-\xi_2|\ge c_1\delta^b$, then
\[
 \|U[\xi_1,\xi_2]\|_{L^r}
 \le\frac{\|U(\xi_1)\|_{L^r}+\|U(\xi_2)\|_{L^r}}
              {|\xi_1-\xi_2|}
 \le Cn^{-1/2+\epsilon_0}\delta^{-b}.
\]
By \eqref{eq:recursion-divided-difference-identities} and the definition
of $U$, we have
\begin{align*}
 U[\xi_1,\xi_2]
 &=\ind_{\mathcal A_1}u^\top(G_{11}-\Pi_{11})[\xi_1,\xi_2]w
 =\ind_{\mathcal A_1}\bigl(
       u^\top G(\xi_1)\diag(0,I_n)G(\xi_2)w
       -S(\xi_1,\xi_2)u^\top M(\xi_1,\xi_2)w\bigr).
\end{align*}
On $\mathcal A_1^c$, using \eqref{eq:raw-polynomial-bound} and
\eqref{eq:recursion-stability}, we have
\begin{align*}
 &\bigl|u^\top G(\xi_1)\diag(0,I_n)G(\xi_2)w
       -S(\xi_1,\xi_2)u^\top M(\xi_1,\xi_2)w\bigr|\\
 &\qquad\le\|G(\xi_1)\|_\op\|G(\xi_2)\|_\op
       +|S(\xi_1,\xi_2)|\,\|M(\xi_1,\xi_2)\|_\op
 \le C(n^4\delta^{-2}+\delta^{-1}).
\end{align*}
Thus, similar to the proof of \eqref{eq:recursion-raw-moments}, choosing $D\ge 7r$ in
$\PP(\mathcal A_1^c)\le C_Dn^{-D}$ gives
\begin{align*}
 &\Bigl\|\ind_{\mathcal A_1^c}\bigl(
       u^\top G(\xi_1)\diag(0,I_n)G(\xi_2)w
       -S(\xi_1,\xi_2)u^\top M(\xi_1,\xi_2)w\bigr)\Bigr\|_{L^r}\\
 &\qquad\le C(n^4\delta^{-2}+\delta^{-1})\PP(\mathcal A_1^c)^{1/r}
 \le C_Dn^{4-D/r}(\log n)^2\le C_{r}n^{-2},
\end{align*}
where we used $\delta^{-1}\le\log n$.
Combining the $\mathcal A_1$ and
$\mathcal A_1^c$ parts above, we have uniformly for $\xi_1,\xi_2\in\zeta(\Gamma_\delta^\circ)$,
\begin{equation}\label{eq:recursion-two-resolvent-moments}
 \begin{aligned}
 &\bigl\|u^\top G(\xi_1)\diag(0,I_n)G(\xi_2)w
       -S(\xi_1,\xi_2)u^\top M(\xi_1,\xi_2)w\bigr\|_{L^r}\\
 &\qquad\le\|U[\xi_1,\xi_2]\|_{L^r}+Cn^{-2}
 \le Cn^{-1/2+\epsilon_0}\delta^{-b}+Cn^{-2}
 \le Cn^{-1/2+\epsilon_0}\delta^{-b}.
 \end{aligned}
\end{equation}

\proofstep{The noise covariance.}
By the definition $\eta(\xi)=\frac1{\sqrt n}v^\top G_{12}(\xi)\varepsilon$ and
$\E[\varepsilon\varepsilon^\top | X]=\sigma^2I_n$,
\begin{align*}
 \E[\eta(\xi_1)\eta(\xi_2) | X]
 &=\frac1n v^\top G_{12}(\xi_1)
       \E[\varepsilon\varepsilon^\top | X]G_{12}(\xi_2)^\top v
=\frac{\sigma^2}{n}
       v^\top G(\xi_1)\diag(0,I_n)G(\xi_2)v.
\end{align*}
Here $G_{12}(\xi_2)^\top=G_{21}(\xi_2)$ follows from
$G(\xi_2)^\top=G(\xi_2)$. Taking expectations and using
\eqref{eq:noise-kernel} and
\eqref{eq:recursion-two-resolvent-moments} with $u=w=v$ gives
\begin{align*}
 \bigl|\E[\eta(\xi_1)\eta(\xi_2)]-\Xi^{\rm n}(\xi_1,\xi_2)\bigr|
 &\le\frac{\sigma^2}{n}
       \bigl\|v^\top G(\xi_1)\diag(0,I_n)G(\xi_2)v
             -S(\xi_1,\xi_2)v^\top M(\xi_1,\xi_2)v\bigr\|_{L^2}\\
 &\le Cn^{-3/2+\epsilon_0}\delta^{-b}
 \le Cn^{-3/2+\epsilon_0}(\log n)^b\le\frac Cn.
\end{align*}

\proofstep{Second-moment calculation.} This part is essentially \citet[Lemma~5.2, Eq.~(5.2)]{Yang2024} for $k=2$. However, it needs additional bulk and edge regularity assumptions (see \citet[Assumption~2.5]{Yang2024}). Here we adapt the proof of \citet[Lemma~5.2]{Yang2024} in our setting.  We apply Lemma~\ref{lem:cumulant-expansion} to the independent entries and evaluate each cumulant contribution below.
Recall we assume that $\mathsf Z_{ki}=Z_{ik}/\sqrt n$, and
\[
 \E\mathsf Z_{ki}=0,\qquad
 \E\mathsf Z_{ki}^2=n^{-1},\qquad
 \sup_{k,i}\E|\sqrt n\,\mathsf Z_{ki}|^q\le M_q
 \text{ for }q\ge3,
\]
which are also needed in \citet[Eq.~(5.1)]{Yang2024}.

We keep $H_j=v_j^\top(G(\xi_j)-\Pi(\xi_j))w_j$ in its linearized
form. Equation~\eqref{eq:recursion-raw-moments} gives
\begin{align*}
 \|\sqrt n\,H_j\|_{L^r}
 &=\sqrt n\,\|v_j^\top(G(\xi_j)-\Pi(\xi_j))w_j\|_{L^r}
 \le Cn^{\epsilon_0},\\
 \E(H_1H_2)&=\frac1n\E\bigl[(\sqrt n\,H_1)(\sqrt n\,H_2)\bigr].
\end{align*}
Thus the normalization uses $\sqrt n$ and introduces no division by
$\xi_j$. The same linearized variables underlie
\citet[Eq.~(5.5)]{Yang2024}.

For a deterministic feature vector $u$ with $\|u\|_2\le1$, the identities
$G(\xi_j)^\top=G(\xi_j)$, $\mathsf u_k=\Sigma^{1/2}e_k$, and
$\mathsf e_i=(0,e_i)$ give
\begin{align*}
      &\sum_{k=1}^p|u^\top G(\xi_j)\mathsf u_k|^2
 +\sum_{i=1}^n|u^\top G(\xi_j)\mathsf e_i|^2
 =\|\Sigma^{1/2}G_{11}(\xi_j)u\|_2^2
   +\|G_{21}(\xi_j)u\|_2^2\\
 &\qquad \le L\|G_{11}(\xi_j)u\|_2^2
   +\|G_{21}(\xi_j)u\|_2^2
   \le(L+1)\|G(\xi_j)u\|_2^2
 \le(L+1)\|G(\xi_j)\|_\op^2.
\end{align*}
The last expression is at most $C\delta^{-4}$ on $\mathcal A_1$ by
\eqref{eq:preliminary-spectral}. Its $L^r$ norm is also controlled, since
\begin{align*}
 \bigl\|\|G(\xi_j)\|_\op^2\bigr\|_{L^r}
 &=\bigl\|\|G(\xi_j)\|_\op\bigr\|_{L^{2r}}^2
 \le C\delta^{-4}
\end{align*}
by \eqref{eq:recursion-raw-moments}.
Together with the anisotropic and averaged bounds in
\eqref{eq:recursion-raw-moments}, these are the moment and row-sum
bounds used in \citet[Eqs.~(3.20)--(3.21) and~(5.50)--(5.51)]{Yang2024}.

To expand $\E(H_1H_2)$ in the independent entries, the upper-left block
of $G(\xi_1)^{-1}G(\xi_1)=I_{p+n}$ and the definition of $D(\xi)=I_p+m(\xi)\Sigma$ in \eqref{eq:covariance-functions} give
\[
 \mathcal XG_{21}(\xi_1)=I_p+G_{11}(\xi_1),\qquad
 D(\xi_1)^{-1}+m(\xi_1)D(\xi_1)^{-1}\Sigma=I_p.
\]
Using $\Pi_{11}(\xi_1)=-D(\xi_1)^{-1}$ and
$\mathcal X=\sum_{k=1}^p\sum_{i=1}^n\mathsf Z_{ki}\mathsf u_ke_i^\top$,
we obtain
\begin{equation}\label{eq:recursion-entry-identity}
 \begin{aligned}
 H_1
 &=v_1^\top\bigl(G_{11}(\xi_1)+D(\xi_1)^{-1}\bigr)w_1\\
 &=v_1^\top D(\xi_1)^{-1}\mathcal XG_{21}(\xi_1)w_1
   +m(\xi_1)v_1^\top D(\xi_1)^{-1}\Sigma G_{11}(\xi_1)w_1\\
 &=\sum_{k=1}^p\sum_{i=1}^n
   \bigl(v_1^\top D(\xi_1)^{-1}\mathsf u_k\bigr)\mathsf Z_{ki}
   \mathsf e_i^\top G(\xi_1)w_1
   +m(\xi_1)v_1^\top D(\xi_1)^{-1}\Sigma G_{11}(\xi_1)w_1.
 \end{aligned}
\end{equation}
For each $(k,i)$, hold all other entries fixed and define the leave-one-out function
\[
 f_{ki}(t)=\left.
   \bigl(\mathsf e_i^\top G(\xi_1)w_1\bigr)H_2
   \right|_{\mathsf Z_{ki}=t},\qquad t\in\R.
\]
Here $\mathsf Z_{ki}$ is replaced by $t$ in both $G(\xi_1)$ and $H_2$.
Since $\xi_1,\xi_2\notin\R$, this function is smooth for every real $t$.
Recall $\chi_j = \chi_j(Z_{11}) = \chi_j(\sqrt{n} \mathsf Z_{ki}) = n^{j/2}\chi_j(\mathsf Z_{ki})$. Multiply \eqref{eq:recursion-entry-identity} by $H_2$ and apply the
sixth-order expansion in Lemma~\ref{lem:cumulant-expansion} with
$\mathsf Z=\mathsf Z_{ki}$ and $f=f_{ki}$ for all $(k,i)$, conditionally on the other entries, to obtain
\begin{equation}\label{eq:recursion-cumulant-expansion}
 \begin{aligned}
 \E(H_1H_2)
 ={}&m(\xi_1)\E\bigl[
     \bigl(v_1^\top D(\xi_1)^{-1}\Sigma G_{11}(\xi_1)w_1\bigr)H_2\bigr]\\
 &+\sum_{q=1}^{5}
      \frac{\chi_{q+1}}{q!\,n^{(q+1)/2}}
      \sum_{k=1}^p\sum_{i=1}^n
      \bigl(v_1^\top D(\xi_1)^{-1}\mathsf u_k\bigr)
      \E f_{ki}^{(q)}(\mathsf Z_{ki})\\
 &+\sum_{k=1}^p\sum_{i=1}^n
      \bigl(v_1^\top D(\xi_1)^{-1}\mathsf u_k\bigr)R_{ki},
 \end{aligned}
\end{equation}
where $R_{ki}$ is the conditional expansion remainder averaged over the
other entries. This is the entry expansion in
\citet[Eqs.~(5.6)--(5.9)]{Yang2024}, written for $H_1,H_2$.

Taking expectations of the conditional remainder
bound \eqref{eq:cumulant-expansion-remainder} with fixed $0<\rho<1/2$ gives
\begin{equation}\label{eq:recursion-cumulant-remainder}
 |R_{ki}|
 \le C\E|\mathsf Z_{ki}|^7\,
       \E\sup_{|t|\le n^{-1/2+\rho}}|f_{ki}^{(6)}(t)|
       +C\E\bigl[|\mathsf Z_{ki}|^7
       \ind_{\{|\mathsf Z_{ki}|>n^{-1/2+\rho}\}}\bigr]\,
       \E\sup_{t\in\R}|f_{ki}^{(6)}(t)|.
\end{equation}
Here the entry moments factor out of the expectations over the other
entries because $f_{ki}(t)$ depends only on those entries once $t$ is
fixed. 

We next bound $f_{ki}^{(6)}(t)$ for $t$ in the first term of \eqref{eq:recursion-cumulant-remainder}.
For each fixed $(k,i)$, define the matrices obtained by replacing
$\mathsf Z_{ki}$ by $t\in\R$:
\begin{align*}
 \mathcal X^{(ki)}(t)
 =\mathcal X+(t-\mathsf Z_{ki})\mathsf u_ke_i^\top,
 \widehat S^{(ki)}(t)
 =\mathcal X^{(ki)}(t)^\top\mathcal X^{(ki)}(t),
 G^{(ki)}(\xi;t)
 =\begin{pmatrix}
   -I_p&\mathcal X^{(ki)}(t)\\
   \mathcal X^{(ki)}(t)^\top&-\xi I_n
  \end{pmatrix}^{-1}.
\end{align*}
In particular, $\widehat S^{(ki)}(\mathsf Z_{ki})=\widehat S$, whereas
$\mathcal X^{(ki)}(0)$ depends only on the other entries.
For $s,t\in\R$, expansion of the Gram matrices gives
\begin{align*}
 \widehat S^{(ki)}(t)-\widehat S^{(ki)}(s)
 ={}&(t-s)\bigl(e_i\mathsf u_k^\top\mathcal X^{(ki)}(s)
          +\mathcal X^{(ki)}(s)^\top\mathsf u_ke_i^\top\bigr)
       +(t-s)^2\|\mathsf u_k\|_2^2 e_ie_i^\top.
\end{align*}
Thus Weyl's inequality \citep[Theorem~4.3.1]{HornJohnson2013} yields
\begin{equation}\label{eq:recursion-entry-perturbation}
 \begin{aligned}
 \max_{1\le a\le n}
 \bigl|\lambda_a(\widehat S^{(ki)}(t))
       -\lambda_a(\widehat S^{(ki)}(s))\bigr|
 &\le\|\widehat S^{(ki)}(t)-\widehat S^{(ki)}(s)\|_\op\\
 &\le2\sqrt L\,\|\mathcal X^{(ki)}(s)\|_\op|t-s|+L|t-s|^2.
 \end{aligned}
\end{equation}
Fix $c,C$ as in \eqref{eq:preliminary-spectral} and define
\begin{equation}\label{eq:recursion-replacement-event}
 \mathcal A_{ki}
 =\Bigl\{\|\mathcal X^{(ki)}(0)\|_\op\le\sqrt C+1,\quad
   \inf_{\xi\in\mathcal T^+}
     \dist\bigl(\xi,\spec(\widehat S^{(ki)}(0))\bigr)
        \ge\tfrac c2\delta^2\Bigr\}.
\end{equation}
On $\mathcal A_1\cap\{|Z_{ik}|\le n^\rho\}$, apply
\eqref{eq:recursion-entry-perturbation} with $s=\mathsf Z_{ki}$ and $t=0$.
Since $\rho<1/2$ and $\delta^{-1}\le\log n$, for all sufficiently large $n$,
\begin{align*}
 \|\mathcal X^{(ki)}(0)\|_\op
 &\le\|\mathcal X\|_\op+\sqrt L\,|\mathsf Z_{ki}|
 \le\sqrt C+\sqrt L\,n^{-1/2+\rho}\le\sqrt C+1,\\
 \|\widehat S^{(ki)}(0)-\widehat S\|_\op
 &\le2\sqrt{LC}\,n^{-1/2+\rho}+Ln^{-1+2\rho}
 =o(\delta^2),\\
 \inf_{\xi\in\mathcal T^+}
    \dist\bigl(\xi,\spec(\widehat S^{(ki)}(0))\bigr)
 &\ge c\delta^2-\|\widehat S^{(ki)}(0)-\widehat S\|_\op
 \ge\tfrac c2\delta^2.
\end{align*}
Consequently, $\mathcal A_1\cap\{|Z_{ik}|\le n^\rho\}\subset\mathcal A_{ki}$.
For every fixed $D>0$, choose $q$ with $q\rho\ge D$ to obtain
\begin{align*}
 \PP(\mathcal A_{ki}^c)
 &\le\PP(\mathcal A_1^c)+\PP(|Z_{ik}|>n^\rho)
 \le C_Dn^{-D}+M_qn^{-q\rho}\le C_Dn^{-D}.
\end{align*}
The event $\mathcal A_{ki}$ is independent of
$\mathsf Z_{ki}$.

On $\mathcal A_{ki}$, applying \eqref{eq:recursion-entry-perturbation}
with $s=0$ gives
\begin{align*}
 \sup_{|t|\le n^{-1/2+\rho}}
 \|\widehat S^{(ki)}(t)-\widehat S^{(ki)}(0)\|_\op
 &\le2\sqrt L(\sqrt C+1)n^{-1/2+\rho}+Ln^{-1+2\rho}
 \le\tfrac c4\delta^2,\\
 \inf_{\substack{|t|\le n^{-1/2+\rho}\\\xi\in\mathcal T^+}}
 \dist\bigl(\xi,\spec(\widehat S^{(ki)}(t))\bigr)
 &\ge\tfrac c2\delta^2-
   \sup_{|t|\le n^{-1/2+\rho}}
     \|\widehat S^{(ki)}(t)-\widehat S^{(ki)}(0)\|_\op
 \ge\tfrac c4\delta^2,\\
 \sup_{|t|\le n^{-1/2+\rho}}\|\mathcal X^{(ki)}(t)\|_\op
 &\le\sqrt C+1+\sqrt L\,n^{-1/2+\rho}\le\sqrt C+2.
\end{align*}
The block formula \eqref{eq:linearization} therefore gives, on $\mathcal A_{ki}$,
\begin{align*}
 \sup_{\substack{|t|\le n^{-1/2+\rho}\\\xi\in\mathcal T^+}}
 \|G^{(ki)}(\xi;t)\|_\op
 &\le1+\sup_{\substack{|t|\le n^{-1/2+\rho}\\\xi\in\mathcal T^+}}
  \frac{1+\|\mathcal X^{(ki)}(t)\|_\op^2}
       {\dist\bigl(\xi,\spec(\widehat S^{(ki)}(t))\bigr)}
 \le C\delta^{-2}.
\end{align*}

With the embeddings in \eqref{eq:recursion-vector-embedding}, the matrix in
\eqref{eq:nonreal-resolvent-derivatives-identity} is
$\Delta_{ki}=\mathsf u_k\mathsf e_i^\top+\mathsf e_i\mathsf u_k^\top$,
with $\|\Delta_{ki}\|_\op\le\sqrt L$.
The derivative identity in the proof of Lemma~\ref{lem:nonreal-resolvent} gives
\begin{align}\label{eq:recursion-entry-derivative-identity}
 \partial_t^aG^{(ki)}(\xi;t)
 &=(-1)^aa!G^{(ki)}(\xi;t)
       \bigl(\Delta_{ki}G^{(ki)}(\xi;t)\bigr)^a,
      \text{ for $a \ge 0$.}
\end{align}
Leibniz's rule and $\|\Pi(\xi_2)\|_\op\le C$ now yield
\begin{align*}
 f_{ki}^{(6)}(t)
 &=\sum_{a=0}^6\binom6a
   \bigl(\mathsf e_i^\top\partial_t^aG^{(ki)}(\xi_1;t)w_1\bigr)
   \bigl(v_2^\top\partial_t^{6-a}
         (G^{(ki)}(\xi_2;t)-\Pi(\xi_2))w_2\bigr),\\
 |f_{ki}^{(6)}(t)|
 &\le C\bigl(1+\|G^{(ki)}(\xi_1;t)\|_\op
                +\|G^{(ki)}(\xi_2;t)\|_\op\bigr)^8,\\
 \sup_{|t|\le n^{-1/2+\rho}}|f_{ki}^{(6)}(t)|
 &\le C(1+\delta^{-2})^8\le C\delta^{-16}
       \qquad\text{on }\mathcal A_{ki}.
\end{align*}
For arbitrary real $t$ needed for the second term in \eqref{eq:recursion-cumulant-remainder}, \eqref{eq:raw-polynomial-bound} applies to
$G^{(ki)}(\xi_j;t)$ for every realization of the other entries. Thus
\begin{align*}
 \sup_{t\in\R}|f_{ki}^{(6)}(t)|
 &\le C(1+n^2\delta^{-1})^8\le Cn^{16}\delta^{-8},\\
 \E\sup_{|t|\le n^{-1/2+\rho}}|f_{ki}^{(6)}(t)|
 &\le C\delta^{-16}+Cn^{16}\delta^{-8}\PP(\mathcal A_{ki}^c)
 \le C\delta^{-16}+C_Dn^{16-D}\delta^{-8}.
\end{align*}
For every fixed $q>7$, we have
\begin{align*}
 \E|\mathsf Z_{ki}|^7&\le Cn^{-7/2}, \qquad 
 \E\bigl[|\mathsf Z_{ki}|^7
       \ind_{\{|\mathsf Z_{ki}|>n^{-1/2+\rho}\}}\bigr]
\le C_qn^{-7/2-\rho(q-7)}.
\end{align*}
Substituting these derivative and tail bounds into
\eqref{eq:recursion-cumulant-remainder} gives
\[
 |R_{ki}|\le Cn^{-7/2}\delta^{-16}
 +C_{D,q}n^{16}\delta^{-8}
       \bigl(n^{-7/2-D}+n^{-7/2-\rho(q-7)}\bigr)
 \le Cn^{-7/2}\delta^{-16},
\]
after choosing $D,q$ sufficiently large. The remainder term on the third line in
\eqref{eq:recursion-cumulant-expansion} therefore satisfies
\begin{align}
 \left|\sum_{k=1}^p\sum_{i=1}^n
   \bigl(v_1^\top D(\xi_1)^{-1}\mathsf u_k\bigr)R_{ki}\right|
 &\le Cn^{-7/2}\delta^{-16}
   \sum_{k=1}^p\sum_{i=1}^n|v_1^\top D(\xi_1)^{-1}\mathsf u_k|\notag\\
 &\le Cn^{-5/2}\delta^{-16}\sqrt p\,
            \|\Sigma^{1/2}D(\xi_1)^{-1}v_1\|_2
            \le Cn^{-2}\delta^{-16}.
 \label{eq:recursion-cumulant-remainder-sum}
\end{align}

We now evaluate the terms of orders $q=1,\ldots,5$ on the second line of
\eqref{eq:recursion-cumulant-expansion}. 
For deterministic $u,w\in\C^{p+n}$ and
$\xi\in\{\xi_1,\xi_2\}$, as in \eqref{eq:recursion-entry-derivative-identity}, we have
 \begin{align}\label{eq:recursion-entry-derivatives}
 \partial_{ki}(u^\top G(\xi)w)
 ={}&-(u^\top G(\xi)\mathsf u_k)(\mathsf e_i^\top G(\xi)w)
      -(u^\top G(\xi)\mathsf e_i)(\mathsf u_k^\top G(\xi)w),\\\notag
 \frac12\partial_{ki}^2(u^\top G(\xi)w)
 ={}&(u^\top G(\xi)\mathsf u_k)(\mathsf e_i^\top G(\xi)\mathsf u_k)
                       (\mathsf e_i^\top G(\xi)w)
                       +(u^\top G(\xi)\mathsf u_k)(\mathsf e_i^\top G(\xi)\mathsf e_i)
                       (\mathsf u_k^\top G(\xi)w)\\\notag
 &+(u^\top G(\xi)\mathsf e_i)(\mathsf u_k^\top G(\xi)\mathsf u_k)
                       (\mathsf e_i^\top G(\xi)w)
+(u^\top G(\xi)\mathsf e_i)(\mathsf u_k^\top G(\xi)\mathsf e_i)
                       (\mathsf u_k^\top G(\xi)w). 
 \end{align}

\proofstep{The second cumulant.}
Since $\chi_2=1$,
\eqref{eq:recursion-entry-derivatives} and
$\sum_k \mathsf u_k\mathsf u_k^\top=\Sigma$ give the exact cancellation
\begin{equation}\label{eq:recursion-second-cumulant-identity}
 \begin{aligned}
 &m(\xi_1)\E\bigl[
   (v_1^\top D(\xi_1)^{-1}\Sigma G_{11}(\xi_1)w_1)H_2\bigr]
 +\frac1n\sum_{k=1}^p\sum_{i=1}^n
   (v_1^\top D(\xi_1)^{-1}\mathsf u_k)\E f_{ki}'(\mathsf Z_{ki})\\
 &\quad=\E\bigl[(m(\xi_1)-\widehat m(\xi_1))
    (v_1^\top D(\xi_1)^{-1}\Sigma G_{11}(\xi_1)w_1)H_2\bigr]\\
 &\qquad-\frac1n\E\bigl[
    (v_1^\top D(\xi_1)^{-1}\Sigma G_{12}(\xi_1)
           G_{21}(\xi_1)w_1)H_2\bigr]\\
 &\qquad-\frac1n\E\bigl[
       (v_1^\top D(\xi_1)^{-1}\Sigma G_{11}(\xi_2)v_2)
       (w_1^\top G(\xi_1)\diag(0,I_n)G(\xi_2)w_2)\bigr]\\
 &\qquad-\frac1n\E\bigl[
       (v_1^\top D(\xi_1)^{-1}\Sigma G_{11}(\xi_2)w_2)
       (w_1^\top G(\xi_1)\diag(0,I_n)G(\xi_2)v_2)\bigr].
 \end{aligned}
\end{equation}
By H\"older's inequality and
\eqref{eq:recursion-raw-moments}, the first two
terms in \eqref{eq:recursion-second-cumulant-identity} are bounded by
\begin{align*}
      \big|\frac1n\E\bigl[
            (v_1^\top D(\xi_1)^{-1}\Sigma G_{12}(\xi_1)
                   G_{21}(\xi_1)w_1)H_2\bigr]\big|
                   &\le\frac Cn\bigl\|\|G(\xi_1)\|_\op\bigr\|_{L^8}^{2}
                   \|H_2\|_{L^2}
                   \le Cn^{-3/2+\epsilon_0}\delta^{-4},\\
      \big|\E\bigl[(m(\xi_1)-\widehat m(\xi_1))
 (v_1^\top D(\xi_1)^{-1}\Sigma G_{11}(\xi_1)w_1)H_2\bigr]\big|
 &\le C\|\widehat m(\xi_1)-m(\xi_1)\|_{L^4}
   \bigl\|\|G(\xi_1)\|_\op\bigr\|_{L^4}\|H_2\|_{L^2}\\
& \le Cn^{-3/2+2\epsilon_0}\delta^{-2}.
\end{align*}
For the last two terms in \eqref{eq:recursion-second-cumulant-identity}, we handle them uniformly by
taking $(u,w)=(v_2,w_2)$ and $(u,w)=(w_2,v_2)$, respectively.
Since $\Pi_{11}(\xi_2)=-D(\xi_2)^{-1}$ and
$M(\xi_1,\xi_2)=D(\xi_1)^{-1}\Sigma D(\xi_2)^{-1}$,
\eqref{eq:recursion-raw-moments} and
\eqref{eq:recursion-two-resolvent-moments} give 
\begin{align*}
 \bigl\|v_1^\top D(\xi_1)^{-1}\Sigma G_{11}(\xi_2)u
            +v_1^\top M(\xi_1,\xi_2)u\bigr\|_{L^2}
 &\le Cn^{-1/2+\epsilon_0},\\
 \bigl\|w_1^\top G(\xi_1)\diag(0,I_n)G(\xi_2)w
       -S(\xi_1,\xi_2)w_1^\top M(\xi_1,\xi_2)w\bigr\|_{L^2}
 &\le Cn^{-1/2+\epsilon_0}\delta^{-b}.
\end{align*}
Moreover, \eqref{eq:recursion-stability} and
\eqref{eq:recursion-raw-moments} imply
\begin{align*}
 |v_1^\top M(\xi_1,\xi_2)u|
 &\le\|M(\xi_1,\xi_2)\|_\op\le C,\\
 \bigl\|w_1^\top G(\xi_1)\diag(0,I_n)G(\xi_2)w\bigr\|_{L^2}
 &\le\bigl\|\|G(\xi_1)\|_\op\bigr\|_{L^4}
       \bigl\|\|G(\xi_2)\|_\op\bigr\|_{L^4}
 \le C\delta^{-4}.
\end{align*}
Applying the Cauchy--Schwarz inequality
therefore gives, for each of the two choices of $(u,w)=(v_2,w_2)$ and $(u,w)=(w_2,v_2)$,
\begin{align}
 &\Biggl|-\frac1n\E\bigl[
   (v_1^\top D(\xi_1)^{-1}\Sigma G_{11}(\xi_2)u)
   (w_1^\top G(\xi_1)\diag(0,I_n)G(\xi_2)w)\bigr]-\frac{S(\xi_1,\xi_2)}n
   (v_1^\top M(\xi_1,\xi_2)u)
   (w_1^\top M(\xi_1,\xi_2)w)\Biggr|\notag\\
 &\quad\le\frac{C\delta^{-4}}n
   \bigl\|v_1^\top D(\xi_1)^{-1}\Sigma G_{11}(\xi_2)u
          +v_1^\top M(\xi_1,\xi_2)u\bigr\|_{L^2}\notag\\
 &\qquad+\frac Cn
   \bigl\|w_1^\top G(\xi_1)\diag(0,I_n)G(\xi_2)w
          -S(\xi_1,\xi_2)w_1^\top M(\xi_1,\xi_2)w\bigr\|_{L^2}\notag\\
 &\quad\le Cn^{-3/2+\epsilon_0}
                  (\delta^{-4}+\delta^{-b}).
 \label{eq:recursion-second-cumulant-products}
\end{align}
The sum of the above leading terms for $(u,w)=(v_2,w_2)$ and $(u,w)=(w_2,v_2)$ gives $\Xi^{\rm g}$ in \eqref{eq:leading-covariance}:
\begin{align*}
      \Xi^{\rm g}(\xi_1,\xi_2)&=\frac{S(\xi_1,\xi_2)}n\bigl[
    (v_1^\top M(\xi_1,\xi_2)v_2)
    (w_1^\top M(\xi_1,\xi_2)w_2)
    +(v_1^\top M(\xi_1,\xi_2)w_2)
    (w_1^\top M(\xi_1,\xi_2)v_2)\bigr].
\end{align*}
Combining \eqref{eq:recursion-second-cumulant-products} with the preceding
bounds on the first two terms in \eqref{eq:recursion-second-cumulant-identity}
yields
\begin{equation}\label{eq:recursion-second-cumulant}
 \begin{aligned}
 &m(\xi_1)\E\bigl[
   (v_1^\top D(\xi_1)^{-1}\Sigma G_{11}(\xi_1)w_1)H_2\bigr]
 +\frac1n\sum_{k=1}^p\sum_{i=1}^n
   (v_1^\top D(\xi_1)^{-1}\mathsf u_k)\E f_{ki}'(\mathsf Z_{ki})\\
 &\qquad=\Xi^{\rm g}(\xi_1,\xi_2)
       +O(n^{-3/2+2\epsilon_0}\delta^{-C}).
 \end{aligned}
\end{equation}

For the remaining cumulants we record the bounds needed when summing
entry derivatives. In \eqref{eq:recursion-entrywise-moments}, $u,w$ range over
$v_1,w_1,v_2,w_2,\mathsf u_1,\ldots,\mathsf u_p,\mathsf e_1,\ldots,\mathsf e_n$.
For fixed $r\ge2$, choose a fixed $R\ge r$ with $2/R\le\epsilon_0$.
Since there are $O(n^2)$ pairs, \eqref{eq:recursion-raw-moments} gives
\begin{equation}\label{eq:recursion-entrywise-moments}
 \begin{aligned}
 \Bigl\|\max_{j,u,w}
       |u^\top(G(\xi_j)-\Pi(\xi_j))w|\Bigr\|_{L^r}
 &\le\biggl(\sum_{j,u,w}
       \E|u^\top(G(\xi_j)-\Pi(\xi_j))w|^R\biggr)^{1/R}\\
 &\le Cn^{2/R}n^{-1/2+\epsilon_0}
 \le Cn^{-1/2+2\epsilon_0}.
 \end{aligned}
\end{equation}
For each deterministic feature vector $u\in\C^p$ of bounded norm,
$\mathsf e_i^\top\Pi(\xi_j)u=0$ by \eqref{eq:linearization} and
\eqref{eq:recursion-vector-embedding}. Hence
\begin{equation}\label{eq:recursion-signed-sample-sum}
 \Bigl\|\sum_{i=1}^n\mathsf e_i^\top G(\xi_j)u\Bigr\|_{L^r}
 =\Bigl\|\Bigl(\sum_{i=1}^n\mathsf e_i\Bigr)^\top
               (G(\xi_j)-\Pi(\xi_j))u\Bigr\|_{L^r}
 \le C\sqrt n\,n^{-1/2+\epsilon_0}=Cn^{\epsilon_0}.
\end{equation}
For bounded feature
vectors $u,w$ and $j,\ell\in\{1,2\}$,
\begin{align}
 \notag \sum_{k=1}^p|v_1^\top D(\xi_1)^{-1}\mathsf u_k|\,|\mathsf u_k^\top G(\xi_j)u|
 &\le\|\Sigma^{1/2}D(\xi_1)^{-1}v_1\|_2
       \|\Sigma^{1/2}G_{11}(\xi_j)u\|_2
 \le C\|G(\xi_j)\|_\op,\\ \notag 
 \sum_{k=1}^p|v_1^\top D(\xi_1)^{-1}\mathsf u_k|\,
       |u^\top G(\xi_j)\mathsf u_k|\,|\mathsf u_k^\top G(\xi_\ell)w|
 &\le C\|G(\xi_j)\|_\op\|G(\xi_\ell)\|_\op,\\
 \sum_{k=1}^p|v_1^\top D(\xi_1)^{-1}\mathsf u_k|\,|\mathsf u_k^\top G(\xi_j)\mathsf u_k|
 &\le C\sqrt p\,\|G(\xi_j)\|_\op. 
 \label{eq:recursion-weighted-feature-sums}
 \end{align}
Their $L^r$ norms are therefore bounded by $C\delta^{-2}$,
$C\delta^{-4}$, and $C\sqrt n\,\delta^{-2}$, respectively.

\proofstep{The third cumulant.} We now bound the third cumulant term ($q=2$) on the second line of \eqref{eq:recursion-cumulant-expansion}.
 First,  $\|D(\xi_1)^{-1}\|_\op\le C$ gives
\begin{align*}
 \sum_{k=1}^p\sum_{i=1}^n
   |v_1^\top D(\xi_1)^{-1}\mathsf u_k|
 &=n\sum_{k=1}^p|v_1^\top D(\xi_1)^{-1}\mathsf u_k|\le n\sqrt p\,
       \|\Sigma^{1/2}D(\xi_1)^{-1}v_1\|_2
 \le Cn^{3/2}.
\end{align*}
For fixed $k,i$, define the following terms, suppressing their
dependence on $k,i$ for simplicity:
\begin{align*}
 T_1&=\bigl(\mathsf e_i^\top G(\xi_1)\mathsf u_k\bigr)^2
   (\mathsf e_i^\top G(\xi_1)w_1)H_2,\\
 T_2(u,w)&=(\mathsf e_i^\top G(\xi_1)\mathsf u_k)
   (\mathsf e_i^\top G(\xi_1)w_1)
   (\mathsf u_k^\top G(\xi_2)u)
   (\mathsf e_i^\top G(\xi_2)w),\\
 T_3(u,w)&=(\mathsf e_i^\top G(\xi_1)w_1)
   (\mathsf e_i^\top G(\xi_2)\mathsf u_k)
   (\mathsf u_k^\top G(\xi_2)u)
   (\mathsf e_i^\top G(\xi_2)w),\\
 T_4&=(\mathsf e_i^\top G(\xi_1)w_1)
   (\mathsf e_i^\top G(\xi_2)v_2)
   (\mathsf u_k^\top G(\xi_2)\mathsf u_k)
   (\mathsf e_i^\top G(\xi_2)w_2),\\
 T_5&=(\mathsf e_i^\top G(\xi_1)\mathsf u_k)
   (\mathsf e_i^\top G(\xi_1)\mathsf e_i)
   (\mathsf u_k^\top G(\xi_1)w_1)H_2,\\
 T_6&=(\mathsf e_i^\top G(\xi_1)\mathsf e_i)
   (\mathsf u_k^\top G(\xi_1)\mathsf u_k)
   (\mathsf e_i^\top G(\xi_1)w_1)H_2,\\
 T_7(u,w)&=(\mathsf e_i^\top G(\xi_1)\mathsf e_i)
   (\mathsf u_k^\top G(\xi_1)w_1)
   (\mathsf u_k^\top G(\xi_2)u)
   (\mathsf e_i^\top G(\xi_2)w),\\
 T_8&=(\mathsf e_i^\top G(\xi_2)\mathsf e_i)
   (\mathsf e_i^\top G(\xi_1)w_1)
   (\mathsf u_k^\top G(\xi_2)v_2)
   (\mathsf u_k^\top G(\xi_2)w_2).
\end{align*}
Equation~\eqref{eq:recursion-entry-derivatives} and
$G(\xi_j)^\top=G(\xi_j)$ give
\begin{align}
 \frac12\partial_{ki}^2(\mathsf e_i^\top G(\xi_1)w_1)H_2
 &=T_1+2T_5+T_6,\notag\\
 \partial_{ki}(\mathsf e_i^\top G(\xi_1)w_1)\partial_{ki}H_2
 &=T_2(v_2,w_2)+T_2(w_2,v_2)+T_7(v_2,w_2)+T_7(w_2,v_2),\notag\\
 \frac12(\mathsf e_i^\top G(\xi_1)w_1)\partial_{ki}^2H_2
 &=T_3(v_2,w_2)+T_3(w_2,v_2)+T_4+T_8.
 \label{eq:recursion-third-cumulant-components}
\end{align}
Thus Leibniz's rule yields the exact expansion
\begin{align}
 \notag f_{ki}''(\mathsf Z_{ki})
 &=\partial_{ki}^2(\mathsf e_i^\top G(\xi_1)w_1)H_2
 +2\partial_{ki}(\mathsf e_i^\top G(\xi_1)w_1)\partial_{ki}H_2
 +(\mathsf e_i^\top G(\xi_1)w_1)\partial_{ki}^2H_2\\
 &=2\bigl(T_1+T_2(v_2,w_2)+T_2(w_2,v_2)
          +T_3(v_2,w_2)+T_3(w_2,v_2)+T_4\bigr)\notag\\
 &\quad+4T_5+2\bigl(T_6+T_7(v_2,w_2)+T_7(w_2,v_2)+T_8\bigr).
 \label{eq:recursion-third-cumulant-products}
\end{align}
For $T_6, T_7, T_8$, as $\mathsf e_i^\top G(\xi_j)\mathsf e_i \approx \mathsf e_i^\top\Pi(\xi_j)\mathsf e_i=m(\xi_j)$, we introduce the following terms
\begin{align}
 \tilde{T}_6
 &=m(\xi_1)(\mathsf u_k^\top G(\xi_1)\mathsf u_k)
   (\mathsf e_i^\top G(\xi_1)w_1)H_2,\notag\\
 \tilde{T}_7(u,w)
 &=m(\xi_1)(\mathsf u_k^\top G(\xi_1)w_1)
   (\mathsf u_k^\top G(\xi_2)u)
   (\mathsf e_i^\top G(\xi_2)w),\notag\\
 \tilde{T}_8
 &=m(\xi_2)(\mathsf e_i^\top G(\xi_1)w_1)
   (\mathsf u_k^\top G(\xi_2)v_2)
   (\mathsf u_k^\top G(\xi_2)w_2).
 \label{eq:recursion-third-cumulant-diagonal-replacements}
\end{align}
The third cumulant term ($q=2$) on the second line of \eqref{eq:recursion-cumulant-expansion} becomes
\begin{align}
 \notag     &\sum_{k=1}^p\sum_{i=1}^n
        (v_1^\top D(\xi_1)^{-1}\mathsf u_k)\E f_{ki}''(\mathsf Z_{ki}) \\
        &= 2\sum_{k=1}^p\sum_{i=1}^n
      (v_1^\top D(\xi_1)^{-1}\mathsf u_k)
      \E\bigl[\tilde{T}_6+\tilde{T}_7(v_2,w_2)
                       +\tilde{T}_7(w_2,v_2)+\tilde{T}_8\bigr]
                       + T_{\text{omit}},\label{eq:recursion-third-cumulant-sums}
      \end{align}
where all the other terms in \eqref{eq:recursion-third-cumulant-products} are in the omitted term $T_{\text{omit}}$. We first bound $T_{\text{omit}}$.
Since $\Pi_{12}=\Pi_{21}=0$ and $|H_2|\le C+\|G(\xi_2)\|_\op$,
the contribution of the first bracket in \eqref{eq:recursion-third-cumulant-products} to the left-hand side of \eqref{eq:recursion-third-cumulant-sums} is bounded as follows:
\begin{align}
 &2\sum_{k=1}^p\sum_{i=1}^n
   |v_1^\top D(\xi_1)^{-1}\mathsf u_k|\,
   \E\bigl[|T_1|+|T_2(v_2,w_2)|+|T_2(w_2,v_2)|
            +|T_3(v_2,w_2)|+|T_3(w_2,v_2)|+|T_4|\bigr]\notag\\
 &\quad\le C\sum_{k=1}^p\sum_{i=1}^n
   |v_1^\top D(\xi_1)^{-1}\mathsf u_k|\,
   \E\biggl[
    \Bigl(\max_{j,u,w}|u^\top(G(\xi_j)-\Pi(\xi_j))w|\Bigr)^3
    \bigl(1+\|G(\xi_1)\|_\op+\|G(\xi_2)\|_\op\bigr)\biggr]\notag\\
 &\quad\le Cn^{3/2}
   \Bigl\|\max_{j,u,w}|u^\top(G(\xi_j)-\Pi(\xi_j))w|\Bigr\|_{L^6}^3
   \biggl(1+\sum_{j=1}^2
      \bigl\|\|G(\xi_j)\|_\op\bigr\|_{L^2}\biggr)\notag\\
 &\quad\le Cn^{3/2}(n^{-1/2+2\epsilon_0})^3\delta^{-2}
 =Cn^{6\epsilon_0}\delta^{-2},
 \label{eq:recursion-third-cumulant-mixed-bound}
\end{align}
by H\"older's inequality, \eqref{eq:recursion-entrywise-moments}, and
\eqref{eq:recursion-raw-moments}.

For the contribution of $4T_5$ in \eqref{eq:recursion-third-cumulant-products}, the first inequality in
\eqref{eq:recursion-weighted-feature-sums} gives
\begin{align}
 &4\sum_{k=1}^p\sum_{i=1}^n
  |v_1^\top D(\xi_1)^{-1}\mathsf u_k|\,\E|T_5|\le Cn\,\E\biggl[
   \max_{k,i}|\mathsf e_i^\top G(\xi_1)\mathsf u_k|
   \|G(\xi_1)\|_\op^2|H_2|\biggr]\notag\\
 &\quad\le Cn\,
   \Bigl\|\max_{k,i}|\mathsf e_i^\top G(\xi_1)\mathsf u_k|\Bigr\|_{L^4}
   \bigl\|\|G(\xi_1)\|_\op\bigr\|_{L^8}^2\|H_2\|_{L^2}\notag\\
 &\quad\le Cn\,n^{-1/2+2\epsilon_0}\delta^{-4}n^{-1/2+\epsilon_0}
 =Cn^{3\epsilon_0}\delta^{-4},
 \label{eq:recursion-third-cumulant-h2-bound}
\end{align}
where the last line again uses \eqref{eq:recursion-entrywise-moments}
and \eqref{eq:recursion-raw-moments}.

It remains to bound the four differences
$T_6-\tilde{T}_6$,
$T_7(v_2,w_2)-\tilde{T}_7(v_2,w_2)$,
$T_7(w_2,v_2)-\tilde{T}_7(w_2,v_2)$, and $T_8-\tilde{T}_8$.
The third inequality in
\eqref{eq:recursion-weighted-feature-sums} yields
\begin{align*}
 &2\sum_{k=1}^p\sum_{i=1}^n
   |v_1^\top D(\xi_1)^{-1}\mathsf u_k|\,
   \E|T_6-\tilde{T}_6|
   \le Cn^{3/2}\,
   \E\biggl[
    \Bigl(\max_{j,u,w}|u^\top(G(\xi_j)-\Pi(\xi_j))w|\Bigr)^2
    \|G(\xi_1)\|_\op|H_2|\biggr]\\
 &\quad\le Cn^{3/2}
   \Bigl\|\max_{j,u,w}|u^\top(G(\xi_j)-\Pi(\xi_j))w|\Bigr\|_{L^8}^2
   \bigl\|\|G(\xi_1)\|_\op\bigr\|_{L^4}\|H_2\|_{L^2}\\
 &\quad\le Cn^{3/2}(n^{-1/2+2\epsilon_0})^2
                  \delta^{-2}n^{-1/2+\epsilon_0}
 =Cn^{5\epsilon_0}\delta^{-2}.
\end{align*}
The second inequality in \eqref{eq:recursion-weighted-feature-sums} gives
\begin{align}
 &2\sum_{k=1}^p\sum_{i=1}^n
   |v_1^\top D(\xi_1)^{-1}\mathsf u_k|\,
   \E\bigl[|T_7(v_2,w_2)-\tilde{T}_7(v_2,w_2)|
   +|T_7(w_2,v_2)-\tilde{T}_7(w_2,v_2)|
       +|T_8-\tilde{T}_8|\bigr]\notag\\
 &\quad\le Cn\,\E\biggl[
    \Bigl(\max_{j,u,w}|u^\top(G(\xi_j)-\Pi(\xi_j))w|\Bigr)^2
    \bigl(\|G(\xi_1)\|_\op+\|G(\xi_2)\|_\op\bigr)^2\biggr]\notag\\
 &\quad\le Cn\,
   \Bigl\|\max_{j,u,w}|u^\top(G(\xi_j)-\Pi(\xi_j))w|\Bigr\|_{L^4}^2
   \biggl(\sum_{j=1}^2
     \bigl\|\|G(\xi_j)\|_\op\bigr\|_{L^4}\biggr)^2\notag\\
 &\quad\le Cn(n^{-1/2+2\epsilon_0})^2\delta^{-4}
 =Cn^{4\epsilon_0}\delta^{-4}.
 \label{eq:recursion-third-cumulant-replacement-bound}
\end{align}
The triangle inequality and the four bounds above give
\begin{align}\label{eq:recursion-third-cumulant-omit-bound}
 |T_{\text{omit}}|
 &\le C\bigl(n^{6\epsilon_0}\delta^{-2}
       +n^{3\epsilon_0}\delta^{-4}
       +n^{5\epsilon_0}\delta^{-2}
       +n^{4\epsilon_0}\delta^{-4}\bigr)
       \le Cn^{6\epsilon_0}\delta^{-4}.
\end{align}
We now bound the retained terms in \eqref{eq:recursion-third-cumulant-sums}, using $|m(\xi_j)|=|\mu_j|^{-1}\le C$.
For $\tilde{T}_6$, sum over $i$ before taking absolute values.
The third inequality in \eqref{eq:recursion-weighted-feature-sums},
H\"older's inequality, \eqref{eq:recursion-signed-sample-sum}, and
\eqref{eq:recursion-raw-moments} give
\begin{align*}
 &\Biggl|2\sum_{k=1}^p\sum_{i=1}^n
   (v_1^\top D(\xi_1)^{-1}\mathsf u_k)\E\tilde{T}_6\Biggr|
   =2|m(\xi_1)|\Biggl|\E\biggl[
   \biggl(\sum_{k=1}^p(v_1^\top D(\xi_1)^{-1}\mathsf u_k)
                 (\mathsf u_k^\top G(\xi_1)\mathsf u_k)\biggr)
   \biggl(\sum_{i=1}^n\mathsf e_i^\top G(\xi_1)w_1\biggr)H_2
   \biggr]\Biggr|\\
 &\le C\sqrt n\,
   \bigl\|\|G(\xi_1)\|_\op\bigr\|_{L^4}
   \Bigl\|\sum_{i=1}^n\mathsf e_i^\top G(\xi_1)w_1\Bigr\|_{L^4}
   \|H_2\|_{L^2}
   \le C\sqrt n\,\delta^{-2}n^{\epsilon_0}n^{-1/2+\epsilon_0}
 =Cn^{2\epsilon_0}\delta^{-2}.
\end{align*}
For each of the two ordered choices $(u,w)=(v_2,w_2)$ and
$(u,w)=(w_2,v_2)$, the second inequality in
\eqref{eq:recursion-weighted-feature-sums} gives
\begin{align*}
 &\Biggl|2\sum_{k=1}^p\sum_{i=1}^n
   (v_1^\top D(\xi_1)^{-1}\mathsf u_k)\E\tilde{T}_7(u,w)\Biggr|\\
 &=2|m(\xi_1)|\Biggl|\E\biggl[
   \biggl(\sum_{k=1}^p(v_1^\top D(\xi_1)^{-1}\mathsf u_k)
       (\mathsf u_k^\top G(\xi_1)w_1)
       (\mathsf u_k^\top G(\xi_2)u)\biggr)
   \biggl(\sum_{i=1}^n\mathsf e_i^\top G(\xi_2)w\biggr)
   \biggr]\Biggr|\\
 &\le C\bigl\|\|G(\xi_1)\|_\op\bigr\|_{L^4}
       \bigl\|\|G(\xi_2)\|_\op\bigr\|_{L^4}
       \Bigl\|\sum_{i=1}^n\mathsf e_i^\top G(\xi_2)w\Bigr\|_{L^2}
       \le C\delta^{-2}\delta^{-2}n^{\epsilon_0}
 =Cn^{\epsilon_0}\delta^{-4},
\end{align*}
where the last two inequalities also use H\"older's inequality,
\eqref{eq:recursion-raw-moments}, and
\eqref{eq:recursion-signed-sample-sum}.
The second inequality in \eqref{eq:recursion-weighted-feature-sums}
with $j=\ell=2$ similarly gives
\begin{align*}
 &\Biggl|2\sum_{k=1}^p\sum_{i=1}^n
   (v_1^\top D(\xi_1)^{-1}\mathsf u_k)\E\tilde{T}_8\Biggr|\\
 &=2|m(\xi_2)|\Biggl|\E\biggl[
   \biggl(\sum_{k=1}^p(v_1^\top D(\xi_1)^{-1}\mathsf u_k)
       (\mathsf u_k^\top G(\xi_2)v_2)
       (\mathsf u_k^\top G(\xi_2)w_2)\biggr)
   \biggl(\sum_{i=1}^n\mathsf e_i^\top G(\xi_1)w_1\biggr)
   \biggr]\Biggr|\\
 &\le C\bigl\|\|G(\xi_2)\|_\op\bigr\|_{L^4}^2
       \Bigl\|\sum_{i=1}^n\mathsf e_i^\top G(\xi_1)w_1\Bigr\|_{L^2}
       \le C\delta^{-4}n^{\epsilon_0}
 =Cn^{\epsilon_0}\delta^{-4}.
\end{align*}
Adding these four contributions yields
\begin{align*}
 &\Biggl|2\sum_{k=1}^p\sum_{i=1}^n
   (v_1^\top D(\xi_1)^{-1}\mathsf u_k)
   \E\bigl[\tilde{T}_6+\tilde{T}_7(v_2,w_2)
                    +\tilde{T}_7(w_2,v_2)+\tilde{T}_8\bigr]\Biggr|
                    \le C\bigl(n^{2\epsilon_0}\delta^{-2}
                  +n^{\epsilon_0}\delta^{-4}\bigr)
 \le Cn^{2\epsilon_0}\delta^{-4}.
\end{align*}
Combining the above bound with
\eqref{eq:recursion-third-cumulant-sums} and
\eqref{eq:recursion-third-cumulant-omit-bound} gives the bound of the third cumulant term on the second line of \eqref{eq:recursion-cumulant-expansion}
\begin{align}
 \left|\frac{\chi_3}{2n^{3/2}}
       \sum_{k=1}^p\sum_{i=1}^n
       (v_1^\top D(\xi_1)^{-1}\mathsf u_k)\E f_{ki}''(\mathsf Z_{ki})\right|
 &\le Cn^{-3/2+6\epsilon_0}\delta^{-C}.
 \label{eq:recursion-third-cumulant}
\end{align}

\proofstep{The fourth cumulant.} We now bound the fourth cumulant term ($q=3$) on the second line of \eqref{eq:recursion-cumulant-expansion}. The procedure is similar to that in the third cumulant case and thus we only give the main steps and omit the details.
We have the expansion
\begin{equation}\label{eq:recursion-fourth-cumulant-expansion}
\begin{aligned}
 \frac{f_{ki}^{(3)}(\mathsf Z_{ki})}{6}
 ={}&\frac16\partial_{ki}^3(\mathsf e_i^\top G(\xi_1)w_1)H_2
  +\frac12\partial_{ki}^2(\mathsf e_i^\top G(\xi_1)w_1)\partial_{ki}H_2\\
  & +\frac12\partial_{ki}(\mathsf e_i^\top G(\xi_1)w_1)\partial_{ki}^2H_2
  +\frac16(\mathsf e_i^\top G(\xi_1)w_1)\partial_{ki}^3H_2.
\end{aligned}
\end{equation}
We apply $\partial_{ki}^aG=(-1)^aa!G(\Delta_{ki}G)^a$ again to the above expansion.
Each $\Delta_{ki}=\mathsf u_k\mathsf e_i^\top+
\mathsf e_i\mathsf u_k^\top$ switches between the feature (top $p$) and sample (last $n$)
blocks. The endpoints of a scalar matrix product are its two outer
vectors. If they belong to the same block, an odd number of
$\Delta_{ki}$ factors requires a feature--sample factor (i.e., the factor involves the $(1,2)$ or $(2,1)$ block of $G$); if they belong
to different blocks, an even number requires one.
These cases already occur in \eqref{eq:recursion-third-cumulant-components}:
$\partial_{ki}H_2$ has two feature endpoints and one $\Delta_{ki}$,
and contributes the feature--sample factor
$\mathsf e_i^\top G(\xi_2)w$ in $T_2(u,w)$ and $T_7(u,w)$, for
$(u,w)=(v_2,w_2)$ and $(u,w)=(w_2,v_2)$.
The factor $\partial_{ki}^2(\mathsf e_i^\top G(\xi_1)w_1)$ has
sample--feature endpoints and two $\Delta_{ki}$ factors; its contribution
is $2(T_1+2T_5+T_6)$ after multiplication by $H_2$, and $T_1,T_5,T_6$
contain three, one, and one feature--sample factors, respectively.

We count products before collecting equal scalar products, retaining
their multiplicities. By \eqref{eq:recursion-entry-derivative-identity},
the term with $a$ derivatives on the first factor and $3-a$ on $H_2$
has $2^a2^{3-a}=8$ products in
\eqref{eq:recursion-fourth-cumulant-expansion}, each with coefficient
$-1$. The first group retains $H_2$ without expanding it. Expanding the
two rank-one summands of each $\Delta_{ki}$ gives the following counts; each row corresponds to one term of \eqref{eq:recursion-fourth-cumulant-expansion};
columns $0,2,4$ specify the number of feature--sample factors in each product:
\begin{equation}\label{eq:recursion-fourth-cumulant-counts}
 \begin{array}{c|c@{\qquad}c@{\qquad}c|c}
 \text{derivative orders}&0&2&4&\text{total products}\\\hline
 (3,0)&1&6&1&8\\
 (2,1)&0&6&2&8\\
 (1,2)&1&4&3&8\\
 (0,3)&0&4&4&8
 \end{array}
\end{equation}
Thus there are $32$ products: eight contain $H_2$, and the other $24$
consist of $14$ products with two feature--sample factors, nine with
four, and exactly one with none. The last product comes from the
derivative orders $(1,2)$ and is
\begin{equation}\label{eq:recursion-fourth-cumulant-leading-product}
 -(\mathsf e_i^\top G(\xi_1)\mathsf e_i)
  (\mathsf e_i^\top G(\xi_2)\mathsf e_i)
  (\mathsf u_k^\top G(\xi_1)w_1)
  (v_2^\top G(\xi_2)\mathsf u_k)(\mathsf u_k^\top G(\xi_2)w_2).
\end{equation}

For the eight products containing $H_2$, the derivative formula
\eqref{eq:recursion-entry-derivative-identity}, the first inequality in
\eqref{eq:recursion-weighted-feature-sums}, and the Cauchy--Schwarz inequality give
\begin{align*}
 &\sum_{k=1}^p\sum_{i=1}^n|v_1^\top D(\xi_1)^{-1}\mathsf u_k|
       |\partial_{ki}^3(\mathsf e_i^\top G(\xi_1)w_1)|\\
 &\quad\le C\|G(\xi_1)\|_\op^3
       \sum_{k=1}^p\sum_{i=1}^n|v_1^\top D(\xi_1)^{-1}\mathsf u_k|
       \bigl(|\mathsf u_k^\top G(\xi_1)w_1|
             +|\mathsf e_i^\top G(\xi_1)w_1|\bigr)
 \le Cn\|G(\xi_1)\|_\op^4.
\end{align*}
Applying H\"older's inequality as in
\eqref{eq:recursion-third-cumulant-h2-bound}, now with four resolvent
norm factors, and then \eqref{eq:recursion-raw-moments}, yields
\begin{align}
 &\frac{|\chi_4|}{6n^2}\sum_{k=1}^p\sum_{i=1}^n
   |v_1^\top D(\xi_1)^{-1}\mathsf u_k|\,
   \E\bigl[|\partial_{ki}^3(\mathsf e_i^\top G(\xi_1)w_1)|\,|H_2|\bigr]
   \notag\\
 &\quad\le\frac Cn\E\bigl[\|G(\xi_1)\|_\op^4|H_2|\bigr]
 \le\frac Cn\bigl\|\|G(\xi_1)\|_\op\bigr\|_{L^8}^4\|H_2\|_{L^2}
 \le Cn^{-3/2+\epsilon_0}\delta^{-8}.
 \label{eq:recursion-fourth-cumulant-h2-bound}
\end{align}

For the $23$ products with at least two feature--sample factors, use
the same entrywise-moment argument as in
\eqref{eq:recursion-third-cumulant-mixed-bound}. Here two factors are
bounded by the maximum in \eqref{eq:recursion-entrywise-moments}, and
the remaining three by resolvent norms. Their total absolute
contribution is at most
\begin{align}
 &\frac{23C}{n^2}\sum_{k=1}^p\sum_{i=1}^n
   |v_1^\top D(\xi_1)^{-1}\mathsf u_k|\,
   \E\biggl[
    \Bigl(\max_{j,u,w}|u^\top(G(\xi_j)-\Pi(\xi_j))w|\Bigr)^2
    \bigl(1+\|G(\xi_1)\|_\op+\|G(\xi_2)\|_\op\bigr)^3\biggr]
    \notag\\
 &\quad\le Cn^{-1/2}
   \Bigl\|\max_{j,u,w}|u^\top(G(\xi_j)-\Pi(\xi_j))w|\Bigr\|_{L^4}^2
   \bigl\|1+\|G(\xi_1)\|_\op+\|G(\xi_2)\|_\op\bigr\|_{L^6}^3
   \notag\\
 &\quad\le Cn^{-1/2}(n^{-1/2+2\epsilon_0})^2\delta^{-6}
 =Cn^{-3/2+4\epsilon_0}\delta^{-6},
 \label{eq:recursion-fourth-cumulant-mixed-bound}
\end{align}
by H\"older's inequality and \eqref{eq:recursion-raw-moments}.

Finally, replace the five resolvent factors in
\eqref{eq:recursion-fourth-cumulant-leading-product} by their $\Pi$
counterparts, one at a time. The telescoping difference has exactly
five terms. Each contains one $G-\Pi$ factor and at least two unchanged
feature--feature factors. As in
\eqref{eq:recursion-third-cumulant-replacement-bound}, bound the
changed factor by \eqref{eq:recursion-entrywise-moments} and sum the
two feature factors using the second inequality in
\eqref{eq:recursion-weighted-feature-sums}.
The same Cauchy--Schwarz inequality applies when a $G(\xi_j)$ in that
inequality is replaced by $\Pi(\xi_j)$, since
$\|\Pi(\xi_j)\|_\op\le C$. Thus the total replacement error is bounded by
\begin{align}
 &\frac{|\chi_4|}{n^2}\sum_{k=1}^p\sum_{i=1}^n
  |v_1^\top D(\xi_1)^{-1}\mathsf u_k|\,
  \E\Bigl|
   (\mathsf e_i^\top G(\xi_1)\mathsf e_i)
   (\mathsf e_i^\top G(\xi_2)\mathsf e_i)(\mathsf u_k^\top G(\xi_1)w_1)
   (v_2^\top G(\xi_2)\mathsf u_k)(\mathsf u_k^\top G(\xi_2)w_2)
   \notag\\
 &\qquad \qquad \qquad \qquad \qquad \qquad \qquad \qquad -m(\xi_1)m(\xi_2)
   (\mathsf u_k^\top\Pi(\xi_1)w_1)
   (v_2^\top\Pi(\xi_2)\mathsf u_k)
   (\mathsf u_k^\top\Pi(\xi_2)w_2)\Bigr|\notag\\
 &\quad\le\frac{5C}{n}\E\biggl[
    \max_{j,u,w}|u^\top(G(\xi_j)-\Pi(\xi_j))w|
    \bigl(1+\|G(\xi_1)\|_\op+\|G(\xi_2)\|_\op\bigr)^4\biggr]
    \notag\\
 &\quad\le\frac Cn
   \Bigl\|\max_{j,u,w}|u^\top(G(\xi_j)-\Pi(\xi_j))w|\Bigr\|_{L^2}
   \bigl\|1+\|G(\xi_1)\|_\op+\|G(\xi_2)\|_\op\bigr\|_{L^8}^4
   \notag\\
 &\quad\le Cn^{-1}n^{-1/2+2\epsilon_0}\delta^{-8}
 =Cn^{-3/2+2\epsilon_0}\delta^{-8}.
 \label{eq:recursion-fourth-cumulant-replacement-bound}
\end{align}
Combining \eqref{eq:recursion-fourth-cumulant-h2-bound},
\eqref{eq:recursion-fourth-cumulant-mixed-bound}, and
\eqref{eq:recursion-fourth-cumulant-replacement-bound} bounds all errors by
$Cn^{-3/2+4\epsilon_0}\delta^{-C}$.
As $\Pi_{11}(\xi_j)=-D(\xi_j)^{-1}$, 
\begin{equation}\label{eq:recursion-fourth-cumulant}
 \begin{aligned}
 &\frac{\chi_4}{6n^2}\sum_{k=1}^p\sum_{i=1}^n
       (v_1^\top D(\xi_1)^{-1}\mathsf u_k)\E f_{ki}^{(3)}(\mathsf Z_{ki})\\
 &\quad=\frac{\chi_4m(\xi_1)m(\xi_2)}n
       \sum_{k=1}^p\prod_{j=1}^2
       (v_j^\top D(\xi_j)^{-1}\mathsf u_k)(w_j^\top D(\xi_j)^{-1}\mathsf u_k)
       +O(n^{-3/2+4\epsilon_0}\delta^{-C}).
 \end{aligned}
\end{equation}

\proofstep{The fifth and sixth cumulants.}
For $q=4,5$, we use Leibniz's rule and
$\partial_{ki}^aG=(-1)^aa!G(\Delta_{ki}G)^a$ to express
$f_{ki}^{(q)}$ as a fixed number of products. Keeping the last entry
adjacent to $w_1$ in each product and bounding the remaining entries
by operator norms gives
\begin{align*}
 |f_{ki}^{(q)}(\mathsf Z_{ki})|
 &\le C_q\bigl(1+\|G(\xi_1)\|_\op+\|G(\xi_2)\|_\op\bigr)^{q+1}
       \bigl(|\mathsf u_k^\top G(\xi_1)w_1|
             +|\mathsf e_i^\top G(\xi_1)w_1|\bigr).
\end{align*}
Moreover, the Cauchy--Schwarz inequality gives
\begin{align*}
 &\sum_{k=1}^p\sum_{i=1}^n|v_1^\top D(\xi_1)^{-1}\mathsf u_k|
       \bigl(|\mathsf u_k^\top G(\xi_1)w_1|
             +|\mathsf e_i^\top G(\xi_1)w_1|\bigr)
      \le Cn\|G(\xi_1)\|_\op.
\end{align*}
Consequently, \eqref{eq:recursion-raw-moments} bounds the terms with $q=4,5$
in \eqref{eq:recursion-cumulant-expansion} as follows:
\begin{align}
 &\left|\frac{\chi_{q+1}}{q!\,n^{(q+1)/2}}
       \sum_{k=1}^p\sum_{i=1}^n
       (v_1^\top D(\xi_1)^{-1}\mathsf u_k)\E f_{ki}^{(q)}(\mathsf Z_{ki})\right|
       \le C_qn^{-(q+1)/2}n\,
       \E\bigl(1+\|G(\xi_1)\|_\op+\|G(\xi_2)\|_\op\bigr)^{q+2}\notag\\
 &\qquad\le C_qn^{-(q-1)/2}\delta^{-2q-4}
 \le Cn^{-3/2}\delta^{-14}.
 \label{eq:recursion-fifth-sixth-cumulants}
\end{align}

In \eqref{eq:recursion-cumulant-expansion}, the first term together with
the $q=1$ contribution is evaluated by \eqref{eq:recursion-second-cumulant}.
The $q=2$, $q=3$, and $q=4,5$ contributions are controlled by
\eqref{eq:recursion-third-cumulant}, \eqref{eq:recursion-fourth-cumulant},
and \eqref{eq:recursion-fifth-sixth-cumulants}, respectively, while
\eqref{eq:recursion-cumulant-remainder-sum} bounds the remainder sum.
Substituting these estimates gives
\begin{align*}
 \E(H_1H_2)
 ={}&\Xi^{\rm g}(\xi_1,\xi_2)
 +\frac{\chi_4m(\xi_1)m(\xi_2)}n
    \sum_{k=1}^p\prod_{j=1}^2
    (v_j^\top D(\xi_j)^{-1}\mathsf u_k)(w_j^\top D(\xi_j)^{-1}\mathsf u_k)\\
 &+O(n^{-3/2+8\epsilon_0}\delta^{-C}+n^{-2}\delta^{-16}).
\end{align*}
H\"older's inequality and $\|x\|_4\le\|x\|_2$ give
\begin{align*}
 \sum_{k=1}^p\prod_{j=1}^2
    |v_j^\top D(\xi_j)^{-1}\mathsf u_k|\,
    |w_j^\top D(\xi_j)^{-1}\mathsf u_k|
 &\le\prod_{j=1}^2
    \|\Sigma^{1/2}D(\xi_j)^{-1}v_j\|_2
    \|\Sigma^{1/2}D(\xi_j)^{-1}w_j\|_2\le C,\\
 \bigl|\E(H_1H_2)-\Xi^{\rm g}(\xi_1,\xi_2)\bigr|
 &\le\frac Cn+Cn^{-3/2+8\epsilon_0}(\log n)^C
                   +Cn^{-2}(\log n)^{16}
 \le\frac Cn,
\end{align*}
where we fix $\epsilon_0<1/32$. This proves the first inequality in
\eqref{eq:moment-recursion}. The entire proof above applies directly to complex bilinear forms.

Finally, the real entries of $X,\Sigma,\varepsilon$ imply
$G(\bar\xi)=\overline{G(\xi)}$,
$\Pi(\bar\xi)=\overline{\Pi(\xi)}$, and, for real $v_2,w_2,v$,
\[
 \overline{H_2}
 =v_2^\top(G(\bar\xi_2)-\Pi(\bar\xi_2))w_2,
 \qquad
 \overline{\eta(\xi_2)}=\eta(\bar\xi_2).
\]
Since $\bar\xi_2\in\zeta(\Gamma_\delta^\circ)$, the same bounds apply with the
kernels $\Xi^{\rm g}(\xi_1,\bar\xi_2)$ and
$\Xi^{\rm n}(\xi_1,\bar\xi_2)$.
\end{proof}


\section{Asymptotic Normality of Bocca}\label{app:bocca-normality}
We prove Theorem~\ref{thm:bocca-normality}. Throughout this section,
$\delta=1/\log n$, $K\asymp\log n$, and Assumption~\ref{ass:bocca-normality}
holds. Constants may depend on this assumption and on the fixed
bounds for $K/\log n$. We retain the notation of \eqref{eq:linearization} and
\eqref{eq:covariance-functions}; in particular, $Q(\xi)$ denotes the
sample resolvent, whereas $\mathcal Q_{\mu\nu}$ is the deterministic
quantity in \eqref{eq:normality-kernel}. By Lemma~\ref{lem:geometry},
$q_\delta=h_\delta U_\delta$ has real part $h_\delta u_\delta$ and
$q_\delta(0)=h_\delta\alpha_\delta$.

Write $\mathsf H_n=p^{-1}\sum_{j=1}^p\delta_{\lambda_j}$. Recall
\begin{equation}\label{eq:normality-boundary}
 D_n(\mu)=\gamma\int\frac{\lambda^2}{|\lambda-\mu|^2}\mathsf H_n(\dd\lambda),\quad
 \mathcal C_n= \{\mu\in\C:D_n(\mu) = 1\},\quad
 M_n=h_\delta\max_{\mu\in\mathcal C_n}u_\delta(\mu).
\end{equation}
We also introduce the following auxiliary concepts needed across the proofs in this section. Let 
$\Gamma_{\delta}^*= \{\mu:D_n(\mu) = 1-\delta\}$.

\subsection{A Contour Inside the Original Circle}

The following lemma gives the geometric properties related to the contours $\mathcal C_n$ and $\Gamma_{\delta}^*$. The results are similar to the ones in Appendix~\ref{app:calibration}.

\begin{lemma}\label{lem:normality-geometry}
For all sufficiently large $n$, $\mathcal C_n$ in
\eqref{eq:normality-boundary} and
$\Gamma_{\delta}^*=\{ \mu:D_n(\mu) = 1-\delta\}$ are smooth
Jordan curves enclosing $[\ell,L]\cup\{0\}$. Their lengths are
bounded, their distances from $[\ell,L]\cup\{0\}$ are bounded
away from zero, and their distance from each other is $O(\delta)$.
The curve $\Gamma_{\delta}^*$ lies inside $\Gamma_\delta$, and
$\dist(\mathcal C_n,\Gamma_\delta)\ge c\delta$ and
$\dist(\Gamma_{\delta}^*,\Gamma_\delta)\ge c\delta$. Moreover,
\begin{equation}\label{eq:normality-potential-gaps}
 M_n-h_\delta\alpha_\delta\ge c,\qquad
 h_\delta-M_n>0,\qquad
 \max_{\Gamma_{\delta}^*}\operatorname{Re}q_\delta=M_n+O(\delta).
\end{equation}
On the closed region between $\Gamma_{\delta}^*$ and $\Gamma_\delta$, we have
\begin{equation}\label{eq:normality-calibration}
 m(\zeta(\mu))=-\frac1\mu,\qquad
 \operatorname{Im}m(\zeta(\mu))
 =\frac{\operatorname{Im}\mu}{|\mu|^2}.
\end{equation}
For every $\mu,\nu$ in this region, recalling $\mathcal Q_{\mu\nu}$ in \eqref{eq:normality-covariance-kernel-auxiliary},
$|\mathcal Q_{\mu\nu}|\le1-\delta$,
$|\zeta(\mu)-\zeta(\nu)|\ge\delta|\mu-\nu|$, and
\begin{equation}\label{eq:normality-spectral-separation}
 \dist\bigl(\zeta(\mu),\operatorname{supp}\varrho\bigr)\ge c\delta^2,
 \qquad |m'(\zeta(\mu))|\ge c.
\end{equation}
\end{lemma}

\begin{proof}
Fix $1<\gamma_0\le\gamma\le\gamma_1<\infty$ for all large $n$.
By \eqref{eq:normality-tube-condition}, for some fixed $c_0>0$,
\[
 \inf_{x\in[\ell,L]}D_n(x+\ii\eta_0)\ge1+c_0,
 \qquad
 \sup_{x\in\R}|\partial_xD_n(x+\ii\eta_0)|
 \le2\gamma_1L^2\eta_0^{-3}.
\]
Moreover, for $\mu=x+\ii y$ away from $[\ell,L]$,
\[
 \partial_xD_n(\mu)
 =2\gamma\int\frac{\lambda^2(\lambda-x)}
 {((\lambda-x)^2+y^2)^2}\,\mathsf H_n(\dd\lambda),
 \qquad
 \partial_yD_n(\mu)
 =-2\gamma y\int\frac{\lambda^2}
 {((\lambda-x)^2+y^2)^2}\,\mathsf H_n(\dd\lambda).
\]
Thus $D_n(x+\ii y)$ decreases in $y>0$.
The first equality above gives, for sufficiently small $r>0$,
$D_n>1+c_0/2$ whenever $\dist(\mu,[\ell,L])\le2r$.
Decreasing $r$ further, the bound
$
 D_n(\mu)\ge\frac{\gamma_0}{(1+|\mu|/\ell)^2}
$
also gives $D_n>1$ on $|\mu|\le2r$; choose in addition
$4r<L(\sqrt{\gamma_0}-1)$.

For $t\in[1-\delta,1]$, write $\Gamma_{n,t}=\{\mu:D_n(\mu)=t\}$.
On the real axis,
\[
 \partial_xD_n(x)=\frac2n\sum_{j=1}^p
                  \frac{\lambda_j^2}{(\lambda_j-x)^3}
 \begin{cases}>0,&x<\ell,\\<0,&x>L,\end{cases}
 \qquad D_n(0)=\gamma>1,\qquad
 \lim_{|x|\to\infty}D_n(x)=0.
\]
Consequently, the real crossings are given by
\begin{equation}\label{eq:normality-real-crossings}
 \Gamma_{n,t}\cap\R=\{a_-(t),a_+(t)\},
 \qquad a_-(t)<0<L<a_+(t).
\end{equation}
For $a_-(t)<x<a_+(t)$, strict monotonicity in $y>0$ gives
exactly one $y>0$ with $D_n(x+\ii y)=t$.
This upper arc and its reflection form $\Gamma_{n,t}$.
For $\delta\le1/2$, put $R_*=L+L\sqrt{2\gamma_1}$. The bound
$D_n(\mu)\le\gamma_1L^2/(|\mu|-L)^2$ for $|\mu|>L$ gives
\begin{equation}\label{eq:normality-level-separation}
 \sup_{\mu\in\Gamma_{n,t}}|\mu|\le R_*,
 \qquad
 \dist(\Gamma_{n,t},[\ell,L]\cup\{0\})\ge2r.
\end{equation}
All fixed-order derivatives of $D_n$ are consequently uniformly
bounded near these level sets. If $|y|\ge r$, the formula for
$\partial_yD_n$ gives
$|\partial_yD_n|\ge2\gamma_0\ell^2r/(R_*+L)^4$.
If $|y|<r$, \eqref{eq:normality-level-separation} implies that
$x$ is outside $[\ell,L]$ by at least $\sqrt3r$, and the formula
for $\partial_xD_n$ gives the same lower bound. Thus
\[
      \inf_{\mu:\,1-\delta\le D_n(\mu)\le1}
      |\nabla D_n(\mu)|\ge c>0, \text{ where for $\mu=x+\ii y$,} \nabla D_n(x+\ii y) = \partial_xD_n(x+\ii y)+ \ii \partial_yD_n(x+\ii y).
\]
The implicit-function theorem makes each $\Gamma_{n,t}$ smooth,
including at its real crossings. As the derivatives have the above bounds,  we have
$\operatorname{length}(\Gamma_{n,t})\le C$.
Near the real crossings, write the local coordinate of $\Gamma_{n,t}$ as
$x_\pm(y)+\ii y$, where $x_\pm(0)=a_\pm(t)$.
Since $|\partial_xD_n|\ge c$ and $|\partial_yD_n|\le C|y|$ there,
\begin{equation}\label{eq:normality-real-crossing-charts}
 |x_\pm'(y)|
 =\left|\frac{\partial_yD_n}{\partial_xD_n}\right|\le C|y|,
 \qquad |x_\pm(y)-a_\pm(t)|\le Cy^2
 \quad (|y|\le c).
\end{equation}
Finally, the flow $\dot\mu=\nabla D_n/|\nabla D_n|^2$ satisfies
$\frac{\dd}{\dd s}D_n(\mu(s))=1$ and $|\dot\mu|\le c^{-1}$.
It follows that the Hausdorff distance satisfies
$d_{\mathrm H}(\Gamma_{n,t},\Gamma_{n,1})\le c^{-1}(1-t)$.
In particular, $\mathcal C_n=\Gamma_{n,1}$ and
$\Gamma_{\delta}^*=\Gamma_{n,1-\delta}$ satisfy the asserted bound on the distance between the two curves.

\proofstep{Potential bounds.}
For $s>L$ and $|\mu-L|=s$, the denominator estimate used in
\eqref{eq:calibration-denominator-bound} gives
\[
 D_n(\mu)\le\frac1n\sum_{j=1}^p
 \frac{\lambda_j^2}{(s-L+\lambda_j)^2}
 \le\frac{\gamma L^2}{s^2}.
\]
Thus $\mathcal C_n$ lies in $|\mu-L|\le\sqrt\gamma L$ and
$\Gamma_{\delta}^*$ lies in $|\mu-L|\le\sqrt\gamma L/\sqrt{1-\delta}$.
For $0<\delta\le1/4$, $(1-\delta)^{-1/2}\le1+3\delta/4$, so
\[
 \dist(\Gamma_{\delta}^*,\Gamma_\delta)
 \ge\sqrt\gamma L\bigl(1+\delta-(1-\delta)^{-1/2}\bigr)
 \ge\tfrac14\sqrt\gamma L\delta.
\]
The same bound holds for $\mathcal C_n$, and both curves lie
inside $\Gamma_\delta$. Lemma~\ref{lem:geometry}(i) gives $M_n<h_\delta$.

For $k$ in \eqref{eq:k-T-h}, abbreviate $\Kell(k^2)$ and $\Kell(1-k^2)$
by $\Kell$ and $\Kell_{\mathrm c}$, respectively.
The point $-r$ lies inside $\mathcal C_n$. On $[-r,0]$, \eqref{eq:U-real-app} gives
\[
 -q_\delta'(x)
 =\frac{h_\delta}{\Kell}
 \frac{T_\delta'(x)}
 {\sqrt{(k^2-T_\delta(x)^2)(1-T_\delta(x)^2)}}\ge c>0,
 \qquad T_\delta'(x)=\frac{2R_\delta}{(R_\delta+L-x)^2}.
\]
Indeed, $0<T_\delta(x)<k<1$ there, {$\epsilon\le k^2\le1-\epsilon$ for some fixed
$\epsilon>0$, uniformly for all large $n$ and sufficiently small $\delta$}, and
$h_\delta/\Kell=\pi/\Kell_{\mathrm c}\asymp1$,
as in \eqref{eq:geometry-compact-bounds}.
The maximum principle on the {domain enclosed by $\mathcal C_n$ with the slit $[\ell,L]$ removed}, where $\operatorname{Re}q_\delta=0$ on $[\ell,L]$,
yields
\[
 M_n-h_\delta\alpha_\delta
 \ge q_\delta(-r)-q_\delta(0)
 =\int_{-r}^0 -q_\delta'(x)\dd x\ge cr.
\]
{For any closed disk $\{|\mu-z|\le s\}\subset\mathcal A_\delta$,
differentiating Poisson's formula
\citep[Chapter~4, Section~6.3, Theorem~22]{Ahlfors1979} gives
\[
 |q_\delta'(z)|=|\nabla\operatorname{Re}q_\delta(z)|
 \le\frac2s\sup_{|\mu-z|\le s}|\operatorname{Re}q_\delta(\mu)|
 \le\frac{2h_\delta}{s}.
\]
Together with \eqref{eq:outer-collar-app}, $h_\delta\asymp1$, and
$\dist(\mathcal C_n\cup\Gamma_{\delta}^*,[\ell,L])\ge2r$, this gives
\[
 \sup_{\dist(z,\mathcal C_n\cup\Gamma_{\delta}^*)\le s}|q_\delta'(z)|\le C,
 \qquad 0<s<\min(r/2,\varepsilon_*/4),
\]
where $s$ is fixed and $\varepsilon_*$ is as in
\eqref{eq:outer-collar-app}.}
Their Hausdorff distance therefore gives
\[
 \left|\max_{\Gamma_{\delta}^*}\operatorname{Re}q_\delta-M_n\right|
 \le C d_{\mathrm H}(\Gamma_{\delta}^*,\mathcal C_n)\le C\delta,
\]
which completes \eqref{eq:normality-potential-gaps}.

\proofstep{Calibration.}
The closed region between $\Gamma_{\delta}^*$ and $\Gamma_\delta$ is
\begin{equation}\label{eq:normality-calibration-region}
 \mathcal R_n=\{\mu:|\mu-L|\le R_\delta,\ D_n(\mu)\le1-\delta\}.
\end{equation}
It is bounded uniformly in $n$, and
\eqref{eq:normality-level-separation} with $t=1-\delta$ gives
\begin{equation}\label{eq:normality-region-separation}
 \dist(\mathcal R_n,[\ell,L]\cup\{0\})\ge2r.
\end{equation}
As in \eqref{eq:zeta-calibration-stability} and \eqref{eq:zeta-difference},
for $\mu,\nu\in\mathcal R_n$,
\[
 |\mathcal Q_{\mu\nu}|\le\sqrt{D_n(\mu)D_n(\nu)}\le1-\delta,
 \qquad
 |\zeta(\mu)-\zeta(\nu)|\ge\delta|\mu-\nu|,
 \qquad \operatorname{Re}\zeta'(\mu)\ge\delta.
\]
The identities \eqref{eq:f-definition} and \eqref{eq:imaginary-part-zeta}
give $f(-1/\mu)=\zeta(\mu)$ and
$\operatorname{Im}\zeta(\mu)=(1-D_n(\mu))\operatorname{Im}\mu$.
The uniqueness and conjugation argument in the proof of
Proposition~\ref{prop:calibration} therefore identifies
$m(\zeta(\mu))=-1/\mu$ for nonreal $\mu$.
For real $\mu$, the same argument using \eqref{eq:xi-prime},
$f'(-1/\mu)=\mu^2(1-D_n(\mu))>0$, identifies the inverse and
places its image outside $\operatorname{supp}\varrho$.
These arguments apply whenever $D_n(\mu)<1$, and prove
\eqref{eq:normality-calibration}.

The uniform derivative bound for $D_n$ argued above allows a fixed $c_1>0$
such that $D_n(\nu)\le1-\delta/2$ on $|\nu-\mu|\le c_1\delta$
for every $\mu\in\mathcal R_n$. On the boundary of this disk,
\eqref{eq:zeta-difference} gives
$
 |\zeta(\nu)-\zeta(\mu)|\ge\tfrac\delta2|\nu-\mu|
 =\tfrac{c_1}2\delta^2.
$
Rouch\'e's theorem, also used in \eqref{eq:calibration-image-inclusion},
now gives
\[
 B(\zeta(\mu),c_1\delta^2/4)
 \subset\zeta\bigl(B(\mu,c_1\delta)\bigr)
 \subset\C\setminus\operatorname{supp}\varrho, \text{ where } B(z,r)=\{w:|w-z|\le r\}.
\]
By \eqref{eq:imaginary-part-zeta}, \eqref{eq:xi-prime}, and
$1-D_n(\nu)\ge\delta/2>0$,
\[
 \zeta(\nu)\in\R
 \ \Longrightarrow\ (1-D_n(\nu))\operatorname{Im}\nu=0
 \ \Longrightarrow\ \nu\in\R.
\]
The inclusion $\zeta\bigl(B(\mu,c_1\delta)\bigr)
\subset\C\setminus\operatorname{supp}\varrho$ above then follows from $\operatorname{supp}\varrho\subset\R$.
This proves the distance bound in \eqref{eq:normality-spectral-separation}.
Finally, \eqref{eq:mprime}, $|\mu|\le C$, and
$|\zeta'(\mu)|\le1+D_n(\mu)\le2$ give
$|m'(\zeta(\mu))|\ge1/(2C^2)$, completing the proof.
\end{proof}

\subsection{Quantitative Fluctuation Inputs}

For $N\ge1$, define
\[
 \mathcal P_2(N)
 =\bigl\{\pi\text{ a partition of }\{1,\ldots,N\}:
          |A|=2\text{ for every }A\in\pi\bigr\}.
\]
Thus $\mathcal P_2(N)=\varnothing$ for odd $N$, and sums over this
empty set are understood to be zero.

Recall $\Gamma_{\delta}^*=\{\mu:D_n(\mu)=1-\delta\}$. Put
\begin{equation}\label{eq:normality-nonreal-contour}
 \Gamma_{\delta}^{*,\circ}
 =\{\mu\in\Gamma_{\delta}^*:|\operatorname{Im}\mu|\ge n^{-3}\}.
\end{equation}
With the spectral tube condition in Assumption~\ref{ass:bocca-normality}, we can have the following higher-order moment bounds generalizing the covariance results in Lemma~\ref{lem:recursion}.

\begin{lemma}\label{lem:normality-moments}
 For deterministic bounded vectors $v_j, w_j$, let
$H_j=v_j^\top(G(\xi_j)-\Pi(\xi_j))w_j$,
$\xi_j=\zeta(\mu_j)$, and $\mu_j\in\Gamma_{\delta}^{*,\circ}$.
Define $\Xi_{ij}=\Xi^{\rm g}(\xi_i,\xi_j)+\Xi^4_{ij}$, where $\Xi^{\rm g}$ is defined in
\eqref{eq:leading-covariance} and
\begin{equation}\label{eq:normality-fourth-kernel}
 \Xi^4_{ij}=\frac{\chi_4m(\xi_i)m(\xi_j)}n
 \sum_{k=1}^p\prod_{a\in\{i,j\}}
 (v_a^\top D(\xi_a)^{-1}\mathsf u_k)
 (w_a^\top D(\xi_a)^{-1}\mathsf u_k).
\end{equation}
For each fixed $N\ge1$, we have
\begin{equation}\label{eq:normality-wick}
 \E\prod_{j=1}^N H_j
 =\sum_{\pi\in\mathcal P_2(N)}
       \prod_{\{i,j\}\in\pi}\Xi_{ij}
   +O_N\bigl(n^{-N/2-1/4}\delta^{-C_N}\bigr).
\end{equation}
Also $|\E H_j|\le C_\epsilon n^{-1+\epsilon}
\delta^{-C_\epsilon}$ for every fixed sufficiently small
$\epsilon>0$.

For every fixed $r$ and sufficiently small $\epsilon>0$, recalling $\Xi^{\rm n}$ in \eqref{eq:noise-kernel}, the noise
field $\eta(\xi)=n^{-1/2}v^\top G_{12}(\xi)\varepsilon$ satisfies
\begin{equation}\label{eq:normality-noise-concentration}
 \left\|\E[\eta(\xi_1)\eta(\xi_2)| X]
                  -\Xi^{\rm n}(\xi_1,\xi_2)\right\|_{L^r}
 \le C_{r,\epsilon}n^{-3/2+\epsilon}\delta^{-C_{r,\epsilon}}.
\end{equation}
There exists an event $\mathcal A'$ with
$\PP(\mathcal A'^{c})\le C_Dn^{-D}$ for any prescribed $D>0$,
on which the empirical branch extends analytically across the region between
$\Gamma_{\delta}^*$ and $\Gamma_\delta$ and obeys the bounds in
\eqref{eq:empirical-control} there, with a fixed power of
$\delta^{-1}$. In particular, on $\mathcal A'$,
\begin{equation}\label{eq:normality-contour-equality}
 \widehat\beta_v(g_K)
 =-\frac1{2\pi\ii}\oint_{\Gamma_{\delta}^*}
       \frac{1-g_K(\mu)}\mu\widehat F(\mu)\dd\mu.
\end{equation}
The right-hand side is real by conjugation symmetry.
\end{lemma}

\begin{proof}
\proofstep{Uniform resolvent bounds.} We first prove results similar to those in Lemma~\ref{lem:preliminary} for the contour $\mathcal C_n$. This part of the proof is also similar to the proof of Lemma~\ref{lem:preliminary}, but we can directly implement the results in \citet{KnowlesYin2017} with Assumption~\ref{ass:bocca-normality}.
Let $a_\pm=a_\pm(1)$ be the real crossings in
\eqref{eq:normality-real-crossings}, and set $e_\pm=\zeta(a_\pm)$.
The separation bounds in \eqref{eq:normality-level-separation}
give $|a_\pm|\le C$ and
$\min\{|a_\pm|,\dist(a_\pm,[\ell,L])\}\ge c$.
As in \eqref{eq:comparison-support-interval},
\citet[Lemmas~2.4--2.6]{KnowlesYin2017} therefore give
\begin{equation}\label{eq:normality-support-regularity}
 \supp\varrho=[e_-,e_+],\qquad
 c\le e_-<e_+\le C,\qquad e_+-e_->L.
\end{equation}
Indeed, $D_n(a_-)=1$ and \eqref{eq:zeta-formula} imply
\[
 e_-=\frac{a_-^2}{n}\sum_{j=1}^p
                  \frac{\lambda_j}{(\lambda_j-a_-)^2}\ge c,
 \qquad
 e_+-e_-\ge a_+-a_->L.
\]
The remaining conditions on $e_\pm$ in
\citet[Definition~2.7(i)]{KnowlesYin2017} follow from
\begin{equation}\label{eq:normality-edge-regularity}
 \begin{gathered}
 \min_{1\le j\le p}|m(e_\pm+\ii0_+)+\lambda_j^{-1}|
 =\min_{1\le j\le p}\frac{|a_\pm-\lambda_j|}{|a_\pm|\lambda_j}\ge c,\\
 \zeta'(a_\pm)=0,\quad
 |\zeta''(a_\pm)|
 =\frac2n\sum_{j=1}^p\frac{\lambda_j^2}{|a_\pm-\lambda_j|^3}\ge c.
 \end{gathered}
\end{equation}

The argument proving \eqref{eq:normality-calibration} in
Lemma~\ref{lem:normality-geometry} applies on $\{\mu:D_n(\mu)<1\}$.
Taking limits to $\mathcal C_n\cap\C^+$ from
$\{\mu\in\C^+:D_n(\mu)<1\}$, where
$\operatorname{Im}\zeta(\mu)>0$ by \eqref{eq:imaginary-part-zeta}, by \eqref{eq:stieltjes-transform-definition}, we have for $\mu\in\mathcal C_n\cap\C^+$, 
\begin{equation}\label{eq:normality-boundary-density}
 \begin{gathered}
 E=\zeta(\mu),
 m(E+\ii0_+)=-\frac1\mu,  \qquad 
 \frac{\dd\varrho}{\dd E}(E)
 =\frac1\pi\operatorname{Im}m(E+\ii0_+)
 =\frac{\operatorname{Im}\mu}{\pi|\mu|^2}.
 \end{gathered}
\end{equation}
To identify the image of $\mathcal C_n\cap\C^+$, write its local coordinate
as $\mu=x+\ii y(x)$. Since $D_n(\mu)=1$,
\begin{equation}\label{eq:normality-boundary-monotonicity}
 \operatorname{Re}\zeta'(\mu)
 =\frac{2y(x)^2}{n}\sum_{j=1}^p\frac{\lambda_j^2}{|\lambda_j-\mu|^4}>0,
 \qquad
 \frac{\dd}{\dd x}\zeta(x+\ii y(x))
 =\frac{|\zeta'(\mu)|^2}{\operatorname{Re}\zeta'(\mu)}>0.
\end{equation}
Its endpoint values are $e_\pm$, so
$\zeta(\mathcal C_n\cap\C^+)=(e_-,e_+)$.
For $\mu=x+\ii y$ with small $y>0$,
\eqref{eq:normality-real-crossing-charts} and
\eqref{eq:zeta-difference} give, for one sign,
$|\mu-a_\pm|\le Cy$ and $|\zeta(\mu)-e_\pm|\le2|\mu-a_\pm|\le Cy$.
Thus, for each fixed $\tau>0$ with
$[e_-+\tau,e_+-\tau]\ne\varnothing$,
\begin{equation}\label{eq:normality-bulk-regularity}
 \inf_{E\in[e_-+\tau,e_+-\tau]}
       \frac{\dd\varrho}{\dd E}(E)
 =\frac1\pi\inf_{E\in[e_-+\tau,e_+-\tau]}
       \operatorname{Im}m(E+\ii0_+)\ge c_\tau>0.
\end{equation}
This verifies \citet[Definition~2.7(ii)]{KnowlesYin2017}.

Let $\mathcal R_n$ be the region in
\eqref{eq:normality-calibration-region}, and choose $c_1>0$ so small that
\[
 \mathcal V_n=\{\xi:\dist(\xi,\zeta(\mathcal R_n))<c_1\delta^2\},
 \qquad
 \dist(\mathcal V_n,[e_-,e_+])\ge c\delta^2,
\]
as permitted by \eqref{eq:normality-spectral-separation}.
Theorem~3.12 of \citet{KnowlesYin2017} gives
\begin{equation}\label{eq:normality-spectral-event}
 \begin{gathered}
 \mathcal A_0=\{\spec(\widehat S)
                  \subset[e_--n^{-1/2},e_++n^{-1/2}]\},
 \qquad \PP(\mathcal A_0^c)\le C_Dn^{-D},\\
 \inf_{\xi\in\mathcal V_n}\dist(\xi,\spec(\widehat S))\ge c\delta^2,
 \qquad
 \sup_{\xi\in\mathcal V_n}\|G(\xi)\|_\op\le C\delta^{-2}
 \quad\text{on }\mathcal A_0.
 \end{gathered}
\end{equation}
The last line uses $n^{-1/2}=o(\delta^2)$ and
\eqref{eq:linearization}. Recall $\prec$ in \eqref{eq:stochastic-domination}.
For deterministic $u,w\in\C^{p+n}$ with $\|u\|_2,\|w\|_2\le1$,
\citet[Theorems~3.6--3.7]{KnowlesYin2017} give
\begin{equation}\label{eq:normality-local-laws}
 \sup_{\xi\in\mathcal V_n}
  |\ind_{\mathcal A_0}u^\top(G(\xi)-\Pi(\xi))w|
       \prec n^{-1/2}\delta^{-C},
 \qquad
 \sup_{\xi\in\mathcal V_n}
  |\ind_{\mathcal A_0}(\widehat m(\xi)-m(\xi))|
       \prec n^{-1}\delta^{-C}.
\end{equation}
For the first bound in \eqref{eq:normality-local-laws}, fix $r_0>0$ with $4r_0<e_-$ for all large $n$.
In the upper half-plane, apply \citet[Theorem~3.6]{KnowlesYin2017} when
$|\xi|\ge r_0$ and $\operatorname{Im}\xi\ge c\delta^2$, and
\citet[Theorem~3.7]{KnowlesYin2017} when $\operatorname{Re}\xi\ge r_0$ and
$\dist(\operatorname{Re}\xi,[e_-,e_+])\ge c\delta^2$.
For the remaining points with $|\xi|\ge r_0$, apply
\citet[Theorem~3.6]{KnowlesYin2017} above height $n^{-1/2}$ and
Lemma~\ref{lem:gap-continuation}.
The argument for $|\xi|<r_0$ in the proof of
Lemma~\ref{lem:preliminary} gives
\eqref{eq:normality-local-laws} on all of $\mathcal V_n$.
For the second bound in \eqref{eq:normality-local-laws}, if $t_i$ is the $(i-\tfrac12)/n$ quantile of
$\varrho$, \citet[Theorem~3.12]{KnowlesYin2017} give
\[
 \begin{gathered}
 |s_i-t_i|\prec n^{-2/3}\min\{i,n+1-i\}^{-1/3},
 \\
 \sup_{\xi\in\mathcal V_n}|\widehat m(\xi)-m(\xi)|
 \le\frac{C\delta^{-4}}n
       \biggl(1+\sum_{i=1}^n|s_i-t_i|\biggr)
 \prec n^{-1}\delta^{-4}
 \quad\text{on }\mathcal A_0.
 \end{gathered}
\]
Equations \eqref{eq:imaginary-part-zeta} and
\eqref{eq:nonreal-resolvent-bound} give
\begin{equation}\label{eq:normality-raw-resolvent}
 \begin{gathered}
 |\operatorname{Im}\zeta(\mu)|
   =\delta|\operatorname{Im}\mu|\ge\delta n^{-3},\\
 \|G(\zeta(\mu))\|_\op+|\widehat m(\zeta(\mu))|
 \le C\bigl(1+|\operatorname{Im}\zeta(\mu)|^{-1}\bigr)
 \le Cn^3\delta^{-1},
 \qquad \mu\in\Gamma_{\delta}^{*,\circ}.
 \end{gathered}
\end{equation}
Combining \eqref{eq:normality-local-laws}--\eqref{eq:normality-raw-resolvent}
as in \eqref{eq:recursion-raw-moments}, and applying Cauchy's integral formula
on disks of radius $c\delta^2$ on $\mathcal A_0$, yields, for each
fixed $r\ge2$, $k\ge0$, and $\epsilon>0$,
\begin{equation}\label{eq:normality-resolvent-moments}
 \begin{aligned}
 \sup_{\mu\in\Gamma_{\delta}^{*,\circ}}
 \bigl\|\partial_\xi^k[u^\top(G(\xi)-\Pi(\xi))w]
                          \big|_{\xi=\zeta(\mu)}\bigr\|_{L^r}
 &\le C_{r,k,\epsilon}n^{-1/2+\epsilon}\delta^{-C_{r,k,\epsilon}},\\
 \sup_{\mu\in\Gamma_{\delta}^{*,\circ}}
 \bigl\|\partial_\xi^k(\widehat m(\xi)-m(\xi))
                          \big|_{\xi=\zeta(\mu)}\bigr\|_{L^r}
 &\le C_{r,k,\epsilon}n^{-1+\epsilon}\delta^{-C_{r,k,\epsilon}},\\
 \sup_{\mu\in\Gamma_{\delta}^{*,\circ}}
       \bigl\|\|G(\zeta(\mu))\|_\op\bigr\|_{L^r}
 &\le C_r\delta^{-2}.
 \end{aligned}
\end{equation}
For the derivatives on $\mathcal A_0^c$, use
$\|\partial_\xi^kG(\xi)\|_\op\le k!\|G(\xi)\|_\op^{k+1}$
and choose the failure exponent in
\eqref{eq:normality-spectral-event} sufficiently large.

\proofstep{Higher-order moments.} This part of the proof extends the proof of Lemma~\ref{lem:recursion}. We extend the cumulant argument in Lemma~\ref{lem:recursion}, adapted from the proof of \citet[Lemma~5.2]{Yang2024},
using \eqref{eq:normality-resolvent-moments} and
\eqref{eq:normality-spectral-event}, established above under
Assumption~\ref{ass:bocca-normality}.
For
$A\subseteq\{2,\ldots,N\}$, write
\[
 b_k=v_1^\top D(\xi_1)^{-1}\mathsf u_k,\qquad
 g_{ki}=\mathsf e_i^\top G(\xi_1)w_1,\qquad
 P_A=\prod_{a\in\{2,\ldots,N\}\setminus A}H_a.
\]
The empty
products equal one. Set $J=2N+10$ and, for $2\le j\le J$, define
\begin{equation}\label{eq:normality-cumulant-terms}
 T_j=\frac{\chi_j}{(j-1)!\,n^{j/2}}
             \sum_{k=1}^p\sum_{i=1}^n
                  b_k\,\E\partial_{ki}^{j-1}(g_{ki}P_\varnothing).
\end{equation}
By \eqref{eq:recursion-entry-identity} and
Lemma~\ref{lem:cumulant-expansion},
\begin{equation}\label{eq:normality-cumulant-expansion}
 \E(H_1P_\varnothing)
 =m(\xi_1)\E\bigl[
       (v_1^\top D(\xi_1)^{-1}\Sigma G_{11}(\xi_1)w_1)
       P_\varnothing\bigr]
       +\sum_{j=2}^{J}T_j+\mathcal E_N.
\end{equation}
The Leibniz rule gives
\begin{equation}\label{eq:normality-derivative-allocations}
 \begin{gathered}
 \partial_{ki}^{r}G(\xi)=(-1)^rr!\,G(\xi)
          \bigl(\Delta_{ki}G(\xi)\bigr)^r,\qquad
 \Delta_{ki}=\mathsf u_k\mathsf e_i^\top+
                         \mathsf e_i\mathsf u_k^\top,\\
 \frac{\partial_{ki}^{j-1}(g_{ki}P_\varnothing)}{(j-1)!}
 =\sum_{\substack{r_0+\sum_{a=2}^Nr_a=j-1\\r_0,r_2,\ldots,r_N\ge0}}
       \frac{\partial_{ki}^{r_0}g_{ki}}{r_0!}
       \prod_{a=2}^N\frac{\partial_{ki}^{r_a}H_a}{r_a!}.
 \end{gathered}
\end{equation}
Apply \eqref{eq:cumulant-expansion-remainder} with $q=J$.
The entry perturbation bound \eqref{eq:recursion-entry-perturbation}
and derivative identity \eqref{eq:recursion-entry-derivative-identity},
using \eqref{eq:normality-spectral-event} and the bounded entry moments,
bound each entry's remainder by $C_Nn^{-(J+1)/2}\delta^{-C_N}$,
as in the argument leading to \eqref{eq:recursion-cumulant-remainder-sum}.
Consequently,
\[
 |\mathcal E_N|
 \le C_Nn^{-(J+1)/2}\delta^{-C_N}
       n\sum_{k=1}^p|b_k|
 \le C_Nn^{1-J/2}\delta^{-C_N},
\]
where $\sum_k|b_k|\le\sqrt p\,(\sum_k|b_k|^2)^{1/2}\le C\sqrt p$
and $p\asymp n$.

We choose a fixed $\epsilon>0$ in all bounds of
\eqref{eq:normality-resolvent-moments} sufficiently
small that their total contribution is at most $n^\epsilon$.
The calculation in \eqref{eq:recursion-second-cumulant-identity}
and \eqref{eq:recursion-divided-difference-identities} gives
\begin{equation}\label{eq:normality-second-cumulant}
 \begin{aligned}
 \!\!\! m(\xi_1)\E\bigl[
   (v_1^\top D(\xi_1)^{-1}\Sigma G_{11}(\xi_1)w_1)
       P_\varnothing\bigr]+T_2
=\sum_{a=2}^N
      \Xi^{\rm g}(\xi_1,\xi_a)\E P_{\{a\}}
      +O_N\bigl(n^{-(N+1)/2+\epsilon}\delta^{-C_N}\bigr).
 \end{aligned}
\end{equation}
Specifically, the two remainders in the first two lines on the
right of \eqref{eq:recursion-second-cumulant-identity} satisfy
\[
 \begin{aligned}
 &\left|\E\bigl[(m-\widehat m)(\xi_1)
     (v_1^\top D(\xi_1)^{-1}\Sigma G_{11}(\xi_1)w_1)
     P_\varnothing\bigr]\right|\\
 &\quad+\frac1n\left|\E\bigl[
     (v_1^\top D(\xi_1)^{-1}\Sigma G_{12}(\xi_1)
                           G_{21}(\xi_1)w_1)P_\varnothing\bigr]\right|
                           \le C_N n^{-(N+1)/2+\epsilon}\delta^{-C_N}.
 \end{aligned}
\]
For each remaining term of \eqref{eq:recursion-second-cumulant-identity}, substitute
$\Pi_{11}(\xi_a)=-D(\xi_a)^{-1}$ and
$\Pi_{11}[\xi_1,\xi_a]=S(\xi_1,\xi_a)M(\xi_1,\xi_a)$.
The corresponding bilinear errors have $L^r$ norms bounded by
$Cn^{-1/2+\epsilon}\delta^{-C}$, by
\eqref{eq:normality-resolvent-moments} and Cauchy's integral formula.
H\"older's inequality and
$\|P_{\{a\}}\|_{L^r}\le C_N n^{-(N-2)/2+\epsilon}\delta^{-C_N}$
then give \eqref{eq:normality-second-cumulant}.

For $T_3$ in \eqref{eq:normality-cumulant-terms},
set $j=3$ in \eqref{eq:normality-derivative-allocations}.
The three groups below contain, respectively, zero, one, and two
distinct differentiated factors among $H_2,\ldots,H_N$:
\begin{equation}\label{eq:normality-third-cumulant-expansion}
 \begin{aligned}
 \partial_{ki}^2(g_{ki}P_\varnothing)
 =&(\partial_{ki}^2g_{ki})P_\varnothing
 +\sum_{a=2}^N
      (2\partial_{ki}g_{ki}\,\partial_{ki}H_a+
                    g_{ki}\partial_{ki}^2H_a)P_{\{a\}}\\
 &\qquad +2\sum_{2\le a<b\le N}
              g_{ki}\partial_{ki}H_a\partial_{ki}H_bP_{\{a,b\}}.
 \end{aligned}
\end{equation}
The estimates used in \eqref{eq:recursion-third-cumulant} give
\begin{equation}\label{eq:normality-third-cumulant-bounds}
 \begin{aligned}
 \Bigl\|\sum_{k,i}b_k\partial_{ki}^2g_{ki}\Bigr\|_{L^r}
       &\le Cn^{1/2+\epsilon}\delta^{-C},\\
 \Bigl\|\sum_{k,i}b_k
 (2\partial_{ki}g_{ki}\,\partial_{ki}H_a+
                    g_{ki}\partial_{ki}^2H_a)\Bigr\|_{L^r}
       &\le Cn^\epsilon\delta^{-C},\\
 \Bigl\|\sum_{k,i}b_kg_{ki}
                     \partial_{ki}H_a\partial_{ki}H_b\Bigr\|_{L^r}
       &\le Cn^{-1/2+\epsilon}\delta^{-C} \text{ if } a\ne b.
 \end{aligned}
\end{equation}
Here the first two bounds of \eqref{eq:normality-third-cumulant-bounds} use
\eqref{eq:recursion-weighted-feature-sums} and, for bounded $w\in\C^p$,
\[
 \Bigl\|\sum_{i=1}^n\mathsf e_i^\top G(\xi)w\Bigr\|_{L^r}
 =\sqrt n\,
   \bigl\|\bar e^\top(G(\xi)-\Pi(\xi))w\bigr\|_{L^r}
 \le Cn^\epsilon\delta^{-C},
 \qquad \bar e=n^{-1/2}\sum_i\mathsf e_i.
\]
For the third bound in \eqref{eq:normality-third-cumulant-bounds}, each monomial contains three mixed-block
entries and a factor $\mathsf u_k^\top G(\xi)w$ or
$w^\top G(\xi)\mathsf u_k$ to which
\eqref{eq:recursion-weighted-feature-sums} applies.
Since $\|P_A\|_{L^r}\le
C_Nn^{-(N-1-|A|)/2+\epsilon}\delta^{-C_N}$, H\"older's inequality gives
$|T_3|\le C_Nn^{-(N+1)/2+\epsilon}\delta^{-C_N}$.

For $T_4$, the allocation $r_0=1$, $r_a=2$ in
\eqref{eq:normality-derivative-allocations} contributes
\[
 \begin{aligned}
 -\frac{\chi_4}{n^2}\sum_{a=2}^N\sum_{k,i}b_k\,
 \E\bigl[&
 (\mathsf e_i^\top G(\xi_1)\mathsf e_i)
 (\mathsf e_i^\top G(\xi_a)\mathsf e_i)
 (\mathsf u_k^\top G(\xi_1)w_1) (v_a^\top G(\xi_a)\mathsf u_k)
 (\mathsf u_k^\top G(\xi_a)w_a)P_{\{a\}}\bigr].
 \end{aligned}
\]
Replacing these five resolvent factors by their $\Pi$ counterparts,
as in \eqref{eq:recursion-fourth-cumulant}, gives
\begin{equation}\label{eq:normality-fourth-cumulant}
 T_4=\sum_{a=2}^N\Xi^4_{1a}\E P_{\{a\}}
       +O_N\bigl(n^{-(N+1)/2+\epsilon}\delta^{-C_N}\bigr).
\end{equation}
To bound the other allocations and $T_j$ for $j\ge5$, put
$s=\#\{a\ge2:r_a>0\}$ and let $t$ count factors of the form
$\mathsf e_i^\top G(\xi)u$ or $u^\top G(\xi)\mathsf e_i$,
with $u \in \C^p$ in the feature block, in a monomial of
\eqref{eq:normality-derivative-allocations}.
The block structure gives
\[
 r_0+\sum_{a=2}^N r_a=j-1,\qquad
 t\ge\ind_{\{r_0\ {\rm even}\}}
                  +\#\{a\ge2:r_a\ {\rm odd}\}.
\]
If a monomial contains $\mathsf u_k^\top G(\xi)u$ or its transpose,
with $u\in\{v_a,w_a:1\le a\le N\}$, apply
\eqref{eq:recursion-weighted-feature-sums} to its sum over $k$.
Using \eqref{eq:recursion-entrywise-moments} with
\eqref{eq:normality-resolvent-moments}, its absolute contribution
is bounded by
\begin{equation}\label{eq:normality-cumulant-power-count}
 C_N n^{-j/2}n\,n^{-t/2}n^{-(N-1-s)/2+\epsilon}\delta^{-C_N}
 =C_N n^{-N/2+(3-j+s-t)/2+\epsilon}\delta^{-C_N}.
\end{equation}
For $j\ge5$, $j-s+t\ge4$. For $j=4$ the only exception is
$(r_0,s,r_a,t)=(1,1,2,0)$, retained above.
If there is no such factor, then
$\sum_k|b_k|\le C\sqrt p$ and $t\ge1+2s$, so the bound becomes
\[
 C_N n^{-N/2+(4-j+s-t)/2+\epsilon}\delta^{-C_N}
 \le C_N n^{-(N+1)/2+\epsilon}\delta^{-C_N},
 \qquad j\ge4.
\]
These inequalities prove the remainder bound in
\eqref{eq:normality-fourth-cumulant} and
$\sum_{j=5}^J|T_j|\le
C_Nn^{-(N+1)/2+\epsilon}\delta^{-C_N}$.
Substitution into \eqref{eq:normality-cumulant-expansion} gives
\begin{equation}\label{eq:normality-product-recursion}
 \E\prod_{j=1}^N H_j
 =\sum_{j=2}^N\Xi_{1j}\E\prod_{a\notin\{1,j\}}H_a
  +O_N\bigl(n^{-(N+1)/2+\epsilon}\delta^{-C_N}\bigr).
\end{equation}
For $N=1$ this gives
$|\E H_1|\le C_\epsilon n^{-1+\epsilon}\delta^{-C_\epsilon}$.
For general fixed $N$, iterate
\eqref{eq:normality-product-recursion}, use
$|\Xi_{ij}|\le Cn^{-1}\delta^{-C}$, and choose $\epsilon<1/4$
to obtain \eqref{eq:normality-wick}.

\proofstep{The noise covariance.} This part also extends the proof of Lemma~\ref{lem:recursion} on the noise covariance.
Independence and $\E(\varepsilon\varepsilon^\top)=\sigma^2I_n$,
together with \eqref{eq:recursion-divided-difference-identities},
give
\[
 \begin{aligned}
 \E[\eta(\xi_1)\eta(\xi_2)\mid X]-\Xi^{\rm n}(\xi_1,\xi_2)
 &=\frac{\sigma^2}{n}
        v^\top(G_{11}-\Pi_{11})[\xi_1,\xi_2]v,\\
 \bigl\|\E[\eta(\xi_1)\eta(\xi_2)\mid X]
                          -\Xi^{\rm n}(\xi_1,\xi_2)\bigr\|_{L^r}
 &\le C_{r,\epsilon}n^{-3/2+\epsilon}\delta^{-C_{r,\epsilon}},
 \end{aligned}
\]
where the second line follows from
\eqref{eq:normality-resolvent-moments}.

\proofstep{Branch condition.} We finally prove $\widehat F$ is analytic on a neighborhood of $\mathcal R_n$. By \eqref{eq:normality-local-laws}, the event
\[
 \mathcal A'=\mathcal A_0\cap
 \Bigl\{\sup_{\xi\in\mathcal V_n}
           |\widehat m(\xi)-m(\xi)|\le n^{-3/4}\Bigr\}
 \quad\text{satisfies}\quad
 \PP(\mathcal A'^{c})\le C_Dn^{-D}.
\]
As in Lemma~\ref{lem:empirical-control}, we fix $b>6$. Uniformly for $\mu\in\mathcal R_n$ and
$|\xi-\zeta(\mu)|=c_1\delta^b$, 
\eqref{eq:empirical-local-lower-bound} and
\eqref{eq:normality-spectral-separation} give, on $\mathcal A'$,
\[
 |m(\xi)+1/\mu|
 \ge c|\xi-\zeta(\mu)|
 =cc_1\delta^b>n^{-3/4}
 \ge|\widehat m(\xi)-m(\xi)|.
\]
As in the argument following \eqref{eq:empirical-local-lower-bound}
in the proof of Lemma~\ref{lem:empirical-control},
Rouch\'e's theorem gives a unique simple zero $\widehat\zeta(\mu)$
of $\widehat m(\xi)+1/\mu$ in this disk.
As in \eqref{eq:empirical-root-displacement},
\begin{equation}\label{eq:normality-empirical-control}
 \begin{gathered}
 \widehat m(\widehat\zeta(\mu))=-1/\mu,\qquad
 \sup_{\mu\in\mathcal R_n}|\widehat\zeta(\mu)-\zeta(\mu)|
       \le Cn^{-3/4},\\
 \sup_{\mu\in\mathcal R_n}\|Q(\widehat\zeta(\mu))\|_\op
       \le C\delta^{-2}.
 \end{gathered}
\end{equation}
Uniqueness of these simple roots gives an analytic extension of
the prescribed branch from $\Gamma_\delta$ and
$\widehat\zeta(\bar\mu)=\overline{\widehat\zeta(\mu)}$.
Consequently,
$\widehat F(\mu)=n^{-1}(Xv)^\top Q(\widehat\zeta(\mu))Y$
is analytic on a neighborhood of $\mathcal R_n$ and satisfies
$\widehat F(\bar\mu)=\overline{\widehat F(\mu)}$.
Cauchy's theorem gives \eqref{eq:normality-contour-equality};
its reality follows from this identity and
$g_K(\bar\mu)=\overline{g_K(\mu)}$.
\end{proof}

\subsection{A Lower Bound for the Variance}

Recall $\Gamma_{\delta}^*=\{\mu:D_n(\mu)=1-\delta\}$ with $D_n(\mu)=\gamma\int\frac{\lambda^2}{|\lambda-\mu|^2}\mathsf H_n(\dd\lambda)$. We have the following lemma on the scale of the variance $\sigma^2(g_K)$.

\begin{lemma}\label{lem:normality-variance-bounds}
For $K\asymp\log n$, put
$A_{n,K}=\exp(K(M_n-h_\delta\alpha_\delta))$ and $\tau_n^2=\sigma^2+\beta^\top\Sigma\beta$.
For $\psi(\mu)=(1-g_K(\mu))/\mu$, we have $\|\psi\|_{\infty,\Gamma_{\delta}^*}\asymp A_{n,K}$ and
\begin{equation}\label{eq:normality-J-bounds}
 \frac{c\tau_n^2}{nK}A_{n,K}^2\le \sigma^2(g_K)
 \le \frac{C\tau_n^2K}{n}A_{n,K}^2.
\end{equation}
\end{lemma}

\begin{proof}
Lemma~\ref{lem:normality-geometry} gives $c\le|\mu|\le C$ when $\mu\in\Gamma_\delta^*$ and
$
 |\max_{\Gamma_\delta^*}\operatorname{Re}q_\delta-M_n|
 \le C\delta.
$
By \eqref{eq:normality-potential-gaps},
\eqref{eq:geometry-compact-bounds}, and $K\asymp\log n$,
$n^c\le A_{n,K}\le n^C$ for all sufficiently large $n$.
Using \eqref{eq:gK}, $\sinh x\le|\cosh(x+\ii y)|\le\cosh x$
for $x\ge0$, and $K\delta=O(1)$, we obtain
\begin{align*}
 \|g_K\|_{\infty,\Gamma_\delta^*}
 &\le\frac{\cosh(K(M_n+C\delta))}
              {\cosh(Kh_\delta\alpha_\delta)}
 \le2e^{CK\delta}A_{n,K}\le CA_{n,K},\\
 \|g_K\|_{\infty,\Gamma_\delta^*}
 &\ge\frac{\sinh(K(M_n-C\delta))}
              {\cosh(Kh_\delta\alpha_\delta)}
 \ge\tfrac14e^{-CK\delta}A_{n,K}\ge cA_{n,K}.
\end{align*}
The lower bound uses $K(M_n-C\delta)\to\infty$.
Since $A_{n,K}\to\infty$, we have
\begin{align*}
 \|\psi\|_{\infty,\Gamma_\delta^*}
 \le\frac{1+\|g_K\|_{\infty,\Gamma_\delta^*}}
              {\inf_{\mu\in\Gamma_\delta^*}|\mu|}
 \le CA_{n,K} \qquad 
 \|\psi\|_{\infty,\Gamma_\delta^*}
 \ge\frac{\|g_K\|_{\infty,\Gamma_\delta^*}-1}
              {\sup_{\mu\in\Gamma_\delta^*}|\mu|}
 \ge cA_{n,K}.
\end{align*}

For $\mu,\nu$ in $\mathcal R_n$ in
\eqref{eq:normality-calibration-region}, Lemma~\ref{lem:normality-geometry} gives
\begin{align*}
 \|R_\mu\|_\op
 &\le\frac1{\dist(\mu,[\ell,L])}\le C,
 \qquad
 \|A_{\mu\nu}\|_\op
 \le L\|R_\mu\|_\op\|R_\nu\|_\op\le C,\\
 |1-\mathcal Q_{\mu\nu}|
 &\ge1-|\mathcal Q_{\mu\nu}|\ge\delta, \qquad
 \sum_{j=1}^p|b_j(\mu)|^2
 \le\|\Sigma^{1/2}R_\mu v\|_2^2
       \|\Sigma^{1/2}R_\mu\beta\|_2^2
 \le C\|\beta\|_2^2.
\end{align*}
Since $|\mu\nu|\le C$,
\eqref{eq:normality-kernel} and the Cauchy--Schwarz inequality give
\begin{align*}
 |\mathcal K_n(\mu,\nu)|
 &\le\frac{C(\|\beta\|_2^2+\sigma^2)}{\delta}
   +C\Bigl(\sum_{j=1}^p|b_j(\mu)|^2\Bigr)^{1/2}
      \Bigl(\sum_{j=1}^p|b_j(\nu)|^2\Bigr)^{1/2}
      \le\frac{C(\|\beta\|_2^2+\sigma^2)}{\delta}
 \le\frac{C\tau_n^2}{\delta},
\end{align*}
where $\|\beta\|_2^2\le\tau_n^2/\ell$ and $\sigma^2\le\tau_n^2$.
The separation of $\mathcal R_n$ from $[\ell,L]$ in
\eqref{eq:normality-region-separation}, together with
$|1-\mathcal Q_{\mu\nu}|\ge\delta$ for $\mu,\nu\in\mathcal R_n$,
makes the integrand in \eqref{eq:normality-variance} holomorphic in each
variable on a neighborhood of $\mathcal R_n$.
Deforming both contours to $\Gamma_\delta^*$ therefore gives
\begin{align*}
 \sigma^2(g_K)
 &=-\frac1{4\pi^2n}\oint_{\Gamma_\delta^*}\!\oint_{\Gamma_\delta^*}
       \psi(\mu)\psi(\nu)\mathcal K_n(\mu,\nu)\dd\mu\dd\nu,\\
 |\sigma^2(g_K)|
 &\le\frac1{4\pi^2n}\oint_{\Gamma_\delta^*}\!\oint_{\Gamma_\delta^*}
       |\psi(\mu)|\,|\psi(\nu)|\,|\mathcal K_n(\mu,\nu)|
       \,|\dd\mu|\,|\dd\nu|\\
 &\le\frac{C\tau_n^2}{n\delta}
       \operatorname{length}(\Gamma_\delta^*)^2
       \|\psi\|_{\infty,\Gamma_\delta^*}^2
 \le\frac{C\tau_n^2}{n\delta}A_{n,K}^2
 \le\frac{C\tau_n^2K}{n}A_{n,K}^2,
\end{align*}
using $\operatorname{length}(\Gamma_\delta^*)\le C$ by Lemma~\ref{lem:normality-geometry}. This proves the upper bound in
\eqref{eq:normality-J-bounds}.

\proofstep{Variance lower bound.}
For the lower bound, work first on $\Gamma_\delta$ and use the
operator $\mathcal I_\psi H=(2\pi\ii)^{-1}\oint_{\Gamma_\delta}\psi(\mu)H(\mu)\dd\mu$ defined in Appendix~\ref{app:moments}.

We decompose $\sigma^2(g_K)=\sigma_X^2(g_K)+\sigma_\varepsilon^2(g_K)$, where
\begin{equation}\label{eq:normality-variance-noise-part}
      \sigma_\varepsilon^2(g_K)
      =-\frac{\sigma^2}{4\pi^2n}
            \oint_{\Gamma_\delta}\!\oint_{\Gamma_\delta}
            \frac{\psi(\mu)\psi(\nu)}{1-\mathcal Q_{\mu\nu}}
            v^\top A_{\mu\nu}v\dd\mu\dd\nu
      =\frac{\sigma^2}{4\pi^2n}
            \oint_{\Gamma_\delta}\!\oint_{\Gamma_\delta}
            \frac{\psi(\mu)\overline{\psi(\nu)}}{1-\mathcal Q_{\mu,\bar\nu}}
            v^\top A_{\mu,\bar\nu}v\dd\mu\dd\bar\nu    
\end{equation}
is the variance term contributed by the noise and 
$\sigma_X^2(g_K)$ is the remaining contribution from the design $X$. We first bound $\sigma_\varepsilon^2(g_K)$ from below.

Write
$f_j(\mu)=\lambda_j/(\sqrt n(\lambda_j-\mu))$ and
$a_j(\mu)=\sqrt{\lambda_j}/(\lambda_j-\mu)$.
Recalling the eigendecomposition $\Sigma=\sum_{j=1}^p\lambda_j o_j o_j^\top$, we have, for $\mu,\nu\notin[\ell,L]$,
\begin{align*}
 A_{\mu\nu}
 &=\Sigma(\Sigma-\mu I_p)^{-1}(\Sigma-\nu I_p)^{-1}
 =\sum_{j=1}^p
   \frac{\lambda_j}{(\lambda_j-\mu)(\lambda_j-\nu)}
   o_j o_j^\top = \sum_{j=1}^p (a_j(\mu) o_j) (a_j(\nu)o_j)^\top,\\
 \mathcal Q_{\mu\nu}
 &=\frac1n\tr(\Sigma A_{\mu\nu})
 =\frac1n\sum_{j=1}^p
   \frac{\lambda_j^2}{(\lambda_j-\mu)(\lambda_j-\nu)} = \sum_{j=1}^p f_j(\mu) f_j(\nu).
\end{align*}

Set $f(\mu)=(f_1(\mu),\ldots,f_p(\mu))^\top$. For $m\ge1$, define
its $m$th tensor power by
\[
 f(\mu)^{\otimes m}
 =\biggl(\prod_{r=1}^m f_{j_r}(\mu)\biggr)_{1\le j_1,\ldots,j_m\le p}
 \in\C^{p^m},
 \qquad f(\mu)^{\otimes0}=1.
\]
By \eqref{eq:calibration-denominator-bound},
$\sup_{\mu\in\Gamma_\delta}\|f(\mu)\|_2^2\le(1+\delta)^{-2}<1$.
Thus the following series converges uniformly for
$\mu,\nu\in\Gamma_\delta$:
\begin{equation}\label{eq:normality-geometric-series}
 \frac1{1-\mathcal Q_{\mu,\bar\nu}}
 =\sum_{m\ge0}\mathcal Q_{\mu,\bar\nu}^m
 =\sum_{m\ge0}
   \bigl(f(\nu)^{\otimes m}\bigr)^*f(\mu)^{\otimes m}.
\end{equation}
Choose $\mu_*\in\mathcal C_n$ attaining
$\operatorname{Re}q_\delta(\mu_*)=M_n$ and set
$\phi(\mu)=\sum_jf_j(\mu)\overline{f_j(\mu_*)}
=\mathcal Q_{\mu,\bar\mu_*}$.
Then $\sum_j|f_j(\mu_*)|^2=D_n(\mu_*)=1$, $\phi(\mu_*)=1$, and
$|\phi'(\mu_*)|=|\nabla D_n(\mu_*)|/2\ge c$.
Since $D_n(\mu_*)=1$, for every $m\ge0$ we have
\[
 \|f(\mu_*)^{\otimes m}\|_2^2=D_n(\mu_*)^m=1,
 \qquad
 \bigl(f(\mu_*)^{\otimes m}\bigr)^*f(\mu)^{\otimes m}
 =\phi(\mu)^m.
\]
Projecting $ {\mathcal I_\psi(a_jf^{\otimes m})}$ onto the unit vector $f(\mu_*)^{\otimes m}$ therefore gives
\begin{equation}\label{eq:normality-tensor-projection}
 \begin{aligned}
 |\mathcal I_\psi(a_j\phi^m)|^2
 =\bigl|\bigl(f(\mu_*)^{\otimes m}\bigr)^*
         \mathcal I_\psi(a_jf^{\otimes m})\bigr|^2
 \le\|f(\mu_*)^{\otimes m}\|_2^2
       \|\mathcal I_\psi(a_jf^{\otimes m})\|_2^2
 =\|\mathcal I_\psi(a_jf^{\otimes m})\|_2^2.
 \end{aligned}
\end{equation}

Assembling all the results above, we have
\begin{align}
 \sigma_\varepsilon^2(g_K)
 &=\frac{\sigma^2}{4\pi^2n}
       \oint_{\Gamma_\delta}\!\oint_{\Gamma_\delta}
       \frac{\psi(\mu)\overline{\psi(\nu)}}{1-\mathcal Q_{\mu,\bar\nu}}
       v^\top A_{\mu,\bar\nu}v\dd\mu\dd\bar\nu\notag\\
 &=\frac{\sigma^2}{4\pi^2n}\sum_{j=1}^p(v^\top o_j)^2
      \sum_{m\ge0} \oint_{\Gamma_\delta}\!\oint_{\Gamma_\delta}
      {\psi(\mu)a_j(\mu)\overline{\psi(\nu)a_j(\nu)}}
            {\mathcal Q^m_{\mu,\bar\nu}}\dd\mu\dd\bar\nu\notag\\
 &=\frac{\sigma^2}{4\pi^2n}\sum_{j=1}^p(v^\top o_j)^2\sum_{m\ge0}
       \biggl\|\oint_{\Gamma_\delta}
       \psi(\mu)a_j(\mu)f(\mu)^{\otimes m}\dd\mu\biggr\|_2^2\notag\\
 &=\frac{\sigma^2}{n}\sum_{j=1}^p(v^\top o_j)^2
       \sum_{m\ge0}\|\mathcal I_\psi(a_jf^{\otimes m})\|_2^2\notag\\
 &\overset{\eqref{eq:normality-tensor-projection}}{\ge}\frac{\sigma^2}{n}\sum_{j=1}^p(v^\top o_j)^2
       \sum_{m\ge0}|\mathcal I_\psi(a_j\phi^m)|^2\notag\\
 &\ge\frac{\sigma^2(1-t^2)}{n}\sum_{j=1}^p(v^\top o_j)^2
       \left|\mathcal I_\psi\!\left(\frac{a_j}{1-t\phi}\right)\right|^2,
 \text{ for all } 0<t<1.\label{eq:normality-noise-projection}
\end{align}
The last inequality is by the Cauchy--Schwarz inequality for the sequence $\{t^m\}_{m\ge0}$.

\proofstep{The residue.}
Fix a small Jordan curve $\Lambda$ enclosing $[\ell,L]$ but not
zero, on which
$\operatorname{Re}q_\delta<h_\delta\alpha_\delta-c$.
It may be fixed uniformly for large $n$, by Lemma~\ref{lem:geometry}.
By \eqref{eq:zeta-formula},
\begin{equation}\label{eq:normality-projection-roots}
 \begin{aligned}
 1-\phi(\mu)
 =1-\frac1n\sum_{j=1}^p
       \frac{\lambda_j^2}{(\mu-\lambda_j)(\bar\mu_*-\lambda_j)}
 =\frac{\zeta(\mu)-\zeta(\bar\mu_*)}{\mu-\bar\mu_*},
 \end{aligned}
\end{equation}
with the right-hand side equal to $\zeta'(\bar\mu_*)$ at $\mu=\bar\mu_*$.
Since $D_n(\mu_*)=1$, \eqref{eq:imaginary-part-zeta} gives
$\operatorname{Im}\zeta(\mu_*)=0$. Hence, by
\eqref{eq:zeta-formula},
$
 \zeta(\bar\mu_*)=\overline{\zeta(\mu_*)}=\zeta(\mu_*)
$ is real.
So \eqref{eq:normality-projection-roots} gives
$
 (1-\phi)(\mu_*)=0$,
$|(1-\phi)'(\mu_*)|=|\phi'(\mu_*)|
 =\tfrac12|\nabla D_n(\mu_*)|\ge c,
$
and hence $\mu_*$ is a simple zero of $1-\phi$.

We next prove that $ {\mu_*}$ is the only zero of $1-\phi$ outside $[\ell,L]$. Let
$\theta_1<\cdots<\theta_d$ be the distinct eigenvalues of $\Sigma$ and put
$
 P(\mu)=\bigl(\zeta(\mu)-\zeta(\mu_*)\bigr)
         \prod_{k=1}^d(\mu-\theta_k).
$
By \eqref{eq:zeta-formula}, $P$ is a monic polynomial of degree
$d+1$, and, for $1\le k<d$,
$
 \lim_{x\downarrow\theta_k}\zeta(x)=+\infty
$, $
 \lim_{x\uparrow\theta_{k+1}}\zeta(x)=-\infty.
$
Since $\zeta(\mu_*)\in\R$, continuity gives at least one zero
of $P$ in each interval $(\theta_k,\theta_{k+1})\subset[\ell,L]$.
If $\mu_*\notin\R$, $P$ also vanishes at the two distinct points
$\mu_*$ and $\bar\mu_*$. If $\mu_*\in\R$, then
\eqref{eq:normality-projection-roots} gives
$
 P(\mu)=(\mu-\mu_*)\bigl(1-\phi(\mu)\bigr)
         \prod_{k=1}^d(\mu-\theta_k).
$
Because $\mu_*\notin[\ell,L]$ and $1-\phi$ has a simple zero
at $\mu_*$, this zero of $P$ has multiplicity two.
In either case, the $d-1$ zeros in $(\theta_1, \theta_d)$ and these two zeros $\mu_*$ and $\bar\mu_*$,
counted with multiplicity, exhaust the degree $d+1$ of $P$.
Thus cancellation of one factor $\mu-\bar\mu_*$ in
\eqref{eq:normality-projection-roots} leaves $\mu_*$ as the only
zero of $1-\phi$ outside $[\ell,L]$, and this zero is simple.

Choose $\Lambda$ sufficiently close to $[\ell,L]$ that
Lemma~\ref{lem:normality-geometry} gives
$\dist(\mathcal C_n,\Lambda)\ge2r$ for some fixed $r>0$ and all
large $n$. The uniform bound
$\sup_{|\mu-\mu_*|\le r}|\phi''(\mu)|\le C$ and
$|\phi'(\mu_*)|\ge c$ give, by Taylor's formula and after
decreasing $r$ if necessary,
\begin{equation}\label{eq:normality-projection-local-bound}
 |\phi(\mu)-1|
 \ge |\phi'(\mu_*)|\,|\mu-\mu_*|-\frac C2|\mu-\mu_*|^2
 \ge\frac c2|\mu-\mu_*|,
 \qquad |\mu-\mu_*|\le r.
\end{equation}
Let $E\subset\C\setminus[\ell,L]$ be a fixed compact
set containing the closed region between $\Lambda$ and
$\Gamma_\delta$, including $\mathcal C_n$, for all large $n$.
We claim that
\begin{equation}\label{eq:normality-projection-separation}
 \inf_{\substack{\mu\in E\\|\mu-\mu_*|\ge r}}
       |1-\phi(\mu)|\ge c_r>0,
 \qquad \inf_{\mu\in\Lambda}|1-\phi(\mu)|\ge c_r.
\end{equation}
The second inequality follows from the first since
$\dist(\mu_*,\Lambda)\ge2r$.
If the first inequality fails, compactness gives a subsequence,
still indexed by $n$, and maximizers $\mu_{*,n}\in\mathcal C_n$
with corresponding transforms $\phi_n$ such that
\begin{align*}
 &\nu_n\in E,\qquad |\nu_n-\mu_{*,n}|\ge r,
       \qquad |1-\phi_n(\nu_n)|\longrightarrow0,\\
 &\nu_n\longrightarrow\nu_\infty\in E,\qquad
   \mu_{*,n}\longrightarrow\mu_{*,\infty}\in E,
       \qquad |\nu_\infty-\mu_{*,\infty}|\ge r.
\end{align*}
By the
compactness of probability laws on $[\ell,L]$, pass to a further
subsequence of $\gamma_n=p/n$ on which $\gamma_n\to\gamma_\infty\in(1,\infty)$ and
$\mathsf H_n\Rightarrow\mathsf H$.
The denominators below are bounded away from zero on compact subsets
of $\C\setminus[\ell,L]$. Weak convergence therefore gives, for every compact
$E'\subset\C\setminus[\ell,L]$,
\begin{align*}
 \phi_n(\mu)
 &=\gamma_n\int\frac{\lambda^2}
       {(\lambda-\mu)(\lambda-\bar\mu_{*,n})}
       \,\mathsf H_n(\dd\lambda),\\
 \phi_\infty(\mu)
 &:=\gamma_\infty\int\frac{\lambda^2}
       {(\lambda-\mu)(\lambda-\bar\mu_{*,\infty})}
       \,\mathsf H(\dd\lambda),
 \qquad \sup_{\mu\in E'}|\phi_n(\mu)-\phi_\infty(\mu)|
       \longrightarrow0,\\
 |1-\phi_\infty(\nu_\infty)|
 &\le |1-\phi_n(\nu_n)|
       +\sup_{\mu\in E}|\phi_n(\mu)-\phi_\infty(\mu)|
       +|\phi_\infty(\nu_n)-\phi_\infty(\nu_\infty)|
       \longrightarrow0.
\end{align*}
Thus $\phi_\infty(\nu_\infty)=1$, whereas
$\phi_\infty(\mu)=O(|\mu|^{-1})$ as $|\mu|\to\infty$, so
$1-\phi_\infty$ is not identically zero. On the other hand, choose
$
 0<s<\frac13\min\bigl(
       \dist(\nu_\infty,[\ell,L]),
       |\nu_\infty-\mu_{*,\infty}|\bigr).
$
By the zero characterization above following \eqref{eq:normality-projection-roots}, $\mu_{*,n}$ is the only zero of $1-\phi_n$ outside $[\ell,L]$. Then, for all large $n$,
\[
 \{\mu:|\mu-\nu_\infty|<s,\ 1-\phi_n(\mu)=0\}
 =\varnothing.
\]
Hurwitz's theorem \citep[Chapter~5, Theorem~2]{Ahlfors1979},
applied on the disk $\{|\mu-\nu_\infty|<s\}$, gives $1-\phi_\infty(\nu_\infty)\ne0$,
a contradiction. This proves \eqref{eq:normality-projection-separation}.

We now are ready to bound \eqref{eq:normality-noise-projection} from below.
Take $t=1-K^{-1}$. For all sufficiently large $n$,
$(1-t)\sup_{|\mu-\mu_*|=r}|\phi(\mu)|\le C/K<cr/2$.
Thus \eqref{eq:normality-projection-local-bound} and Rouch\'e's
theorem give a unique zero $\mu_t$ of $1-t\phi$ in
$|\mu-\mu_*|<r$. Since $\phi(\mu_t)=t^{-1}$,
\eqref{eq:normality-projection-local-bound} yields
\begin{align*}
      \frac c2|\mu_t-\mu_*|
      \le |\phi(\mu_t)-1|
      =\frac{1-t}{t}=\frac1{K-1}\le\frac2K,
\end{align*}
which proves $\mu_t=\mu_*+O(K^{-1})$ and
$\dist(\mu_t,[\ell,L])\ge c$.
The bound on $\phi''$ also gives, for all sufficiently large $n$,
\begin{equation}\label{eq:normality-projection-local-bound-2}
\begin{aligned}
 &|(1-t\phi)'(\mu_t)|
 =t|\phi'(\mu_t)|
 \ge t\bigl(|\phi'(\mu_*)|-C|\mu_t-\mu_*|\bigr)
 \ge c/4,\\
 &0<\frac{\sqrt\ell}{L+C}
 \le |a_j(\mu_t)|
 =\frac{\sqrt{\lambda_j}}{|\lambda_j-\mu_t|}
 \le\frac{\sqrt L}{c}<\infty,
 \text{ for all } 1\le j\le p.
\end{aligned}
\end{equation}
The Cauchy--Schwarz inequality gives
$|\phi(\mu)|^2\le D_n(\mu)$, so
$D_n(\mu_t)\ge t^{-2}>1$. Thus $\mu_t$ lies inside
$\mathcal C_n$, hence inside $\Gamma_\delta$, and its proximity
to $\mu_*$ keeps it outside $\Lambda$.
Equation~\eqref{eq:normality-projection-separation} excludes any
other zero between the contours. 
We have
\[
\left|\frac1{2\pi i}\oint_\Lambda
\frac{\psi a_j}{1-t\phi}\,d\mu\right|
\le
\frac{\operatorname{length}(\Lambda)}{2\pi}
\frac{\|\psi\|_{\infty,\Lambda}\|a_j\|_{\infty,\Lambda}}
     {\inf_\Lambda|1-t\phi|}
\le C,
\]
where the last inequality is by \eqref{eq:normality-projection-separation} and thus $\inf_{\Lambda}|1-t\phi|
\ge \inf_{\Lambda}|1-\phi|
   -K^{-1}\sup_{\Lambda}|\phi|
\ge c_r-C/K\ge c_r/2.$
The residue theorem
\citep[Chapter~4, Theorem~17]{Ahlfors1979} therefore gives
\begin{equation}\label{eq:normality-noise-residue}
 \mathcal I_\psi\!\left(\frac{a_j}{1-t\phi}\right)
 =-\frac{\psi(\mu_t)a_j(\mu_t)}{t\phi'(\mu_t)}
   +\frac1{2\pi\ii}\oint_\Lambda
       \frac{\psi a_j}{1-t\phi}\dd\mu=-\frac{\psi(\mu_t)a_j(\mu_t)}{t\phi'(\mu_t)}+O(1).
\end{equation}
Since $\operatorname{Re}q_\delta(\mu_t)\ge M_n-C/K$, 
$|\psi(\mu_t)|\ge cA_{n,K}$. Also,
\eqref{eq:normality-level-separation} and
$\mu_t=\mu_*+O(K^{-1})$ give $|\phi'(\mu_t)|\le C$.
Together with \eqref{eq:normality-projection-local-bound-2}
and $A_{n,K}\to\infty$, \eqref{eq:normality-noise-residue} gives $|\mathcal I_\psi(\frac{a_j}{1-t\phi})|
\ge cA_{n,K}$.
Equation~\eqref{eq:normality-noise-projection} implies
$\sigma_\varepsilon^2(g_K)\ge c\sigma^2A_{n,K}^2/(nK)$.

\proofstep{The design contribution.} We now bound $\sigma_X^2(g_K)=\sigma^2(g_K)-\sigma_\varepsilon^2(g_K)$ from below, where $\sigma_\varepsilon^2(g_K)$ is given by \eqref{eq:normality-variance-noise-part}.
Set
$u_\mu=\Sigma^{1/2}R_\mu v$,
$w_\mu=\Sigma^{1/2}R_\mu\beta$, and
$C_\mu=\mu(u_\mu w_\mu^\top+w_\mu u_\mu^\top)/2$.
Here $(C_\mu)_{jj}=\mu b_j(\mu)$ and
$
 2\tr(C_\mu C_\nu)
 =\mu\nu\bigl((v^\top A_{\mu\nu}v)(\beta^\top A_{\mu\nu}\beta)
                  +(v^\top A_{\mu\nu}\beta)^2\bigr).
$
Thus \eqref{eq:normality-variance}, \eqref{eq:normality-kernel}, and
\eqref{eq:normality-variance-noise-part} give
\begin{align}\label{eq:normality-variance-design-part}
 \sigma_X^2(g_K)
 &=\frac1{4\pi^2n}\oint_{\Gamma_\delta}\!\oint_{\Gamma_\delta}
       \psi(\mu)\overline{\psi(\nu)}\biggl(
       \frac{2\tr(C_\mu C_\nu^*)}{1-\mathcal Q_{\mu,\bar\nu}}
       +\chi_4\sum_{j=1}^p(C_\mu)_{jj}\overline{(C_\nu)_{jj}}
       \biggr)\dd\mu\dd\bar\nu\\
       &=\frac1{4\pi^2n}\oint_{\Gamma_\delta}\!\oint_{\Gamma_\delta}
       \psi(\mu)\overline{\psi(\nu)}\biggl(
       {2\tr(C_\mu C_\nu^*)}\sum_{m\ge 1}\mathcal Q_{\mu,\bar\nu}^m + {2\tr(C_\mu C_\nu^*)}
       +\chi_4\sum_{j=1}^p(C_\mu)_{jj}\overline{(C_\nu)_{jj}}
       \biggr)\dd\mu\dd\bar\nu    \notag 
\end{align}
using $C_\mu^\top=C_\mu$, $C_{\bar\nu}=C_\nu^*$, and \eqref{eq:normality-geometric-series}.
The last two terms above have the combined kernel
\begin{equation}\label{eq:normality-design-zero-order}
 \begin{aligned}
 &2\tr(C_\mu C_\nu^*)
       +\chi_4\sum_{j=1}^p(C_\mu)_{jj}\overline{(C_\nu)_{jj}}
=2\sum_{i,j=1}^p(C_\mu)_{ij}\overline{(C_\nu)_{ij}}
       +\chi_4\sum_{j=1}^p(C_\mu)_{jj}\overline{(C_\nu)_{jj}}\\
 &\quad=4\sum_{1\le i<j\le p}(C_\mu)_{ij}\overline{(C_\nu)_{ij}}
       +(2+\chi_4)\sum_{j=1}^p(C_\mu)_{jj}\overline{(C_\nu)_{jj}}.
 \end{aligned}
\end{equation}
The kernel in \eqref{eq:normality-design-zero-order} is positive
semidefinite, since
$2+\chi_4=\E Z_{11}^4-1\ge0$. Similar to \eqref{eq:normality-noise-projection},
applying the
projection argument in \eqref{eq:normality-tensor-projection}
entrywise to the terms
$m\ge1$ in \eqref{eq:normality-variance-design-part} gives
\begin{equation}\label{eq:normality-design-projection}
      \sigma_X^2(g_K)\ge\frac2n\sum_{m\ge1}
      \|\mathcal I_\psi(C_\mu\phi^m)\|_F^2
      \ge\frac{2(1-t^2)}n
      \left\|\mathcal I_\psi\!\left(
                \frac{C_\mu\phi}{1-t\phi}\right)\right\|_F^2.  
\end{equation}
Applying the residue calculation in \eqref{eq:normality-noise-residue}
entrywise and using $\phi(\mu_t)=t^{-1}$, we obtain
\begin{equation}\label{eq:normality-design-residue}
 \begin{aligned}
 \mathcal I_\psi\!\left(\frac{C_\mu\phi}{1-t\phi}\right)
 &=-\frac{\psi(\mu_t)}{t^2\phi'(\mu_t)}C_{\mu_t}
 +\frac1{2\pi\ii}\oint_\Lambda
       \frac{\psi(\mu)\phi(\mu)C_\mu}{1-t\phi(\mu)}\dd\mu.
 \end{aligned}
\end{equation}
Still take $t=1-K^{-1}$. 
Since $\Lambda$ is fixed and separated from $[\ell,L]$ and
$\|v\|_2=1$, we have
$
 \sup_{\mu\in\Lambda}\|C_\mu\|_F
 \le\sup_{\mu\in\Lambda}
       |\mu|\,\|u_\mu\|_2\|w_\mu\|_2
 \le C\|\beta\|_2.
$
The same contour bounds used in \eqref{eq:normality-noise-residue} give
\begin{align*}
 \left\|\frac1{2\pi\ii}\oint_\Lambda
       \frac{\psi\phi C_\mu}{1-t\phi}\dd\mu\right\|_F
 &\le\frac{\operatorname{length}(\Lambda)}{2\pi}
       \frac{\|\psi\|_{\infty,\Lambda}\|\phi\|_{\infty,\Lambda}}
            {\inf_{\mu\in\Lambda}|1-t\phi(\mu)|}
       \sup_{\mu\in\Lambda}\|C_\mu\|_F\le C\|\beta\|_2.
\end{align*}
The bounds on $\psi(\mu_t)$ and $\phi'(\mu_t)$ used in
\eqref{eq:normality-noise-residue} also give
$
 {\left|\frac{\psi(\mu_t)}{t^2\phi'(\mu_t)}\right|}
 \ge c|\psi(\mu_t)|\ge cA_{n,K}.
$
Moreover, \eqref{eq:normality-level-separation} and
$\mu_t=\mu_*+O(K^{-1})$ give $|\mu_t|\ge c$, while
\eqref{eq:normality-projection-local-bound-2} bounds the singular
values of $\Sigma^{1/2}R_{\mu_t}$ above and away from zero. Hence
\[
 \|C_{\mu_t}\|_F^2
 =\frac{|\mu_t|^2}{2}
   \bigl(\|u_{\mu_t}\|_2^2\|w_{\mu_t}\|_2^2
          +|u_{\mu_t}^*w_{\mu_t}|^2\bigr)
 \ge c\|\beta\|_2^2.
\]
Taking Frobenius norms in \eqref{eq:normality-design-residue}
and using $A_{n,K}\to\infty$, we conclude that
\[
 \left\|\mathcal I_\psi\!\left(
       \frac{C_\mu\phi}{1-t\phi}\right)\right\|_F
 \ge (cA_{n,K}-C)\|\beta\|_2
 \ge cA_{n,K}\|\beta\|_2.
\]
As $1-t^2\asymp K^{-1}$, \eqref{eq:normality-design-projection} gives
$\sigma_X^2(g_K)\ge c\|\beta\|_2^2A_{n,K}^2/(nK)$.
Finally, $\sigma^2+\|\beta\|_2^2\asymp\tau_n^2$ follows from
$\ell I_p\preceq\Sigma\preceq LI_p$. This proves the lower bound in
\eqref{eq:normality-J-bounds}.
\end{proof}

\subsection{Estimating the Degree Parameter}\label{app:pf:normality-estimated-M}

\begin{lemma}\label{lem:normality-estimated-M}
Under Assumption~\ref{ass:bocca-normality}, the estimator in
\eqref{eq:normality-estimated-M} satisfies
\[
 \widehat M_n-M_n=O_P(n^{-1/8}).
\]
\end{lemma}

\begin{proof}
Recall from \eqref{eq:normality-support-regularity} that
$\supp\varrho=[e_-,e_+]$, where $e_\pm=\zeta(a_\pm)$ and
$a_\pm=a_\pm(1)$ are the real crossings in
\eqref{eq:normality-real-crossings}. The boundary characterization in
\citet[Theorems~1.1 and~4.2]{SilversteinChoi1995}, together with
\eqref{eq:imaginary-part-zeta}, gives
\begin{equation}\label{eq:normality-boundary-parametrization}
 \mathcal C_n\cap\{\operatorname{Im}\mu\ge0\}
 =\left\{-\frac1{m(E+\ii0_+)}:e_-\le E\le e_+\right\}.
\end{equation}
Indeed, \eqref{eq:normality-boundary-density} and
\eqref{eq:normality-boundary-monotonicity} parametrize the upper arc
$\mathcal C_n\cap\C^+$ by $e_-<E<e_+$.
The endpoint values $-1/m(e_\pm+\ii0_+)=a_\pm$ give the endpoints
of the closed upper arc $\mathcal C_n\cap\{\operatorname{Im}\mu\ge0\}$.

The regularity bounds verified in the proof of
Lemma~\ref{lem:normality-moments} are uniform in $n$.
With $\eta_n=n^{-1/4}$,
\citet[Corollary~3.9 and Theorem~3.12]{KnowlesYin2017} give,
for each fixed sufficiently small $\epsilon>0$,
\begin{equation}\label{eq:normality-degree-local-law}
      \sup_{e_--c_0\le E\le e_++c_0}|\widehat m(E+\ii\eta_n)-m(E+\ii\eta_n)|
 =O_P(n^{-3/4+\epsilon}),
 |s_1-e_-|+|s_n-e_+|=O_P(n^{-2/3+\epsilon}),
\end{equation}
where $c_0>0$ is a sufficiently small constant such that $\PP([s_1,s_n]\subset [e_--c_0,e_++c_0])\to 1$.

For all $z,w$ in the region
$
 \{E+\ii\eta:e_--c_0\le E\le e_++c_0,\ 0\le\eta\le c_0\},
$
we claim
\begin{equation}\label{eq:normality-degree-local-law-2}
      |m(z)|\ge c,\qquad
      \left|\frac1{m(z)}-\frac1{m(w)}\right|\le C|z-w|^{1/2},  
\end{equation}
where $m(E)=m(E+\ii0_+)$ on the real axis.
To prove these bounds, \eqref{eq:normality-boundary-parametrization}
and \eqref{eq:normality-level-separation} first give
$c\le|-1/m(E+\ii0_+)|\le C$ for $e_-\le E\le e_+$.
On $\mathcal C_n\cap\C^+$, the derivative identity
\eqref{eq:normality-boundary-monotonicity} and the endpoint charts
\eqref{eq:normality-real-crossing-charts} give, for each fixed
sufficiently small $r>0$, $\operatorname{Re}\zeta'(\mu)\ge c_r>0$ 
for all $\mu\in\mathcal C_n\cap\C^+$ with $\min\{|\mu-a_-|,|\mu-a_+|\}\ge r$.
The implicit-function theorem therefore gives uniform inverse
bounds on this part of the arc. At $\mu=a_\pm$,
\eqref{eq:normality-edge-regularity} and Taylor's formula give
\[
 \zeta(\mu)-e_\pm
 =\frac{\zeta''(a_\pm)}2(\mu-a_\pm)^2
   +O(|\mu-a_\pm|^3),\qquad |\zeta''(a_\pm)|\ge c,
\]
which gives the $1/2$-H\"older bound near the support endpoints
$z=e_-$ and $z=e_+$. Together these inverse bounds prove \eqref{eq:normality-degree-local-law-2}.

Compare the empirical grid with
$E_j^\circ=e_-+j(e_+-e_-)/N_n$.
By $\eta_n=n^{-1/4}$, \eqref{eq:normality-degree-local-law}, and \eqref{eq:normality-degree-local-law-2},
\begin{align*}
 \max_{0\le j\le N_n}
 \left|\widehat\mu_j+\frac1{m(E_j^\circ+\ii0_+)}\right|
 &\le Cn^{-1/8}
   +O_P\bigl(n^{-3/4+\epsilon}+n^{-1/3+\epsilon/2}\bigr)=O_P(n^{-1/8}).
\end{align*}
This error is $o_P(\delta)$, whereas
$\operatorname{dist}(\mathcal C_n,\Gamma_\delta)\ge c\delta$
and $\mathcal C_n$ is uniformly separated from the slit.
Hence all transformed grid points lie in
$\Ddisk_\delta\setminus[\ell,L]$ with probability tending to one.

The derivatives of $h_\delta u_\delta$ are uniformly bounded near
$\mathcal C_n$, including just outside $\mathcal C_n$ by
\eqref{eq:outer-collar-app}. The H\"older bound in \eqref{eq:normality-degree-local-law-2} and the grid size
$O(N_n^{-1})$ give a discretization error $O(N_n^{-1/2})$ in
\eqref{eq:normality-boundary-parametrization}.
Conjugation symmetry makes the maximum in {\eqref{eq:normality-boundary}} equal to the maximum over $ {\mathcal C_n\cap\{\operatorname{Im}\mu\ge0\}}$, so we have
\[
 \left|h_\delta\max_{0\le j\le N_n}u_\delta(\widehat\mu_j)-M_n\right|
 \le O_P(n^{-1/8})+O(N_n^{-1/2})=O_P(n^{-1/8}).
\]
Finally, $M_n-h_\delta\alpha_\delta\ge c$ by
Lemma~\ref{lem:normality-geometry}, so the lower truncation of $\widehat M_n=\max\{
      h_\delta\alpha_\delta,\,
      h_\delta\max_{0\le j\le N_n}u_\delta(\widehat\mu_j)\}$ in \eqref{eq:normality-estimated-M} is
also inactive with probability tending to one.
\end{proof}

\subsection{Estimating the Variance}

\begin{lemma}[Variance estimation]\label{lem:normality-variance-estimation}
Under Assumption~\ref{ass:bocca-normality}, for any fixed
$0<c_-<c_+<\infty$,
\begin{equation}\label{eq:normality-variance-consistency}
 \max_{\substack{K\in\mathbb N\\c_-\log n\le K\le c_+\log n}}
 \left|\frac{\widehat\sigma^2(g_K)}{\sigma^2(g_K)}-1\right|
 \overset{p}{\longrightarrow}0.
\end{equation}
In particular, for $\tilde K$ in \eqref{eq:normality-degree},
${\widehat\sigma^2(g_{\widetilde K})}
      /{\sigma^2(g_{\widetilde K})}\overset{p}{\longrightarrow}1$.
\end{lemma}

\begin{proof}
As both $\sigma^2(g_K)$ and $\widehat\sigma^2(g_K)$ scale quadratically with $Y$, without loss of generality, we can assume $\tau_n=1$ and replace $Y$ by $Y/\tau_n$, where
$\tau_n^2=\sigma^2+\beta^\top\Sigma\beta$. So
$\|\beta\|_2\le\ell^{-1/2}$ and $\sigma\le1$.
{Since $\underline\sigma^2\le\tau_n^2\le\bar\sigma^2+LB_0^2$, this rescaling preserves uniform noise moment bounds and a positive lower bound on the noise standard deviation.}
Throughout the proof, we condition on the event $\mathcal A'$ in
Lemma~\ref{lem:normality-moments} with probability tending to one.
{The empirical branch exists on $\mathcal A'$, so the safeguard in the definition of $\widehat\sigma^2(g_K)$ is inactive there and does not affect convergence in probability.}

\proofstep{Cross-product terms.}
For a fixed small $\epsilon>0$, we aim to show that each row below has estimation rate $O_P(n^{-1/2+\epsilon}\delta^{-C_\epsilon})$ uniformly for
$(\mu,\nu)\in\Gamma_{\delta}^*\times\Gamma_{\delta}^*$:
\[
\begin{array}{c|c}
 \text{estimator}&\text{population truth}\\[2pt]\hline
 n^{-1}(Xv)^\top\widehat Q_\mu\widehat Q_\nu Xv
 &\displaystyle\frac{v^\top A_{\mu\nu}v}{1-\mathcal Q_{\mu\nu}}\\[8pt]
 n^{-1}(Xv)^\top\widehat Q_\mu\widehat Q_\nu Y
 &\displaystyle\frac{v^\top A_{\mu\nu}\beta}{1-\mathcal Q_{\mu\nu}}\\[8pt]
 n^{-1}Y^\top\widehat Q_\mu\widehat Q_\nu Y
 &\displaystyle\frac{\beta^\top A_{\mu\nu}\beta+\sigma^2/(\mu\nu)}
                         {1-\mathcal Q_{\mu\nu}}\\[8pt]
 n^{-1}\tr(\widehat Q_\mu\widehat Q_\nu)
 &\displaystyle\frac1{\mu\nu(1-\mathcal Q_{\mu\nu})}.
\end{array}
\]
The model $Y=X\beta+\varepsilon$ gives 
\begin{align*}
 \frac1n(Xv)^\top\widehat Q_\mu\widehat Q_\nu Y
 &=\frac1n(Xv)^\top\widehat Q_\mu\widehat Q_\nu X\beta
   +\frac1n(Xv)^\top\widehat Q_\mu\widehat Q_\nu\varepsilon,\\
 \frac1nY^\top\widehat Q_\mu\widehat Q_\nu Y
 &=\frac1n(X\beta)^\top\widehat Q_\mu\widehat Q_\nu X\beta
   +\frac{\sigma^2}{n}\tr(\widehat Q_\mu\widehat Q_\nu)
   +\frac2n(X\beta)^\top\widehat Q_\mu\widehat Q_\nu\varepsilon
   +\frac1n\bigl(\varepsilon^\top\widehat Q_\mu\widehat Q_\nu\varepsilon
               -\sigma^2\tr(\widehat Q_\mu\widehat Q_\nu)\bigr).
\end{align*}
Therefore, we essentially have four types of terms:
\begin{equation}\label{eq:var-four-types}
      \frac1n(Xu)^\top Q(\xi_1)Q(\xi_2)Xw, \quad \frac1n\tr(\widehat Q_\mu\widehat Q_\nu),  \quad \frac1n(Xu)^\top Q(\xi_1)Q(\xi_2)\varepsilon,   \quad \frac1n\varepsilon^\top Q(\xi_1)Q(\xi_2)\varepsilon.
\end{equation}
We start with the first type. Fix deterministic $u,w\in\C^p$ with bounded norms,
and put $\xi_1=\zeta(\mu)$, $\xi_2=\zeta(\nu)$, and
$U(\xi)=u^\top(G_{11}(\xi)-\Pi_{11}(\xi))w$. Recall 
$h[\xi_1,\xi_2]=(h(\xi_1)-h(\xi_2))/(\xi_1-\xi_2)$, with
$h[\xi,\xi]=h'(\xi)$.
By \eqref{eq:covariance-functions}, \eqref{eq:linearization}, and
\eqref{eq:recursion-divided-difference-identities},
\begin{align*}
 \frac1n(Xu)^\top Q(\xi_1)Q(\xi_2)Xw
 =u^\top G_{11}[\xi_1,\xi_2]w
 =S(\xi_1,\xi_2)u^\top M(\xi_1,\xi_2)w+U[\xi_1,\xi_2]
 =\frac{u^\top A_{\mu\nu}w}{1-\mathcal Q_{\mu\nu}}
      +U[\xi_1,\xi_2],
\end{align*}
where the last equality follows from \eqref{eq:normality-calibration}, which gives
$D(\xi_1)^{-1}=-\mu R_\mu$ and $D(\xi_2)^{-1}=-\nu R_\nu$.
Choose sufficiently small $c>0$ so that
$B(\xi_j,2c\delta^2)\subset\mathcal V_n$ for $j=1,2$.
The local law \eqref{eq:normality-local-laws} and Cauchy's
estimate \eqref{eq:Cauchy-derivative} give, uniformly for
$\mu,\nu\in\Gamma_\delta^*$,
\begin{align*}
 |U[\xi_1,\xi_2]|
 &\le
 \begin{cases}
 \displaystyle\sup_{0\le t\le1}
   |U'(\xi_2+t(\xi_1-\xi_2))|,&|\xi_1-\xi_2|\le c\delta^2,\\[4pt]
 \displaystyle\frac{|U(\xi_1)|+|U(\xi_2)|}{|\xi_1-\xi_2|},
   &|\xi_1-\xi_2|>c\delta^2,
 \end{cases}\\
 &\le C\delta^{-2}\sup_{\xi\in\mathcal V_n}|U(\xi)|
 =O_P(n^{-1/2+\epsilon}\delta^{-C_\epsilon}).
\end{align*}
The first case includes $\xi_1=\xi_2$, with
$U[\xi_1,\xi_1]=U'(\xi_1)$.

Applying the resolvent identity $\widehat Q_\mu-Q(\zeta(\mu))
=(\widehat\zeta(\mu)-\zeta(\mu))
      \widehat Q_\mu Q(\zeta(\mu))$, we have
\begin{align*}
 \widehat Q_\mu\widehat Q_\nu
       -Q(\zeta(\mu))Q(\zeta(\nu))
 &=\bigl(\widehat Q_\mu-Q(\zeta(\mu))\bigr)\widehat Q_\nu
       +Q(\zeta(\mu))\bigl(\widehat Q_\nu-Q(\zeta(\nu))\bigr)\\
 &=(\widehat\zeta(\mu)-\zeta(\mu))
       \widehat Q_\mu Q(\zeta(\mu))\widehat Q_\nu
       +(\widehat\zeta(\nu)-\zeta(\nu))
       Q(\zeta(\mu))\widehat Q_\nu Q(\zeta(\nu)).
\end{align*}
By \eqref{eq:normality-spectral-event} and
\eqref{eq:normality-support-regularity},
$\|X\|_\op^2/n=\|\widehat S\|_\op\le C$ on $\mathcal A'$.
Together with \eqref{eq:normality-empirical-control}, the above results give
\begin{align*}
 &\sup_{\mu,\nu\in\Gamma_\delta^*}
 \left|\frac1n(Xu)^\top\bigl(
     \widehat Q_\mu\widehat Q_\nu
       -Q(\zeta(\mu))Q(\zeta(\nu))\bigr)Xw\right|\\
 &\qquad\le C\frac{\|X\|_\op^2}{n}
       \sup_{\omega\in\Gamma_\delta^*}
           |\widehat\zeta(\omega)-\zeta(\omega)|\sup_{\mu,\nu\in\Gamma_\delta^*}\bigl(
       \|\widehat Q_\mu Q(\zeta(\mu))\widehat Q_\nu\|_\op
       +\|Q(\zeta(\mu))\widehat Q_\nu Q(\zeta(\nu))\|_\op\bigr)\\
 &\qquad\le C\frac{\|X\|_\op^2}{n}\delta^{-6}
       \sup_{\omega\in\Gamma_\delta^*}
           |\widehat\zeta(\omega)-\zeta(\omega)|
 \le Cn^{-3/4}\delta^{-6},
\end{align*}
where the last two inequalities are by
\eqref{eq:normality-spectral-event} and
\eqref{eq:normality-empirical-control}. Combining the preceding estimates yields
\begin{equation}\label{eq:normality-design-product-error}
 \sup_{\mu,\nu\in\Gamma_\delta^*}
 \left|\frac1n(Xu)^\top\widehat Q_\mu\widehat Q_\nu Xw
       -\frac{u^\top A_{\mu\nu}w}{1-\mathcal Q_{\mu\nu}}\right|
 =O_P(n^{-1/2+\epsilon}\delta^{-C_\epsilon}).
\end{equation}
The bound for the trace term $n^{-1}\tr(Q(\xi_1)Q(\xi_2))$ is similar. Combining
$n^{-1}\tr(Q(\xi_1)Q(\xi_2))=\widehat m[\xi_1,\xi_2]$, the rate of $\hat m - m$ in \eqref{eq:normality-local-laws}, and the identity
$m[\xi_1,\xi_2]=S(\xi_1,\xi_2)$ by
\eqref{eq:recursion-divided-difference-identities}, we have
\begin{equation}\label{eq:normality-trace-product-error}
 \sup_{\mu,\nu\in\Gamma_\delta^*}
 \left|\frac1n\tr(\widehat Q_\mu\widehat Q_\nu)
       -\frac1{\mu\nu(1-\mathcal Q_{\mu\nu})}\right|
 =O_P(n^{-1/2+\epsilon}\delta^{-C_\epsilon}).
\end{equation}

We then deal with the last two types in \eqref{eq:var-four-types} involving the noise. 
On $\mathcal A'$, just like above, \eqref{eq:normality-empirical-control} and
\eqref{eq:normality-spectral-event} give
$\|\widehat Q_\mu\widehat Q_\nu\|_\op\le C\delta^{-4}$ and
$\|X\|_\op\le C\sqrt n$.
 For $u=v$ or $u=\beta$,
Lemma~\ref{lem:normality-noise-form-moments} gives
\begin{equation}\label{eq:normality-noise-product-moments}
 \begin{aligned}
 \Big\|\frac1n(Xu)^\top\widehat Q_\mu\widehat Q_\nu\varepsilon
       \Big\|_{{L^r(\varepsilon\mid X)}}
 &\le\frac{C_r\sigma}{n}\|X\|_\op\|u\|_2
       \|\widehat Q_\mu\widehat Q_\nu\|_\op
 \le C_r n^{-1/2}\delta^{-4},\\
 \Big\|\frac{\varepsilon^\top\widehat Q_\mu\widehat Q_\nu
                    \varepsilon
             -\sigma^2\tr(\widehat Q_\mu\widehat Q_\nu)}n
       \Big\|_{{L^r(\varepsilon\mid X)}}
 &\le\frac{C_r\sigma^2}{n}
       \|\widehat Q_\mu\widehat Q_\nu\|_F
 \le C_r n^{-1/2}\delta^{-4},
 \end{aligned}
\end{equation}
where $L^r(\varepsilon\mid X)$ denotes the $L^r$ norm of $\varepsilon$ given $X$, and we use $\sigma\le1$, $\|v\|_2=1$,
$\|\beta\|_2\le\ell^{-1/2}$, and $\|B\|_F\le\sqrt n\|B\|_\op$.

We then use an $\epsilon$-net argument to make these bounds in \eqref{eq:normality-noise-product-moments} uniform for $(\mu,\nu)\in\Gamma_{\delta}^*\times\Gamma_{\delta}^*$. Let $T(\mu,\nu)$ denote either
term on the left of
\eqref{eq:normality-noise-product-moments}, i.e.,
\[
  T(\mu,\nu) = \frac1n(Xu)^\top\widehat Q_\mu\widehat Q_\nu\varepsilon \text{ or }  T(\mu,\nu) =  \frac{\varepsilon^\top\widehat Q_\mu\widehat Q_\nu
  \varepsilon
-\sigma^2\tr(\widehat Q_\mu\widehat Q_\nu)}n.
\]

By Lemma~\ref{lem:normality-geometry},
$\operatorname{length}(\Gamma_\delta^*)\le C$.
Choose a net $\mathcal N_n\subset\Gamma_\delta^*$ with arclength
spacing at most $n^{-2}$ and $|\mathcal N_n|\le Cn^2$.
Applying a union bound and Markov's inequality to
\eqref{eq:normality-noise-product-moments} gives, for $r>4/\epsilon$,
\begin{align*}
 &\PP\Big(\max_{\mu,\nu\in\mathcal N_n}|T(\mu,\nu)|
          >n^{-1/2+\epsilon}\delta^{-4}\,\Big|\,X\Big)
\le C_r|\mathcal N_n|^2n^{-r\epsilon}
 \le C_r n^{4-r\epsilon}=o(1)
 \quad\text{on }\mathcal A'.
\end{align*}
Next control the error between the points of $\mathcal N_n$. By
\eqref{eq:normality-level-separation} and
$|\mathcal Q_{\mu\mu}|\le1-\delta$, we have
$
 \inf_{\mu\in\Gamma_\delta^*}
       |\frac1{\mu^2(1-\mathcal Q_{\mu\mu})}|\ge c.
$
Thus the diagonal case of \eqref{eq:normality-trace-product-error}
and \eqref{eq:zhat-derivative} imply
\[
 \widehat\zeta'(\mu)
 =\bigl(\mu^2 n^{-1}\tr(\widehat Q_\mu^2)\bigr)^{-1},\qquad
 \sup_{\mu\in\Gamma_\delta^*}|\widehat\zeta'(\mu)|=O_P(1),\qquad
 \partial_\mu\widehat Q_\mu
 =\widehat\zeta'(\mu)\widehat Q_\mu^2.
\]
Consequently, differentiating $T(\mu, \nu)$ for either case, we have
\begin{align*}
 \sup_{\mu,\nu\in\Gamma_\delta^*}
   \bigl(|\partial_\mu T(\mu,\nu)|+|\partial_\nu T(\mu,\nu)|\bigr)
 &\le C\delta^{-6}\left(1+\frac{\|\varepsilon\|_2^2}{n}\right)
       \sup_{\omega\in\Gamma_\delta^*}|\widehat\zeta'(\omega)|
       =O_P(\delta^{-6}),
 \end{align*}
since $\E\|\varepsilon\|_2^2/n=\sigma^2\le1$.
We then have
\begin{equation}\label{eq:normality-noise-products-uniform}
 \sup_{\mu,\nu\in\Gamma_\delta^*}|T(\mu,\nu)|
 \le\max_{\mu,\nu\in\mathcal N_n}|T(\mu,\nu)|
       +O_P(n^{-2}\delta^{-6})
 =O_P(n^{-1/2+\epsilon}\delta^{-C_\epsilon}).
\end{equation}

In summary, \eqref{eq:normality-design-product-error},
\eqref{eq:normality-trace-product-error}, and
\eqref{eq:normality-noise-products-uniform} now give the claimed rate for all the rows of the table above. 
For $\mu,\nu\in\Gamma_\delta^*$,
\eqref{eq:normality-level-separation} and
Lemma~\ref{lem:normality-geometry} give
$c\le|\mu|,|\nu|\le C$ and
$\delta\le|1-\mathcal Q_{\mu\nu}|\le2$ for fixed $c,C>0$.
Hence, uniformly for $\mu,\nu\in\Gamma_\delta^*$,
\begin{equation}\label{eq:normality-trace-equivalent-bounds}
 \frac1{2C^2}
 \le\left|\frac1{\mu\nu(1-\mathcal Q_{\mu\nu})}\right|
 \le\frac1{c^2\delta}.
\end{equation}
Also, \eqref{eq:normality-covariance-kernel-auxiliary} and
\eqref{eq:normality-level-separation} give
$\|A_{\mu\nu}\|_\op\le C$. Since
$\|v\|_2=1$, $\|\beta\|_2\le\ell^{-1/2}$, and $\sigma\le1$,
\[
 \frac{\max\bigl\{
       |v^\top A_{\mu\nu}v|,\,
       |v^\top A_{\mu\nu}\beta|,\,
       |\beta^\top A_{\mu\nu}\beta+\sigma^2/(\mu\nu)|
       \bigr\}}{|1-\mathcal Q_{\mu\nu}|}
 \le C\delta^{-1}.
\]
Moreover, \eqref{eq:normality-trace-product-error} and
\eqref{eq:normality-trace-equivalent-bounds} imply
\[
 \PP\biggl(
       \inf_{\mu,\nu\in\Gamma_\delta^*}
       \left|\frac1n\tr(\widehat Q_\mu\widehat Q_\nu)\right|
       \ge\frac1{4C^2}\biggr)\longrightarrow1.
\]
These bounds control the reciprocal trace and the products in
\eqref{eq:normality-estimated-kernel}.
Substituting the estimates above into \eqref{eq:normality-estimated-kernel} yields
\begin{equation}\label{eq:normality-estimated-kernel-error}
 \sup_{\mu,\nu\in\Gamma_{\delta}^*}
 \Big|\widehat{\mathcal K}_n(\mu,\nu)-\mathcal K_n(\mu,\nu)
       +\chi_4\mu\nu\sum_{j=1}^p b_j(\mu)b_j(\nu)\Big|
 =O_P(n^{-1/2+\epsilon}\delta^{-C_\epsilon}).
\end{equation}

\proofstep{Fourth cumulant.} We now show the fourth cumulant term in \eqref{eq:normality-estimated-kernel-error} is negligible when using $g_K$ for the variance.
Let $\widehat\sigma_0^2(g_K)$ denote the double integral in
\eqref{eq:normality-estimated-variance} before taking the maximum,
{and set $\widehat\sigma_0^2(g_K)=0$ when the empirical branch does not exist}.
By \eqref{eq:normality-empirical-control}, the resolvent identity,
and \eqref{eq:zhat-derivative},
$
 \big(\frac1n\tr(\widehat Q_\mu\widehat Q_\nu)\big)^{-1} = \mu\nu \widehat\zeta[\mu, \nu].
$
As in the proof of
Lemma~\ref{lem:global-fourth-moment},
\eqref{eq:normality-empirical-control} gives injectivity of
$\widehat\zeta$ and $\widehat\zeta'(\mu)\ne0$ on $\mathcal R_n$.
Thus $(\frac1n\tr(\widehat Q_\mu\widehat Q_\nu))^{-1}$ and $\widehat{\mathcal K}_n$ are holomorphic
on a neighborhood of $\mathcal R_n\times\mathcal R_n$, with
$
 \widehat{\mathcal K}_n(\bar\mu,\bar\nu)
 =\overline{\widehat{\mathcal K}_n(\mu,\nu)}$ and $\widehat\sigma_0^2(g_K)\in\R$. Also, $|\mathcal Q_{\mu\nu}|\le1-\delta$ on $\mathcal R_n\times\mathcal R_n$
from Lemma~\ref{lem:normality-geometry}.
By Cauchy's theorem, we can {move} the integration contours of both $ {\widehat\sigma_0^2(g_K)}$ and $\sigma^2(g_K)$ from $\Gamma_\delta$
to $\Gamma_\delta^*$.

Let $\Lambda$ be the fixed inner contour from the proof of
Lemma~\ref{lem:normality-variance-bounds}.
Since the poles of $b_j$ lie in $[\ell,L]$ and
$b_j(\mu)=O(|\mu|^{-2})$ at infinity,
\[
 \oint_\Lambda b_j(\mu)\dd\mu=0,\qquad
 \oint_{\Gamma_\delta^*}(1-g_K(\mu))b_j(\mu)\dd\mu
 =-\oint_\Lambda g_K(\mu)b_j(\mu)\dd\mu.
\]
Consequently, the fourth cumulant term in $ {\sigma^2(g_K)}$ can be written as
\begin{equation*}\label{eq:normality-variance-error-decomposition}
 \begin{aligned}
      -\frac{\chi_4}{4\pi^2n}
      \oint_{\Gamma_\delta^*}\!\oint_{\Gamma_\delta^*}
     {(1-g_K(\mu))(1-g_K(\nu))} \sum_{j=1}^p b_j(\mu)b_j(\nu)
      \dd\mu\dd\nu = \frac{\chi_4}{n}\sum_{j=1}^p
    \biggl(\frac1{2\pi\ii}\oint_\Lambda
                 g_K(\mu)b_j(\mu)\dd\mu\biggr)^2
 \end{aligned}
\end{equation*}
The same cosh and resolvent estimates used in the proof of
Lemma~\ref{lem:normality-variance-bounds} apply on $\Lambda$.
By \eqref{eq:gK}, the choice
$\sup_\Lambda\operatorname{Re}q_\delta\le h_\delta\alpha_\delta-c$,
and \eqref{eq:normality-covariance-kernel-auxiliary},
\begin{align*}
 \|g_K\|_{\infty,\Lambda}
 &\le\frac{\cosh(K(h_\delta\alpha_\delta-c))}
              {\cosh(Kh_\delta\alpha_\delta)}
 \le2e^{-cK},\\
 \sup_{\mu\in\Lambda}\sum_{j=1}^p|b_j(\mu)|^2
 &\le\sup_{\mu\in\Lambda}
       \|\Sigma^{1/2}R_\mu v\|_2^2
       \|\Sigma^{1/2}R_\mu\beta\|_2^2
 \le C\|\beta\|_2^2\le C.
\end{align*}
Minkowski's inequality now gives
\begin{align*}
 \left|\frac{\chi_4}{n}\sum_{j=1}^p
    \biggl(\frac1{2\pi\ii}\oint_\Lambda
                 g_K(\mu)b_j(\mu)\dd\mu\biggr)^2\right|
 &\le\frac{|\chi_4|}{4\pi^2n}
    \biggl(\oint_\Lambda |g_K(\mu)|
       \Bigl(\sum_{j=1}^p|b_j(\mu)|^2\Bigr)^{1/2}|\dd\mu|\biggr)^2
 \le\frac Cn e^{-2cK}.
\end{align*}
Using \eqref{eq:normality-estimated-kernel-error} and
Lemma~\ref{lem:normality-variance-bounds}, we obtain
\begin{equation}\label{eq:normality-integrated-variance-error}
 \begin{aligned}
 |\widehat\sigma_0^2(g_K)-\sigma^2(g_K)|
 &\le\frac{\operatorname{length}(\Gamma_\delta^*)^2}{4\pi^2n}
       \left\|\frac{1-g_K}{\mu}\right\|_{\infty,\Gamma_\delta^*}^{2} \! O_P(n^{-1/2+\epsilon}\delta^{-C_\epsilon}) +\frac Cn e^{-2cK}\\
 &\le\frac{A_{n,K}^2}{n}
       O_P(n^{-1/2+\epsilon}\delta^{-C_\epsilon})
       +\frac Cn e^{-2cK}.
 \end{aligned}
\end{equation}
The above bound holds for all $K$ uniformly.

For notational simplicity, all maxima and minima below are over integers
$c_-\log n\le K\le c_+\log n$.
By \eqref{eq:normality-potential-gaps}, there is $c_0>0$ such that,
for all sufficiently large $n$,
\[
 M_n-h_\delta\alpha_\delta\ge c_0,\qquad
 \min_K A_{n,K}^2
 =\min_K e^{2K(M_n-h_\delta\alpha_\delta)}
 \ge n^{2c_-c_0}.
\]
Dividing \eqref{eq:normality-integrated-variance-error} by the lower
bound in \eqref{eq:normality-J-bounds}, and using
$\delta=(\log n)^{-1}$ and $0<\epsilon<1/2$, gives
\begin{align*}
 \max_K\left|\frac{\widehat\sigma_0^2(g_K)}{\sigma^2(g_K)}-1\right|
 &\le O_P(n^{-1/2+\epsilon}\delta^{-C_\epsilon}\log n)
       +C\max_K\frac{K e^{-2cK}}{A_{n,K}^2}\\
 &\le O_P\bigl(n^{-1/2+\epsilon}(\log n)^{C_\epsilon+1}\bigr)
       +C(\log n)n^{-2c_-(c+c_0)}
 =o_P(1).
\end{align*}
We then show the floor ${\|Y\|_2^2/n^2}$ in \eqref{eq:normality-estimated-variance} is inactive with high probability. Using 
$\E\|Y\|_2^2/n=\sigma^2+\beta^\top\Sigma\beta=1$ gives
\[
 \max_K\frac{\|Y\|_2^2/n^2}{\sigma^2(g_K)}
 \le C\frac{\|Y\|_2^2}{n}\max_K\frac{K}{A_{n,K}^2}
 =O_P\bigl((\log n)n^{-2c_-c_0}\bigr)=o_P(1).
\]
Consequently, since $\PP(\mathcal A')\to1$,
\begin{equation}\label{eq:normality-variance-floor-inactive}
 \PP\biggl(
       \min_K\frac{\widehat\sigma_0^2(g_K)}{\sigma^2(g_K)}>\frac12,
       \quad\max_K\frac{\|Y\|_2^2/n^2}{\sigma^2(g_K)}<\frac12
       \biggr)\longrightarrow1.
\end{equation}
On this event, $\widehat\sigma^2(g_K)=\widehat\sigma_0^2(g_K)>0$
for every $c_-\log n\le K\le c_+\log n$, proving \eqref{eq:normality-variance-consistency}.

Finally, we consider the case of $\tilde K$. As $0<u_\delta<1$ on
$\Ddisk_\delta\setminus[\ell,L]$, we have 
$ {h_\delta\alpha_\delta\le\widehat M_n\le h_\delta}$.
Equations~\eqref{eq:normality-degree} and
\eqref{eq:geometry-compact-bounds}, together with $0<U_n<1$, yield
fixed $0<c_-<c_+<\infty$ such that, for all sufficiently large $n$,
almost surely,
\[
 c_-\log n
 \le\frac{\log n+a\log\log n}{2h_\delta}
 \le\widetilde K
 \le\frac{\log n+a\log\log n}{2h_\delta\alpha_\delta}+2
 \le c_+\log n.
\]
Then we can apply \eqref{eq:normality-variance-consistency} to $\widetilde K$ to get
\begin{equation}\label{eq:normality-selected-variance}
 \left|\frac{\widehat\sigma^2(g_{\widetilde K})}
                  {\sigma^2(g_{\widetilde K})}-1\right|
 \le\max_K\left|\frac{\widehat\sigma^2(g_K)}{\sigma^2(g_K)}-1\right|
 \overset{p}{\longrightarrow}0.
\end{equation}
\end{proof}

\subsection{Proof of Theorem~\ref{thm:bocca-normality}}

\begin{proof}
Put $\tau_n^2=\sigma^2+\beta^\top\Sigma\beta>0$ and
$(\widetilde Y,\widetilde\beta,\widetilde\varepsilon,\widetilde\sigma)
=(Y,\beta,\varepsilon,\sigma)/\tau_n$. Then
\[
 \widetilde\sigma^2+\widetilde\beta^\top\Sigma\widetilde\beta=1,
 \qquad \|\widetilde\beta\|_2\le\ell^{-1/2},
 \qquad 0<\widetilde\sigma\le1.
\]
As $\widehat \beta_v$ scales linearly with $\tau_n$ and the variance estimator ${\widehat\sigma^2(g_K)}$ and the variance $\sigma^2(g_K)$ scale quadratically with $\tau_n$,
it suffices to work with the case $\tau_n=1$; we suppress tildes below.
{As in the proof of Lemma~\ref{lem:normality-variance-estimation}, $\underline\sigma^2\le\tau_n^2\le\bar\sigma^2+LB_0^2$ ensures that the rescaled noise retains uniform moment bounds and a positive lower bound on its standard deviation.}

Recall $K\asymp\log n$, 
$\psi(\mu)=(1-g_K(\mu))/\mu$ and
$A_{n,K}=e^{K(M_n-h_\delta\alpha_\delta)}$, as in
Lemma~\ref{lem:normality-variance-bounds}.
Use $D_\psi,N_\psi,R_\psi$ from \eqref{eq:Ipsi_decompose}, with
$\Gamma_\delta^\circ$ replaced by $\Gamma_{\delta}^{*,\circ}
 =\{\mu\in\Gamma_{\delta}^*:|\operatorname{Im}\mu|\ge n^{-3}\}$ from
\eqref{eq:normality-nonreal-contour} and $\Gamma_{\delta}^*=\{\mu:D_n(\mu)=1-\delta\}$ with $D_n(\mu)=\gamma\int\frac{\lambda^2}{|\lambda-\mu|^2}\mathsf H_n(\dd\lambda)$.
Conjugation symmetry gives
$D_\psi,N_\psi\in\R$.
Throughout the proof, we work on the $X$-measurable event $\mathcal A'$ in
Lemma~\ref{lem:normality-moments}, for which
$\PP(\mathcal A'^c)\le C_Dn^{-D}$ with $D>0$ chosen to be sufficiently large.

By \eqref{eq:normality-spectral-event},
\eqref{eq:normality-empirical-control}, and the resolvent identity,
\begin{align*}
 \sup_{\mu\in\Gamma_\delta^*}
 \|\widehat Q_\mu-Q(\zeta(\mu))\|_\op
 &\le\sup_{\mu\in\Gamma_\delta^*}
       |\widehat\zeta(\mu)-\zeta(\mu)|
       \|\widehat Q_\mu\|_\op\|Q(\zeta(\mu))\|_\op
 \le Cn^{-3/4}\delta^{-4},\\
 \|\ind_{\mathcal A'}R_\psi\|_{L^2}
 &\le C A_{n,K}\delta^{-4}(n^{-3/4}+n^{-5/4}).
\end{align*}
The second line follows from \eqref{eq:Ipsi_decompose}, as in the proof of Lemma~\ref{lem:field-moments}: on $\mathcal A'$,
$\|X\|_\op\le C\sqrt n$ and
\[
 \bigl\|n^{-1}(Xv)^\top
       (\widehat Q_\mu-Q(\zeta(\mu)))\varepsilon
       \bigr\|_{L^2(\varepsilon\mid X)}
 =\frac{\sigma}{n}
       \|(\widehat Q_\mu-Q(\zeta(\mu)))Xv\|_2
 \le Cn^{-5/4}\delta^{-4},
\]
where $L^2(\varepsilon|X)$ denotes the $L^2$ norm taken over $\varepsilon$ conditional on $X$.
Also, \eqref{eq:normality-real-crossing-charts} gives
$\operatorname{length}(\Gamma_\delta^*\setminus
\Gamma_\delta^{*,\circ})\le Cn^{-3}$.
Applying the calculation in \eqref{eq:omitted-arcs}, with
\eqref{eq:normality-empirical-control}, yields
\[
 \biggl\|\ind_{\mathcal A'}\frac1{2\pi\ii}
       \int_{\Gamma_\delta^*\setminus\Gamma_\delta^{*,\circ}}
       \psi(\mu)(\widehat F(\mu)-F(\mu))\dd\mu\biggr\|_{L^2}
 \le Cn^{-3}\delta^{-C}A_{n,K}.
\]
Since \eqref{eq:normality-J-bounds} gives
$\sigma(g_K)\ge cA_{n,K}/\sqrt{nK}$, we have
\begin{align*}
 &\frac1{\sigma(g_K)}\biggl(
       \|\ind_{\mathcal A'}R_\psi\|_{L^2}
       +\biggl\|\frac{\ind_{\mathcal A'}}{2\pi\ii}
       \int_{\Gamma_\delta^*\setminus\Gamma_\delta^{*,\circ}}
       \psi(\mu)(\widehat F(\mu)-F(\mu))\dd\mu
       \biggr\|_{L^2}\biggr)\\
 &\quad\le\frac{C A_{n,K}\delta^{-C}(n^{-3/4}+n^{-3})}{\sigma(g_K)}
 \le C\delta^{-C}\sqrt K\,(n^{-1/4}+n^{-5/2})=o(1).
\end{align*}

For the complement of $\mathcal A'$, Lemma~\ref{lem:global-fourth-moment},
\eqref{eq:normality-raw-resolvent}, and
$\|\|\varepsilon\|_2\|_{L^4}\le C\sqrt n$ give
\[
 \|\widehat\beta_v(g_K)\|_{L^4}
 \le Cn\|\psi\|_{\infty,\Gamma_\delta},\qquad
 \|D_\psi\|_{L^4}+\|N_\psi\|_{L^4}
 \le Cn^3\delta^{-1}A_{n,K}.
\]
By \eqref{eq:gK-bound}, \eqref{eq:geometry-compact-bounds}, and
$K\asymp\log n$,
$1+\|\psi\|_{\infty,\Gamma_\delta}\le n^C A_{n,K}$.
H\"older's inequality and \eqref{eq:normality-J-bounds} therefore give,
for sufficiently large $D$,
\[
 \frac{\bigl\|\ind_{\mathcal A'^c}
       (\widehat\beta_v(g_K)-\beta_v+v^\top g_K(\Sigma)\beta
          +D_\psi+N_\psi)\bigr\|_{L^2}}{\sigma(g_K)}
 \le C_D n^{C+1/2-D/4}\sqrt K=o(1).
\]
Combining these bounds with \eqref{eq:bocca-bias},
\eqref{eq:Fhat-F-decomposition}, and
\eqref{eq:normality-contour-equality}, we obtain
\begin{equation}\label{eq:normality-leading-expansion}
 \widehat\beta_v(g_K)-\beta_v+v^\top g_K(\Sigma)\beta
 =-D_\psi-N_\psi+o_{P}(\sigma(g_K)).
\end{equation}

Take $v_1=v_2=v$ and $w_1=w_2=\beta$ in
\eqref{eq:leading-covariance} and \eqref{eq:normality-fourth-kernel}.
By \eqref{eq:normality-calibration} and \eqref{eq:covariance-functions},
\[
 D(\zeta(\mu))^{-1}=-\mu R_\mu,\qquad
 M(\zeta(\mu),\zeta(\nu))=\mu\nu A_{\mu\nu},\qquad
 S(\zeta(\mu),\zeta(\nu))
 =\frac1{\mu\nu(1-\mathcal Q_{\mu\nu})}.
\]
Consequently, \eqref{eq:normality-kernel} and \eqref{eq:noise-kernel} give
\[
 \Xi_{12}+\Xi^{\rm n}(\zeta(\mu),\zeta(\nu))
 =\frac1n\mathcal K_n(\mu,\nu),
 \qquad \xi_1=\zeta(\mu),\quad\xi_2=\zeta(\nu).
\]
The two terms integrate to $\sigma_X^2(g_K)$ and
$\sigma_\varepsilon^2(g_K)$ in
\eqref{eq:normality-variance-design-part} and
\eqref{eq:normality-variance-noise-part}, respectively.
For either kernel $\Theta(\mu,\nu)=\Xi_{12}$ or
$\Theta(\mu,\nu)=\Xi^{\rm n}(\zeta(\mu),\zeta(\nu))$,
\eqref{eq:leading-covariance}, \eqref{eq:normality-fourth-kernel},
\eqref{eq:noise-kernel}, and Lemma~\ref{lem:normality-geometry} give
$
 \sup_{\mu,\nu\in\Gamma_\delta^*}|\Theta(\mu,\nu)|
 \le Cn^{-1}\delta^{-C}.
$
Using \eqref{eq:normality-real-crossing-charts} and
Lemmas~\ref{lem:normality-geometry} and~\ref{lem:normality-variance-bounds},
the omitted-arc error satisfies
 \begin{align}
 \notag &\biggl|\frac1{(2\pi\ii)^2}
       \biggl(\oint_{\Gamma_\delta^*}\!\oint_{\Gamma_\delta^*}
             -\int_{\Gamma_\delta^{*,\circ}}\!\int_{\Gamma_\delta^{*,\circ}}\biggr)
       \psi(\mu)\psi(\nu)\Theta(\mu,\nu)\dd\mu\dd\nu\biggr|
       \le\frac1{4\pi^2}
       \iint_{(\Gamma_\delta^*)^2\setminus(\Gamma_\delta^{*,\circ})^2}
       |\psi(\mu)\psi(\nu)\Theta(\mu,\nu)|\,|\dd\mu|\,|\dd\nu|\\
       \notag  &\quad\le\frac{2\operatorname{length}
                    (\Gamma_\delta^*\setminus\Gamma_\delta^{*,\circ})
                    \operatorname{length}(\Gamma_\delta^*)}{4\pi^2}
       \|\psi\|_{\infty,\Gamma_\delta^*}^2
       \sup_{\mu,\nu\in\Gamma_\delta^*}|\Theta(\mu,\nu)|\\
 &\quad\le C n^{-3}A_{n,K}^2\,n^{-1}\delta^{-C}
       =Cn^{-4}\delta^{-C}A_{n,K}^2. \label{eq:normality-omitted-kernel-arcs}
 \end{align}
Thus Lemma~\ref{lem:normality-moments},
\eqref{eq:normality-wick} with $N=2$, and
\eqref{eq:normality-noise-concentration} imply, for sufficiently small
$0<\epsilon<1/4$,
\begin{align*}
 &|\E D_\psi|
 \le C_\epsilon n^{-1+\epsilon}\delta^{-C_\epsilon}A_{n,K},\\
 &|\E D_\psi^2-\sigma_X^2(g_K)|
 \le C\delta^{-C}A_{n,K}^2(n^{-5/4}+n^{-4}),\\
 &\bigl\|\E[N_\psi^2\mid X]-\sigma_\varepsilon^2(g_K)\bigr\|_{L^1}
 \le C_\epsilon\delta^{-C_\epsilon}A_{n,K}^2
       (n^{-3/2+\epsilon}+n^{-4}).
\end{align*}
Dividing by the lower bound in \eqref{eq:normality-J-bounds} gives ${|\E D_\psi|}/{\sigma(g_K)}
\le C_\epsilon n^{-1/2+\epsilon}
      \delta^{-C_\epsilon}\sqrt K=o(1)$ and
\begin{align*}
 \frac{|\E D_\psi^2-\sigma_X^2(g_K)|
       +\|\E[N_\psi^2\mid X]-\sigma_\varepsilon^2(g_K)\|_{L^1}}
      {\sigma^2(g_K)}
 &\le C_\epsilon K\delta^{-C_\epsilon}
       (n^{-1/4}+n^{-1/2+\epsilon}+n^{-3})=o(1).
\end{align*}
Moreover, 
$\E[N_\psi\mid X]=0$ and
$\E(D_\psi N_\psi)=\E[D_\psi\E(N_\psi\mid X)]=0$.
We have therefore proved
\begin{equation}\label{eq:normality-relative-moments}
 \begin{aligned}
 \E D_\psi=o(\sigma(g_K)),
 \Var(D_\psi)=\sigma_X^2(g_K)+o(\sigma^2(g_K)),
 \E[N_\psi^2\mid X]=\sigma_\varepsilon^2(g_K)
                    +o_P(\sigma^2(g_K)).
 \end{aligned}
\end{equation}
The $L^2$ bounds used to obtain \eqref{eq:normality-leading-expansion}
also give
\[
 \frac{\Var(\widehat\beta_v(g_K))}{\sigma^2(g_K)}
 =\frac{\Var(D_\psi)+\E N_\psi^2}{\sigma^2(g_K)}+o(1)
 =1+o(1).
\]
The asserted consistency of $\widehat\sigma^2(g_K)$ follows from
\eqref{eq:normality-variance-consistency}.

\proofstep{Asymptotic normality.}
First consider the indices for which $\beta\ne0$. The bound proved after
\eqref{eq:normality-design-residue} gives
$
 \sigma_X^2(g_K)\ge\frac{c\|\beta\|_2^2}{nK}A_{n,K}^2>0.
$
For each fixed $N\ge1$, integrate \eqref{eq:normality-wick} with
$v_j=v$, $w_j=\beta$ for $1\le j\le N$, and weight
$(2\pi\ii)^{-N}\prod_{j=1}^N\psi(\mu_j)$ over
$(\Gamma_\delta^{*,\circ})^N$.
Applying \eqref{eq:normality-wick} to $w_j=\beta/\|\beta\|_2$,
$1\le j\le N$, and rescaling gives
\begin{align*}
 &\frac1{\sigma_X(g_K)^N}\biggl|\E D_\psi^N
   -\sum_{\pi\in\mathcal P_2(N)}
       \prod_{\{i,j\}\in\pi}\biggl(
       \frac1{(2\pi\ii)^2}
       \int_{\Gamma_\delta^{*,\circ}}\!\int_{\Gamma_\delta^{*,\circ}}
       \psi(\mu_i)\psi(\mu_j)\Xi_{ij}\dd\mu_i\dd\mu_j
       \biggr)\biggr|\\
 &\quad\le C_N n^{-1/4}\delta^{-C_N}
       \biggl(\frac{\|\beta\|_2 A_{n,K}}
                    {\sqrt n\,\sigma_X(g_K)}\biggr)^N
 \le C_N n^{-1/4}\delta^{-C_N}K^{N/2}=o(1).
\end{align*}
Likewise, applying \eqref{eq:normality-omitted-kernel-arcs} with
$\beta/\|\beta\|_2$ gives the following bound on the error on the omitted arcs:
\[
 \frac{Cn^{-4}\delta^{-C}\|\beta\|_2^2A_{n,K}^2}
      {\sigma_X^2(g_K)}
 \le Cn^{-3}\delta^{-C}K=o(1).
\]
Hence we have
\begin{equation}\label{eq:normality-integrated-wick}
 \E\biggl(\frac{D_\psi}{\sigma_X(g_K)}\biggr)^N
 =\begin{cases}
 (N-1)!!+o(1),&N\text{ even},\\
 o(1),&N\text{ odd}.
 \end{cases}
\end{equation}
The method of moments \citep[Theorem~3.3.26]{Durrett2019} gives
$
 {D_\psi}/{\sigma_X(g_K)}\overset{d}{\longrightarrow}N(0,1).
$
By \eqref{eq:normality-noise-projection} and
\eqref{eq:normality-design-projection},
$0\le\sigma_X(g_K)/\sigma(g_K)\le1$. Thus, for each fixed $t\in\R$,
\begin{equation}\label{eq:normality-design-characteristic}
 \begin{aligned}
 &\biggl|\E e^{\ii tD_\psi/\sigma(g_K)}
       -\exp\biggl(-\frac{t^2\sigma_X^2(g_K)}{2\sigma^2(g_K)}\biggr)\biggr|\le\sup_{|u|\le|t|}
       \bigl|\E e^{\ii uD_\psi/\sigma_X(g_K)}-e^{-u^2/2}\bigr|
       \longrightarrow0.
 \end{aligned}
\end{equation}
If $\beta=0$, then $D_\psi=\sigma_X(g_K)=0$, so the left-hand side
of \eqref{eq:normality-design-characteristic} is zero.

For the noise term, \eqref{eq:Ipsi_decompose} gives
\[
 N_\psi=\sum_{i=1}^n a_i\varepsilon_i,\text{ where }
 a_i=\frac1{2\pi\ii\sqrt n}\int_{\Gamma_\delta^{*,\circ}}
       \psi(\mu)v^\top G_{12}(\zeta(\mu))e_i\dd\mu\in\R,\text{ for }1\le i\le n.
\]
Since $\Pi_{12}=0$, \eqref{eq:normality-resolvent-moments} and
Minkowski's inequality imply, for every fixed $r\ge2$,
\[
 \max_{1\le i\le n}\|a_i\|_{L^r}
 \le\frac{CA_{n,K}}{\sqrt n}
       \sup_{\substack{1\le i\le n\\\mu\in\Gamma_\delta^{*,\circ}}}
       \|v^\top G_{12}(\zeta(\mu))e_i\|_{L^r}
 \le C_{r,\epsilon}n^{-1+\epsilon/2}
       \delta^{-C_{r,\epsilon}}A_{n,K}.
\]
Set $c_i=\sigma a_i/\sigma(g_K)$ for $1\le i\le n$ and choose $r>2/\epsilon$.
Using \eqref{eq:normality-J-bounds} and $\sigma\le1$, a union bound gives
\[
 \PP\bigl(\max_{1\le i\le n}|c_i|>
       n^{-1/2+\epsilon}\delta^{-C}\sqrt K\bigr)
 \le C_{r,\epsilon}n^{1-r\epsilon/2}=o(1).
\]
Together with \eqref{eq:normality-relative-moments}, this proves
\begin{equation}\label{eq:normality-noise-coefficients}
 \max_{1\le i\le n}|c_i|=o_P(1),\qquad
 \sum_{i=1}^n c_i^2
 =\frac{\E[N_\psi^2\mid X]}{\sigma^2(g_K)}
 =\frac{\sigma_\varepsilon^2(g_K)}{\sigma^2(g_K)}+o_{P}(1).
\end{equation}
By Taylor's formula,
$|\E e^{\ii u\varepsilon_i/\sigma}-e^{-u^2/2}|
\le C|u|^3$ for $u\in\R$, as {$\varepsilon_i/\sigma$ has mean zero, variance one, and a uniformly bounded third absolute moment}.
Conditional independence and telescoping the products therefore yield
\begin{align*}
 &\bigl|\E[e^{\ii tN_\psi/\sigma(g_K)}\mid X]
                 -e^{-t^2\sum_{i=1}^n c_i^2/2}\bigr|=\biggl|\prod_{i=1}^n
       \E[e^{\ii t c_i\varepsilon_i/\sigma}\mid X]
           -\prod_{i=1}^n e^{-t^2c_i^2/2}\biggr|\\
 &\quad\le C_t\sum_{i=1}^n|c_i|^3
       \le C_t\max_{1\le i\le n}|c_i|\sum_{i=1}^n c_i^2=o_P(1).
\end{align*}
The left-hand side is at most $2$, so its expectation tends to zero.
Since $|e^{-x}-e^{-y}|\le\min\{1,|x-y|\}$ for $x,y\ge0$,
\eqref{eq:normality-noise-coefficients} gives, for each fixed $t\in\R$,
\begin{align*}
 &\biggl|\E e^{\ii t(D_\psi+N_\psi)/\sigma(g_K)}
       -\exp\biggl(-\frac{t^2\sigma_\varepsilon^2(g_K)}
                          {2\sigma^2(g_K)}\biggr)
        \E e^{\ii tD_\psi/\sigma(g_K)}\biggr|\\
 &\quad\le
       \E\bigl|\E[e^{\ii tN_\psi/\sigma(g_K)}\mid X]
                  -e^{-t^2\sum_{i=1}^n c_i^2/2}\bigr|+\E\min\biggl\{1,\frac{t^2}{2}\biggl|\sum_{i=1}^n c_i^2
                  -\frac{\sigma_\varepsilon^2(g_K)}{\sigma^2(g_K)}\biggr|\biggr\}
       \longrightarrow0.
\end{align*}
Using \eqref{eq:normality-design-characteristic} and
$\sigma_X^2(g_K)+\sigma_\varepsilon^2(g_K)=\sigma^2(g_K)$ gives
\[
 \E e^{\ii t(D_\psi+N_\psi)/\sigma(g_K)}
 =\exp\biggl(-\frac{t^2(\sigma_X^2(g_K)+\sigma_\varepsilon^2(g_K))}
                   {2\sigma^2(g_K)}\biggr)+o(1)
 =e^{-t^2/2}+o(1).
\]
The characteristic-function criterion
\citep[Theorem~3.3.17]{Durrett2019} and
\eqref{eq:normality-leading-expansion} now prove
\eqref{eq:normality-biased}.

\proofstep{The data-driven degree.}
Put
\[
 t_n=\frac{\log n+a\log\log n}{2M_n},\qquad
 \widehat t_n=\frac{\log n+a\log\log n}{2\widehat M_n},\qquad
 K_n^\circ=\lceil t_n+U_n\rceil.
\]
The lower bounds in \eqref{eq:geometry-compact-bounds},
\eqref{eq:normality-potential-gaps}, and
\eqref{eq:normality-estimated-M}, together with
Lemma~\ref{lem:normality-estimated-M}, give
\[
 |\widehat t_n-t_n|
 =\frac{\log n+a\log\log n}{2M_n\widehat M_n}
       |\widehat M_n-M_n|
 =O_P(n^{-1/8}\log n)=o_P(1).
\]
For $x,y\in\R$,
$\int_0^1\ind\{\lceil x+u\rceil\ne\lceil y+u\rceil\}\dd u
\le\min\{1,|x-y|\}$.
Since $U_n$ is independent of $(X,Y)$, conditioning on $(X,Y)$ gives
\begin{equation}\label{eq:normality-degree-coupling}
 \PP(\widetilde K\ne K_n^\circ)
 \le\E\min\{1,|\widehat t_n-t_n|\}\longrightarrow0.
\end{equation}
Also, $K_n^\circ$ is independent of $(X,Y)$ and
$K_n^\circ\in\{\lceil t_n\rceil,\lceil t_n\rceil+1\}$ almost surely.
For either deterministic value $k\in\{\lceil t_n\rceil,\lceil t_n\rceil+1\}$, we have $k\asymp\log n$ and,
by \eqref{eq:gK-bound} and \eqref{eq:normality-J-bounds},
\begin{align*}
 \frac{|v^\top g_k(\Sigma)\beta|}{\sigma(g_k)}
 &\le\frac{2\|\beta\|_2 e^{-kh_\delta\alpha_\delta}}
              {c e^{k(M_n-h_\delta\alpha_\delta)}/\sqrt{nk}}
 \le C\sqrt{nk}\,e^{-kM_n}
 \le C(\log n)^{(1-a)/2}=o(1),\\
 \sigma^2(g_k)
 &\le\frac{Ck}{n}e^{2k(M_n-h_\delta\alpha_\delta)}
 \le C(\log n)^{a+1}n^{-h_\delta\alpha_\delta/M_n}
 \le C(\log n)^{a+1}n^{-c}\longrightarrow0.
\end{align*}
Here $t_n\le k\le t_n+2$, and
$h_\delta\alpha_\delta/M_n\ge\alpha_\delta\ge c>0$ by
\eqref{eq:normality-potential-gaps} and \eqref{eq:geometry-compact-bounds}.
Let $\Phi$ be the standard normal distribution function.
Applying \eqref{eq:normality-biased} and the bias bound to the two
deterministic degrees, then using independence of $K_n^\circ$ and
\eqref{eq:normality-degree-coupling}, gives, for every $x\in\R$,
\begin{align*}
 &\biggl|\PP\biggl(
       \frac{\widehat\beta_v(g_{\widetilde K})-\beta_v}
            {\sigma(g_{\widetilde K})}\le x\biggr)-\Phi(x)\biggr|
            \le\PP(\widetilde K\ne K_n^\circ)
       +\max_{k\in\{\lceil t_n\rceil,\lceil t_n\rceil+1\}}
       \biggl|\PP\biggl(
       \frac{\widehat\beta_v(g_k)-\beta_v}{\sigma(g_k)}\le x\biggr)
                 -\Phi(x)\biggr|\longrightarrow0.
\end{align*}

By \eqref{eq:normality-selected-variance}, we have
\[
 \frac{\widehat\sigma^2(g_{\widetilde K})}{\sigma^2(g_{\widetilde K})}
 \overset{p}{\longrightarrow}1,\qquad
 \PP(\widehat\sigma(g_{\widetilde K})=0)
 \le\PP\biggl(\biggl|
       \frac{\widehat\sigma^2(g_{\widetilde K})}{\sigma^2(g_{\widetilde K})}
          -1\biggr|>\frac12\biggr)\longrightarrow0.
\]
{On the event $\widehat\sigma(g_{\widetilde K})>0$, whose probability tends to one,} Slutsky's theorem gives
\[
 \frac{\widehat\beta_v(g_{\widetilde K})-\beta_v}
      {\widehat\sigma(g_{\widetilde K})}
 =\biggl(\frac{\widehat\sigma^2(g_{\widetilde K})}
                 {\sigma^2(g_{\widetilde K})}\biggr)^{-1/2}
       \frac{\widehat\beta_v(g_{\widetilde K})-\beta_v}
            {\sigma(g_{\widetilde K})}
 \overset{d}{\longrightarrow}N(0,1),
\]
which proves \eqref{eq:normality-unbiased}.
Finally, rescaling by $\tau_n$ preserves all standardized
limits and variance ratios in the original model.
\end{proof}

\section{The Scalar Lower-Bound Experiments}\label{app:lower}
We prove the properties of the two hypotheses constructed in Section~\ref{app:lower-construction}. Throughout this section, $a=\kappa^{-1}+\upsilon$ and the other quantities are those defined in Section~\ref{app:lower-construction}.

\subsection{The Supported Priors and Their Separation}\label{app:pf:lem:lower-priors}
\begin{proof}[Proof of Lemma~\ref{lem:lower-priors}]
We verify that the priors have the required support, that their target values are separated, and that their data laws are close.

\proofstep{Regression parameters and support.} Fix $j\in\{0,1\}$. The design after \eqref{eq:final-transform} is
$X=(X_1,X_{-1})=(\mathsf Y+\xi,\mathsf X)$. Define the regression noise
\begin{align*}
 \varepsilon
 =Y-X\beta^{(j)}(V)
 =Y-b_jX_1-(1-b_j)\mathsf XV\vartheta_j
 =(1-b_j)\varepsilon^\circ-(1+b_j)\xi.
\end{align*}
Conditional on $V$, $\varepsilon$ and $X$ are jointly Gaussian, and
\begin{align*}
 \Cov(\varepsilon,X_1\mid V)
 =\bigl((1-b_j)\varpi_j-(1+b_j)s_0\bigr)I_n=0,
 \Cov(\varepsilon,\operatorname{vec}X_{-1}\mid V)=0.
\end{align*}
Hence $\varepsilon$ is independent of $X=(X_1,X_{-1})$ conditional on $V$, and
\begin{align*}
 \Var(\varepsilon\mid V)
 =\bigl((1-b_j)^2\varpi_j+(1+b_j)^2s_0\bigr)I_n
 =\frac{4\varpi_js_0}{\varpi_j+s_0}I_n=\sigma_j^2I_n.
\end{align*}
Together with $X_1=\mathsf XV\vartheta_j+\varepsilon^\circ+\xi$,
this proves \eqref{eq:final-parameters}.

Fix $V\in\mathcal C$. Using the mass bound \eqref{eq:lower-measure-size}, we have
\[
 V^\top\Sigma_-(V)V=(\mathsf Q^{-1}+\mathsf D)^{-1},
 \qquad
 \vartheta_j^\top V^\top\Sigma_-(V)V\vartheta_j
 =\tau\int\lambda\,\nu_j(\dd\lambda)+O(\tau\upsilon^2).
\]

Write $V=O_V\mathsf Q^{1/2}$, where $O_V^\top O_V=I_s$, and put $\mathsf S=(I_s+\mathsf Q^{1/2}\mathsf D\mathsf Q^{1/2})^{-1}$.
Complete $O_V$ to an orthogonal matrix $O=(O_V,O_V^\perp)$.
By \eqref{eq:final-parameters},
\begin{equation}\label{eq:lower-prior-covariance-blocks}
 \diag(1,O)^\top\Sigma_j(V)\diag(1,O)
 =\begin{pmatrix}
  (\Sigma_j(V))_{11}
   &\vartheta_j^\top\mathsf Q^{1/2}\mathsf S&0\\
  \mathsf S\mathsf Q^{1/2}\vartheta_j&\mathsf S&0\\
  0&0&I_{d-s}
 \end{pmatrix}.
\end{equation}
Since $\mathsf D\succ0$, we have $0\prec\mathsf S\prec I_s$.
The Schur-complement formula for the leading $(s+1)\times(s+1)$ block of
$I_p-\Sigma_j(V)$ in the coordinates of
\eqref{eq:lower-prior-covariance-blocks} gives
\begin{equation}\label{eq:lower-prior-upper-schur}
 \Sigma_j(V)\preceq I_p
 \quad\Longleftrightarrow\quad
 1-(\Sigma_j(V))_{11}
 -\vartheta_j^\top\mathsf Q^{1/2}\mathsf S
   (I_s-\mathsf S)^{-1}\mathsf S\mathsf Q^{1/2}\vartheta_j\ge0.
\end{equation}
When $\mathsf S\succ\kappa^{-1}I_s$, applying the Schur-complement formula to
$\Sigma_j(V)-\kappa^{-1}I_p$ gives
\begin{equation}\label{eq:lower-prior-lower-schur}
 \Sigma_j(V)\succeq\kappa^{-1}I_p
 \quad\Longleftrightarrow\quad
 (\Sigma_j(V))_{11}-\kappa^{-1}
 -\vartheta_j^\top\mathsf Q^{1/2}\mathsf S
   (\mathsf S-\kappa^{-1}I_s)^{-1}\mathsf S\mathsf Q^{1/2}\vartheta_j\ge0.
\end{equation}
For the upper spectral bound $\Sigma_j(V)\preceq I_p$, the quadratic term in
\eqref{eq:lower-prior-upper-schur} satisfies
\[
 \begin{aligned}
 \vartheta_j^\top\mathsf Q^{1/2}\mathsf S
       (I_s-\mathsf S)^{-1}\mathsf S\mathsf Q^{1/2}\vartheta_j
 \le \vartheta_j^\top\mathsf D^{-1}\vartheta_j
 =\tau\int\frac{\lambda}{1-\lambda}\,\nu_j(\dd\lambda)
 \le C\tau,
 \end{aligned}
\]
where the last bound uses \eqref{eq:lower-measure-endpoints} and
$\mathsf S(I_s-\mathsf S)^{-1}\mathsf S
 =(\mathsf Q^{1/2}\mathsf D\mathsf Q^{1/2})^{-1}-\mathsf S$.

For the lower spectral bound $\Sigma_j(V)\succeq\kappa^{-1}I_p$, set
$\mathsf C=(1-\kappa^{-1})\mathsf Q^{-1}-\kappa^{-1}\mathsf D$. On $\mathcal C$, we have $\|\mathsf Q-I_s\|_{\mathrm{op}}\le\upsilon^2$ and hence $\|\mathsf Q^{-1}-I_s\|_{\mathrm{op}} \le 2\upsilon^2$.
We then have 
\[
      \|\mathsf C - \diag((z_k-\kappa^{-1})/z_k)\|_{\mathrm{op}} = O(\upsilon^2), \text{ thus }
 \mathsf C\succeq
 \frac12\diag\left(\frac{z_k-\kappa^{-1}}{z_k}\right)
 \succeq\frac12\diag\left(\frac{z_k-a}{z_k}\right)
 \succ0,
\]
where the second inequality is by $z_k-\kappa^{-1}>\upsilon$.
Thus $\mathsf S\succ\kappa^{-1}I_s$, and the quadratic term in
\eqref{eq:lower-prior-lower-schur} satisfies
\[
 \begin{aligned}
 \vartheta_j^\top\mathsf Q^{1/2}\mathsf S
  (\mathsf S-\kappa^{-1}I_s)^{-1}\mathsf S
  \mathsf Q^{1/2}\vartheta_j
 &\le \vartheta_j^\top\mathsf C^{-1}\vartheta_j
 \le C\tau\int\frac{\lambda}{\lambda-a}\,\nu_j(\dd\lambda)
 \le C\tau.
 \end{aligned}
\]
Here we used \eqref{eq:lower-measure-endpoints}, $\mathsf S\preceq I_s$, and
$\mathsf Q^{1/2}\mathsf S(\mathsf S-\kappa^{-1}I_s)^{-1}
\mathsf Q^{1/2}=\mathsf C^{-1}$.
The spectral gaps in \eqref{eq:lower-prior-upper-schur} and
\eqref{eq:lower-prior-lower-schur} satisfy 
\begin{equation}\label{eq:lower-X1-variance}
      \begin{aligned}
       1-(\Sigma_j(V))_{11}
       &=1-2s_0+O(\tau\upsilon^2)
       =\frac{1-\kappa^{-1}}2+O(\tau\upsilon^2),\\
       (\Sigma_j(V))_{11}-\kappa^{-1}
       &=2s_0-\kappa^{-1}+O(\tau\upsilon^2)
       =\frac{1-\kappa^{-1}}2+O(\tau\upsilon^2).
      \end{aligned}
\end{equation}
The Schur-complement conditions \eqref{eq:lower-prior-upper-schur} and
\eqref{eq:lower-prior-lower-schur} therefore give
$\kappa^{-1}I_p\preceq\Sigma_j(V)\preceq I_p$ for small $\tau$ and large $n$.
Finally, we have
\[
 |b_j|\le C\tau,\qquad
 \|\beta^{(j)}(V)\|_2\le C\sqrt\tau,\qquad
 \frac{4s_0}{3}\le
 \frac{4\varpi_js_0}{\varpi_j+s_0}\le2s_0,
\]
which proves the coefficient and noise bounds.

\proofstep{Separation of two hypotheses.}
As $\beta_1^{(j)}(V) = b_j$ for every $V$, \eqref{eq:lower-measure-separation} gives
\[
 b_0-b_1
 =
 \frac{2s_0\tau\int\lambda\,(\nu_1-\nu_0)(\dd\lambda)}
 {(2s_0-\tau\int\lambda\dd\nu_0)
  (2s_0-\tau\int\lambda\dd\nu_1)}
 \ge c\tau e^{-K\mathfrak q(0)}.
\]
By \eqref{eq:u0-alpha} in Lemma~\ref{lem:geometry}\textup{(iii)} and
the degree choice \eqref{eq:lower-degree}, we have
\[
 \frac{\mathfrak q(0)}h
 =\alpha(a^{-1},(p-1)/n)
 =\alpha(\kappa,p/n)+O(\upsilon+n^{-1}),
 \qquad
 K=\frac{\log n}{2h}+O(1).
\]
As $(\upsilon+n^{-1})\log n=o(1)$, we have
$e^{-K\mathfrak q(0)}\asymp n^{-\alpha(\kappa,\gamma)/2}$.

\proofstep{Total variation.}
Recall that $\widetilde P_j$ and $\widetilde P_j^{\mathcal C}$ are the
laws of $(\mathsf X,\mathsf Y)$ under hypothesis $j$, mixed over
$V\sim\pi$ and $V\sim\pi_{\mathcal C}=\pi(\,\cdot\mid\mathcal C)$, respectively. For each 
$j=0,1$, the conditional law of $(\mathsf X,\mathsf Y)$ given $V$ is
the same under these two mixing laws. The data-processing inequality
for total variation \citep[Theorem~7.4]{PolyanskiyWu2025} therefore gives
$
 \operatorname{TV}(\widetilde P_j^{\mathcal C},\widetilde P_j)
 \le\operatorname{TV}(\pi_{\mathcal C},\pi).
$
By \eqref{eq:prior-event}, $\pi(\mathcal C)>0$ for all sufficiently large
$n$ and thus
\[
 \operatorname{TV}(\pi_{\mathcal C},\pi)
 =\frac12\int
   \left|\frac{\mathbf 1_{\mathcal C}(V)}{\pi(\mathcal C)}-1\right|
   \pi(\dd V)
 =\pi(\mathcal C^c).
\]
The law $P_j$ is obtained from $\widetilde P_j^{\mathcal C}$ by adjoining
an independent $\xi\sim N(0,s_0I_n)$ and applying
\eqref{eq:final-transform}. These operations are the same for $j=0,1$.
Applying the data-processing inequality again gives
\[
 \begin{aligned}
 \operatorname{TV}(P_0,P_1)
 &\le
       \operatorname{TV}(\widetilde P_0^{\mathcal C},
                     \widetilde P_1^{\mathcal C})
\le\operatorname{TV}(\widetilde P_0,\widetilde P_1)
       +\sum_{j=0}^1\operatorname{TV}(
          \widetilde P_j^{\mathcal C},\widetilde P_j)\\
 &\le\operatorname{TV}(\widetilde P_0,\widetilde P_1)
       +2\operatorname{TV}(\pi_{\mathcal C},\pi)
=\operatorname{TV}(\widetilde P_0,\widetilde P_1)
       +2\pi(\mathcal C^c)
 \le\frac14+o(1)\le\frac13.
 \end{aligned}
\]
The fourth inequality follows from \eqref{eq:TVaux} and \eqref{eq:prior-event}, for sufficiently
small fixed $\tau$ and all sufficiently large $n$.
\end{proof}

\subsection{Spectral Measures With the Required Transform Bound}\label{app:pf:lem:lower-measure}

\begin{proof}[Proof of Lemma~\ref{lem:lower-measure}]

We construct the measures from the residues of
$1/\cosh(K\mathfrak q)$ with $\mathfrak q$ defined in Appendix~\ref{app:lower-construction}. Recall we set $\gamma^\circ=d/n$, $\upsilon = n^{-1/8}$,
$a=\kappa^{-1}+\upsilon$, $x_\pm=(1\pm(\gamma^\circ)^{-1/2})^2$, $\rho, h$ in \eqref{eq:lower-rho-h}, and use \eqref{eq:geometry-compact-bounds}
with $(R_\delta,L,\ell)=(\sqrt{\gamma^\circ},1,a)$. For all
sufficiently large $n$, we have
\begin{equation}\label{eq:lower-parameter-bounds}
 \begin{gathered}
 c\le a\le1-c,\quad 1+c\le\gamma^\circ\le C,
 \quad x_--\upsilon\ge c,
 \quad
 c\le\mathfrak q(0)\le h-c,\quad h\le C.
 \end{gathered}
\end{equation}
Here $[x_-,x_+]$ is the support of the Marchenko--Pastur law of standard Gaussian covariance. See Lemma~\ref{lem:weighted-Wishart}.

\proofstep{Atomic measures.}
Set $f(z)=\cosh(K\mathfrak q(z))$. Since $K$ is an integer,
Lemma~\ref{lem:geometry}\textup{(iv)} makes $f$ single valued and
holomorphic on a neighborhood of
$\{z:|z-1|\le\sqrt{\gamma^\circ}\}$.
For the branch continued through the upper half-disk from
$(1-\sqrt{\gamma^\circ},a)$ (see Lemma~\ref{lem:geometry}\textup{(ii)}), define $\theta(\lambda) = \ii \mathfrak q(\lambda+\ii0_+)$, for $a<\lambda<1$.
Put $T(z)=(\sqrt{\gamma^\circ}-1+z)/(\sqrt{\gamma^\circ}+1-z)$.
Equations~\eqref{eq:sequence-of-maps-app} and \eqref{eq:U-real-app}
give
\begin{equation}\label{eq:lower-slit-phase}
 \begin{aligned}
 &\theta(a)=0,\qquad \theta(1)=\pi,\qquad
 f(\lambda)=\cos(K\theta(\lambda)),\\
 &\theta'(\lambda)=
 \frac{\pi T'(\lambda)}
 {\Kell(1-\rho^2)\sqrt{(T(\lambda)^2-\rho^2)(1-T(\lambda)^2)}}
 \asymp\frac1{\sqrt{(\lambda-a)(1-\lambda)}}>0.
 \end{aligned}
\end{equation}
Indeed, \eqref{eq:lower-parameter-bounds} gives $c\le\rho\le1-c$,
and, for $a<\lambda<1$,
\[
 \begin{aligned}
 T'(\lambda)&=\frac{2\sqrt{\gamma^\circ}}
                   {(\sqrt{\gamma^\circ}+1-\lambda)^2}\asymp1,
 T(\lambda)-\rho
 =\frac{2\sqrt{\gamma^\circ}(\lambda-a)}
 {(\sqrt{\gamma^\circ}+1-\lambda)(\sqrt{\gamma^\circ}+1-a)}
 \asymp\lambda-a,\\
 1-T(\lambda)^2
 &=\frac{4\sqrt{\gamma^\circ}(1-\lambda)}
          {(\sqrt{\gamma^\circ}+1-\lambda)^2}\asymp1-\lambda.
 \end{aligned}
\]
Thus $f$ has $K$ simple zeros in $(a,1)$, given by
\begin{equation}\label{eq:lower-cosh-zeros}
 \begin{gathered}
 a<z_1<\cdots<z_K<1,\text{ where }
 \theta(z_r)=\frac{(r-\tfrac12)\pi}{K},\\
 f'(z_r)=(-1)^rK\theta'(z_r)\ne0,
 \text{ for } 1\le r\le K.
 \end{gathered}
\end{equation}
Since $f(a)=1$, $f(1)=(-1)^K$, and
$0<\operatorname{Re}\mathfrak q(z)<h$ for
$|z-1|<\sqrt{\gamma^\circ}$, $z\notin[a,1]$, these are all its
zeros in the disk.
Define the real signed measure $\eta$ and its Cauchy transform by
\begin{equation}\label{eq:lower-signed-measure}
 \begin{aligned}
 \eta=-\sum_{r=1}^K\frac{\delta_{z_r}}{f'(z_r)}
      =\sum_{r=1}^K\frac{(-1)^{r+1}}{K\theta'(z_r)}
                       \delta_{z_r},
 \mathcal J(z)=\int_a^1\frac{\eta(\dd\lambda)}{\lambda-z},
                    \text{ for } z\in\C\setminus[a,1].
 \end{aligned}
\end{equation}
 By
\eqref{eq:lower-slit-phase}, we have
\[
 \|\eta\|_{\TV}
 =\sum_{r=1}^K\frac1{K\theta'(z_r)}
 \le\frac CK\sum_{r=1}^K\sqrt{(z_r-a)(1-z_r)}
 \le C.
\]
For the total variation measure $|\eta|$, define
\begin{equation}\label{eq:lower-jordan-parts}
 \begin{gathered}
 \eta^+=\frac{|\eta|+\eta}{2}, \quad
 \eta^-=\frac{|\eta|-\eta}{2}, \quad
 w(\lambda)=\frac{(\lambda-a)(1-\lambda)}{\lambda^2},\\
 \nu_0=w\eta^+, \quad \nu_1=w\eta^-, \quad
 \nu_1-\nu_0=-w\eta.
 \end{gathered}
\end{equation}
Since $w(z_r)>0$ for $1\le r\le K$,
$|\operatorname{supp}(\nu_0)\cup\operatorname{supp}(\nu_1)|=K$.
Moreover, for $j=0,1$,
\begin{equation}\label{eq:lower-weighted-measure-bounds}
 \begin{aligned}
      &\nu_0(\R)+\nu_1(\R)
 =\int_a^1w(\lambda)|\eta|(\dd\lambda)
 \le\frac{(1-a)^2}{4a^2}\|\eta\|_{\TV}\le C,\\
 &\int_a^1\biggl(\frac{\lambda}{1-\lambda}
              +\frac{\lambda}{\lambda-a}\biggr)\nu_j(\dd\lambda)
 \le\int_a^1\frac{1-a}{\lambda}|\eta|(\dd\lambda)
 \le\frac{1-a}{a}\|\eta\|_{\TV}\le C.
 \end{aligned}
\end{equation}
This proves \eqref{eq:lower-measure-size} and
\eqref{eq:lower-measure-endpoints}.

\proofstep{Separation.}
By \eqref{eq:outer-collar-app}, with
$(R_\delta,L,\ell)=(\sqrt{\gamma^\circ},1,a)$, there are fixed
$\varepsilon_*,C_*>0$ such that
\begin{equation}\label{eq:lower-collar-bounds}
 \sup_{\bigl||z-1|-\sqrt{\gamma^\circ}\bigr|\le\varepsilon_*}
       |\mathfrak q'(z)|\le C_*,\qquad
 \operatorname{Re}\mathfrak q(z)\ge h
 \quad\text{if }\sqrt{\gamma^\circ}\le|z-1|
                         \le\sqrt{\gamma^\circ}+\varepsilon_*.
\end{equation}
For $u>0$ and $v\in\R$,
$|\cosh(u+\ii v)|^2=\sinh^2u+\cos^2v\ge\sinh^2u$ and
$|\sinh(u+\ii v)|^2\le\cosh^2u$. Consequently, wherever
$\operatorname{Re}\mathfrak q(z)>0$,
\begin{equation}\label{eq:lower-reciprocal-bounds}
 \begin{aligned}
 \left|\frac1{f(z)}\right|
 &\le\frac1{\sinh(K\operatorname{Re}\mathfrak q(z))},\\
 \left|\left(\frac1f\right)'(z)\right|
 &=K|\mathfrak q'(z)|
       \frac{|\sinh(K\mathfrak q(z))|}{|\cosh(K\mathfrak q(z))|^2}
 \le K|\mathfrak q'(z)|
       \frac{\cosh(K\operatorname{Re}\mathfrak q(z))}
            {\sinh^2(K\operatorname{Re}\mathfrak q(z))}.
 \end{aligned}
\end{equation}
In particular, $1/f$ is holomorphic on a neighborhood of the annulus
$\{z:\sqrt{\gamma^\circ}\le|z-1|
                         \le\sqrt{\gamma^\circ}+\varepsilon_*\}$.
For $|z-1|<\sqrt{\gamma^\circ}$, $z\notin[a,1]$, the residue
theorem \citep[Chapter~4, Theorem~17]{Ahlfors1979} gives
\begin{equation}\label{eq:lower-cauchy-decomposition}
 \begin{aligned}
 \mathcal E(z)
 &:=\frac1{2\pi\ii}\oint_{|\xi-1|=\sqrt{\gamma^\circ}}
                  \frac{\dd\xi}{f(\xi)(\xi-z)}
=\frac1{f(z)}+\sum_{r=1}^K\frac1{f'(z_r)(z_r-z)}
  =\frac1{f(z)}-\mathcal J(z).
 \end{aligned}
\end{equation}
Deforming the circle $|\xi-1|=\sqrt{\gamma^\circ}$ to
$|\xi-1|=\sqrt{\gamma^\circ}+\varepsilon_*/2$ gives
\[
 \mathcal E(z)=\frac1{2\pi\ii}
 \oint_{|\xi-1|=\sqrt{\gamma^\circ}+\varepsilon_*/2}
                \frac{\dd\xi}{f(\xi)(\xi-z)}.
\]
Thus $\mathcal E$ extends holomorphically to
$|z-1|<\sqrt{\gamma^\circ}+\varepsilon_*/2$, and
$\mathcal J=1/f-\mathcal E$ there for $z\notin[a,1]$.
The residue theorem also gives
\[
 \int_a^1\lambda^k\eta(\dd\lambda)
 =-\sum_{r=1}^K\frac{z_r^k}{f'(z_r)}
 =-\frac1{2\pi\ii}
   \oint_{|\xi-1|=\sqrt{\gamma^\circ}+\varepsilon_*/2}
                   \frac{\xi^k}{f(\xi)}\dd\xi,
 \text{ for } k=0,1.
\]
For $|\xi-1|=\sqrt{\gamma^\circ}+\varepsilon_*/2$ and
$|z-1|\le\sqrt{\gamma^\circ}+\varepsilon_*/4$, we have
$|\xi-z|\ge\varepsilon_*/4$ and $|\xi|\le C$.
Differentiating the contour integral and using
\eqref{eq:lower-collar-bounds}--\eqref{eq:lower-reciprocal-bounds}
therefore yields
\begin{equation}\label{eq:lower-contour-error}
 \begin{aligned}
 &\sup_{|z-1|\le\sqrt{\gamma^\circ}+\varepsilon_*/4}
       \bigl(|\mathcal E(z)|+|\mathcal E'(z)|\bigr)
   +\sum_{k=0}^1\left|\int_a^1\lambda^k\eta(\dd\lambda)\right|\\
  &\quad \le C\sup_{|\xi-1|=\sqrt{\gamma^\circ}+\varepsilon_*/2}
                      \frac1{|f(\xi)|}
 \le\frac C{\sinh(Kh)}\le Ce^{-Kh}.
 \end{aligned}
\end{equation}
By \eqref{eq:lower-jordan-parts}, $\nu_1-\nu_0=-w\eta$ and $-\lambda w(\lambda)
=\lambda-(1+a)+\frac a\lambda$. 
By
the definitions in \eqref{eq:lower-signed-measure} and \eqref{eq:lower-cauchy-decomposition}, $a\int_a^1\lambda^{-1}\eta(\dd\lambda)=a\mathcal J(0)= a(1/f(0)-\mathcal E(0))$. We then have
\begin{equation}\label{eq:lower-residue-separation}
 \begin{aligned}
      &\int_a^1\lambda\,(\nu_1-\nu_0)(\dd\lambda)
 =\int_a^1\lambda\,\eta(\dd\lambda)
   -(1+a)\int_a^1\eta(\dd\lambda)+a\mathcal J(0)\\
 &\qquad =\frac a{\cosh(K\mathfrak q(0))}+O(e^{-Kh})
 \ge e^{-K\mathfrak q(0)}
          \bigl(a-Ce^{-K(h-\mathfrak q(0))}\bigr)
 \ge c e^{-K\mathfrak q(0)},
 \end{aligned}
\end{equation}
where the second equality is by \eqref{eq:lower-contour-error} and the last inequality follows from \eqref{eq:lower-parameter-bounds}
and $K\asymp\log n$, proving \eqref{eq:lower-measure-separation}.

In the last step, we prove \eqref{eq:lower-transform}. Recall  in
     \eqref{eq:lower-kernel} that
$
      \mathsf K_x(\lambda)=
      \frac{\lambda^2(\gamma^\circ x-\gamma^\circ+1-\lambda)}
      {\lambda(\gamma^\circ-1+\lambda)+\gamma^\circ x(1-\lambda)}.
$
      For $x\in[x_--\upsilon,x_++\upsilon]$, define
      \begin{equation}\label{eq:lower-kernel-roots}
       z_\pm=\frac{1+\gamma^\circ(x-1)
          \pm\gamma^\circ\sqrt{(x-x_-)(x-x_+)}}2,
      \end{equation}
     which are roots of the denominator of $\mathsf K_x(\lambda)$. We then have
      \begin{equation}\label{eq:lower-kernel-division}
       \begin{aligned}
            &\lambda(\gamma^\circ-1+\lambda)
            +\gamma^\circ x(1-\lambda)
            =(\lambda-z_+)(\lambda-z_-),\\
       &w(\lambda)\mathsf K_x(\lambda)
       =\lambda-(1+a)
       +
       \frac{(a-\gamma^\circ x)\lambda
       +\gamma^\circ x+a(\gamma^\circ-1)}
       {(\lambda-z_+)(\lambda-z_-)}.
       \end{aligned}
      \end{equation}
      Integrating the last identity above over $(a,1)$ gives
      \begin{equation}\label{eq:lower-integrated-kernel}
       \begin{aligned}
       \int_a^1w(\lambda)\mathsf K_x(\lambda)\eta(\dd\lambda)
       &=\int_a^1\lambda\,\eta(\dd\lambda)
         -(1+a)\int_a^1\eta(\dd\lambda)+(a-\gamma^\circ x)
          \frac{z_+\mathcal J(z_+)-z_-\mathcal J(z_-)}{z_+-z_-}\\
       &\qquad+\bigl(\gamma^\circ x+a(\gamma^\circ-1)\bigr)
          \frac{\mathcal J(z_+)-\mathcal J(z_-)}{z_+-z_-}.
       \end{aligned}
      \end{equation}
      When $z_+=z_-=z$, the above divided differences of $\mathcal J$ and
      $z\mathcal J$ mean $\mathcal J'(z)$ and
      $\mathcal J(z)+z\mathcal J'(z)$, respectively.
      By \eqref{eq:lower-parameter-bounds}, for
      $x\in[x_--\upsilon,x_++\upsilon]$,
      $
       |a-\gamma^\circ x|
       +\bigl|\gamma^\circ x+a(\gamma^\circ-1)\bigr|\le C
      $, so \eqref{eq:lower-integrated-kernel} is well-defined.

      For $x\in[x_-,x_+]$, write
      $x=1+(\gamma^\circ)^{-1}+2(\gamma^\circ)^{-1/2}\cos t$,
      $t\in[0,\pi]$. Then \eqref{eq:lower-kernel-roots} gives
      \begin{equation}\label{eq:lower-band-roots}
       z_\pm=1+\sqrt{\gamma^\circ}e^{\pm\ii t},\qquad
       |z_+-z_-|=2\sqrt{\gamma^\circ}|\sin t|.
      \end{equation}
The following estimates hold for both $F(z)=\mathcal J(z)$
and $F(z)=z\mathcal J(z)$. Equations~\eqref{eq:lower-band-roots},
\eqref{eq:lower-reciprocal-bounds},
\eqref{eq:lower-cauchy-decomposition}, and
\eqref{eq:lower-contour-error} give
\[
 |F(z_\pm)|
 \le C\biggl(\frac1{|f(z_\pm)|}+|\mathcal E(z_\pm)|\biggr)
 \le\frac C{\sinh(Kh)}+Ce^{-Kh}\le Ce^{-Kh}.
\]

      Thus, if $|\sin t|\ge K^{-1}$,
      \begin{equation}\label{eq:lower-far-divided-difference}
       \left|\frac{F(z_+)-F(z_-)}{z_+-z_-}\right|
       \le\frac{|F(z_+)|+|F(z_-)|}
                    {2\sqrt{\gamma^\circ}|\sin t|}
       \le\frac{Ce^{-Kh}}{|\sin t|}.
      \end{equation}
      If $|\sin t|<K^{-1}$, then
      \begin{equation}\label{eq:lower-segment-distance}
       \begin{aligned}
       \sup_{0\le s\le1}
       \Bigl|\big|(1-s)z_++sz_--1\big|-\sqrt{\gamma^\circ}\Bigr|
       &=\sqrt{\gamma^\circ}(1-|\cos t|)
       \le\sqrt{\gamma^\circ}\sin^2t\le CK^{-2}.
       \end{aligned}
      \end{equation}
For $z=(1-s)z_++sz_-$, $0\le s\le1$,
\eqref{eq:lower-collar-bounds} and \eqref{eq:lower-segment-distance}
give
$
 \operatorname{Re}\mathfrak q(z)\ge h-CK^{-2}\ge h/2$ and
 $
|\mathfrak q'(z)|\le C.
$
Hence \eqref{eq:lower-reciprocal-bounds},
\eqref{eq:lower-cauchy-decomposition}, and
\eqref{eq:lower-contour-error} imply
\[
 \begin{aligned}
 \frac1{|f(z)|}+K^{-1}\left|\left(\frac1f\right)'(z)\right|
 &\le Ce^{-K\operatorname{Re}\mathfrak q(z)}
 \le Ce^{-Kh+C/K}\le Ce^{-Kh},\\
 |F'(z)|
 &\le C\biggl(\frac1{|f(z)|}
      +\left|\left(\frac1f\right)'(z)\right|
      +|\mathcal E(z)|+|\mathcal E'(z)|\biggr)
 \le CKe^{-Kh}.
 \end{aligned}
\]
Consequently, we have
      \begin{equation}\label{eq:lower-near-divided-difference}
     \!\!\!\!\!\!  \left|\frac{F(z_+)-F(z_-)}{z_+-z_-}\right|
       =\left|\int_0^1F'\bigl((1-s)z_-+sz_+\bigr)\dd s\right|
       \le\sup_{0\le s\le1}
               \bigl|F'\bigl((1-s)z_-+sz_+\bigr)\bigr|
         \le CKe^{-Kh}.
      \end{equation}
      Substituting \eqref{eq:lower-far-divided-difference},
      \eqref{eq:lower-near-divided-difference}, and
      \eqref{eq:lower-contour-error} into
      \eqref{eq:lower-integrated-kernel} yields
      \begin{equation}\label{eq:lower-band-transform}
       \left|\int_a^1 w(\lambda)\mathsf K_x(\lambda)
       \eta(\dd\lambda)\right|
       \le\frac{CKe^{-Kh}}{1+K|\sin t|},\qquad t\in[0,\pi].
      \end{equation}
      For $x\in[x_-,x_+]$,
      $
       (x-x_-)(x_+-x)=4(\gamma^\circ)^{-1}\sin^2t$ and
       $
       x_+-x_-=4(\gamma^\circ)^{-1/2}\asymp1.
      $
      Hence the definition of $\mathcal W(x)$ in \eqref{eq:edge-weight} gives
   \begin{equation}\label{eq:lower-band-transform-Kx}
          \frac{K^2}{(1+K|\sin t|)^2}
          \le\frac1{K^{-2}+\sin^2t}
          \le\frac C{K^{-2}+\min(x-x_-,x_+-x)}
          \le C\mathcal W(x).
   \end{equation}
      For $x\in[x_--\upsilon,x_-]$ and
      $x\in[x_+,x_++\upsilon]$, respectively,
      \eqref{eq:lower-kernel-roots} implies, for all sufficiently large $n$, 
      \begin{equation}\label{eq:lower-exterior-root-distance}
       \begin{aligned}
       \max_{j\in\{+,-\}}|z_j-(1-\sqrt{\gamma^\circ})|
       &\le C\sqrt{x_--x}\le C\sqrt\upsilon<\varepsilon_*,\\
       \max_{j\in\{+,-\}}|z_j-(1+\sqrt{\gamma^\circ})|
       &\le C\sqrt{x-x_+}\le C\sqrt\upsilon<\varepsilon_*.
       \end{aligned}
      \end{equation}
For either range of $x$ and $z=(1-s)z_++sz_-$, $0\le s\le1$,
\eqref{eq:lower-collar-bounds},
\eqref{eq:lower-exterior-root-distance},
\eqref{eq:lower-reciprocal-bounds}, and
\eqref{eq:lower-contour-error} give
\[
 \begin{aligned}
 \operatorname{Re}\mathfrak q(z)&\ge h-C\sqrt\upsilon\ge h/2,
                  \qquad |\mathfrak q'(z)|\le C,\\
 |F'(z)|&\le CKe^{-K\operatorname{Re}\mathfrak q(z)}+Ce^{-Kh}
 \le CKe^{-Kh+CK\sqrt\upsilon}\le CKe^{-Kh},
 \end{aligned}
\]
since $K\sqrt\upsilon=o(1)$.
Equations
      \eqref{eq:lower-near-divided-difference} and
      \eqref{eq:lower-integrated-kernel} therefore yield, for $  x\in[x_--\upsilon,x_-]\cup[x_+,x_++\upsilon]$
   \begin{equation}\label{eq:lower-band-transform-squared-bound}
      \left|\int_a^1w(\lambda)\mathsf K_x(\lambda)\eta(\dd\lambda)\right|^2
      \le CK^2e^{-2Kh}
      \le Ce^{-2Kh}\mathcal W(x),
   \end{equation}
      where the second inequality is by \eqref{eq:edge-weight} and $\upsilon K^2=o(1)$,
      $
       \mathcal W(x)\ge\frac1{\sqrt{\upsilon^2+K^{-4}}}\ge cK^2.
      $

Combining \eqref{eq:lower-band-transform-squared-bound} for $  x\in[x_--\upsilon,x_-]\cup[x_+,x_++\upsilon]$ with \eqref{eq:lower-band-transform} and \eqref{eq:lower-band-transform-Kx} for $ x\in[x_-, x_+]$, for every
$x\in[x_--\upsilon,x_++\upsilon]$, we obtain
\begin{equation}\label{eq:lower-residue-transform}
 \begin{aligned}
 \left|\int_a^1\mathsf K_x(\lambda)(\nu_1-\nu_0)(\dd\lambda)\right|^2
 &=\left|\int_a^1w(\lambda)\mathsf K_x(\lambda)\eta(\dd\lambda)\right|^2
 \le Ce^{-2Kh}\mathcal W(x).
 \end{aligned}
\end{equation}
This proves \eqref{eq:lower-transform}.
\end{proof}
      
\subsection{Concentration of the Mixing Matrix}

For a twice continuously differentiable real-valued function $\phi$
on the positive definite symmetric $s\times s$ matrices and a symmetric
direction $E\in\R^{s\times s}$, write
\[
 D^2\phi(\mathsf Q)[E,E]
 =\left.\frac{\dd^2}{\dd t^2}\phi(\mathsf Q+tE)\right|_{t=0}.
\]
We view the Hessian as a self-adjoint operator on symmetric matrices
with Frobenius inner product $\langle E,F\rangle_F=\tr(EF)$.
For $A=(a_{ij}),B\in\R^{s\times s}$, the Kronecker product
$A\otimes B=(a_{ij}B)_{1\le i,j\le s}$ is an $s^2\times s^2$ block
matrix. Writing $\operatorname{vec}(E)$ for the vector obtained by
stacking the columns of $E$, for symmetric $A,E\in\R^{s\times s}$ we have
$
 \tr(AEAE)
 =\operatorname{vec}(E)^\top(A\otimes A)\operatorname{vec}(E).
$

\begin{lemma}[Concentration of the full Gram matrix]\label{lem:prior-concentration}
For the density \eqref{eq:Vprior},
\begin{equation}\label{eq:prior-event}
 \pi(\mathcal C^c)
 \le \exp(-c\sqrt n+C(\log n)^3)=o(n^{-D_0})
\end{equation}
for every fixed $D_0>0$. On $\mathcal C$,
$\|V\|_\op\le2$ and
$\lambda_{\min}(\Sigma_-(V))\ge a-C\upsilon^2$.
\end{lemma}
\begin{proof}
It suffices to prove the tail bound
\[
 \pi(\|\mathsf Q-I_s\|_\op>t)
 \le e^{-cdt^2+Cs^2\log d},\qquad 0<t<1/2,
\]
and then take $t=\upsilon^2$. A change of variables gives the density of $\mathsf Q=V^\top V$, with respect to $\dd\mathsf Q=\prod_{k\le l}\dd\mathsf Q_{kl}$, proportional to $e^{\phi(\mathsf Q)}$, where
\[
 \phi(\mathsf Q)=
 \frac{d-s-1}{2}\log\det\mathsf Q
 -\frac n2\log\det(I_s+\mathsf D^{1/2}\mathsf Q\mathsf D^{1/2})
 -\frac d2\tr\bigl(\diag(t_k)\mathsf Q\bigr).
\]
We bound the tail integral and the normalizing integral separately below.

\proofstep{Decrease from the mode.}
The choice of $t_k$ gives $\nabla\phi(I_s)=0$. For a symmetric direction $E$, put
$\mathsf B=\mathsf D^{1/2}(I_s+\mathsf D^{1/2}\mathsf Q\mathsf D^{1/2})^{-1}\mathsf D^{1/2}$.
Since $0\preceq\mathsf B\preceq\mathsf Q^{-1}$ and
$\mathsf B\otimes\mathsf B\preceq\mathsf Q^{-1}\otimes\mathsf Q^{-1}$,
 \begin{align}
      \label{eq:Gram-concavity-bound}
 D^2\phi(\mathsf Q)[E,E]
 &=-\frac{d-s-1}{2}\tr(\mathsf Q^{-1}E\mathsf Q^{-1}E)
   +\frac n2\tr(\mathsf B E\mathsf B E)\\
 &\le-\frac{d-s-1-n}{2}\tr(\mathsf Q^{-1}E\mathsf Q^{-1}E)
 \le-cd\,\tr(\mathsf Q^{-1}E\mathsf Q^{-1}E).\notag
 \end{align}
Integration along the segment from $I_s$ to $\mathsf Q$ yields
\begin{equation}\label{eq:Gram-concavity}
 \phi(\mathsf Q)-\phi(I_s)
 \le-cd\,\tr(\mathsf Q-I_s-\log\mathsf Q).
\end{equation}

\proofstep{Normalizer and tail.}
On $\|\mathsf Q-I_s\|_F\le d^{-1/2}$,
$\|\mathsf Q^{-1}\|_\op\le(1-d^{-1/2})^{-1}$.
The preceding Hessian formula \eqref{eq:Gram-concavity-bound} therefore gives the operator norm of
$D^2\phi(\mathsf Q)$ induced by the Frobenius norm
\[
 \begin{aligned}
 \|D^2\phi(\mathsf Q)\|_{F\to F}
 &:=\sup_{\substack{E=E^\top\\\|E\|_F=1}}
       |D^2\phi(\mathsf Q)[E,E]|\le\frac{d-s-1+n}{2}\|\mathsf Q^{-1}\|_\op^2
 \le Cd.
 \end{aligned}
\]

For $E=E^\top$ with $\|E\|_F\le d^{-1/2}$,
$\lambda_{\min}(I_s+uE)\ge1-d^{-1/2}>0$ for $0\le u\le1$.
Taylor's formula and $\nabla\phi(I_s)=0$ give
 \begin{align}\label{eq:Gram-Taylor-lower-bound}
 \phi(I_s+E)-\phi(I_s)
 &=\int_0^1(1-u)D^2\phi(I_s+uE)[E,E]\dd u
 \ge-\frac{Cd}{2}\|E\|_F^2\ge-C.
 \end{align}
Consequently, with $\mathcal Q_d = \{\mathsf Q=\mathsf Q^\top, \|\mathsf Q-I_s\|_F\le d^{-1/2}\}$, we have
 \begin{align}\label{eq:Gram-Taylor-lower-bound-integral}
 \int_{\mathsf Q\succ0}e^{\phi(\mathsf Q)}\dd\mathsf Q
 \ge\int_{\mathcal Q_d}
       e^{\phi(I_s)-C}\dd\mathsf Q
 \ge e^{\phi(I_s)-C}
       \biggl(\frac{2}{s\sqrt d}\biggr)^{s(s+1)/2}
 \ge e^{\phi(I_s)-Cs^2\log d}.
 \end{align}
Here the first inequality applies $\mathcal Q_d\subseteq\{\mathsf Q\succ0\}$ and  \eqref{eq:Gram-Taylor-lower-bound} for $E= \mathsf Q - I_s$; the second uses the volume bound $\operatorname{vol}(\mathcal Q_d) {\ge} (2/(s\sqrt d))^{s(s+1)/2}$ as the cube 
$\tilde{\mathcal Q}_d := \{\max_{1\le k\le l\le s}|\mathsf Q_{kl}-(I_s)_{kl}|
\le(s\sqrt d)^{-1}\}$ satisfies $\tilde{\mathcal Q}_d \subseteq {\mathcal Q}_d$ by
$\|\mathsf Q-I_s\|_F\le
s\max_{1\le k\le l\le s}|\mathsf Q_{kl}-(I_s)_{kl}|$ and has volume
$\operatorname{vol}(\tilde{\mathcal Q}_d) = (2/(s\sqrt d))^{s(s+1)/2}$; and the
last inequality uses $s=O(\log d)$.

For $\mathsf Q\succ0$ with $\|\mathsf Q-I_s\|_\op>t$ and $0<t<1/2$,
\[
 \begin{aligned}
 \tr(\mathsf Q-I_s-\log\mathsf Q)
 &=\sum_{k=1}^s
   \bigl(\lambda_k(\mathsf Q)-1-\log\lambda_k(\mathsf Q)\bigr)
   \ge\min\bigl\{t-\log(1+t),-t-\log(1-t)\bigr\}
 \ge\frac{t^2}{3}.
 \end{aligned}
\]
Let $c_0>0$ be the constant in \eqref{eq:Gram-concavity} and set
$r=c_0d/2$. Since $2r=c_0d$, on the same set,
\[
 \begin{aligned}
 e^{\phi(\mathsf Q)-\phi(I_s)}
 &\le e^{-r\tr(\mathsf Q-I_s-\log\mathsf Q)}
        e^{-r\tr(\mathsf Q-I_s-\log\mathsf Q)}
\le e^{-rt^2/3}
        e^{-r\tr(\mathsf Q-I_s-\log\mathsf Q)}.
 \end{aligned}
\]
Since $r\asymp d$ and $s=O(\log d)$, Lemma~\ref{lem:matrix-volume},
with its scalar parameter $t=r$, gives
\[
 \begin{aligned}
 \int_{\mathsf Q\succ0}
   e^{-r\tr(\mathsf Q-I_s-\log\mathsf Q)}\dd\mathsf Q
 &=e^{rs}r^{-s(r+(s+1)/2)}\pi^{s(s-1)/4}
   \prod_{k=1}^s\Gamma\biggl(r+\frac{s+2-k}{2}\biggr)\le e^{Cs^2\log d},
 \end{aligned}
\]
where the last inequality is by the result in Lemma~\ref{lem:matrix-volume} as well.
Thus
\[
 \int_{\substack{\mathsf Q\succ0\\\|\mathsf Q-I_s\|_\op>t}}
     e^{\phi(\mathsf Q)-\phi(I_s)}\dd\mathsf Q
 \le e^{-rt^2/3+Cs^2\log d}.
\]
Combining this bound with \eqref{eq:Gram-Taylor-lower-bound-integral}, we have
\[
 \begin{aligned}
 \pi(\|\mathsf Q-I_s\|_\op>t)
 &={\displaystyle
     \int_{\substack{\mathsf Q\succ0\\\|\mathsf Q-I_s\|_\op>t}}
        e^{\phi(\mathsf Q)}\dd\mathsf Q}\Big /
        {\displaystyle\int_{\mathsf Q\succ0}
        e^{\phi(\mathsf Q)}\dd\mathsf Q}
        \le\frac{e^{\phi(I_s)-rt^2/3+C_1s^2\log d}}
            {e^{\phi(I_s)-C_2s^2\log d}}
 \le e^{-cdt^2+Cs^2\log d}.
 \end{aligned}
\]
Since
$d\upsilon^4\asymp\sqrt n$ and $s^2\log d=O((\log n)^3)$, the choice $t=\upsilon^2$ yields \eqref{eq:prior-event}.

On $\mathcal C$, $\|V\|_\op^2=\|\mathsf Q\|_\op\le1+\upsilon^2$. The nonzero eigenvalues of $V\mathsf D V^\top$ are those of $\mathsf D^{1/2}\mathsf Q\mathsf D^{1/2}$, whose largest eigenvalue is at most $(1+\upsilon^2)(1-a)/a$. Hence we have
\[
 \lambda_{\min}(\Sigma_-(V))
 \ge\left(1+(1+\upsilon^2)\frac{1-a}{a}\right)^{-1}
 \ge a-C\upsilon^2.
\]
\end{proof}

\subsection{Conditional Likelihood and Total Variation}

For a real symmetric matrix
$A=O\diag(a_1,\ldots,a_n)O^\top$ with $O^\top O=I_n$ and a function $f:\R\to\R$,  we define
\[
 f(A)
 :=O\diag\bigl(f(a_1),\ldots, f(a_n)\bigr)O^\top,
 \qquad
 \frac1n\tr f(A)=\frac1n\sum_{i=1}^nf(a_i).
\]

\begin{lemma}[Conditional Gaussian comparison]\label{lem:conditional-comparison}
      Recall that for $j=0,1$, $\widetilde P_j$ is the
law of $(\mathsf X,\mathsf Y)$ in 
\eqref{eq:auxiliary-regression-observations} under hypothesis $j$, mixed over
$V\sim\pi$.
The laws $\widetilde P_0,\widetilde P_1$ have the same design marginal. Given $\mathsf X$, the columns of $V$ are independent and
\begin{equation}\label{eq:posterior-columns}
 V_k\mid\mathsf X\sim
 N_d\left(0,(dt_kI_d+\mathsf D_{kk}\mathsf X^\top\mathsf X)^{-1}\right).
\end{equation}
The conditional law of $\mathsf Y \mid \mathsf X$ under hypothesis $j$ is
$N(0,\mathcal V_j(\mathsf X))$, where
\begin{equation}\label{eq:conditional-variance-Vj}
       \mathcal V_j(\mathsf X)=
       \varpi_jI_n+\tau\sum_{k=1}^s w_{jk}\mathsf X
       (dt_kI_d+\mathsf D_{kk}\mathsf X^\top\mathsf X)^{-1}\mathsf X^\top.
\end{equation}
Moreover, with $K$ chosen as in \eqref{eq:lower-degree}, for sufficiently small fixed $\tau$ and large $n$, we have
\begin{equation}\label{eq:TVaux}
 \operatorname{TV}(\widetilde P_0,\widetilde P_1)
 \le C\tau(ne^{-2Kh}+K^2/n)^{1/2}+Ce^{-cn\upsilon^2} \le 1/4.
\end{equation}
\end{lemma}
\begin{proof}
The common mixing law $V\sim\pi$ gives the same design $\mathsf X$ marginal under both hypotheses. 

We next  verify the formula for $\mathcal V_j(\mathsf X)$.
Since $\Sigma_-(V)^{-1}=I_d+V\mathsf D V^\top$, the density of the design $\mathsf X$ is
\[
 p(\mathsf X\mid V)
 =(2\pi)^{-nd/2}\det(I_d+V\mathsf D V^\top)^{n/2}
 \exp\biggl(-\frac12\|\mathsf X\|_F^2
            -\frac12\sum_{k=1}^s\mathsf D_{kk}
                              \|\mathsf X V_k\|_2^2\biggr).
\]
Multiplying by \eqref{eq:Vprior} cancels the determinant factor.
Bayes' formula gives the conditional density as follows:
\begin{equation}\label{eq:lower-posterior-density}
 \begin{aligned}
 \pi(V\mid\mathsf X)
 =(2\pi)^{-ds/2}\prod_{k=1}^s
       \det(dt_kI_d+\mathsf D_{kk}\mathsf X^\top\mathsf X)^{1/2}\exp\biggl(-\frac12\sum_{k=1}^s
       V_k^\top(dt_kI_d+\mathsf D_{kk}\mathsf X^\top\mathsf X)V_k
       \biggr).
 \end{aligned}
\end{equation}
This density factors over $V_1,\ldots,V_s$, proving their conditional
independence and \eqref{eq:posterior-columns}.
For $j=0,1$, since
$\mathsf XV\vartheta_j=
\sqrt\tau\sum_{k=1}^s\sqrt{w_{jk}}\,\mathsf X V_k$,
\begin{equation}
      \label{eq:conditional-Gaussian-law}
       \mathsf XV\vartheta_j\mid\mathsf X
       \sim N_n\biggl(0,\tau\sum_{k=1}^s w_{jk}\mathsf X
             (dt_kI_d+\mathsf D_{kk}\mathsf X^\top\mathsf X)^{-1}
             \mathsf X^\top\biggr)
       =N_n\bigl(0,\mathcal V_j(\mathsf X)-\varpi_jI_n\bigr).
\end{equation}
The regression model \eqref{eq:auxiliary-regression-observations} gives
$\mathsf Y\mid(\mathsf X,V)\sim
N_n(\mathsf XV\vartheta_j,\varpi_jI_n)$.
Convolving it with \eqref{eq:conditional-Gaussian-law}, for $y\in\R^n$, the conditional density of $\mathsf Y\mid\mathsf X$ is
\begin{equation}\label{eq:lower-conditional-response-density}
 \begin{aligned}
 p_j(y\mid\mathsf X)
 &=\int_{\R^{d\times s}}
       \frac{\exp\bigl(-\|y-\mathsf XV\vartheta_j\|_2^2/
                            (2\varpi_j)\bigr)}
            {(2\pi\varpi_j)^{n/2}}
       \pi(V\mid\mathsf X)\dd V
       =\frac{\exp\bigl(-y^\top\mathcal V_j(\mathsf X)^{-1}y/2\bigr)}
          {(2\pi)^{n/2}\det(\mathcal V_j(\mathsf X))^{1/2}}.
 \end{aligned}
\end{equation}
Thus $\mathsf Y\mid\mathsf X\sim N_n(0,\mathcal V_j(\mathsf X))$.

Let $u\in\R^n$ satisfy $\|u\|_2=1$ and
$(\mathsf X\mathsf X^\top/d)u=xu$, i.e., $u$ is an eigenvector of $\mathsf X\mathsf X^\top/d$ with eigenvalue $x$. For $1\le k\le s$,
\[
 \begin{aligned}
 (dt_kI_d+\mathsf D_{kk}\mathsf X^\top\mathsf X)\mathsf X^\top u
 &=d(t_k+\mathsf D_{kk}x)\mathsf X^\top u,\\
 \mathsf X(dt_kI_d+\mathsf D_{kk}\mathsf X^\top\mathsf X)^{-1}
       \mathsf X^\top u
 &=\frac{\mathsf X\mathsf X^\top u}{d(t_k+\mathsf D_{kk}x)}
 =\frac{x}{t_k+\mathsf D_{kk}x}u.
 \end{aligned}
\]
Thus the same $u$ is an eigenvector of $\mathcal V_j(\mathsf X)$,
with eigenvalue $v_j(x)$, for $j=0,1$:
\begin{equation}\label{eq:lower-conditional-variance-eigenvalues}
 \mathcal V_j(\mathsf X)u=v_j(x)u,\qquad
 v_j(x)=\varpi_j+
 \tau\sum_{k=1}^s\frac{w_{jk}x}{t_k+\mathsf D_{kk}x}.
\end{equation}

By \eqref{eq:lower-atomic-weights} and
\eqref{eq:lower-auxiliary-parameters}, we have
\[
 \varpi_1-\varpi_0
 =-\tau\int_a^1\lambda\,(\nu_1-\nu_0)(\dd\lambda)
 =-\tau\sum_{k=1}^s z_k(w_{1k}-w_{0k}).
\]
Substituting this identity and
\eqref{eq:lower-prior-precision-parameters} into
\eqref{eq:lower-conditional-variance-eigenvalues} gives, for
$x\in[x_--\upsilon,x_++\upsilon]$,
\begin{equation}\label{eq:lower-response-variance-difference}
 \begin{aligned}
 v_1(x)-v_0(x)
 &=\tau\sum_{k=1}^s(w_{1k}-w_{0k})
       \biggl(\frac{x}{t_k+\mathsf D_{kk}x}-z_k\biggr)\\
 &=\tau\int_a^1
 \biggl(
 \frac{x}{1-(1-\lambda)/\gamma^\circ-(s+1)/d
                     +(1-\lambda)x/\lambda}
 -\lambda
 \biggr)(\nu_1-\nu_0)(\dd\lambda)\\
 &=\tau\int_a^1\mathsf K_x(\lambda)(\nu_1-\nu_0)(\dd\lambda)
   +O(\tau K/n).
 \end{aligned}
\end{equation}
The last equality uses \eqref{eq:lower-kernel} and
 \eqref{eq:lower-measure-size}, with
$t_k+\mathsf D_{kk}x\ge t_k\ge c$, $s=O(K)$, and $d\asymp n$.

\proofstep{The spectral event.}
Given $V$, write $\mathsf X=\mathsf Z\Sigma_-(V)^{1/2}$, where
$\mathsf Z\in\R^{n\times d}$ has iid $N(0,1)$ entries and is independent
of $V$. Choose $U_V\in\R^{d\times(d-s)}$ such that
$U_V^\top U_V=I_{d-s}$ and $V^\top U_V=0$, and set $P_V=U_VU_V^\top$.
Since $\Sigma_-(V)^{-1}=I_d+V\mathsf D V^\top$ with
$\mathsf D\succeq0$, we have $\Sigma_-(V)U_V=U_V$ and
\begin{equation}\label{eq:lower-spectral-order}
 P_V\preceq\Sigma_-(V)\preceq I_d,\qquad
 \frac{(\mathsf ZU_V)(\mathsf ZU_V)^\top}{d}
 \preceq\frac{\mathsf X\mathsf X^\top}{d}
 \preceq\frac{\mathsf Z\mathsf Z^\top}{d}.
\end{equation}
Conditionally on $V$, $\mathsf ZU_V$ has iid $N(0,1)$ entries.
By \eqref{eq:aspect-regime}, \eqref{eq:lower-degree}, and
\eqref{eq:lower-measure-size},
\[
 \frac{s}{d}=O\Bigl(\frac{\log n}{n}\Bigr)=o(\upsilon),
 1-\sqrt{n/d}\ge c>0 \text{ thus }  d-s\ge n, \frac{s}{d}\le\frac{\upsilon}{8},
 1-\sqrt{n/d}-\frac{s}{d}-\frac{\upsilon}{8}>0,
\]
for sufficiently large $n$.
Lemma~\ref{lem:gaussian-singular-values}, applied with
$t=\upsilon\sqrt d/8$, gives
\begin{align*}
 \PP\bigl(s_{\max}(\mathsf Z)>
       \sqrt d+\sqrt n+\upsilon\sqrt d/8\mid V\bigr)
 &\le e^{-d\upsilon^2/128},\\
 \PP\bigl(s_{\min}(\mathsf ZU_V)<
       \sqrt{d-s}-\sqrt n-\upsilon\sqrt d/8\mid V\bigr)
 &\le e^{-d\upsilon^2/128}.
\end{align*}
Outside the two events above, \eqref{eq:lower-spectral-order} yields
\begin{align*}
 \lambda_{\max}(\mathsf X\mathsf X^\top/d)
 &\le\frac{s_{\max}(\mathsf Z)^2}{d}
 \le\bigl(1+\sqrt{n/d}+\upsilon/8\bigr)^2
 \le x_++\frac{\upsilon}{2}+\frac{\upsilon^2}{64}
 \le x_++\upsilon,\\
 \lambda_{\min}(\mathsf X\mathsf X^\top/d)
 &\ge\frac{s_{\min}(\mathsf ZU_V)^2}{d}
 \ge\bigl(\sqrt{1-s/d}-\sqrt{n/d}-\upsilon/8\bigr)^2\\
 &\ge\bigl(1-\sqrt{n/d}-s/d-\upsilon/8\bigr)^2
 \ge x_--\frac{2s}{d}-\frac{\upsilon}{4}
 \ge x_--\upsilon.
\end{align*}
Here $x_\pm=(1\pm\sqrt{n/d})^2$ and
$\sqrt{1-s/d}\ge1-s/d$. Let
$\mathcal E=\{\operatorname{spec}(\mathsf X\mathsf X^\top/d)
\subset[x_--\upsilon,x_++\upsilon]\}$. We therefore have
\begin{equation}\label{eq:lower-spectral-event-probability}
 \PP(\mathcal E^c\mid V)\le2e^{-d\upsilon^2/128},\qquad
 \PP(\mathcal E^c)=\E[\PP(\mathcal E^c\mid V)]
 \le2e^{-d\upsilon^2/128}\le2e^{-cn\upsilon^2}.
\end{equation}

As $\mathsf Y\mid\mathsf X\sim N_n(0,\mathcal V_j(\mathsf X))$, \eqref{eq:lower-gaussian-tv-frobenius} gives
\begin{equation}\label{eq:Gaussian-comparison-total-variation}
 \begin{aligned}
 \operatorname{TV}(\widetilde P_0,\widetilde P_1)
 &=\E\bigl[\operatorname{TV}\bigl(
       N_n(0,\mathcal V_0(\mathsf X)),
       N_n(0,\mathcal V_1(\mathsf X))\bigr)\bigr]
       \le\frac3{s_0}\E\bigl[
       \ind_{\mathcal E}\|\mathcal V_1(\mathsf X)
                         -\mathcal V_0(\mathsf X)\|_F\bigr]
       +\PP(\mathcal E^c)\\
 &\le\frac3{s_0}\left(\E\left[
       \ind_{\mathcal E}\|\mathcal V_1(\mathsf X)
                         -\mathcal V_0(\mathsf X)\|_F^2
       \right]\right)^{1/2}+\PP(\mathcal E^c),
 \end{aligned}
\end{equation}
where the first inequality also uses $\operatorname{TV}\le1$ on
$\mathcal E^c$ and
$\mathcal V_j(\mathsf X)\succeq\varpi_jI_n\succeq(s_0/2)I_n$ and the last follows from the Cauchy--Schwarz inequality. To prove \eqref{eq:TVaux}, what remains is to bound $\E[
      \ind_{\mathcal E}\|\mathcal V_1(\mathsf X)
                        -\mathcal V_0(\mathsf X)\|_F^2]$.

\proofstep{The covariance difference.}
Since $I_d-\Sigma_-(V)=V\mathsf D V^\top\Sigma_-(V)$,
\[
 \operatorname{rank}\bigl(
       \mathsf Z\mathsf Z^\top/d-\mathsf X\mathsf X^\top/d\bigr)
 =\operatorname{rank}\bigl(
       \mathsf Z(I_d-\Sigma_-(V))\mathsf Z^\top\bigr)
 \le\operatorname{rank}(V\mathsf D V^\top)\le s.
\]
Let $F_T(x)=\frac1n\sum_{i=1}^n\ind_{\{\lambda_i(T)\le x\}}$ be the spectral distribution of a matrix $T$. By \citet[Theorem~A.43]{BaiSilverstein2010} and the rank bound above,
\begin{equation}\label{eq:lower-rank-distribution}
 \sup_{x\in\R}|F_{\mathsf X\mathsf X^\top/d}(x)-F_{\mathsf Z\mathsf Z^\top/d}(x)|
 \le\frac{\operatorname{rank}\bigl(
      \mathsf Z\mathsf Z^\top/d-\mathsf X\mathsf X^\top/d\bigr)}n\le\frac{s}{n}.
\end{equation}
By \eqref{eq:edge-weight}, $\mathcal W\in C^1(\R)$,
$\lim_{|x|\to\infty}\mathcal W(x)=0$, and its total variation satisfies
\begin{equation}\label{eq:lower-weight-variation}
 \begin{aligned}
 \int_{\R}|\mathcal W'(x)|\dd x
 &\le\sum_{e\in\{x_-,x_+\}}\int_{\R}
       \frac{|x-e|}{((x-e)^2+K^{-4})^{3/2}}\dd x
       =4\int_0^\infty\frac{u}{(u^2+K^{-4})^{3/2}}\dd u
 =4K^2.
 \end{aligned}
\end{equation}
Integration by parts, \eqref{eq:lower-rank-distribution}, and
\eqref{eq:lower-weight-variation} therefore give
\begin{align*}
 \left|\frac1n\tr\mathcal W(\mathsf X\mathsf X^\top/d)-\frac1n\tr\mathcal W(\mathsf Z\mathsf Z^\top/d)\right|
 &=\left|\int_{\R}\mathcal W(x)\dd(F_{\mathsf X\mathsf X^\top/d}-F_{\mathsf Z\mathsf Z^\top/d})(x)\right|\\
 &=\left|-\int_{\R}(F_{\mathsf X\mathsf X^\top/d}(x)-F_{\mathsf Z\mathsf Z^\top/d}(x))\mathcal W'(x)\dd x\right|\\
 &\le\sup_{x\in\R}|F_{\mathsf X\mathsf X^\top/d}(x)-F_{\mathsf Z\mathsf Z^\top/d}(x)|
       \int_{\R}|\mathcal W'(x)|\dd x
 \le\frac{4sK^2}{n}=o(1).
\end{align*}
Lemma~\ref{lem:weighted-Wishart} therefore implies
$\E[n^{-1}\sum_{i=1}^n\mathcal W(x_i)]\le C$ for the eigenvalues $x_i$ of $\mathsf X\mathsf X^\top/d$.

By \eqref{eq:lower-conditional-variance-eigenvalues},
\eqref{eq:lower-response-variance-difference}, and
\eqref{eq:lower-transform}, we have
 \begin{align}\label{eq:lower-covariance-difference-proof}
 \E\left[\ind_{\mathcal E}
   \|\mathcal V_1(\mathsf X)-\mathcal V_0(\mathsf X)\|_F^2\right]
 &=\E\Big[\ind_{\mathcal E}
              \sum_{i=1}^n|v_1(x_i)-v_0(x_i)|^2\Big]\\
 &\le C\tau^2\left(
 e^{-2Kh}\E\sum_{i=1}^n\mathcal W(x_i)+\frac{K^2}{n}\right)
 \le C\tau^2(ne^{-2Kh}+K^2/n). \notag
 \end{align}
Substituting the above inequality and  \eqref{eq:lower-spectral-event-probability}
into \eqref{eq:Gaussian-comparison-total-variation} proves \eqref{eq:TVaux}. 
\end{proof}

\begin{lemma}\label{lem:weighted-Wishart}
Let $\mathsf Z$ be $n\times d$ standard Gaussian, with $d/n$ in a compact subset of $(1,\infty)$. For $\mathcal W$ in \eqref{eq:edge-weight} and $K=O(\log n)$,
$\E[n^{-1}\tr\mathcal W(\mathsf Z\mathsf Z^\top/d)]\le C$.
\end{lemma}
\begin{proof}
Recall $\gamma^\circ = d/n$ and $x_\pm=(1\pm(\gamma^\circ)^{-1/2})^2$. Let $\mu_{\gamma^\circ}$ be the Marchenko--Pastur law
\citep[Eq.~(3.1.1)]{BaiSilverstein2010}, with density
\begin{equation}\label{eq:lower-mp-density}
 \mu_{\gamma^\circ}(\dd x)=
 \frac{\gamma^\circ}{2\pi x}
 \sqrt{(x-x_-)(x_+-x)}\,\ind_{(x_-,x_+)}(x)\dd x.
\end{equation}

The bounds on $\gamma^\circ$ imply
\[
 1+c\le\gamma^\circ\le C,\qquad
 c\le x_-<x_+\le4,\qquad
 c\le x_+-x_-=\frac4{\sqrt{\gamma^\circ}}\le4.
\]
By \eqref{eq:edge-weight} and \eqref{eq:lower-mp-density},
\begin{equation}\label{eq:weighted-Wishart-mp-integral}
 \begin{aligned}
 \int_{x_-}^{x_+}\mathcal W(x)\,\mu_{\gamma^\circ}(\dd x)
 &\le C\int_{x_-}^{x_+}
 \biggl(
 \frac{\sqrt{x-x_-}}{\sqrt{(x-x_-)^2+K^{-4}}}
 +\frac{\sqrt{x_+-x}}{\sqrt{(x_+-x)^2+K^{-4}}}
 \biggr)\dd x\\
 &=2C\int_0^{x_+-x_-}
       \frac{t^{1/2}}{\sqrt{t^2+K^{-4}}}\dd t
 \le2C\int_0^{x_+-x_-}t^{-1/2}\dd t\le C.
 \end{aligned}
\end{equation}

Let $x_1,\ldots,x_n$ be the eigenvalues of
$\mathsf Z\mathsf Z^\top/d$, and set
\[
 H=\frac1{\sqrt d}
 \begin{pmatrix}0&\mathsf Z\\ \mathsf Z^\top&0\end{pmatrix}.
\]
We aim to apply \citet[Corollary~2.9]{Bandeira2023} with parameters given by $\E H^2=\diag(I_n,(n/d)I_d)$, $
\|\E H^2\|_\op=1,$ and
\[
 \|\operatorname{Cov}(\operatorname{vec}H)\|_\op
 =\frac1d\sup_{\substack{M\in\R^{(n+d)\times(n+d)}\\\|M\|_F=1}}
       \sum_{i=1}^n\sum_{j=1}^d
       |M_{i,n+j}+M_{n+j,i}|^2
 =\frac2d.
\]
Choose a smooth, compactly supported function $\chi$ such that
\[
 0\le\chi\le1,\qquad \chi(-t)=\chi(t),\qquad
 \chi|_{[-3,3]}=1,\qquad \operatorname{supp}\chi\subset[-4,4].
\]
For example, take
\[
 \chi(t)=
 \begin{cases}
  1,& |t|\le3,\\[2pt]
  \displaystyle
  \frac{e^{-1/(16-t^2)}}{e^{-1/(16-t^2)}+e^{-1/(t^2-9)}},
       &3<|t|<4,\\[6pt]
  0,&|t|\ge4.
 \end{cases}
\]
Put $f(t)=\chi(t)\mathcal W(t^2)$. For $r=0,\ldots,6$,
\[
 \|\mathcal W^{(r)}\|_{\infty,\R}
 \le2K^{2r+2}\sup_{u\in\R}
       \left|\frac{\dd^r}{\dd u^r}(1+u^2)^{-1/2}\right|
 \le C_rK^{2r+2}.
\]

For $K\ge1$, Leibniz's rule and the chain rule therefore give 
\begin{equation}\label{eq:weighted-Wishart-sobolev}
 \|f\|_{W^{6,1}(\R)}
 :=\sum_{j=0}^6\|f^{(j)}\|_{L^1(\R)}
 \le8\sum_{j=0}^6\|f^{(j)}\|_{\infty,[-4,4]}
 \le C\sum_{j=0}^6\sum_{r=0}^j
       \|\mathcal W^{(r)}\|_{\infty,\R}
 \le CK^{14}.
\end{equation}

Let $\rho_H$ be the spectral probability law of the free model
associated with $H$ in \citet[Eq.~(2.1)]{Bandeira2023}.
The noncrossing-pairing formula in
\citet[Definition~4.2]{Bandeira2023}, with the tree count in
\citet[Lemma~3.4]{BaiSilverstein2010}, gives the odd moments
$
 \int_{\R}t^{2k+1}\rho_H(\dd t)=0,
$ for $k \ge 0$,
and, for every integer $k\ge1$, the even moments
\begin{align*}
 \int_{\R}t^{2k}\rho_H(\dd t)
 &=\frac{2}{(n+d)d^k}
       \sum_{r=0}^{k-1}\frac1{r+1}
       \binom{k}{r}\binom{k-1}{r}n^{r+1}d^{k-r}\\
 &=\frac{2n}{n+d}\sum_{r=0}^{k-1}\frac1{r+1}
       \binom{k}{r}\binom{k-1}{r}(\gamma^\circ)^{-r}
 =\frac{2n}{n+d}\int_{x_-}^{x_+}x^k\,\mu_{\gamma^\circ}(\dd x),
\end{align*}
where the last equality is \citet[Lemma~3.1]{BaiSilverstein2010}.
Moreover, $\operatorname{supp}\rho_H\subset[-2,2]$ by
\citet[Lemma~2.5]{Bandeira2023} and $\|\E H^2\|_\op=1$.
Since $f(-t)=f(t)$, applying the Weierstrass approximation theorem
\citep[Theorem~7.26]{Rudin1976} to $x\mapsto f(\sqrt x)$ on $[0,4]$
gives even polynomials
\[
 p_m(t)=\sum_{k=0}^m c_{m,k}t^{2k},\text{ with }
 \|p_m-f\|_{\infty,[-2,2]}\longrightarrow0
 \text{ as } m\to\infty.
\]
The preceding moment identities of $\rho_H$ imply
\begin{align*}
 \int_{\R}p_m(t)\rho_H(\dd t)
 &=c_{m,0}+\frac{2n}{n+d}\sum_{k=1}^m c_{m,k}
              \int_{x_-}^{x_+}x^k\,\mu_{\gamma^\circ}(\dd x)
              =\frac{d-n}{n+d}p_m(0)
   +\frac{2n}{n+d}\int_{x_-}^{x_+}p_m(\sqrt x)\,
                             \mu_{\gamma^\circ}(\dd x).
\end{align*}
Since $\rho_H(\R)=\mu_{\gamma^\circ}([x_-,x_+])=1$ and
$0<x_-<x_+\le4$,
\[
 \begin{aligned}
 &\left|\int_{\R}f(t)\rho_H(\dd t)
       -\frac{d-n}{n+d}f(0)
       -\frac{2n}{n+d}\int_{x_-}^{x_+}f(\sqrt x)\,
                                 \mu_{\gamma^\circ}(\dd x)\right|\\
 &\qquad\le
 \biggl(1+\frac{d-n}{n+d}+\frac{2n}{n+d}\biggr)
       \|p_m-f\|_{\infty,[-2,2]}
 =2\|p_m-f\|_{\infty,[-2,2]}\longrightarrow0.
 \end{aligned}
\]
This proves the following identities:
\begin{equation}\label{eq:weighted-Wishart-trace-identities}
 \begin{aligned}
 \int_{\R}f(t)\rho_H(\dd t)
 &=\frac{d-n}{n+d}f(0)
   +\frac{2n}{n+d}\int_{x_-}^{x_+}f(\sqrt x)\,
                            \mu_{\gamma^\circ}(\dd x),\\
 \frac1{n+d}\tr f(H)
 &=\frac{d-n}{n+d}f(0)
   +\frac2{n+d}\sum_{i=1}^nf(\sqrt{x_i}),
 \end{aligned}
\end{equation}
where the second identity above follows because $H$ has the eigenvalues
$\{\sqrt{x_i},-\sqrt{x_i}:1\le i\le n\}$ and $d-n$ zeros.

Applying \citet[Corollary~2.9]{Bandeira2023} and using
\eqref{eq:weighted-Wishart-sobolev}--\eqref{eq:weighted-Wishart-trace-identities},
we obtain
\begin{equation}\label{eq:weighted-Wishart-smooth-comparison}
 \begin{aligned}
      &\left|\E\left[\frac1n\sum_{i=1}^n f(\sqrt{x_i})\right]
       -\int_{x_-}^{x_+}f(\sqrt x)\,\mu_{\gamma^\circ}(\dd x)\right|
 =\frac{n+d}{2n}
       \left|\E\left[\frac1{n+d}\tr f(H)\right]
             -\int_{\R}f(t)\rho_H(\dd t)\right|\\
 &\qquad \le C\frac{n+d}{n}\|\E H^2\|_\op
       \|\operatorname{Cov}(\operatorname{vec}H)\|_\op
       \|f\|_{W^{6,1}(\R)}
\le C\frac{n+d}{nd}\|f\|_{W^{6,1}(\R)}
 \le\frac{CK^{14}}n.
 \end{aligned}
\end{equation}
Since $\sqrt{x_+}=1+\sqrt{n/d}<2$,
\[
 f(\sqrt x)=\mathcal W(x) \text{ for } x\in[x_-,x_+],\qquad
 0\le\mathcal W(x)-f(\sqrt x)\le2K^2\ind_{\{x>9\}}
 \text{ for } x\ge0.
\]
Lemma~\ref{lem:gaussian-singular-values} with
$t=2\sqrt d-\sqrt n\ge\sqrt d$ gives
\begin{equation*}\label{eq:weighted-Wishart-cutoff-error}
 \begin{aligned}
 0\le\E\left[\frac1n\sum_{i=1}^n
       \big(\mathcal W(x_i)-f(\sqrt{x_i})\big)\right]
 &\le2K^2\PP(s_{\max}(\mathsf Z)>3\sqrt d)
 \le2K^2e^{-(2\sqrt d-\sqrt n)^2/2}
 \le2K^2e^{-d/2}.
 \end{aligned}
\end{equation*}
Combining the above inequality with \eqref{eq:weighted-Wishart-mp-integral} and 
\eqref{eq:weighted-Wishart-smooth-comparison}, we conclude
\[
 \begin{aligned}
 \E\left[\frac1n\tr\mathcal W(\mathsf Z\mathsf Z^\top/d)\right]
 &\le\int_{x_-}^{x_+}\mathcal W(x)\,\mu_{\gamma^\circ}(\dd x)
       +\frac{CK^{14}}n+CK^2e^{-cn}\\
 &\le C+\frac{C(\log n)^{14}}n+C(\log n)^2e^{-cn}\le C.
 \end{aligned}
\]
For $0<K<1$, the claim follows directly from
$0\le\mathcal W(x)\le2K^2\le2$.
\end{proof}
\section{Results on the Elliptic Integral}\label{app:elliptic}

\subsection{Preliminaries on Elliptic Functions and Conformal Maps}
\label{app:elliptic-preliminaries}

We list the preliminaries on the elliptic integral and the conformal map needed for this paper, especially the formulas needed to construct the harmonic potential
in Section~\ref{sec:bias-var}. Throughout this section, for $0<k<1$,
abbreviate $\Kell(k^2)$ and $\Kell(1-k^2)$ by $\Kell$ and
$\Kell_{\mathrm c}$, respectively.

\paragraph{M\"obius transformation.}
For $R>L$, the M\"obius transformation is
\begin{align}
 \nonumber &M_R(z)=\frac{R-L+z}{R+L-z},
  \text{ and } \quad M_R^{-1}(t)=L-R\frac{1-t}{1+t} \\
 &\operatorname{Re}M_R(z)=\frac{R^2-|z-L|^2}{|R+L-z|^2},
 \operatorname{Im}M_R(z)=\frac{2R\operatorname{Im}z}{|R+L-z|^2}. \label{eq:mobius-half-disk-app}
\end{align}

The M\"obius transformation $M_R$ maps the disc $\{|z-L|<R\}$ onto the right half-plane $\{\operatorname{Re}z>0\}$, and the upper disc $\{|z-L|<R, \operatorname{Im}z>0\}$ onto the first quadrant $\{\operatorname{Re}z>0,\operatorname{Im}z>0\}$.
Besides, the squaring map $z \mapsto z^2$ maps the first quadrant conformally onto the upper half-plane 
$\C^+=\{\operatorname{Im}s>0\}$. See Figure~\ref{fig:conformal-app}(a)-(c) for the illustration.

\paragraph{Schwarz--Christoffel mapping.} The Schwarz--Christoffel formula in general maps a half-plane to a polygon \citep{driscoll2002schwarz}. In our paper, we only consider the case of mapping $\C^+$ to a rectangle using the following  Schwarz--Christoffel mapping:
\begin{equation}\label{eq:SC-rectangle-app}
      \Psi_k(s)=\frac12\int_0^s
            \frac{\dd\zeta}{\sqrt{\zeta(k^2-\zeta)(1-\zeta)}}
      =\Fell\!\left(\arcsin\frac{\sqrt{s}}{k}\,\middle|\,k^2\right).
\end{equation}
The square root in the integrand is the analytic branch on $\C^+$
whose boundary values are positive on $(0,k^2)$; $\sqrt s$ is the
principal square root. The $\arcsin$ and elliptic integral above are continued from $(0,k^2)$. The Schwarz--Christoffel mapping
$\Psi_k$ maps $\C^+$ onto
\begin{equation}\label{eq:rectangle-corners-app}
 \mathcal R_k^+=\{w:0<\operatorname{Re}w<\Kell,
                         \ 0<\operatorname{Im}w<\Kell_{\mathrm c}\} \text{ and maps }
                         0,\ k^2,\ 1,\ \infty
                         \mapsto 
                         0,\ \Kell,\ \Kell+\ii\Kell_{\mathrm c},
                               \ \ii\Kell_{\mathrm c},
                              \end{equation}
respectively. Figure~\ref{fig:conformal-app}(d) shows the illustration of the Schwarz--Christoffel mapping after rescaling. 


\subsection{The Harmonic Potential and the Holomorphic Filter}
\label{app:pf:lem:geometry}
The following lemma collects the results related to the elliptic integral and the harmonic potential in \eqref{eq:potential} in
Section~\ref{sec:bias-var}.  Recall
\[
 \mathcal A_\delta=\Ddisk_\delta\setminus[\ell,L], 
 \quad
 h_\delta=\frac{\pi\Kell}{\Kell_{\mathrm c}},  \quad k=\frac{R_\delta-L+\ell}{R_\delta+L-\ell},  \quad \text{and}
 \quad
 T_\delta(z)=\frac{R_\delta-L+z}{R_\delta+L-z}.  
\]

\begin{lemma}\label{lem:geometry}
Let $\delta\ge0$. The following assertions hold.
\begin{enumerate}
\item[\textup{(i)}]
The map $U_\delta$ in \eqref{eq:normalized-U} defines a conformal map
\[
 U_\delta:\mathcal A_\delta\longrightarrow
 \left\{v\in\C/\big((2\pi\ii/h_\delta)\mathbb Z\big):
                    0<\operatorname{Re}v<1\right\}.
\]
The function $u_\delta=\operatorname{Re}U_\delta$ is the unique bounded
solution of \eqref{eq:potential}, with continuous boundary values.

\item[\textup{(ii)}]
The map $q_\delta=h_\delta U_\delta$ takes values in
$\C/(2\pi\ii\mathbb Z)$ and satisfies
$
 \operatorname{Re}q_\delta=h_\delta u_\delta,
 q_\delta(0)=h_\delta u_\delta(0)>0.
$
For $s>0$, write
$
 \mathcal C_s
 =\{z:\bigl||z-L|-R_\delta\bigr|\le s\}.
$
Uniformly over $(R_\delta,L,\ell)$ in any compact subset of
$\{(R,L,\ell):R>L>\ell>0\}$, there exist
$\varepsilon_*,C_*>0$ such that $q_\delta$ extends
holomorphically with values in $\C/(2\pi\ii\mathbb Z)$ to a
neighborhood of $\mathcal C_{\varepsilon_*}$, and
\begin{equation}\label{eq:outer-collar-app}
 \operatorname{Re}q_\delta(z)\ge h_\delta
 \quad\text{for }R_\delta\le|z-L|\le R_\delta+\varepsilon_*,
 \qquad
 \sup_{z\in\mathcal C_{\varepsilon_*}}
 |q_\delta'(z)|\le C_*.
\end{equation}

\item[\textup{(iii)}] 
The identity for $u_\delta(0)$ in \eqref{eq:explicit-potential} holds, and
\begin{equation}\label{eq:u0-alpha}
 \alpha_\delta=u_\delta(0)
 =\alpha\bigl(\kappa,\gamma(1+\delta)^2\bigr)
 =\alpha(\kappa,\gamma)+O(\delta),\qquad \delta\downarrow0.
\end{equation}
For fixed $\kappa>1$, for sufficiently
small $\delta\ge0$, we have the uniform scales
\begin{equation}\label{eq:geometry-compact-bounds}
 \alpha_\delta\asymp1-\alpha_\delta\asymp h_\delta
 \asymp q_\delta(0)\asymp h_\delta-q_\delta(0)\asymp1.
\end{equation}

\item[\textup{(iv)}] 
For every integer $K\ge1$, the filter $g_K$ in \eqref{eq:gK} is
well-defined and extends holomorphically to a neighborhood of
$\overline{\Ddisk_\delta}$. It satisfies $g_K(0)=1$,
$g_K(\bar z)=\overline{g_K(z)}$ on the disk, and the contour formula
\eqref{eq:gK-contour}.
\end{enumerate}
For an arbitrary radius $R>L>\ell>0$, parts~\textup{(i)}, \textup{(ii)}, and~\textup{(iv)} hold with $R_\delta$ replaced by $R$, and the value of the harmonic potential at zero is $\alpha(\kappa,(R/L)^2)$.
\end{lemma}

\begin{figure}[t]
 \centering
 \includegraphics[width=\textwidth]{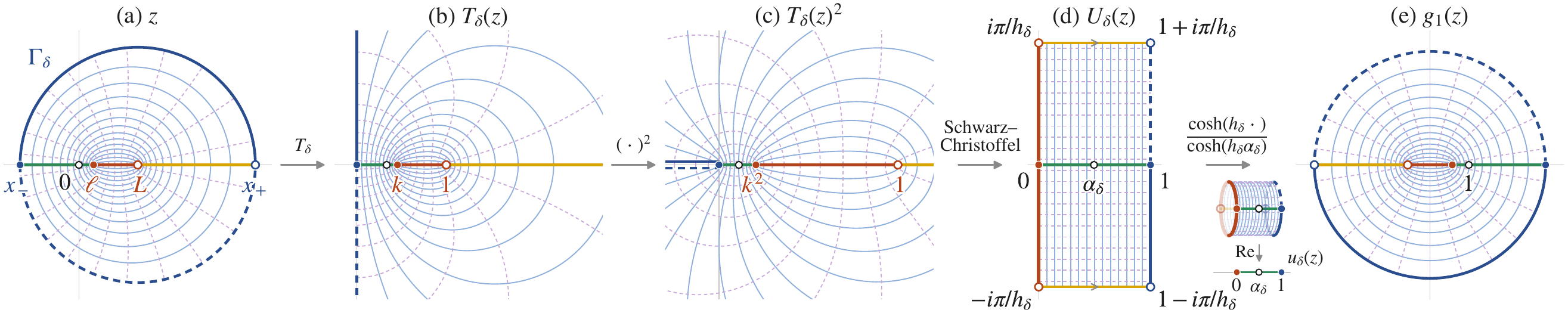}
 \vspace{-20pt}
 \caption{The sequence of maps \eqref{eq:sequence-of-maps-app} in the proof of Lemma~\ref{lem:geometry}. Blue: $\Gamma_\delta$ and its images (upper arc solid, lower arc dashed); red: the slit $[\ell,L]$ and its images; green: the interval $(x_-,\ell)$ and its images; gold: the interval $(L,x_+)$ and its images. The rectangle rolls into the cylinder $\C/((2\pi\ii/h_\delta)\ZZ)$ along the gold edges, and $u_\delta=\operatorname{Re}U_\delta$ is its projection onto the real axis.}
 \label{fig:conformal-app}
 \vspace{-10pt}
\end{figure}

\begin{proof}[Proof of Lemma~\ref{lem:geometry}]
We first construct the normalized coordinate $U_\delta$ by mapping the two
half-disks to rectangles. We then obtain its real part $u_\delta$ and its value at
zero, and prove the extension and contour formula for the filter $g_K$.
Write $x_-=L-R_\delta$ and $x_+=L+R_\delta$.

\proofstep{Part~\textup{(i)}: harmonic potential.}
First consider $\Ddisk_\delta^+=\Ddisk_\delta\cap\{\operatorname{Im}z>0\}$.
By \eqref{eq:mobius-half-disk-app} and \eqref{eq:rectangle-corners-app}, we have
\begin{equation}\label{eq:sequence-of-maps-app}
 \Ddisk_\delta^+
 \xrightarrow{\ T_\delta\ }
 \{t:\operatorname{Re}t>0,\ \operatorname{Im}t>0\}
 \xrightarrow{\ t\mapsto t^2\ }\C^+
 \xrightarrow{\ \Psi_k\ }\mathcal R_k^+ = (0, \Kell) + \ii(0, \Kell_{\mathrm c}).
\end{equation}
Normalize this composition by setting
$
 U_\delta^+(z)=1-{\Psi_k(T_\delta(z)^2)}/{\Kell}.
$
It maps $\Ddisk_\delta^+$ onto the rectangle $(0,1)+\ii(-\pi/h_\delta,0)$ and sends the upper arc and the slit to the two vertical edges of the rectangle respectively. Namely, we have
\begin{equation}\label{eq:U-boundary-app}
      \lim_{\zeta \to z} \operatorname{Re}U^+_\delta(\zeta) = 1, \text{ for } z\in\Gamma_\delta \cap\{\operatorname{Im}z>0\} \quad \text{and} \quad \lim_{\zeta \to z}\operatorname{Re}U^+_\delta(\zeta) = 0, \text{ for } z\in[\ell,L].
\end{equation}
See Figure~\ref{fig:conformal-app}, especially the blue and red lines, for the illustration.

Next define
$U_\delta^-(z)=\overline{U_\delta^+(\bar z)}$ on the lower half-disk $\Ddisk_\delta^-=\Ddisk_\delta\cap\{\operatorname{Im}z<0\}$.
Its image is the reflected rectangle
$(0,1)+\ii(0,\pi/h_\delta)$.
On $(x_-,\ell)$ (see the green lines in Figure~\ref{fig:conformal-app}), the two boundary values of $U_\delta^+$ and $U_\delta^-$ are equal and real.
On $(L,x_+)$ (see the gold lines in Figure~\ref{fig:conformal-app}), they have the same real part and differ by
$2\pi\ii/h_\delta$.
Schwarz reflection \citep[Chapter 4, Theorem 24]{Ahlfors1979} therefore joins them holomorphically as $U_\delta$ in
$\C/((2\pi\ii/h_\delta)\mathbb Z)$.
The two rectangle interiors and these lines cover the {strip $0<\operatorname{Re}v<1$ in} $\C/((2\pi\ii/h_\delta)\mathbb Z)$
exactly once, giving the conformal map in part~\textup{(i)}.


Summarizing the above discussion, we can define the real representative of $U_\delta$ on $(x_-,\ell)$, by $
 U_\delta^+(z)=1-{\Psi_k(T_\delta(z)^2)}/{\Kell}
$ and 
\eqref{eq:SC-rectangle-app}, as
\begin{equation}\label{eq:U-real-app}
 U_\delta(z)
 =1-\frac1{\Kell}\int_0^{T_\delta(z)}
       \frac{\dd t}{\sqrt{(k^2-t^2)(1-t^2)}}
 =1-\frac{\Fell\!\left(\arcsin\frac{T_\delta(z)}k
                    \,\middle|\,k^2\right)}{\Kell}.
\end{equation}
This identifies the construction with \eqref{eq:normalized-U} and its  holomorphic continuation to $\mathcal A_\delta$ aligns with the above definition gluing $U_\delta^+$ and $U_\delta^-$.
Therefore, the real part $\operatorname{Re}U_\delta$ is harmonic and \eqref{eq:U-boundary-app} combined with the reflection implies the boundary conditions in \eqref{eq:potential}. 
The uniqueness of the solution to the Dirichlet problem in \eqref{eq:potential} proves
$u_\delta=\operatorname{Re}U_\delta$ and part~\textup{(i)}.

\proofstep{Part~\textup{(ii)}: complex potential.}
We have $q_\delta(z)=h_\delta U_\delta(z)
\in\C/(2\pi\ii\mathbb Z)$. By \eqref{eq:U-real-app},  $\operatorname{Re}q_\delta=h_\delta u_\delta$ and
$q_\delta(0)=h_\delta\alpha_\delta>0$. 
The reflection construction gives conjugation symmetry. For any  $z \in (\ell,L)$, define $q_{\delta,\pm}(z) = \lim_{\zeta \to z_{\pm}} q_\delta(\zeta)$ where the limit $\zeta \to z_{+}$ is taken from the upper half-plane to the slit and $\zeta \to z_{-}$ is taken from the lower half-plane. This corresponds to the lower and upper sides of the red edge in Figure~\ref{fig:conformal-app}(d). Then we have
\begin{equation}\label{eq:slit-traces-app}
      \operatorname{Re}q_{\delta,\pm}(z)=0, \text{ and } q_{\delta,-}(z)=\overline{q_{\delta,+}(z)}=-q_{\delta,+}(z),\qquad z \in (\ell,L).
\end{equation}
We leave \eqref{eq:outer-collar-app} to the end of the proof.

\proofstep{Part~\textup{(iii)}: value at zero.}
Next evaluate \eqref{eq:U-real-app} at zero. Since $R_\delta>L>\ell>0$, we have $x_-<0<\ell$ and $T_\delta(0) \in (0,k)$.
Thus $U_\delta(0)$ is real and equals $u_\delta(0)$.
Substituting $R_\delta=L\sqrt\gamma(1+\delta)$ and $\ell=L/\kappa$ into \eqref{eq:U-real-app} and comparing to \eqref{eq:alpha-exponent}, we have
 $u_\delta(0) = \alpha(\kappa,\gamma(1+\delta)^2)$ in \eqref{eq:explicit-potential}.
For fixed $\kappa>1$ and the ratio bounds in \eqref{eq:aspect-regime},
the elliptic parameters and real amplitudes remain in compact subsets
of their domains for sufficiently small $\delta$.
Their smooth dependence gives the uniform $O(\delta)$ remainder in
\eqref{eq:u0-alpha}. As $0<u_\delta(0)<1$ and $h_\delta>0$, continuity
then gives \eqref{eq:geometry-compact-bounds}, using
$q_\delta(0)=h_\delta u_\delta(0)$.

\proofstep{Part~\textup{(iv)}: holomorphic filter $g_K$.} As $q_\delta(\bar z)=\overline{q_\delta(z)}$ and on $(L, x_+)$, $q_\delta$ differs by $2\pi \ii$ for the limits from the upper and lower sides (see the two gold lines in Figure~\ref{fig:conformal-app}(d)),  we can define the conformal map $W_\delta$:
\begin{equation}\label{eq:annulus-coordinate-app}
 W_\delta(z)=e^{q_\delta(z)}:
 \mathcal A_\delta\longrightarrow\{1<|W|<e^{h_\delta}\}.
\end{equation}
By \eqref{eq:U-boundary-app},
Schwarz reflection across the circle $\Gamma_\delta$ gives the nonvanishing continuation
\begin{equation}\label{eq:outer-reflection-app}
 W_\delta(z)=\frac{e^{2h_\delta}}
       {\overline{W_\delta(\mathcal R_\delta(z))}},\text{ where the Schwarz reflection is }
 \mathcal R_\delta(z)=L+\frac{R_\delta^2}{\bar z-L},
\end{equation}
for $z$ in a sufficiently small exterior neighborhood of $\Ddisk_\delta$. We can write the filter $g_K$ in \eqref{eq:gK} using \eqref{eq:annulus-coordinate-app} as
\[
 g_K(z) = \frac{\cosh(Kq_\delta(z))}{\cosh(Kq_\delta(0))} =  \frac{W_\delta(z)^K+W_\delta(z)^{-K}}
                   {W_\delta(0)^K+W_\delta(0)^{-K}}.
\]
We therefore define $g_K$ on $\mathcal A_\delta$ and the exterior neighborhood of $\overline{\Ddisk_\delta}$ by \eqref{eq:outer-reflection-app}. What remains is to prove that $g_K$ extends across the slit $[\ell,L]$. Let $W_{\delta,\pm}(z)=e^{q_{\delta,\pm}(z)}$ for $z \in (\ell,L)$. By \eqref{eq:slit-traces-app}, we have $W_{\delta,-}=W_{\delta,+}^{-1}$ but $\cosh(Kq_{\delta,+}(z)) = \cosh(Kq_{\delta,-}(z))$. Therefore, $g_K$ has equal upper and lower limits on $(\ell, L)$, and thus extends continuously across it. That is also why we comment in Remark~\ref{rem:two-constants} that, although it attains equality in \eqref{eq:two-constants}, the exponential function cannot be extended holomorphically across $[\ell,L]$ as required in Theorem~\ref{thm:generic} and why we choose the hyperbolic cosine function in \eqref{eq:gK}. Moreover, as
$
 |g_K(z)|\le2\exp(K(h_\delta-q_\delta(0)))
$ for $z\in\mathcal A_\delta$,
$\{\ell,L\}$ are removable singularities. The removable singularity theorem \citep[Chapter~4, Section~3.1, Theorem~7]{Ahlfors1979} and Morera's theorem \citep[Chapter~4, Section~2.3]{Ahlfors1979}  then make $g_K$ holomorphic on a neighborhood of $\overline{\Ddisk_\delta}$.

\proofstep{{Part~\textup{(iv)}: contour values.}}
Finally, let $\mu(t)=L+R_\delta e^{\ii t}$, $0<t<2\pi$.
Since $T_\delta(\mu(t))=\ii\cot(t/2)$, substitution along the imaginary
axis in \eqref{eq:U-real-app} gives
\[
 y(t)=\int_0^{\cot(t/2)}
       \frac{\dd v}{\sqrt{(k^2+v^2)(1+v^2)}}
 =\Fell\!\left(\arctan\frac{\cot(t/2)}k
                    \,\middle|\,1-k^2\right).
\]
The last equality follows by setting $v=k\tan\varphi$.
Consequently, with the equalities interpreted in their respective
imaginary-period quotients,
\begin{equation}\label{eq:q-outer-app}
 U_\delta(\mu(t))=1-\frac{\ii y(t)}{\Kell},\qquad
 q_\delta(\mu(t))=h_\delta-\frac{\pi\ii y(t)}{\Kell_{\mathrm c}}.
\end{equation}
Substitution into \eqref{eq:gK} proves \eqref{eq:gK-contour}.
As $t$ increases from $0$ to $2\pi$, $y(t)$ decreases from
$\Kell_{\mathrm c}$ to $-\Kell_{\mathrm c}$.
Thus $q_\delta$ gains $2\pi\ii$, and
both endpoint limits of $g_K$ equal
$(-1)^K\cosh(Kh_\delta)/\cosh(Kq_\delta(0))$.

Finally, we prove \eqref{eq:outer-collar-app} in
part~\textup{(ii)} and its uniformity.
As in \eqref{eq:outer-reflection-app}, we extend $q_\delta$ across $\Gamma_\delta$ by the Schwarz reflection
\begin{equation}\label{eq:q-outer-reflection-app}
      q_\delta(z)
      =
      2h_\delta-
      \overline{q_\delta(\mathcal R_\delta(z))}
      \quad\text{in }\C/(2\pi\ii\mathbb Z).
     \end{equation}
Fix any compact parameter set in
$\{(R,L,\ell):R>L>\ell>0\}$.
Since $R_\delta-L+\ell$ has a positive lower bound
on this set, we can choose a common $\varepsilon_*>0$
such that
$
 4\varepsilon_*<R_\delta-L+\ell
$
for all parameters in the set.
In particular, we have
$R_\delta-2\varepsilon_*>L-\ell$ and
$R_\delta+2\varepsilon_*<R_\delta^2/(L-\ell)$.
Thus \eqref{eq:q-outer-reflection-app} extends
$q_\delta$ holomorphically to a neighborhood of
$\mathcal C_{2\varepsilon_*}$.

For $R_\delta\le|z-L|\le R_\delta+2\varepsilon_*$,
the reflected point $\mathcal R_\delta(z)$ lies
inside $\Gamma_\delta$, where
$\operatorname{Re}q_\delta\le h_\delta$. Hence we have
\[
 \operatorname{Re}q_\delta(z)
 =
 2h_\delta-
 \operatorname{Re}q_\delta(\mathcal R_\delta(z))
 \ge h_\delta.
\]
This proves the first assertion in
\eqref{eq:outer-collar-app}.
Together with $0<\operatorname{Re}q_\delta<h_\delta$
inside $\Gamma_\delta$, the reflection formula $\mathcal R_\delta$ in \eqref{eq:outer-reflection-app} gives
$0<\operatorname{Re}q_\delta<2h_\delta$
throughout $\mathcal C_{2\varepsilon_*}$.
Consequently, we have
$
 1\le |e^{q_\delta(z)}|\le e^{2h_\delta},
$ for $z\in\mathcal C_{2\varepsilon_*}$.
For any $z\in\mathcal C_{\varepsilon_*}$, the closed
disk $\{|\zeta-z|\le\varepsilon_*/2\}$ lies in
$\mathcal C_{2\varepsilon_*}$. Cauchy's derivative estimate
\citep[Chapter~4, Section~2.3]{Ahlfors1979}
therefore gives
$|(e^{q_\delta})'(z)|\le 2e^{2h_\delta}/\varepsilon_*$.
Using $(e^{q_\delta})'=e^{q_\delta}q_\delta'$ and
$|e^{q_\delta}|\ge1$, we conclude that
\[
 \sup_{z\in\mathcal C_{\varepsilon_*}}
 |q_\delta'(z)|
 \le \frac{2e^{2h_\delta}}{\varepsilon_*}.
\]
The derivative is independent of the local
representative of $q_\delta$.
Since $h_\delta$ is uniformly bounded above on
the fixed compact parameter set, the last inequality
gives a common constant $C_*$, proving
\eqref{eq:outer-collar-app}.
\end{proof}

\subsection{Optimal Exponent Asymptotics}\label{sec:exponent-asymptotics}

The following proposition formalizes the asymptotic expansion of $\alpha(\kappa,\gamma)$ in Remark~\ref{rem:exponent-asymptotics}.

\begin{proposition}
\label{prop:exponent-asymptotics}
The function $\alpha(\kappa,\gamma)$ is strictly decreasing in $\kappa$
for fixed $\gamma>1$, and strictly decreasing in $\gamma$ for fixed
$\kappa>1$. As $\kappa,\gamma\to\infty$, we have
\begin{align}
 \frac{\sqrt\kappa\,\log(16\gamma)}4\alpha(\kappa,\gamma)
 ={}&1+\frac1\kappa\left(\frac13-\frac2{\log(16\gamma)}\right)
       +\frac1\gamma\left(\frac14+\frac1{2\log(16\gamma)}\right)
       +R(\kappa,\gamma), \text{ with}\label{eq:alpha-joint-second}\\
 |R(\kappa,\gamma)|
 \le{}&C\left(\kappa^{-2}+\frac1{\kappa\gamma}+\gamma^{-2}\right)\notag
\end{align}
for all $\kappa\ge\kappa_0$ and $\gamma\ge\gamma_0$, where
$\kappa_0,\gamma_0>1$. Consequently, we have
\begin{equation}\label{eq:alpha-joint-uniform}
 \left|\frac{\sqrt\kappa\,\log(16\gamma)}4\alpha(\kappa,\gamma)-1\right|
 \le C(\kappa^{-1}+\gamma^{-1}).
\end{equation}
\end{proposition}

\begin{proof}
We first prove monotonicity by comparing the harmonic potentials.
The joint expansion is proved by Taylor expansion and a logarithmic series of the complete elliptic integral of the first kind.

\proofstep{Monotonicity.}
Set $L=1$ and $\ell=\kappa^{-1}$. Similar to \eqref{eq:potential},
let $u_{\kappa,\gamma}$ be harmonic on
$\{z:|z-1|<\sqrt\gamma\}\setminus[\kappa^{-1},1]$,
with boundary values zero on $[\kappa^{-1},1]$ and one on
$|z-1|=\sqrt\gamma$.
Lemma~\ref{lem:geometry} gives
$\alpha(\kappa,\gamma)=u_{\kappa,\gamma}(0)$.
By the strong maximum and minimum principle
\citep[Chapter~4, Theorem~21]{Ahlfors1979},
we have $0<u_{\kappa,\gamma}<1$ in the slit disk.

First, fix $\gamma>1$ and let $\kappa_1<\kappa_2$.
Compare $u_{\kappa_1,\gamma}$ and $u_{\kappa_2,\gamma}$
on the smaller slit disk
$\{z:|z-1|<\sqrt\gamma\}\setminus[\kappa_2^{-1},1]$.
Their difference is harmonic and vanishes on the
outer circle $|z-1|=\sqrt\gamma$ and the common slit
$[\kappa_1^{-1},1]$. On the additional slit
$(\kappa_2^{-1},\kappa_1^{-1})$, the difference is
strictly positive since $u_{\kappa_1,\gamma}>0$
and $u_{\kappa_2,\gamma}=0$.
The strong minimum principle therefore gives
$
 \alpha(\kappa_1,\gamma)
 >\alpha(\kappa_2,\gamma).
$
Next, fix $\kappa>1$ and let $1<\gamma_1<\gamma_2$.
Compare $u_{\kappa,\gamma_1}$ and
$u_{\kappa,\gamma_2}$ on the smaller slit disk
$\{z:|z-1|<\sqrt{\gamma_1}\}
\setminus[\kappa^{-1},1]$.
Their difference is harmonic and vanishes on the
common slit $[\kappa^{-1},1]$.
On the smaller outer circle $|z-1|=\sqrt{\gamma_1}$,
the difference is strictly positive since
$u_{\kappa,\gamma_1}=1$ and $u_{\kappa,\gamma_2}<1$.
Applying the strong minimum principle again yields
$
 \alpha(\kappa,\gamma_1)
 >\alpha(\kappa,\gamma_2).
$

\proofstep{Expansion of $\alpha(\kappa,\gamma)$.}
Write $a=\kappa^{-1}$ and $b=\gamma^{-1}$, so that the notation of \eqref{eq:alpha-exponent} becomes
\[
 \rho=\frac{1-(1-a)\sqrt b}{1+(1-a)\sqrt b},
 \qquad
 \xi=\frac{1-\sqrt b}{1+\sqrt b}.
\]
By the definition of $\Fell(\phi\mid m)$ in  \eqref{eq:Fell}, substituting
$x=\rho\sin t$ gives the algebraic form of the
elliptic integral of the first kind
\citep[Section~22.7]{WhittakerWatson1920}:
\begin{align}
  \nonumber    &\Kell(\rho^2)
      -\Fell(\arcsin(\xi/\rho)\mid\rho^2) 
    =
    \int_{\arcsin(\xi/\rho)}^{\pi/2}
    \frac{\dd t}{\sqrt{1-\rho^2\sin^2 t}} 
    =
    \int_\xi^\rho
    \frac{\dd x}{\sqrt{(\rho^2-x^2)(1-x^2)}}\\
    \quad &=  \sqrt a\bigl(1+(1-a)\sqrt b\bigr) I(a,b), \text{ where } I(a,b)=
    \int_0^1
    \frac{\dd t}
    {\sqrt{(1-at^2)
    (1-(1-a)^2b-at^2)}}, \label{eq:alpha-unit-integral-algebraic}
\end{align}
and the last equality is obtained by substituting
$
 x=\frac{1-at^2-(1-a)\sqrt b}
         {1-at^2+(1-a)\sqrt b}.
$
Landen's transformation \citep[Section~22.42]{WhittakerWatson1920} gives
\[
 \Kell\Bigl(\Bigl(\frac{1-r}{1+r}\Bigr)^2\Bigr)
 =\frac{1+r}{2}\Kell(1-r^2),\qquad 0<r<1.
\]
Apply it with $r=(1-a)\sqrt b$, combine  with \eqref{eq:alpha-unit-integral-algebraic}, and plug into \eqref{eq:alpha-exponent}. We have
\begin{equation}\label{eq:alpha-unit-integral}
 \alpha(\kappa,\gamma) = 1-\frac{\Fell(\arcsin(\xi/\rho)\mid\rho^2)}{\Kell(\rho^2)}
 =\frac{2\sqrt a}{\Kell(1-(1-a)^2b)}I(a,b).
\end{equation}

For $0\le a,b\le1/4$ and $0\le t\le1$,  the integrand of $I(a,b)$ and its second derivatives in $(a,b)$ are
 uniformly bounded on this set. Taylor expansion therefore gives
\begin{equation}\label{eq:alpha-numerator-expansion}
 I(a,b)=1+\frac a3+\frac b2+O(a^2+ab+b^2).
\end{equation}
Following the logarithmic series in
\citet[Chapter~14, Example~21]{WhittakerWatson1920}, we have
\begin{equation}\label{eq:K-log-expansion}
 \Kell(1-\varepsilon)
 =\log\frac4{\sqrt\varepsilon}
  +\frac\varepsilon4\left(\log\frac4{\sqrt\varepsilon}-1\right)
  +O\left(\varepsilon^2\log\frac4{\sqrt\varepsilon}\right),
 \qquad \varepsilon\downarrow0.
\end{equation}
The remainder is uniform for
$0<\varepsilon\le\varepsilon_0<1$ with $\varepsilon_0$ fixed. Set $\varepsilon=(1-a)^2b$.
Since
\[
 \log\frac4{\sqrt\varepsilon}
 =\log(4/\sqrt b)-\log(1-a)=\log(4/\sqrt b)+a+O(a^2),\qquad
 \varepsilon=b+O(ab),
\]
the logarithmic series in \eqref{eq:K-log-expansion} gives
\begin{equation}\label{eq:alpha-denominator-expansion}
 \frac{\Kell(1-(1-a)^2b)}{\log(4/\sqrt b)}
 =1+\frac a{\log(4/\sqrt b)}
   +\frac b4\left(1-\frac1{\log(4/\sqrt b)}\right)
   +O(a^2+ab+b^2).
\end{equation}
The remainder is uniform as $a,b\downarrow0$.
Combining \eqref{eq:alpha-unit-integral},
\eqref{eq:alpha-numerator-expansion}, and
\eqref{eq:alpha-denominator-expansion}, and expanding the reciprocal
of the last denominator, yields
\[
 \frac{\log(4/\sqrt b)}{2\sqrt a}\alpha(\kappa,\gamma)
 =1+a\left(\frac13-\frac1{\log(4/\sqrt b)}\right)
    +b\left(\frac14+\frac1{4\log(4/\sqrt b)}\right)
    +O(a^2+ab+b^2).
\]
Because $2\log(4/\sqrt b)=\log(16\gamma)$, this proves
\eqref{eq:alpha-joint-second} and
\eqref{eq:alpha-joint-uniform}.
\end{proof}

\section{Supporting Results}\label{app:auxiliary}

This section contains the technical results for the proofs in the Appendix.

\begin{lemma}[{\citealp[Theorem~2.1.12]{Muirhead1982}}]
  \label{lem:matrix-volume}
  For $t>0$, integrate over real symmetric positive definite
  $s\times s$ matrices with
  $\dd\mathsf Q=\prod_{1\le i\le j\le s}\dd\mathsf Q_{ij}$.
  Then
\begin{equation}\label{eq:matrix-volume-gamma-integral}
      J_s(t):= \int_{\mathsf Q\succ0}
         e^{-t\tr(\mathsf Q-I_s-\log\mathsf Q)}\dd\mathsf Q
         =e^{ts}t^{-s(t+(s+1)/2)}\pi^{s(s-1)/4}
         \prod_{j=1}^s\Gamma\left(t+\frac{s+2-j}{2}\right).
\end{equation}
  For $t\asymp d$ and $s=O(\log d)$, we have $|\log J_s(t)|
  \lesssim s^2\log d$.
  \end{lemma}
  
  \begin{proof} 
  The identity \eqref{eq:matrix-volume-gamma-integral} is essentially the multivariate gamma integral 
  \citep[Definition~2.1.10 and Theorem~2.1.12]{Muirhead1982},
  with parameter $t+(s+1)/2$ and the substitution
  $\mathsf Y=t\mathsf Q$. 
 We next show the order of $ {\log J_s(t)}$.
  
  Stirling's formula
  \citep[Chapter~12, Section~12.33]{WhittakerWatson1920}
  gives, uniformly for
  $0\le a_j=(s+2-j)/2=O(\log d)$ and $t\asymp d$,
  \[
   \log\Gamma(t+a_j)
   =t(\log t-1)+(a_j-\tfrac12)\log t
   +\tfrac12\log(2\pi)
   +O\Big(\frac{a_j^2+1}{t}\Big).
  \]
  Summing over $j=1,\ldots,s$ yields
  \[
   \begin{aligned}
   \sum_{j=1}^s\log\Gamma(t+a_j)
   &=st(\log t-1)+\frac{s(s+1)}4\log t
   +\frac s2\log(2\pi)+O(s^3/t).
   \end{aligned}
  \]
  Consequently, we have
  \[
   \begin{aligned}
   \log J_s(t)
   &=-\frac{s(s+1)}4\log t
     +\frac{s(s-1)}4\log\pi
     +\frac s2\log(2\pi)+O(s^3/t).
   \end{aligned}
  \]
  Since $t\asymp d$ and $s=O(\log d)$,
  $
   |\log J_s(t)|
   \lesssim s^2\log d.
  $
  \end{proof}

  \begin{lemma}[{\citealp[Theorem~1.1]{DevroyeMehrabianReddad2023}}]\label{lem:lower-gaussian-tv}
      Let $\mathcal V_0,\mathcal V_1\in\R^{n\times n}$ be symmetric
      positive definite matrices. Then
      \begin{equation}\label{eq:lower-gaussian-tv-comparison}
       \operatorname{TV}\bigl(N_n(0,\mathcal V_0),N_n(0,\mathcal V_1)\bigr)
       \le\frac32\min\bigl(1,
         \|\mathcal V_0^{-1/2}(\mathcal V_1-\mathcal V_0)
                        \mathcal V_0^{-1/2}\|_F\bigr).
      \end{equation}
      If $\mathcal V_j\succeq(s_0/2)I_n$ for $j=0,1$ and some $s_0>0$,
      then
      \begin{equation}\label{eq:lower-gaussian-tv-frobenius}
       \operatorname{TV}\bigl(N_n(0,\mathcal V_0),N_n(0,\mathcal V_1)\bigr)
       \le\frac3{s_0}\|\mathcal V_1-\mathcal V_0\|_F.
      \end{equation}
      \end{lemma}

      \begin{proof}
            Equation~\eqref{eq:lower-gaussian-tv-comparison} follows from
            \citet[Theorem~1.1 and Eq.~(2)]{DevroyeMehrabianReddad2023}.
            Under the additional lower bound on the covariances,
            \[
             \|\mathcal V_0^{-1/2}(\mathcal V_1-\mathcal V_0)
                              \mathcal V_0^{-1/2}\|_F
             \le\|\mathcal V_0^{-1}\|_\op\|\mathcal V_1-\mathcal V_0\|_F
             \le\frac2{s_0}\|\mathcal V_1-\mathcal V_0\|_F.
            \]
            Substituting into \eqref{eq:lower-gaussian-tv-comparison} proves
            \eqref{eq:lower-gaussian-tv-frobenius}.
            \end{proof}

\begin{lemma}[{\citealp[Theorem~6.1]{wainwright2019high}}]
\label{lem:gaussian-singular-values}
Let $\mathsf G\in\R^{n\times m}$, $m\ge n$, have iid $N(0,1)$
entries, and let $s_{\min}(\mathsf G)$ and $s_{\max}(\mathsf G)$ denote
its smallest and largest singular values. For every $t\ge0$,
\begin{equation}\label{eq:gaussian-singular-value-tails}
 \begin{aligned}
 \PP\bigl(s_{\max}(\mathsf G)>\sqrt m+\sqrt n+t\bigr)
 \le e^{-t^2/2},
 \PP\bigl(s_{\min}(\mathsf G)<\sqrt m-\sqrt n-t\bigr)
 \le e^{-t^2/2}.
 \end{aligned}
\end{equation}
\end{lemma}

\begin{lemma}\label{lem:normality-noise-form-moments}
  {Suppose $\sigma\ge\underline\sigma>0$ for a fixed $\underline\sigma$.}
  Given $a\in\C^n$ and
  a complex symmetric $B\in\C^{n\times n}$,
  for each fixed $r\ge2$, we have
  \begin{equation}\label{eq:normality-noise-form-moments}
   \|a^\top\varepsilon\|_{L^r}
   \le C_r\sigma\|a\|_2,\qquad
   \|\varepsilon^\top B\varepsilon-\sigma^2\tr B\|_{L^r}
   \le C_r\sigma^2\|B\|_F.
  \end{equation}
\end{lemma}
\begin{proof}
  {The model moment bounds imply $\|\varepsilon_i/\sigma\|_{L^r}\le C_r$ for every fixed $r\ge2$, uniformly in $n$.}
  Applying Rosenthal's inequality \citep[Theorem~9]{Boucheron2005} to the real and imaginary
  parts gives
  \[
   \|a^\top\varepsilon\|_{L^r}
   \le C_r\sigma\biggl(
         \Bigl(\sum_{i=1}^n|a_i|^2\Bigr)^{1/2}
         +\Bigl(\sum_{i=1}^n|a_i|^r\Bigr)^{1/r}\biggr)
   \le C_r\sigma\|a\|_2.
  \]
  Since $\|\varepsilon_i^2-\sigma^2\|_{L^r}
  \le C_r\sigma^2$, the same inequality gives
  \[
   \biggl\|\sum_{i=1}^n B_{ii}(\varepsilon_i^2-\sigma^2)
         \biggr\|_{L^r}
   \le C_r\sigma^2\Bigl(\sum_{i=1}^n|B_{ii}|^2\Bigr)^{1/2}.
  \]
  For the off-diagonal sum, let $\epsilon_i$ be independent Rademacher variables and put $I=\{i:\epsilon_i=1\}$. Then
  $
   \sum_{i\ne j}B_{ij}\varepsilon_i\varepsilon_j
   =4\E_\epsilon\sum_{i\in I,\,j\notin I}
         B_{ij}\varepsilon_i\varepsilon_j.
$
  Conditioning first on $\{\varepsilon_j\}_{j\notin I}$ and using Minkowski's inequality, we have
  \begin{align*}
   \biggl\|\sum_{i\ne j}B_{ij}\varepsilon_i\varepsilon_j
         \biggr\|_{L^r}
   &\le4\E_\epsilon
         \biggl\|\sum_{i\in I,\,j\notin I}
            B_{ij}\varepsilon_i\varepsilon_j\biggr\|_{L^r}\le C_r\sigma\E_\epsilon
         \biggl(\sum_{i\in I}
            \Bigl\|\sum_{j\notin I}B_{ij}\varepsilon_j
               \Bigr\|_{L^r}^2\biggr)^{1/2}\\
   &\le C_r\sigma^2\E_\epsilon
         \Bigl(\sum_{i\in I,\,j\notin I}|B_{ij}|^2\Bigr)^{1/2}
   \le C_r\sigma^2\|B\|_F.
  \end{align*}
  Combining the diagonal and off-diagonal bounds proves
  \eqref{eq:normality-noise-form-moments}.
  \end{proof}

\end{document}